\RequirePackage{tikz}
 \documentclass[11pt]{book}
 
 \usepackage{diss_book}
 
\usepackage[english]{babel}
\usepackage{todonotes}
\usepackage[utf8]{inputenc}
\usepackage[T1]{fontenc}
\usepackage{geometry}
\usepackage{microtype}
\usepackage{authblk}
\usepackage{amsmath,amssymb}
\usepackage{subcaption}
\usepackage{amsthm}
\usepackage{bbm}
\usepackage{enumitem}
\setenumerate{label=\alph*)}
\usepackage{bm}
\usepackage{url}
\usepackage{multirow}
\usepackage{setspace}
\usepackage{mathtools} 
\usepackage{graphicx}
\usepackage{xcolor}
\usepackage{xspace}
\newcommand{\MATLAB}{\textsc{Matlab}\xspace}
\usepackage{booktabs}
\usepackage{csquotes}
\usepackage{soul}
\usepackage{environ} 
\usepackage{hyperref}
\usepackage{trimclip}
\hypersetup{
    colorlinks,
    linkcolor={red!50!black},
    citecolor={blue!50!black},
    urlcolor={blue!80!black}
}
\usepackage{tabularx}

\theoremstyle{definition}
\newtheorem{thm}{Theorem}[chapter]
\newtheorem{defn}[thm]{Definition}
\newtheorem{lem}[thm]{Lemma}
\newtheorem{prop}[thm]{Proposition}
\newtheorem{rem}[thm]{Remark}
\newtheorem{cor}[thm]{Corollary}
\newtheorem{exmp}[thm]{Example}

\renewcommand{\Vec}[1]{\renewcommand*{\arraystretch}{1.2}\begin{pmatrix*}[r]#1\end{pmatrix*}}
\renewcommand{\vec}[1]{\begin{psmallmatrix}#1\end{psmallmatrix}}

\renewcommand{\b}[1]{\mathbf #1}
\newcommand{\bA}{\b \Lambda}
\newcommand{\bby}{\bar{\b y}}
\def\code#1{\texttt{#1}}
\newcommand{\ie}{i.\,e.\ }

\newcommand{\dt}{\Delta t}
\newcommand{\ii}{\mathrm{i}}
\newcommand{\from}{\colon}

\newcommand{\bxi}{\mathbf \xi}

\newcommand{\sgn}{\text{sgn}}
\renewcommand{\O}{\mathcal O}

\newcommand{\R}{\mathbb R}
\newcommand{\C}{\mathbb C}
\newcommand{\N}{\mathbb{N}}
\newcommand{\dd}{\mathrm{d}}
\newcommand{\Cminus}{\overline{\C^-}}

\newcommand{\tm}{\subseteq}
\newcommand{\abs}[1]{\lvert #1\rvert}
\newcommand{\norm}[1]{\lVert #1\rVert}

\newcommand{\qta}{\quad\text{ and }\quad}

\newcommand{\tend}{t_\mathrm{end}}
\newcommand{\ns}{\mkern-5mu}

\DeclareMathOperator{\diag}{diag}
\DeclareMathOperator{\tr}{trace}
\DeclareMathOperator{\NB}{NB}
\DeclareMathOperator{\re}{Re}
\DeclareMathOperator{\im}{Im}
\DeclareMathOperator{\Span}{span}

\usepackage{forest}
\forestset{
	declare toks={elo}{}, 
	anchors/.style={grow'=90, anchor=#1,child anchor=#1,parent anchor=#1},
	dot/.style={tikz+={\fill (.child anchor) circle[radius=#1];}},
	dot/.default=2pt,
	decision edge label/.style n args=3{
		edge label/.expanded={node[midway,auto=#1,anchor=#2,\forestoption{elo}]{\strut$\unexpanded{#3}$}}
	},
	decision/.style={if n=1
		{decision edge label={left}{east}{#1}}
		{decision edge label={right}\label{key}{west}{#1}}
	},
	decision tree/.style={
		for tree={grow'=90,
			s sep=2.5pt, l=0pt, l sep =0.5pt, outer sep =-1.5pt,
			if n children=0{anchors=west}{
				if n=1{anchors=west}{anchors=west}},
			math content,
		},
		anchors=west, outer sep=-1.5pt,
		dot=2pt, for descendants=dot,
		delay={for descendants={split option={content}{;}{content,decision}}},
	},
	rooted tree/.style={
		for tree={
			grow'=90,
			parent anchor=center,
			child anchor=center,
			s sep=2.5pt,
			l sep =1pt,
			if level=0{
				baseline
			}{},
			delay={
				if content={*}{
					content=,
					append={[]}
				}{}
			}
		},
		before typesetting nodes={
			for tree={
				circle,
				fill,
				minimum width=3pt,
				inner sep=0pt,
				child anchor=center,
			},
		},
		before computing xy={
			for tree={
				l=5pt,
			}
		}
	}
}

\DeclareDocumentCommand\ct{o}{\Forest{decision tree [#1]}}
\DeclareDocumentCommand\rt{o}{\Forest{rooted tree [#1]}}
\newcommand{\et}[2]{\resizebox{#1 cm}{!}{\marginbox*{0pt -0.45\height pt 0pt 0pt}{#2}}}
\renewcommand{\tt}[2]{\resizebox{#1 cm}{!}{\marginbox*{0pt -0.25\height pt 0pt 0pt}{#2}}}

\newcommand{\byi}{\b y^{(i)}}
\newcommand{\byj}{\b y^{(j)}}

\newcommand{\yi}{y^{(i)}}

\newcommand{\D}{\b D}

\newcommand{\Fnu}{\mathbf{f}^{[\nu]}}
\newcommand{\Fmu}{\mathbf{f}^{[\mu]}}
  
  \newcommand{\eq}{\text{EQ}}
  \newcommand{\gl}{\text{GL}}

\newcommand{\fnu}{\mathbf{f}^{[\nu]}}
\newcommand{\fmu}{\mathbf{f}^{[\mu]}}

\newcommand{\dF}{\mathcal{F}}

\newcommand{\tn}{t_n}

\usepackage{tikz}
\usepackage{tkz-euclide} 
\usetikzlibrary{patterns}

\definecolor{colorA}{rgb}{0,0.447,0.741}
\definecolor{colorB}{rgb}{0.85,0.325,0.098}
\definecolor{colorE}{rgb}{0.929,0.694,0.125}
\definecolor{colorF}{rgb}{0.494,0.184,0.556}
\definecolor{colorD}{rgb}{0.466,0.674,0.188}
\definecolor{colorC}{rgb}{0.301,0.745,0.933}
\definecolor{colorG}{rgb}{0.635,0.078,0.184}

\newlength{\tickl}    
\tikzset{axes/.style={thick,-latex}}
\tikzset{lineplot/.style={thick}}
\tikzset{arrow/.style={thick,-latex}} 
\tikzset{thick arrow/.style={ultra thick,-latex}} 
\tikzset{grid lines/.style={very thin,gray!30}}	
\tikzset{point/.style={radius=2pt}}
\tikzset{help line/.style={black,thin,dashed}} 

\makeatletter
\newsavebox{\measure@tikzpicture}
\NewEnviron{scaletikzpicturetowidth}[1]{%
	\def\tikz@width{#1}%
	\def\tikzscale{1}\begin{lrbox}{\measure@tikzpicture}%
		\BODY
	\end{lrbox}%
	\pgfmathparse{#1/\wd\measure@tikzpicture}%
	\edef\tikzscale{\pgfmathresult}%
	\BODY
}
\makeatother
\usetikzlibrary{decorations.markings}
\usepackage{pgfplots}

\tikzset{mylabel/.style  args={at #1 #2  with #3}{
		postaction={decorate,
			decoration={
				markings,
				mark= at position #1
				with  \node [#2] {#3};
} } } }

\begin{document}
	\pagenumbering{gobble}

\onecolumn

\begin{center}
	~
	\bigskip
	
	
	\bigskip  
	
	\bigskip

	\bigskip
	\textbf{\textsc{\fontsize{20}{1.5\baselineskip}\selectfont A Unifying Theory for Runge--Kutta-like Time Integrators:\\[0.3\baselineskip] Convergence and Stability}} 
	\bigskip  
	\bigskip
	\bigskip
	\bigskip
	\bigskip
	
	\text{\fontsize{14}{\baselineskip}\selectfont By}
	
	\bigskip
	\textbf{\textsc{\fontsize{14}{\baselineskip}\selectfont Thomas Izgin}}
	
	
	\bigskip
	\bigskip
	\bigskip
	\bigskip
	\bigskip
	
	\text{\fontsize{14}{\baselineskip}\selectfont A thesis submitted in partial fulfillment for the degree of }\\[0.2\baselineskip]
	\text{\fontsize{14}{\baselineskip}\selectfont Doktor der Naturwissenschaften (Dr. rer. nat.)}
	
	\bigskip
	\bigskip
	\text{\fontsize{14}{\baselineskip}\selectfont in the}
	\bigskip
	\bigskip
	
	\bigskip
	\text{\fontsize{14}{\baselineskip}\selectfont Faculty of Mathematics and Natural Sciences}
	
	\bigskip
	\text{\fontsize{14}{\baselineskip}\selectfont University of Kassel}
	\bigskip

	\bigskip
	\text{\fontsize{12}{\baselineskip}\selectfont Date of Submission: November 10, 2023}\\[0.2\baselineskip]
	\text{\fontsize{12}{\baselineskip}\selectfont Disputation Date: February 2, 2024}
	\vfill
	\text{\fontsize{14}{\baselineskip}\selectfont First Reviewer:} \textbf{\fontsize{14}{\baselineskip}\selectfont Prof.\ Dr.\ Andreas Meister}\\[\baselineskip]
	\text{\fontsize{14}{\baselineskip}\selectfont Second Reviewer:} \textbf{\fontsize{14}{\baselineskip}\selectfont Prof.\ Dr.\ Chi-Wang Shu}\\[\baselineskip]
	
	\bigskip
	\bigskip
	
	Kassel,  February 14, 2024
	
\end{center}

\clearpage

{\pagestyle{plain}
	\hypersetup{linkcolor=black}
	\tableofcontents
}

\addcontentsline{toc}{chapter}{Acknowledgements}

\chapter*{Acknowledgements}
This work was done during my doctoral studies at the University of Kassel and was largely funded by the German Research Foundation (DFG, project number 466355003), for which I would like to express my gratitude at this point.

First and foremost, I would like to thank my doctoral supervisor Prof.\ Dr.\ Andreas Meister for his excellent supervision and extraordinary commitment to my advancement. He has not only enabled me to further my education but also to travel to many international conferences and scientists, which have supported me to an incomparable extent. With individual support and a multitude of scientific discussions, he has influenced the present work in many ways. 

I would also like to thank Dr.\ Stefan Kopecz very much for his constructive criticism, which significantly helped me to present research results in a more structured and transparent way. In this context, I would also like to thank him very much for the many scientific discussions, each of which had great value for me.

At this point, I would also like to express my gratitude to all my colleagues from Department 10, who made me feel very welcome. In particular, I would like to single out Veronika Straub, Stephanie Thomas, Stefan Dingel, and Andreas Linß, who became close friends during my doctoral studies and with whom I spent many hours in fruitful discussions.

However, my deepest gratitude also goes to the international scientists who have been invaluable to my research questions. At this point, I would especially like to thank Prof.\ Dr.\ Chi-Wang Shu from Brown University, Prof.\ Dr.\ Juntao Huang from Texas Tech University, and Prof.\ Dr.\ David I.\ Ketcheson from King Abdullah University of Science and Technology (KAUST). I would also like to thank Dr.\ Philipp Öffner for many good advices and great cooperation.

Finally, I would like to thank my wife Daniela from the bottom of my heart, who lovingly accompanied me through every phase of my PhD time.
\mainmatter
\chapter{Introduction}
Many realistic phenomena in the natural sciences, epidemiology and ecology are  modeled  by systems of differential equations that are constrained by restrictions linked to the nature of the problem \cite{lacitignola2021using, colonna2016plasma,  kooijman2000dynamic}.   Solving these equations analytically is not possible in general, necessitating the use of numerical methods to approximate the solution. However, given the model assumptions and the presence of measurement errors, an exact representation of reality cannot be expected anyway. Rather, the goal of the numerical approximation is to retain all properties of the underlying process while achieving approximations within the limits of measurement accuracy.
Two important examples of physical properties are the conservation of quantities and the positivity of certain solution components. For instance in the context of chemical reactions such as the stratospheric reaction problem \cite{sandu2001positive} or the Robertson problem \cite{HNWII}, the total mass is conserved and the modeled densities are non-negative. 

Often, the underlying process to be modeled consists of converting one quantity into the other, which can be represented in a more abstract framework using a special system of ordinary differential equations (ODEs), a so-called \emph{conservative production-destruction system} (PDS). Conservativity in this context means that the production of one quantity is equivalent to the destruction of another, and vice versa. As a result of conservativity, the sum of constituents remains constant in time. A numerical method that mimics this behavior on a discrete level for every chosen time step size $\dt>0$ is called \emph{unconditionally conservative}. Similarly, if the method produces positive approximations for any $\dt>0$ whenever the initial value is positive, the scheme is called \emph{unconditionally positive}.
In many cases additional terms exist that have no counterpart. In such a situation, the corresponding non-conservative PDS may be understood as the sum of a conservative PDS and rest terms. Hence, a non-conservative PDS can always be interpreted as a so-called \emph{production-destruction-rest system} (PDRS) with a conservative PDS part. 
%

Besides the scientifically induced requirement of preserving specific solution properties such as conservativity and positivity, the preservation of these two particular properties also hold significant importance from a purely numerical perspective. First, a numerical method that does not preserve all linear invariants such as conservativity may produce a qualitatively wrong behavior \cite{shampine1986conservation, BDM03, lacitignola2021using}. Second, the preservation of positivity is a desirable property because negative approximations can lead to the failure of the method, see for instance \cite{STKB} and the literature mentioned therein. Preserving the positivity of certain solution components is also crucial in the context of partial differential equations (PDEs). For instance, the calculation of the speed of sound when solving the compressible Euler equations requires the positivity of pressure and density. Another system of PDEs that emphasizes the importance of generating positive approximations is given in \cite{KMpos}, where the right-hand sides of the so-called NPZD model (\emph nutrients, \emph phytoplankton, \emph zooplankton, and \emph detritus) \cite{BDM2005} were used as stiff source terms. In the numerical solution of the resulting PDE, the occurrence of negative approximations can lead to the divergence of the method and therefore necessitates a severe time step constraint for methods that are not unconditionally positive, see \cite{KMpos}.

While high order general linear methods \cite{HNWII,J09} such as Runge--Kutta and linear multistep schemes \cite{B16, HNW1993, HNWII} preserve all linear invariants of the system, unconditional positivity is much harder to obtain. Among the class of linear integrators, unconditional positivity is restricted to first order \cite{Sandu02,BC78}. The implicit Euler method indeed grants the positivity, although methods for solving nonlinear systems coming from implicit schemes do not guarantee positive approximations. 
Higher order linear methods can only guarantee positivity by restricting the time step size, leading to
a significant increase in computational time \cite{Sandu02,bertolazzi1996positive}.

Positive and linear invariants preserving  schemes based on projection techniques were proposed  in \cite{sandu2001positive, nusslein2021positivity}, where at each time step, the negative approximations or the weights of the Runge--Kutta method are changed to guarantee positivity while maintaining the order of the method.  More recently, the issue of positivity preservation was addressed in \cite{blanes2021positivity}, where splitting and exponential methods were combined to construct positive and conservative integrators up to 3rd order for solving nonlinear mass conservative systems of the type $\b y'(t)=\b A(\b y(t),t)\b y(t)$, where $\b A(\b y(\cdot),\cdot)$ is an $N\times N$ matrix-valued function. 

Another approach for preserving positivity is to apply the Patankar-trick \cite{Patankar1980} to an RK method resulting in  a Patankar--Runge--Kutta (PRK) scheme, which guarantees the unconditional positivity of the numerical approximation. However, the PRK method in general does not preserve linear invariants such as conservativity anymore. Still, PRK methods are of interest due to their unconditional positivity.  Furthermore, 
in the context of conservative production-destruction systems, it is possible to improve the PRK method obtaining modified PRK (MPRK) schemes, originally introduced in \cite{BDM03}, which additionally are unconditionally conservative. Second and third order MPRK schemes have been developed and numerically investigated in \cite{KM18,KM18Order3,MPRK3ex}. The idea was then carried out in the context of strong-stability preserving (SSP) Runge--Kutta methods in \cite{SSPMPRK2,SSPMPRK3}, where the resulting SSPMPRK schemes have been applied to solve reactive Euler equations. In \cite{MPDeC}, the authors used the Patankar-trick to develop MPDeC methods, which are modified Patankar (MP) schemes of arbitrary order based on deferred correction methods (DeC). It is worth mentioning that the 5th order MPDeC method was used to preserve a positive water height when solving the shallow water equations \cite{CMOT21}.  Furthermore, an implicit first order MP scheme based on a 3rd order SDIRK method was presented in \cite{MeisterOrtleb2014} and applied to the shallow water equations to guarantee a positive water height. Thereby, it was also proven that the method is of third order away from the wet-dry transition zone.  All these schemes are mass conservative and unconditionally positive. Moreover their efficiency and robustness was proven numerically while integrating stiff PDS.

Among the positive and linear invariants preserving integrators for biochemical systems, 1st and 2nd order  generalized BBKS (gBBKS), which were developed in \cite{BBKS2007,BRBM2008,gBBKS} and named after the authors Bruggeman, Burchard, Kooi, and Sommeijer, and Geometric Conservative (GeCo) schemes \cite{martiradonna2020geco} have been introduced in recent literature. 
These methods fall in the class of  non-standard integrators \cite{mickens1994nonstandard}, as they result as  non-standard versions of explicit first and second order Runge--Kutta schemes, where the advancement in time is modulated by a nonlinear functional dependency on the temporal step size and on the  approximation itself. 
The step size modification thereby guarantees the numerical solution to be unconditionally positive while keeping the accuracy of the underlying method. 
While GeCo schemes are explicit integrators, the gBBKS step size modification function leads to an implicit scheme. 
Nevertheless, nonlinear implicit equations that arise from gBBKS schemes may be reduced to a  scalar nonlinear equation in one single unknown \cite{gBBKS}.

We want to emphasize that the application of the modified Patankar approach on an RK scheme has a great impact on its structure. Indeed, the resulting MPRK scheme is not an RK method anymore. Even more, MPRK schemes do not belong to the class of general linear methods. Therefore, the excessive theory for RK schemes cannot be applied directly to deduce the properties of MPRK methods. As a result, the first constructions of 2nd and 3rd order MPRK schemes in \cite{KM18,KM18Order3} were interlinked with technical proofs using Taylor series expansions. Moreover, due to the nonlinear nature of Patankar-type methods, also a stability analysis for these schemes is not straightforward, yet of high importance.

The first part of my thesis is concerned with developing a comprehensive theory for deriving order conditions of Patankar-type methods. To that end, we generalized the theory of NB-series \cite{Ntrees} by interpreting Patankar-type methods as Runge--Kutta-like schemes with solution-dependent Butcher tableau, which we referred to as non-standard additive Runge--Kutta (NSARK) methods in \cite{NSARK}. Thereby, the main idea was to revisit Butchers approach from \cite{B16} concerning order conditions for RK schemes and apply his techniques to the results for additive Runge--Kutta methods \cite{Ntrees}. Furthermore, we adapted Butcher's proofs in such a way that they remain valid even for the case of solution-dependent Butcher tableaux. In particular, we provided a theorem for arbitrary high order NSARK methods. However, these order conditions may be implicit or not fully reduced. Nevertheless, we were able to trace the reduction of order conditions back to the investigation of polynomial systems of equations, which we were able to solve using the Gröbner basis theory from commutative algebra. We applied this approach deriving the known order conditions for GeCo and MPRK methods from \cite{martiradonna2020geco,KM18,KM18Order3} in a compact manner. Moreover, within the same work \cite{NSARK}, we derived for the first time explicit conditions for 3rd and 4th order GeCo methods as well as 4th order MPRK schemes. 

Even though the first MPRK schemes were introduced about two decades ago in \cite{BDM03} and followed by many further works on positivity-preserving methods, the corresponding theories for a stability analysis and deriving order conditions were first developed in my PhD project. 
In particular, I present in this work a unifying theory for the analysis of Patankar-type schemes concerning their stability and convergence. To that end, we review and extend the corresponding results that were already published during my PhD time. 

A first step in my approach of investigating the stability of MPRK schemes was the observation that the scalar Dahlquist equation $y'(t)=\lambda y(t)$ with $\lambda\in \C^-$ could not be used for the analysis. The reason for that is the fact that MP schemes are applied to real valued systems of equations. One is thus tempted to consider the decoupled PDS
\[\vec{y_1'(t)\\y_2'(t)}=\vec{\lambda y_1(t)\\-\lambda y_1(t)},\quad  \lambda\in \R^-,\] whose first component represents the Dahlquist equation with $\lambda\in \R^-$. However, it turned out that the analysis of this equation is not even sufficient to understand the stability behavior in a more general system with two equations \cite{IKM2122, izgin2022stability}, let alone larger systems. Instead, the main idea was to use the theory of center manifolds for maps from dynamical systems \cite{carr1982,SH98,mccracken1976hopf} to analyze the behavior of the numerical method near steady states when applied to general linear autonomous problems. This approach was first carried out for systems of two equations \cite{IKM2122} and later generalized to arbitrary large linear PDS \cite{izgin2022stability}, already analyzing a second order family of MPRK schemes. The very first stability analysis of further Patankar-type methods followed shortly, which resulted in several publications \cite{IOE22StabMP,gecostab,HIKMS22} during my PhD time. We also want to note that the theory is not limited to linear problems, but can also be applied in the context of certain nonlinear PDS \cite{IKMnonlin22}. Furthermore, we derived a necessary condition for avoiding unrealistic oscillations in \cite{ITOE22}, underlining the numerical results from \cite{IssuesMPRK}, where different modified Patankar methods from \cite{MPDeC,KM18,KM18Order3} were analyzed with respect to oscillatory behavior. Also, recently we investigated the hypothesis that the stability properties may be of global nature when the MPRK scheme is based on a non-negative Butcher tableau \cite{izgin2023nonglobal}, which is mostly based on the master thesis \cite{Schilling2023}.

Altogether, this thesis represents a collection of my work as first author with several collaborators on the stability and convergence of nonlinear time stepping methods. Additionally, I unify in this framework the stability analysis for the above mentioned MPRK schemes by deriving a stability function for NSARK methods. Moreover, we also investigate RK schemes generalizing the notion of $A$-stability. 

The remainder of the thesis is divided into six chapters and an appendix. 

We first review the theoretical fundamentals in Chapter~\ref{chap:Fundamentals}. In particular, RK and additive RK (ARK) methods are introduced. Additionally, we recall the main theorems concerning their stability and order of convergence. Furthermore, we introduce the 
notation for the production-destruction-rest systems together with the main properties of interest.

In Chapter~\ref{chap:NumSchemes} we present the previously mentioned Patankar-type schemes and write them as Runge--Kutta-like methods with solution-dependent Butcher tableau.  

In the following Chapter~\ref{chap:order} we then turn to order conditions for Patankar-type methods giving a unifying and comprehensive theory based on the order conditions for ARK methods. In particular, we investigate GeCo and MPRK reproducing the known order conditions in a compact manner. Furthermore, we give explicit formulations for the conditions of 3rd and 4th order GeCo and 4th order MPRK methods. We also construct a 4th order MPRK method and confirm its order of convergence numerically.

In Chapter~\ref{chap:stab} we present the stability theory based on the center manifold theorem for maps and investigate Patankar-type methods as well as Runge--Kutta schemes. We also provide necessary conditions for non-oscillatory schemes and validate the theoretical results with numerical experiments.

Finally, we come to a conclusion in Chapter~\ref{chap:conclusion}, where we also discuss open questions for future work.

\chapter{Theoretical Fundamentals}\label{chap:Fundamentals}
\section{Runge--Kutta Methods}
Runge--Kutta (RK) methods are numerical schemes to approximate the solution  $\b y\colon [t_0,\tend]\to \R^N$ of the initial value problem (IVP)
\begin{equation}
	\b y'(t) = \b f(\b y(t),t),\quad \b y(t_0) =\b y^0. \label{eq:initprob}	
\end{equation}
Hereafter, we use superscript indices for vectors to better distinguish between iterates of a numerical method and their respective components.
For the sake of simplicity, let us consider a fixed time step size $\dt$ and set $t_{n}=t_0+n\dt$ for $n=1,\dotsc,k$ so that $t_n\in [t_0,\tend]$. A time integrator such as a Runge--Kutta method aims to generate approximations $\b y^n$ to $\b y(t_n)$. In the case of RK schemes, intermediate times \[\xi_j=t_n +c_j\dt, \quad c_j\in [0,1],\quad j=1,\dotsc,s\]
are introduced and a quadrature formula is used to obtain
\[\b y(t_{n+1})-\b y(t_n)=\int_{t_n}^{t_{n+1}}\b y'(t)\dd t=\int_{t_n}^{t_{n+1}}\b f(\b y(t),t)\dd t\approx\dt\sum_{j=1}^sb_j\b f(\b y(\xi_j),\xi_j), \]
where $b_j$ depend on the particular quadrature formula, and $\sum_{j=1}^sb_j=1$ holds true if an interpolatory quadrature formula is used. Since the value of $\b y$ at the intermediate times $\xi_i$, $i=1,\dotsc,s$, is not known in general, we approximate them in a similar manner, \ie
\[\b y(\xi_i)-\b y(t_n)=\int_{t_n}^{\xi_i}\b y'(t)\dd t=\int_{t_n}^{t_n+c_i\dt} \b f(\b y(t),t)\dd t\approx\dt\sum_{j=1}^sa_{ij}\b f(\b y(\xi_j),\xi_j), \]
where $a_{ij}$ again depend on the chosen quadrature rule and $\sum_{\nu=1}^sa_{ij}=c_i$ holds true for interpolatory quadrature formulae. 

Now, denoting the approximation to $\b y(\xi_i)$ by $\byi$, the corresponding $s$-stage Runge--Kutta method for the solution of the IVP \eqref{eq:initprob} is given by
\begin{subequations}\label{eq:rk}
	\begin{align}
		\byi &= \b y^n + \dt\sum_{j=1}^{s} a_{ij} \b f(\byj,t_n+c_j \dt),\quad i=1,\dots,s,\label{eq:stageRK}\\
		\b y^{n+1}&=\b y^n+\dt \sum_{j=1}^s b_j 
		\b f(\byj,t_n+c_j\dt).
	\end{align}
\end{subequations}
It is worth mentioning that a Runge--Kutta method is characterized by its coefficients $a_{ij}$, $b_j$, $c_j$ for $i,j=1,\dots,s$ and can be represented 
by the \emph{Butcher tableau}
\[
\begin{array}{c|c}
	\b c        &    \b A\\\hline
	& \b b
\end{array}
\]
with $\b A = (a_{ij})_{i,j=1,\dots,s}$, $\b b=(b_1,\dots,b_s)$ and $\b c=(c_1,\dots,c_s)^T$. If $\b A$ is a strict lower left triangular matrix, the \emph{stage vectors} $\byi$ can be computed explicitly using \eqref{eq:stageRK}, which is why the corresponding RK method is called \emph{explicit}. Otherwise, the scheme is called \emph{implicit}. If $\b f$ is nonlinear and the RK scheme is implicit, the stage vectors $\byi$ are the solution to a nonlinear system of equations. Nevertheless, the existence of a unique solution can be guaranteed under some time step constrains for Lipschitz continuous (with respect to $\b y$) right-hand sides $\b f$ \cite[Theorem~7.2]{HNW1993}. 
\begin{rem}[{\cite[Section II.2]{HNW1993},\cite{DB2002}}]\label{rem:RKautonom}
	Given the non-autonomous IVP \eqref{eq:initprob}, one may rather consider solving the corresponding autonomous problem 
	\begin{equation}\label{eq:initprob_autonom}
		\b Y'(t)=\b F(\b Y(t))
	\end{equation} with $\b Y(t)=\vec{\b y(t)\\ t}$ and $\b F(\b Y(t))=\vec{\b f(\b Y(t))\\ 1}$.  If the stage vectors are uniquely determined and \begin{equation}\label{eq:sum_aij}
		\sum_{j=1}^sa_{ij}=c_i
	\end{equation}holds, then the approximations of an RK method to the solution $\b y$ of \eqref{eq:initprob} are identical regardless of whether the method was applied to \eqref{eq:initprob} or \eqref{eq:initprob_autonom}.
\end{rem}
\subsection{Additive Runge--Kutta Methods}
A generalization of Runge--Kutta methods are \emph{additive Runge--Kutta} (ARK) schemes, which approximate the solution of the initial value problem, where the right-hand side is split into a sum, that is
\begin{align}\label{ivp}
	\b y'(t)  =\b f(\b y(t),t)= \sum_{\substack{\nu=1}}^N  \fnu(\b y(t),t), \quad  \b y(t_0) & = \b y^0\in \R^d.
\end{align}

The main idea of an ARK method is to apply very
different RK schemes determined by $\b A^{[\nu]}, \b b^{[\nu]}, \b c^{[\nu]}$ to the different addends $\fnu$. A popular class of ARK schemes are Implicit-Explicit (IMEX) RK methods \cite{Crouzeix1980,ARS1997}. For internal consistency, we require that the different RK schemes actually do not differ in $\b c$, \ie 
\begin{equation}
	c_i=c_i^{[\nu]}=\sum_{j=1}^sa_{ij}^{[\nu]}\label{eq:sum_aij_ARK}
\end{equation}
for $i=1,\dotsc,s$ and $\nu=1,\dotsc, N$, see \cite{SG2015}. For standard RK methods this reduces to \eqref{eq:sum_aij}. 
The resulting ARK method reads
\begin{equation}\label{eq:ark}
	\begin{aligned}
		\byi & = \b y^n + \dt \sum_{j=1}^s  \sum_{\substack{\nu=1}}^N a^{[\nu]}_{ij}\fnu(\byj, t_n + c_j\dt), \quad i=1,\dotsc,s,\\
		\b y^{n+1} & = \b y^n + \dt \sum_{j=1}^s \sum_{\substack{\nu=1}}^N  b^{[\nu]}_j\fnu(\byj, t_n + c_j\dt),
	\end{aligned}
\end{equation} 
and the corresponding extended Butcher tableau is given by
\[\arraycolsep=1.4pt\def\arraystretch{1.5}
\begin{array}{c|c|c|c|c}
	\b c &	\b A^{[1]}        &    \b A^{[2]} & \cdots & \b A^{[N]}\\ \hline
	&	\b b^{[1]}        &    \b b^{[2]} & \cdots & \b b^{[N]}
\end{array},
\]where $\b A^{[\nu]}=(a_{ij}^{[\nu]})_{i,j=1,\dotsc,s}$ and $\b b^{[\nu]}=(b_1^{[\nu]},\dotsc, b_s^{[\nu]})$.
The statement of Remark~\ref{rem:RKautonom} follows also in the case of ARK methods from the internal consistency condition \eqref{eq:sum_aij_ARK}, see \cite{SG2015}. As a consequence, it suffices to investigate autonomous systems to understand the order of the method, if \eqref{eq:sum_aij_ARK} is satisfied. 

\section[NB-series and Order Conditions for ARK Methods]{NB-Series and Order Conditions for ARK Methods}
Runge--Kutta (RK) and additive RK schemes belong to \emph{one-step} methods since 
there exists an \emph{incremental map} $\b \Phi$ generating the iterates according to \begin{equation}\label{eq:inc_map}\b y^{n+1}=\b y^n+\dt \b \Phi(t_n,\b y^n,\dt), \quad \b y^0=\b y(t_0),\end{equation}
where implicit schemes are formally represented in their explicit form. For one-step methods, we consider the following notions and results.
\begin{defn}[\cite{HNW1993, HNWII}] Let $\b y\colon [t_0,\tend]\to \R^N$ be the solution to the IVP \eqref{ivp}.  A one-step method for solving \eqref{ivp}
	\begin{enumerate}
		\item has and \emph{order of consistency $p$}, if the \emph{local truncation error} \[\bm \eta(t,\dt) =\b y(t)+\dt \b \Phi(t,\b y(t),\dt) -\b y(t+\dt)\] with $t\in[t_0,\tend]$ and $0\leq \dt\leq \tend-t$ satisfies 
		\[\bm \eta(t,\dt)=\O( \dt^{p+1}), \quad \dt\to 0\] for all $t\in[ t_0,\tend]$.
		\item has an \emph{order of convergence $p$}, if the \emph{global error} \[\b e(t_n,\dt) =\b y^{n} -\b y(t_n)\] satisfies 
		\[\b e(t_n,\dt)=\O( \dt^{p}), \quad \dt\to 0\] for any $t_n=t_0+n\dt\in[t_0,\tend]$.
	\end{enumerate}
\end{defn}
While the local error represents the error of the method generated by a single step starting with exact data, the global error is determined by the difference of the numerical and analytical solutions after $n$ steps. These two notions are deeply interlinked by the following result.
\begin{thm}[{\cite[Theorem 12.2, 12.3]{SM2003},\cite[Theorem 4.10]{DB2002}}]\label{thm:ConsistencyImpConvergence} \hfill
	
	Let $\b y\colon [t_0,\tend]\to \R^N$ be the sufficiently smooth solution to the IVP \eqref{ivp}.
	Furthermore, let the incremental map $\b \Phi$ of the one-step method \eqref{eq:inc_map} for solving \eqref{ivp} be continuous. In addition let $\b \Phi$ be  locally Lipschitz with respect to its second input argument in the sense that
	\[\Vert \b \Phi(t,\b x,\dt)-\b \Phi(t,\b z,\dt)\Vert \leq L_{\b \Phi}\Vert \b x-\b z\Vert\quad \text{ on }\quad  D\times [0,\dt_0] \] for some $0<\dt_0\leq \tend$ and \[D=\{(t,\b z)\mid t_0\leq t\leq t_M, \Vert \b z-\b y^0\Vert \leq C \} \] with some $t_M\leq \tend $ and $C>0$. If the one-step method is consistent of order $p$, then it is also convergent of order $p$.
\end{thm}
If the incremental map satisfies a certain Lipschitz condition specified in Theorem~\ref{thm:ConsistencyImpConvergence}, it thus suffices to study the local truncation error of a method to understand its accuracy, \ie to deduce the order of convergence.

The accuracy of standard RK methods can be understood through the use of trees and B-series, which are formal power series used to represent exact
and approximate solutions of an autonomous initial value problem \cite{B16,HWButcherTrees}.
Similarly, ARK methods can be studied using colored trees and NB-Series \cite{Ntrees}, which we briefly review in the upcoming subsection.  
\subsection{Colored Rooted Trees}
A \emph{rooted tree} is a cycle-free, connected graph with one node designated as the root \cite{B16}. More precisely, a rooted tree can be understood as the underlying undirected graph of an arborescence, for which the root is the uniquely determined node with no incoming arc \cite{KV2012}.  We consider colored rooted trees, in which each node possesses one of $N$ possible colors from the set $\{1,\dotsc, N\}$.  We denote the set of all colored rooted trees, the so-called \emph{$N$-trees}, by $NT$. We indicate the color $\nu\in \{1,\dotsc, N\}$ of the tree represented by $\rt[]$ by writing $\rt[]^{[\nu]}$. In general, a colored rooted tree $\tau$ with a root color $\nu$ can be written in terms of its colored \emph{children} $\tau_1,\dotsc,\tau_k$ by writing 
\begin{equation}\label{eq:tau}
\tau=[\tau_1,\dotsc,\tau_k]^{[\nu]}=[\tau_1^{m_1},\dotsc,\tau_r^{m_r}]^{[\nu]},
\end{equation}
where the children $\tau_1,\dotsc,\tau_k$ are the connected components of $\tau$ when the root together with its edges are removed. Moreover, the neighbors of the root of $\tau$ are the roots of the corresponding children.
In the latter representation of $\tau$ in \eqref{eq:tau}, $m_i$ is the number of copies of $\tau_i$ within $\tau_1,\dotsc,\tau_k$, which already includes the fact that we do not distinguish between trees whose children are permuted.

\begin{exmp}
For simplicity, we consider only one color in this example, that is $N=1$ and $\rt[]^{[1]}=\rt[]$. The children of the tree $\tau=\rt[[],[[],[]]]$ are given by $\tau_1=\rt[]$ and $\tau_2=\rt[[],[]]$ and the respective roots are the lowest nodes. In terms of the representation \eqref{eq:tau} we can write $\rt[[],[[],[]]]=[\rt[],\rt[[],[]]]=[\rt[],[\rt[],\rt[]]]=[\rt[],[\rt[]^2]]=[[\rt[]^2],\rt[]]$. 
\end{exmp}

The \emph{order} of a colored tree $\tau$ is denoted by $\lvert\tau\rvert$ and equals the number of its nodes. We introduce the set $NT_q$ of all $N$-trees up to order $q$. We set $NT_0=\emptyset$ and note that the sets $NT_q$ for $q=1,2,3$ read 
\begin{equation}    
\begin{aligned}
	NT_1&=\{\rt[]^{[\nu]}\mid \nu=1,\dotsc,N\},\\
	NT_2&=NT_1\cup \left\{\et{0.6}{\ct[\hphantom{.}^{[\mu]}[\hphantom{.}^{[\nu]}]]}\bigg| \nu,\mu=1,\dotsc,N\right\},\\
	NT_3&=NT_2\cup \Biggl\{\et{0.6}{\ct[\hphantom{.}^{[\mu]}[\hphantom{.}^{[\nu]}[\hphantom{.}^{[\xi]}]]]}\Bigg| \nu,\mu,\eta=1,\dotsc,N\Biggr\}\cup \Biggl\{\et{1.2}{\ct[\hphantom{.}^{[\mu]}[\hphantom{.}^{[\nu]}][\hphantom{.}^{[\xi]}]]}\Bigg|\nu,\mu,\eta=1,\dotsc,N\Biggr\},
\end{aligned}
\end{equation}
where we used the representation $\left[\rt[]^{[\nu]}\right]^{[\mu]}=\tt{0.6}{\ct[\hphantom{.}^{[\mu]}[\hphantom{.}^{[\nu]}]]}$, $\left[\left[\rt[]^{[\xi]}\right]^{[\nu]}\right]^{[\mu]}=\tt{0.6}{\ct[\hphantom{.}^{[\mu]}[\hphantom{.}^{[\nu]}[\hphantom{.}^{[\xi]}]]]}$ as well as $\left[\rt[]^{[\nu]},\rt[]^{[\xi]}\right]^{[\mu]}=\tt{1.1}{\ct[\hphantom{.}^{[\mu]}[\hphantom{.}^{[\nu]}][\hphantom{.}^{[\xi]}]]}$.
Lastly, the \emph{symmetry} $\sigma$ and \emph{densitity} $\gamma$ of $\tau$ from \eqref{eq:tau} are defined by
\begin{equation}\label{eq:sigmagamma}
\begin{aligned}
	\sigma(\tau)&=\prod_{j=1}^rm_j!\sigma(\tau_j), &&\sigma(\rt[]^{[\nu]})=1,\quad \nu=1,\dotsc,N,\\
	\gamma(\tau)&=\lvert \tau\rvert \prod_{i=1}^k\gamma(\tau_i),&&\gamma(\rt[]^{[\nu]})=1,\quad \nu=1,\dotsc,N.
\end{aligned}
\end{equation}
Observe that $\sigma$ depends on the coloring of $\tau$, while $\gamma$ does not since already $\lvert \tau \rvert$ is independent of the coloring. For instance we find $\sigma(\left[\rt[]^{[1]},\rt[]^{[2]}\right]^{[3]})=1$ since the children are not identical, while $\sigma(\left[\rt[]^{[1]},\rt[]^{[1]}\right]^{[3]})=2$. Meanwhile, we observe $\gamma(\left[\rt[]^{[1]},\rt[]^{[2]}\right]^{[3]})=\gamma(\left[\rt[]^{[1]},\rt[]^{[1]}\right]^{[3]})=3$. The symmetry and density are crucial quantities to describe the expansions of the analytical solution as we will see in the next subsection.

\subsection{Elementary Differentials}
For the following analysis, we assume for simplicity that the system \eqref{ivp} is autonomous, \ie $\fnu(\b y,t)=\fnu(\b y)$. 
We first introduce elementary differentials $\dF\colon NT\to \mathcal C(\R^d,\R^d)$ for colored trees, see \cite{Ntrees}, which are recursively defined by
\begin{equation}\label{eq:elemdiff}
\begin{aligned}
	\dF(\rt[]^{[\nu]})(\b y)&=\Fnu(\b y), \\
	\dF([\tau_1,\dotsc,\tau_k]^{[\nu]})(\b y)&=\sum_{i_1,\dotsc,i_k=1}^d\partial_{i_1\dotsc i_k}\Fnu(\b y)\dF_{i_1}(\tau_1)(\b y)\cdots \dF_{i_k}(\tau_k)(\b y)
\end{aligned}
\end{equation}
for $\nu=1,\dotsc, N$.
An important result in \cite{Ntrees,B16} is the representation of the analytical solution of \eqref{ivp} in terms of an NB-series
\[\NB(u,\b y)=\b y+ \sum_{\tau\in NT}\frac{\dt^{\lvert \tau\rvert}}{\sigma(\tau)}u(\tau)\dF(\tau)(\b y),\]
where $u\colon NT\to \R$, $\b y\in \R^d$ and $\sigma$ is the previously introduced symmetry. Note that $\NB(u,\b y)$ is defined only if $\Fmu\in \mathcal C^\infty$ for $\mu=1,\dotsc, N$. For $\Fmu\in \mathcal C^{p+1}$, we truncate the NB-series and introduce
\[\NB_p(u,\b y)=\b y+ \sum_{\tau\in NT_p}\frac{\dt^{\lvert \tau\rvert}}{\sigma(\tau)}u(\tau)\dF(\tau)(\b y),\]
and point out that $\NB_0(u,\b y)=\b y.$ With that, we can formulate a theorem concerning the NB-series expansion of the solution to the differential equation \eqref{ivp} at some time $t+\dt$.
\begin{thm}[{\cite[Theorem 1]{Ntrees}}]\label{thm:anasolNB}
Let the functions $\Fmu$ from \eqref{ivp} satisfy $\Fmu\in \mathcal C^{p+1}$ for $\mu=1,\dotsc,N$. If $\b y$ solves \eqref{ivp}, then
\[\b y(t+\dt)=\NB_p(\tfrac{1}{\gamma},\b y(t)) +\O(\dt^{p+1}),\]
where $\gamma$ is the density defined in \eqref{eq:sigmagamma}.
\end{thm}
The numerical solution given by one step of an ARK method can also be written as an NB-series $\NB(u,\b y^n)$, with coefficients $u$ recursively determined by
\begin{equation}\label{eq:cond}
\begin{aligned}
	u(\tau)&=\sum_{\nu=1}^N\sum_{i=1}^sb_i^{[\nu]} g_i^{[\nu]}(\tau),\\ 
	g_i^{[\nu]}(\rt[]^{[\mu]})&=\delta_{\nu\mu},&&  \nu,\mu=1,\dotsc,N,  \\
	g_i^{[\nu]}([\tau_1,\dotsc,\tau_l]^{[\mu]})&=\delta_{\nu\mu}\prod_{j=1}^ld_i(\tau_j),&& \nu,\mu=1,\dotsc,N\text{ and } 
	\\
	d_i(\tau)&=\sum_{\nu=1}^N\sum_{j=1}^sa_{ij}^{[\nu]} g_j^{[\nu]}(\tau)
\end{aligned}
\end{equation}
with the Kronecker delta $\delta_{\nu\mu}$, see \cite{Ntrees}. From Theorem \ref{thm:anasolNB} and the fact that elementary differentials are linearly independent \cite{B16,Ntrees}, we obtain the following result.
\begin{thm}[\cite{Ntrees}]\label{thm:ARKcon} An ARK method \eqref{eq:ark} applied to \eqref{ivp} with $\Fnu\in \mathcal C^{p+1}$ is of order $p$ if and only if $u$ determined by \eqref{eq:cond} satisfies
\begin{equation}
	u(\tau)=\frac{1}{\gamma(\tau)} \quad \text{ for all $\tau\in NT_p$.}\label{eq:c=1/gamma}
\end{equation}
\end{thm}

\begin{rem}\label{rem:compute_u}
Based on \cite[Lemma 312B]{B16}, the value of $u$ can be read off from a colored and labeled rooted tree $\tau$. Thereby, a node labeled by $i$ and colored in $\mu$ is represented by $\ct[]_i^{[\mu]}$. It is convenient to also associate with each edge a color; we denote the edge connecting parent node $i$ to child node $j$ by $e_{ij}^{[v]}$, where $\nu$ is the color of node $j$.  We denote the set of labels by $L(\tau)$ and the set of colored edges by $E(\tau)$.

For computing $u(\tau)$, let the root of $\tau$ be labeled by $i$ and colored in $\mu$. Then form the product \[b_{i}^{[\mu]}\prod_{e_{jk}^{[\nu]}\in E(\tau)} a_{jk}^{[\nu]}\] and sum over all elements of $L(\tau)$ ranging over the index set $\{1,\dotsc, s\}$. The result of the sum equals $u(\tau)$.
\end{rem}\newpage
\begin{exmp}
We label the colored rooted tree $\tau=\left[\left[\rt[]^{[\xi]}\right]^{[\nu]}\right]^{[\mu]}$ and represent the result by
\[\ct[\hphantom{.}_i^{[\mu]}[\hphantom{.}_j^{[\nu]}[\hphantom{.}_k^{[\xi]}]]]\]
so that $u(\tau)=\sum_{i,j,k=1}^sb_i^{[\mu]} a_{ij}^{[\nu]} a_{jk}^{[\xi]}$ since $E(\tau)=\left\{e_{ij}^{[\nu]},e_{jk}^{[\xi]}\right\}$ and $L(\tau)=\{i,j,k\}$. 

For the tree $\tau=\left[\rt[]^{[\nu]},\rt[]^{[\xi]}\right]^{[\mu]}$, which we label and represent by
\[\ct[\hphantom{.}_i^{[\mu]}[\hphantom{.}_j^{[\nu]}][\hphantom{.}_k^{[\xi]}]]\]
the value of $u(\tau)$ is $\sum_{i,j,k=1}^sb_i^{[\mu]} a_{ij}^{[\nu]} a_{ik}^{[\xi]}$ as $E(\tau)=\left\{e_{ij}^{[\nu]},e_{ik}^{[\xi]}\right\}$ and $L(\tau)=\{i,j,k\}$. 
\end{exmp}
\section{Linear Stability of Runge--Kutta Methods}\label{sec:stabRK}
The linear stability of a time integration method is usually tackled by the application of the scheme to the linear test equation 
\begin{align}
y'(t)=\lambda y(t), \quad\lambda\in \C^-=\{z\in \C\mid \re(z)<0\}, \label{dahlquist}
\end{align}
which was introduced in 1963 by Dahlquist \cite{D63}. The basic idea behind stability is that the numerical method should replicate the qualitative behavior of the analytic solution in some sense. The central notion linked to the Dahlquist equation is $A$-stability.
\begin{defn}[\cite{D63}]\label{Def:A-stab} A time integration method is called \emph{$A$-stable}, if the sequence of iterates $y^n$ of the method tends to zero, as $n\to \infty$, when applied with fixed $\dt>0$ to any differential equation of the form \eqref{dahlquist}.
\end{defn}
The reason why $A$-stability is of interest may be based on the following heuristic. Consider the difference of two solutions $\b w,\b y$, denoted by $\b u$, of a nonlinear system $\b y'=\b f(\b y)$. Note that $\b u=\b w-\b y$ can be seen as a perturbation. We then linearize the disturbed system $(\b y+\b u)'=\b{w}'=\b f(\b{w})=\b f(\b y+\b u)$, which results in 
\[\b u'=(\b y+\b u)'-\b y'=\b f(\b y+\b u)-\b f(\b y)\approx \b f(\b y)+\b D \b f(\b y)\b u-\b f(\b y)=\b D \b f(\b y)\b u.\]

Freezing the Jacobian $\b D\b f(\b y)$ at a given time $T$ yields a linear system $\b u'=\bA\b u$ for the perturbation, where $\bA$ possibly has complex eigenvalues $\lambda$.  Moreover, the perturbation should disappear as $t\to \infty$, and hence, we consider $\lambda\in \C^-$ in (\ref{dahlquist}) rather than $\lambda\in \C$.
Since this heuristic is not rigorous, I would rather prefer to point out the following motivation. A numerical method that is not capable of mimicking the behavior of the analytical solution to a (scalar) linear test problem is not worth considering for more complex problems.  

Later, the notion of $L$-stability was introduced \cite{HNW1993}. Moreover, for the case of $\lambda\in \R^-$, the notions $A_0$-stable \cite{C73} and $L_0$-stable arise \cite{TGA96}. We also note that more theories have been developed, some of which are suitable for the analysis of RK schemes applied to stiff nonlinear ODEs \cite{DK06, SVV18}. 

For multistep methods zero-stability is a fundamental notion \cite{SM2003}. Some stability properties even introduce a class of schemes, e.g.\ so-called positive and elementary stable non-standard (PESN) schemes \cite{DK06}. 

In this work we focus on $A$-stability.
In the case of an RK method, there exists a rational function $R=\frac{P}{Q}$ such that the method applied to the Dahlquist equation \eqref{dahlquist} reads $y^{n+1}=R(\dt\lambda)y^n$. Hence, the RK method is $A$-stable if and only if $\lvert R(z)\rvert <1$ for all $z\in \C^-$, which is why $R$ is also called the \emph{stability function} of the Runge--Kutta method. Indeed, if we apply the RK method to a linear system 
\begin{equation}\label{eq:Dahlquist_system}
\b y'=\b\Lambda\b y,\quad  \b y(0)=\b y^0,\quad \sigma(\bA)\tm \C^-,
\end{equation} where $\sigma(\bA)$ denotes the spectrum, then the RK method has the same stability properties as applied to the Dahlquist equation with $\lambda$ passing through the eigenvalues of $\bA$, see for instance \cite[Chapter 6]{DB2002}. Hence, if an RK method is $A$-stable, then $\lim_{n\to\infty}\b y^n= \b 0$ holds also for general linear problems \eqref{eq:Dahlquist_system}. 

Even though Definition~\ref{Def:A-stab} does not require the method to be linear, some nonlinear schemes are constructed only for systems of equations as is the case for modified Patankar (MP) methods, see Chapter~\ref{chap:NumSchemes}. Even more, as mentioned in the introduction, the investigation of \[\Vec{y_1'(t)\\y_2'(t)}=\Vec{\lambda y_1(t)\\-\lambda y_1(t)},\quad  \lambda\in \R^-,\] whose first component represents the Dahlquist equation with $\lambda\in \R^-$ is not enough for understanding the stability properties of an MP method applied to more complex linear systems \cite{IKM2122, izgin2022stability}.
Hence, for nonlinear methods it is necessary to investigate general linear systems $\b y'=\b\Lambda\b y$ rather than a scalar equation.
To generalize the notion of $A$-stability in a meaningful way also for nonlinear methods, we consider stability in the sense of Lyapunov, which we recall in the upcoming section.   


\section{Stability in the Sense of Lyapunov}\label{sec:intro_dyn_sys}
In the following, we use $\norm{\ \cdot\  }$ to represent an arbitrary norm in $\R^l$ for $l\in \N$ and $\b D\b f$ denotes the Jacobian of a $\mathcal C^1$-map $\b f$.

Dahlquist already considered to generalize the notion of $A$-stability in \cite{D63} by considering stability in the sense of Lyapunov, which is defined for arbitrary systems of ODEs. Here, the stability near steady states is investigated.
\begin{defn}\label{Def Lyapunov Cont}
Let $\b y^*\in \R^N$ be a steady state solution of a differential equation $\b y'={\b f}(\b y)$, that is ${\b f}(\b y^*)=\b 0$.
\begin{enumerate}
	\item\label{item:lyap_stab} Then $\b y^*$ is called \emph{Lyapunov stable} if, for any $\epsilon>0$, there exists a $\delta=\delta(\epsilon)>0$ such that $\norm{\b y(0)-\b y^*}<\delta$ implies $\norm{\b y(t)-\b y^*}<\epsilon$ for all $t\geq 0$.
	\item If in addition to a), there exists a constant $c>0$ such that $\Vert \b y(0)-\b y^*\Vert<c$ implies  $\Vert \b y(t)-\b y^*\Vert \to 0$ for $t\to \infty$,  we call  $\b y^*$ \emph{asymptotically stable.}
	\item A steady state solution that is not Lyapunov stable is said to be \emph{unstable}.
\end{enumerate}
\end{defn}
In the following, we will also briefly speak of stability instead of Lyapunov stability. Note that in contrast to $A$-stability, these notions are only \emph{global} if $\delta$ and $c$ can be chosen arbitrarily large. Considering the linear system \eqref{eq:Dahlquist_system}, the stability of $\b y^*$ is fully determined by the spectrum $\sigma(\bA)$.

\begin{thm}{(\cite[Theorem 3.23]{DB2002})}\label{Thm:StabLin}
A steady state $\b y^*$ of $\b y'=\bA\b y$ with a matrix $\bA\in \R^{N\times N}$
\begin{enumerate}
	\item 	 is stable if and only if $\max_{\lambda\in \sigma(\bA)}\operatorname{Re}(\lambda)\leq 0$ and all $\lambda$ with $\operatorname{Re}(\lambda)=0$ are associated with a Jordan block of size 1.
	\item  is asymptotically stable if and only if $\max_{\lambda\in \sigma(\bA)}\operatorname{Re}(\lambda)< 0$.
\end{enumerate}
\end{thm}

%

As we are interested in numerical schemes mimicking the stability behavior of the exact solution, we shall consider the following definition, noting that steady states should correspond to fixed points of the method.
\begin{defn}\label{Def_Lyapunov_Diskr}
Let $\b y^*$ be a fixed point of an iteration scheme $\b y^{n+1}=\b g(\b y^n)$, that is $\b y^*=\b g(\b y^*)$. 
\begin{enumerate}
	\item\label{def:stab} Then $\b y^*$ is called \emph{Lyapunov stable} if, for any $\epsilon>0$, there exists a $\delta=\delta(\epsilon)>0$ such that $\norm{\b y^0-\b y^*}<\delta$ implies $\norm{\b y^n- \b y^*}<\epsilon$ for all $n\geq 0$.
	\item If in addition to a), there exists a constant $c>0$ such that $\Vert \b y^0-\b y^*\Vert<c$ implies $\Vert \b y^n-\b y^*\Vert \to 0$ for $n\to \infty$, we call $\b y^*$ \emph{asymptotically stable.}
	\item A fixed point that is not Lyapunov stable is said to be \emph{unstable}.
\end{enumerate}
\end{defn}As before, we may only speak of stability in the following.
For linear methods, such as RK schemes, we have the following result.
\begin{thm}[{\cite[Theorem 3.33]{DB2002}}]\label{Thm:asymStablin}
A fixed point $\b y^*$ of $\b y^{n+1}=\b R\b y^n$ with $\b R\in \R^{N\times N}$
\begin{enumerate}
	\item  is stable if and only if the spectral radius $\rho$ satisfies $\rho(\b R) \leq 1$ and all $\lambda\in \sigma(\b R)$ with $\lvert \lambda\rvert =1$ are associated with a Jordan block of size 1.
	\item  is asymptotically stable if and only if $\rho(\bm R)< 1$. 
\end{enumerate}
\end{thm}

\begin{rem}\label{rem:RKstab_dynsys}
According to Theorem~\ref{Thm:StabLin}, $y^*=0$ is the unique globally asymptotically stable solution of the Dahlquist equation. Also, Theorem~\ref{Thm:asymStablin} tells us that an RK method is $A$-stable if and only if $0$ is an asymptotically stable fixed point of the method when applied to the Dahlquist equation. This also demonstrates that $0$ is a globally asymptotically stable fixed point of the $A$-stable RK method.  We also note that in some literature, such as \cite{B16,HNWII}, $A$-stability of an RK scheme is defined by requiring $\lvert R(z)\rvert \leq 1$ for all $z\in \overline{\C^-}=\{z\in \C\mid \re(z)\leq 0\}$. The idea behind this adaptation is that we may only require that the numerical solution is bounded for bounded solutions of the Dahlquist equation. However, with this notion of $A$-stability, the generalization to linear systems is more involved as Theorem~\ref{Thm:asymStablin} suggests.
\end{rem}
If the method is not linear, the stability properties are a priori only of local nature and their investigation is more complex. As stated by the next theorem, it is in some cases sufficient to investigate the linearized method in order to understand the stability properties of a fixed point.\newpage
\begin{thm}[{\cite[Theorem 1.3.7]{SH98}}]\label{Thm:_Asym_und_Instabil}
Let  $\b y^{n+1}=\b g(\b y^n)$ be an iteration scheme with fixed point $\b y^*$. Suppose the Jacobian $\b D\b g(\b y^*)$ exists. Then
\begin{enumerate}
	\item $\b y^*$ is asymptotically stable if $\rho(\b D\b g(\b y^*))<1$. 
	\item $\b y^*$ is unstable if $\rho(\b D\b g(\b y^*))>1$.
\end{enumerate}
\end{thm}

The above theorem gives sufficient conditions for the stability of fixed points that are \emph{hyperbolic} in the following sense.

\begin{defn}[{\cite[Definition 1.3.6]{SH98}}]\label{Def hyperbolic}
A fixed point $\b y^*$ of an iteration scheme $\b y^{n+1}=\b g(\b y^n)$ is called \emph{hyperbolic} if $\abs\lambda\ne 1$ for all eigenvalues $\lambda$ of $\b D\b g(\b y^*)$. If a fixed point is not hyperbolic, it is called \emph{non-hyperbolic}.
\end{defn}

A generalization of Theorem~\ref{Thm:_Asym_und_Instabil} is the Hartman-Grobman Theorem, which states that a nonlinear iteration scheme and its linearization share the same behavior near hyperbolic fixed points, see \cite[Theorem 1.6.2]{SH98} for the precise statement.

In this work, we will also analyze schemes that require $\dim(\ker(\bA))=k>0$. In such a case, the linear system $\b y'=\bA\b y$ possesses a subspace of steady state solutions, each of which can be stable according to Theorem~\ref{Thm:StabLin} but none of them is asymptotically stable. If the numerical method is steady state preserving, it thus possesses a subspace of fixed points, each of them being non-hyperbolic as we will find out in Chapter~\ref{chap:stab}. Hence, it is also of high importance to understand the stability of non-hyperbolic fixed points. However, for schemes outside the class of general linear methods the stability behavior of a single non-hyperbolic fixed point is in general not captured by the eigenvalues of the corresponding Jacobian, \ie is not guaranteed by $\rho(\b D(\b g(\b y^*)))= 1$ as the following example illustrates.
\begin{exmp}[\cite{Osipenko2009}]\label{exmp:nonhyper}
Consider the generating map defined by
\begin{equation}\label{eq:ex_g}
	\b g(x,y)=\left(x+xy,\frac12(y+x^2)+2x^2y+y^3\right)^T.
\end{equation} We observe $\b g(0,0)=(0,0)^T$ and $\b D\b g(0,0)=\vec{1 & 0\\0 &\tfrac12}$
Now, defining $h(x)=x^2$, we see that the graph of $h$ is invariant under $\b g$ since
\[ \b g(x,h(x))=(x+x^3, x^2+2x^4+x^6)^T= (x+x^3,h(x+x^3))^T.\]
Focusing on the $x$-component, \ie $x+x^3=x(1+x^2)$, we find that the iterates distance from the origin along the graph of $h$, see Figure~\ref{fig:nonhyper} for an illustration.
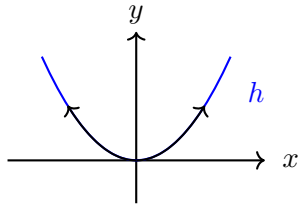
\begin{figure}[!h]
	\centering
	\begin{scaletikzpicturetowidth}{0.3\textwidth}
		\begin{tikzpicture}[scale=\tikzscale]
			\draw [->, thick] (-1.5,0) -- (0,0) -- (1.5,0);
			\node at (1.8,0) {$x$};
			\draw [->, thick] (0,-0.5) -- (0,1.5);
			\node at (0,1.7) {$y$};
			\draw [thick, domain=-1.1:1.1, blue] plot (\x, {\x*\x});
			\draw [->, thick, domain=0:0.8, black] plot (\x, {\x*\x});
			\draw [<-, thick, domain=-0.8:0, black] plot (\x, {\x*\x});
			\node[blue] at (1.4,0.8) {$h$};
		\end{tikzpicture}
	\end{scaletikzpicturetowidth}
	\caption{Graph of $h$ defined by $h(x)=x^2$. The arcs indicate the action of $\b g$ from \eqref{eq:ex_g} on the graph of $h$.}\label{fig:nonhyper}
\end{figure}
From this, we can conclude that the origin is unstable even though the eigenvalues of the Jacobian are $1$ and $\frac12$. 
\end{exmp}
This example demonstrates that, in general, higher-order terms have to be included within the stability analysis of nonlinear methods.
One possibility to decrease the complexity of such a stability analysis is to use the center manifold theory, which allows to assess the stability based on a corresponding iteration on a lower dimensional manifold. Indeed, in Example~\ref{exmp:nonhyper} the map $h$ represents the center manifold.
\section{Center Manifold Theory}\label{sec:stab_dyn_sys}
To study the stability of a non-hyperbolic fixed point $\b y^*$ of an iteration scheme with $\mathcal C^1$-map $\b g$, we make use of an affine linear transformation\footnote{See the proof of Theorem~\ref{Thm_MPRK_stabil} for the details of this transformation.} to obtain a $\mathcal{C}^1$-map $\b G\colon \mathcal M\to \R^{N}$, with $\mathcal M\subset \R^{N}$ being a neighborhood of the origin, which has the form
\begin{equation}
\b G(\b w_1,\b w_2)=\Vec{\b U\b w_1+\b u(\b w_1,\b w_2)\\\b V\b w_2 +\b v(\b w_1,\b w_2)},\label{Form_of_G}
\end{equation}
with $\b w_1\in \R^m$, $\b w_2\in \R^l$ and $m+l=N$. The square matrices $\b U\in\R^{m\times m}$ and $\b V\in\R^{l\times l}$ are such that $\abs\lambda = 1$ holds for all eigenvalues $\lambda$ of $\b U$ and each eigenvalue $\mu$ of $\b V$  satisfies $\abs\mu <1$. The functions $\b u$ and $\b v$ are in $\mathcal{C}^1$ and $\b u,\b v$ as well as their first order derivatives vanish at the origin, that is
\begin{align*}
\b u(\b 0,\b 0)&=\b 0, & \b D\b u(\b 0,\b 0)&=\b 0, & 
\b v(\b 0,\b 0)&=\b 0, & \b D\b v(\b 0,\b 0)&=\b 0,
\end{align*} where $\b 0$ stands for the zero vector or matrix of appropriate size, respectively. 
In particular, the fixed point $\b y^*$ of $\b g$ is mapped to $\b 0$, which is a fixed point of $\b G$ with equal stability properties as $\b y^*$ as  we point out in the proof of Theorem \ref{Thm_MPRK_stabil}. 

Hence, it is sufficient to study the stability of the origin with respect to $\b G$, which is a simplification due to the existence of a center manifold. 

\begin{thm}{(Center Manifold Theorem, \cite[Theorem 2.1, Remark 2.6]{mccracken1976hopf})}\label{Thm:Ex_CM}
Let $\b G$ be defined as in \eqref{Form_of_G} with Lipschitz continuous derivatives on $\mathcal M$.
\begin{enumerate}
	\item\label{item:existence}  (Existence):
	There exists a center manifold for $\b G$, which is locally representable as the graph of a function $\b h\colon\R^m\to \R^l$.  This means, for some $\epsilon>0$ there exists a $\mathcal{C}^1$-function $\b h\colon\R^m\to\R^l$ with $\b h(\b 0)=\b 0$ and $\b D\b h(\b 0)=\b 0$ such that $\Vert \b w_1^0\Vert,\Vert \b w_1^1\Vert  <\epsilon$ and $(\b w_1^1, \b w_2^1)^T=\b G(\b w_1^0,\b h(\b w_1^0))$ imply $\b w_2^1=\b h(\b w_1^1)$.
	\item\label{item:attractivity}  (Local Attractivity): If in addition to \ref{item:existence} the iterates $(\b w_1^n,\b w_2^n)^T$ generated by 
	\begin{align}
		\Vec{\b w_1^{n+1}\\\b w_2^{n+1}}=\b G(\b w_1^n,\b w_2^n)=\Vec{\b U\b w_1^n+\b u(\b w_1^n,\b w_2^n)\\\b V\b w_2^n +\b v(\b w_1^n,\b w_2^n)},\quad \Vec{\b w_1^0\\ \b w_2^0}\in \mathcal M\label{Form_of_G_transformed}
	\end{align}  
	satisfy $\Vert \b w_1^n\Vert,\Vert \b w_2^n\Vert<\epsilon$ for all $n\in \N_0$, then the distance of $(\b w_1^n,\b w_2^n)$ to the center manifold tends to zero for $n\to \infty$, i.\,e.\ $\Vert \b w_2^n-\b h(\b w_1^n)\Vert\to 0$ for $n\to \infty$.
\end{enumerate}
\end{thm}
As will be seen in Theorem \ref{Thm:Stab_CM}, the existence of a center manifold enables the investigation of the stability properties of the origin based on a system with reduced dimension.  This reduced system is obtained by restricting \eqref{Form_of_G} to the center manifold, i.\,e.\  using $\b w_2=\b h(\b w_1)$ which leads to the map
\begin{equation}\label{Flow_center_manifold} \mathcal G(\b w_1)=\b U \b w_1 + \b u(\b w_1,\b h(\b w_1)).
\end{equation}

\begin{thm}{(\cite[Theorem 8]{carr1982})}\label{Thm:Stab_CM} (Stability): Suppose the fixed point $\b 0\in \R^m$ of $\mathcal G$ from \eqref{Flow_center_manifold} is stable, asymptotically stable or unstable. Then the fixed point $\b 0 \in \R^N$ of $\b G$ from \eqref{Form_of_G} is stable, asymptotically stable or unstable, respectively.
\end{thm}
In summary,  the stability of a non-hyperbolic fixed point $\b y^*\in\R^N$ of a map $\b g$ can be determined by investigating the fixed point $\b 0\in\R^m$ of $\mathcal G$, which has a lower complexity due to the reduced dimension $m<N$.

To actually calculate the center manifold we need to solve
\begin{equation*}
(\b w_1^1, \b h(\b w_1^1))^T=\b G(\b w_1^0,\b h(\b w_1^0))=\Vec{\b U\b w_1^0+\b u(\b w_1^0,\b h(\b w_1^0))\\\b V\b h(\b w_1^0) +\b v(\b w_1^0,\b h(\b w_1^0))},
\end{equation*} 
which can be rewritten as
\[\b h(\b U\b w_1^0+\b u(\b w_1^0,\b h(\b w_1^0)))=\b V\b h(\b w_1^0) +\b v(\b w_1^0,\b h(\b w_1^0)).\]
This invariance property offers a way to approximate the center manifold up to an arbitrary order.
\begin{thm}{(\cite[Theorem 7]{carr1982})}\label{Thm:Comp_func_h}
Let $\b h$ be a center manifold for $\b G$ and $\bm\Phi$ be a $\mathcal{C}^1(\R^m, \R^l)$-map with $\bm\Phi(\b 0)=\b 0$ and $\b D\bm\Phi(\b 0)=\b 0$. If
\[\bm\Phi(\b U\b w_1+\b u(\b w_1,\bm\Phi(\b w_1)))-\left(\b V\bm\Phi(\b w_1) +\b v(\b w_1,\bm\Phi(\b w_1))\right)=\mathcal{O}(\Vert \b w_1 \Vert^q)\]  as $\b w_1\to \b 0$ for some $q>1$, then $\b h(\b w_1)=\bm\Phi(\b w_1)+\mathcal{O}(\Vert \b w_1 \Vert^q)$ as $\b w_1\to \b 0$.
\end{thm}
Before we go to theoretical fundamentals on production-destruction-rest systems, let us summarize the sections on stability. We started with $A$-stability which is the central notion for capturing the linear stability properties of general linear methods such as Runge--Kutta schemes. However, we discussed that analyzing a scalar equation is not sufficient to capture the stability behavior of nonlinear methods. Hence, we generalized $A$-stability by considering stability in the sense of Lyapunov. Moreover, we presented tools for analyzing general numerical methods with hyperbolic and non-hyperbolic fixed points, where the analysis of the latter is more challenging as more techniques such as the approximation of the center manifold is required. However, for our purposes this is the interesting case when analyzing Patankar-type methods.
\section{Production-Destruction-Rest Systems}\label{sec:PDRS}
In this work we are interested in methods that are capable of producing positive approximations for any chosen time step size. First focusing on autonomous problems, it is convenient to rewrite the system of ODEs into the form of a \emph{production-destruction system} (PDS)
\begin{equation} \label{eq:PDS}
y_k'(t)= f_k(\b y(t))=\sum_{\nu=1}^N (p_{k\nu}(\b y(t))-d_{k\nu}(\b y(t))),\quad \b y(0)=\b y^0>\b 0,
\end{equation}
where $p_{k\nu}(\b y(t)), d_{k\nu}(\b y(t)) \geq 0$ for all $\b y(t)\geq 0$.  Note that every real valued right-hand side $f_k$ can be split into production and destruction terms setting \[p_{k1}(\b y)=\max\{0, f_k(\b y)\}, \quad d_{k1}(\b y)=-\min\{0, f_k(\b y)\},\quad p_{k\nu}=d_{k\nu}=0,\quad \nu\neq 1.\]
However, using this splitting the production and destruction terms are generally not differentiable. Nevertheless, in view of Theorem~\ref{thm:ConsistencyImpConvergence} we note that if $f_k$ is locally Lipschitz continuous, then so are $p_{k1}$ and $d_{k1}$ as they are the composition of two locally Lipschitz mappings.
\begin{defn}
The PDS \eqref{eq:PDS} is called \emph{positive}, if $\b y(0)>\b 0$ implies $\b y(t)>\b 0$ for all $t>0$. Similarly, a \emph{non-negative} PDS are defined.
\end{defn}
\begin{prop}[ \cite{BDM03}]\label{prop:PDSnonneg}
For non-negative initial data, the PDS \eqref{eq:PDS} is non-negative if $d_{k\nu}(\b y)\to 0$ as $y_k\to 0$ for $k,\nu=1,\dotsc,N$.
\end{prop}

\begin{defn}
We call the PDS \eqref{eq:PDS} \emph{conservative}, if $p_{k\nu}=d_{\nu k}$ for all $k,\nu=1,\dotsc,N$. If in addition we have $p_{kk}=d_{kk}=0$, the PDS is called \emph{fully conservative}.
\end{defn}
\begin{rem}\label{rem:fullyconservative}
Since $p_{kk}=d_{kk}$ cancel out in \eqref{eq:PDS} for a conservative PDS we can assume without loss of generality that $p_{kk}=d_{kk}=0$, \ie that the PDS is always fully conservative. 
\end{rem}
For a conservative PDS, we know that $\sum_{k=1}^Ny_k'(t)=0$, and hence, the sum of the constituents remains constant in time. In general, if a linear combination $\b n^T\b y$ remains constant 
in time, we call it a \emph{linear invariant}. 

It is also worth mentioning that the additive splitting into production and destruction terms is not uniquely determined. For instance, considering
\begin{equation*}
\b y'=\Vec{y_2 + y_4 - y_1\\
	y_1-y_2\\
	y_1 - y_3\\
	y_3 -y_1-y_4},
\end{equation*}
the terms $p_{14}(\b y)=y_4$, $p_{12}(\b y)=y_2$ and $p_{43}(\b y)=y_3$ are a straightforward choice, however, both,
\[p_{21}(\b y)=y_1,\quad p_{34}(\b y)=y_1  \]
and
\[p_{31}(\b y)=y_1,\quad p_{24}(\b y)=y_1\]
complete the splitting into a PDS, where we set $p_{mn}=0$ for the remaining production terms and $p_{k\nu}=d_{\nu k}$.

In this work, we are also interested in positive PDS which are non-autonomous and not conservative. For a transparent notation we split the PDS into a conservative part and rest terms, leading to a \emph{production-destruction-rest system} (PDRS)
\begin{equation}
y_k'(t)= f_k(\b y(t), t)=r_k(\b y(t), t) + \sum_{\nu=1}^N (p_{k\nu}(\b y(t), t)-d_{k\nu}(\b y(t), t)),\quad \b y(0)=\b y^0>\b 0 \label{eq:PDRS_ODE}
\end{equation}
with $k=1,\dotsc,N$ and $p_{k\nu}=d_{\nu k}$. Additionally, the rest term is also split according to
\begin{equation}
r_k(\b y(t), t) = r_k^p(\b y(t), t) - r_k^d(\b y(t), t)\label{eq:rp_rd}
\end{equation}
with $r_k^p,r_k^d\geq 0$ for $t\geq 0,$ $\b y(t)\geq \b 0$.
Note that $r_k^p$ and $r_k^d$ can always be constructed, for example by using the functions $\max$ and $\min$ as above. 
The autonomous version of the PDRS \eqref{eq:PDRS_ODE} was already considered in \cite{IssuesMPRK} and the existence, uniqueness and positivity of the solution of \eqref{eq:PDRS_ODE} was discussed in \cite{FSpos}. In what follows, we are assuming that such a positive solution exists. For later use it is also beneficial to rewrite the PDS as an additive splitting of the form  \eqref{ivp}. 
\begin{rem}\label{rem:PDStoAdditive} Any PDRS \eqref{eq:PDRS_ODE}, \eqref{eq:rp_rd} may be rewritten as an additive splitting of the form
\begin{align*}
	\b f(\b y(t),t)= \sum_{\nu=1}^{N+1}  \Fnu(\b y(t),t)\in \R^N
\end{align*} 
using  $\b f^{[N+1]}(\b y(t),t)=(r_1^p(\b y(t),t),\dotsc,r_N^p(\b y(t),t))^T$ and \[f^{[\nu]}_k(\b y(t),t)=\begin{cases}p_{k\nu}(\b y(t),t), &k\neq \nu, \\-\left(r_k^d(\b y(t),t) + \sum_{\mu=1}^Nd_{k\mu}(\b y(t),t)\right), &k=\nu\end{cases}\]
for $k,\nu=1,\dotsc, N$. To see this, we first point out that $p_{kk}=d_{kk}=0$ can be assumed, see Remark~\ref{rem:fullyconservative}. Hence,
\begin{equation*}
	\begin{aligned}
		f_k&=\sum_{\nu=1}^{N+1}f_k^{[\nu]}
		=\sum_{\substack{\nu=1\\\nu\neq k}}^{N}f_k^{[\nu]} + f_k^{[k]}+ f_k^{[N+1]} \\
		&=\sum_{\nu=1}^Np_{k\nu} -\left(r_k^d + \sum_{\mu=1}^Nd_{k\mu}\right) + r_k^p=r_k+\sum_{\nu=1}^N(p_{k\nu}-d_{k\nu}).
	\end{aligned}
\end{equation*}

\end{rem}

\chapter{Numerical Schemes}\label{chap:NumSchemes}
In this chapter we review  positivity-preserving schemes that additionally preserve at least one linear invariant. For other recent approaches which facilitate positive and conservative numerical approximations, we refer to \cite{AGKM_Oliver,nuesslein2021positivitypreserving,blanes2021positivitypreserving}, some of which even conserve all linear invariants. The following schemes are one-step methods, for which we briefly recall the definition of unconditional conservativity and positivity.
\begin{defn}
	Let $\b y^n$ denote an approximation of $\b y(\tn)$ at time level $\tn$. 
	The corresponding one-step method
	is called
	\begin{itemize}
		\item \textit{unconditionally conservative}, if 
		\[
		\sum_{k=1}^N y_k^{n+1}=\sum_{k=1}^Ny_k^n
		\]
		is satisfied for all $n\in\N_0$ and $\dt>0$.
		\item \textit{unconditionally positive}, if $\b y^n>0$ implies $\b y^{n+1}>0$ for all $n\in \N_0$ and $\dt>0$.
	\end{itemize}
\end{defn}
\section{Non-standard Additive Runge--Kutta Methods}\label{sec:NSARK}
Non-standard additive Runge--Kutta (NSARK) methods are based on ARK schemes \eqref{eq:ark}, where the Butcher tableau is allowed to also depend on the step size and the solution. In particular, NSARK methods are of the form
\begin{equation} \tag{NSARK}\label{eq:nsark}
	\begin{aligned}
		\byi & = \b y^n + \dt \sum_{j=1}^s  \sum_{\substack{\nu=1}}^N a^{[\nu]}_{ij}(\b y^n,\tn,\dt)  \fnu(\byj, \tn + c_j\dt), \quad i=1,\dotsc,s,\\
		\b y^{n+1} & = \b y^n + \dt \sum_{j=1}^s \sum_{\substack{\nu=1}}^N b^{[\nu]}_j(\b y^n,\tn,\dt) \fnu(\byj, \tn + c_j\dt).
	\end{aligned}
\end{equation}
Note that the stages $\byi=\byi(\b y^n)$ may be interpreted as functions of $\b y^n$, so that the dependence of $a^{[\nu]}_{ij}, b^{[\nu]}_j$ on $\b y^n$ might be given implicitly. As a result of this notation, an NSARK method is called \emph{explicit}, if the matrices $\b A^{[\nu]}$ are strict lower left triangular matrices and the dependence of $\b A^{[\nu]}$ as well as $\b b^{[\nu]}$ on $\b y^n$ is only explicit. Otherwise, the NSARK method is called \emph{implicit}.

As we will discover in this chapter, all MP methods based on RK schemes can be written as an NSARK method. Moreover, given a Butcher tableau defined by $\b A,\b b,\b c$, the corresponding MP methods are of the form 
\begin{equation}\label{eq:NSweights}
	\begin{aligned}
		a^{[\nu]}_{ij}(\b y^n,\tn,\dt)&=a_{ij}\gamma_\nu^{[i]}(\b y^n,\tn,\dt),\\
		b^{[\nu]}_j(\b y^n,\tn,\dt)&=b_j\delta_\nu(\b y^n,\tn,\dt)
	\end{aligned}
\end{equation}
for some scheme-dependent functions $\gamma_\nu^{[i]}$ and $\delta_\nu$, which we refer to as \emph{non-standard weights} (NSWs). Investigating NSARK methods allows the comprehensive derivation of a general stability function as well as order conditions for different families of methods. In particular, it turns out that NSARK methods are a valuable formulation for the analysis of so-called modified Patankar--Runge--Kutta (MPRK) methods. Nevertheless, we will be able to deduce also some results for  Geometric Conservative (GeCo) schemes in this work and discuss how to generalize or adapt NSARK schemes to investigate even more nonlinear methods.

	The following proposition formulates sufficient conditions under which an NSARK scheme produces the same approximations for the transformed autonomous system mentioned in Remark~\ref{rem:RKautonom}. 
	\begin{prop}\label{prop:nsarkautonom}Let $\b A,\b b,\b c$ describe an RK method satisfying $\sum_{j=1}^sa_{ij}=c_i$ and $\sum_{j=1}^sb_j=1$.
		Let the stages of the corresponding NSARK method \eqref{eq:nsark} be uniquely determined for some $\dt>0$  and transform the IVP \eqref{ivp}
		into the autonomous system $\b Y'(t)=\sum_{\nu=1}^{N+1}\b F^{[\nu]}(\b Y(t))$ using
		\[\b Y(t)=\vec{\b y(t)\\t},\quad \b F^{[\mu]}(\b Y(t))=\vec{
			\fmu(\b Y(t))\\0
		},\,\, 1\leq \mu\leq N, \quad \b F^{[N+1]}(\b Y(t))=\vec{
			\b 0\\1
		}.\]
		If $\gamma_{N+1}^{[i]},\delta_{N+1}^{[j]}=1$, then the approximations for the solution of the IVP \eqref{ivp} using the NSARK method coincide irrespective of whether the autonomous or non-autonomous system is solved.
		
	\end{prop}
	\begin{proof}
		Since $\gamma_{N+1}^{[i]},\delta_{N+1}^{[j]}=1$, the NSARK method applied to the autonomous system reads
		\begin{equation}\label{eq:NSARK_autonom}
			\begin{aligned}
				Y_k^{(i)}=& Y_k^n + \dt\sum_{j=1}^{s}\left(\sum_{\nu=1}^Na_{ij}^{[\nu]}(\b Y^n,\dt) F_k^{[\nu]}(\b Y^{(j)})+  a_{ij} F_k^{[N+1]}(\b Y^{(j)})\right),\\
				Y_k^{n+1}=& Y_k^n + \dt\sum_{j=1}^{s}\left(\sum_{\nu=1}^Nb_j^{[\nu]}(\b Y^n,\dt) F_k^{[\nu]}(\b Y^{(j)})+  b_j F_k^{[N+1]}(\b Y^{(j)})\right).
			\end{aligned}
		\end{equation}
		Thus, for $k=N+1$ we find
		\begin{equation} \label{eq:tk_tn+1}
			\begin{aligned}
				t_{(i)}&=\tn + \dt \sum_{j=1}^{s}a_{ij}=\tn+c_i\dt,\\
				t_{n+1} &= \tn +\dt \sum_{j=1}^sb_j=\tn+\dt.
			\end{aligned}
		\end{equation}
		Furthermore, for $k\leq N$, we end up with
		{\allowdisplaybreaks
			\begin{align*}
				y_k^{(i)}=& y_k^n + \dt\sum_{j=1}^{s}\sum_{\nu=1}^Na_{ij}^{[\nu]}(\b y^n,\tn,\dt) f_k^{[\nu]}(\b y^{(j)},t_{(j)}),\\
				y_k^{n+1}=& y_k^n + \dt\sum_{j=1}^{s}\sum_{\nu=1}^Nb_j^{[\nu]}(\b y^n,\tn,\dt) f_k^{[\nu]}(\b y^{(j)},t_{(j)}).
		\end{align*}}
		
		Substituting \eqref{eq:tk_tn+1} into these equations, the proof is finished by noting that the resulting systems always possess a unique solution due to our preconditions.
	\end{proof}
	As a consequence of Proposition~\ref{prop:nsarkautonom} we may consider only autonomous problems for deriving order conditions, if the method satisfies the assumptions of the proposition.
	
	\section{Modified Patankar--Runge--Kutta}\label{sec:MPRK}
	
	The main idea of modified Patankar--Runge--Kutta (MPRK) methods  \cite{BDM03,KM18,KM18Order3,MPRK3ex} is to apply an explicit Runge--Kutta (RK) method to a production-destruction systems (PDS) \eqref{eq:PDS} and use the modified Patankar-trick. We extend this approach also to production-destruction-rest systems (PDRS) \eqref{eq:PDRS_ODE}, \eqref{eq:rp_rd} where we only apply the Patankar-trick to the rest term. This means, that $r_k^p$ will not be weighted and $r_k^d$ will be treated like a destruction term. 
	\begin{defn}\label{def:MPRKdefn}
		Given an explicit $s$-stage RK method described by a non-negative Butcher array, \ie $\b A,\b b,\b c \geq\b 0$ we define the corresponding MPRK schemes applied to \eqref{eq:PDRS_ODE}, \eqref{eq:rp_rd} by
		\begin{equation}\tag{MPRK}\label{eq:MPRK_PDRS}
			\begin{aligned}
				y_k^{(i)}=& y_k^n + \dt\sum_{j=1}^{i-1}a_{ij}\left( r_k^p(\b y^{(j)}, \tn + c_j\dt) + \sum_{\nu=1}^N  p_{k\nu}(\b y^{(j)}, \tn + c_j\dt)\frac{y_\nu^{(i)}}{\pi_\nu^{(i)}}\right. \\
				& \left. - \left(r_k^d(\b y^{(j)}, \tn + c_j\dt) +\sum_{\nu=1}^Nd_{k\nu}(\b y^{(j)}, \tn + c_j\dt) \right)\frac{y_k^{(i)}}{\pi_k^{(i)}}\right),\quad k=1,\dotsc,s,\\
				y_k^{n+1}=& y_k^n + \dt\sum_{j=1}^s b_j\left( r_k^p(\b y^{(j)}, \tn + c_j\dt) + \sum_{\nu=1}^N  p_{k\nu}(\b y^{(j)}, \tn + c_j\dt)\frac{y_\nu^{n+1}}{\sigma_\nu}\right. \\
				& \left. - \left(r_k^d(\b y^{(j)}, \tn + c_j\dt) +\sum_{\nu=1}^Nd_{k\nu}(\b y^{(j)}, \tn + c_j\dt) \right)\frac{y_k^{n+1}}{\sigma_k}\right),
			\end{aligned}
		\end{equation}
		where $\pi_\nu^{(i)},\sigma_\nu$ are the so-called \emph{Patankar-weight denominators} (PWDs) and positive for any $\dt\geq 0$ as well as independent of the corresponding numerators $y_k^{(i)}$ and $y_k^{n+1}$, respectively. 
	\end{defn}
	MPRK schemes are of considerable interest and widely used such as in the context of ecosystems \cite{HenseBeckmann2010,HenseBurchard2010,WHK2013,BMZ2007,BMZ2009,MeisterBenz2010} or ocean models \cite{SemeniukDastoor2017,BBKMNU2006}. Further applications can be found in the context of magneto-thermal winds \cite{Gressel2017} or warm-hot intergalactic mediums \cite{KlarMuecket2010} as well as in that of the SIR epidemic model \cite{WS21}. 
	\begin{rem}
		In matrix notation, \eqref{eq:MPRK_PDRS} can be rewritten as
		\begin{equation}\label{eq:MPRK_PDRS_matrix}
			\begin{aligned}
				\b M^{(i)}\byi&= \b y^n+\dt \sum_{j=1}^{i-1}a_{ij} \b r^p(\b y^{(j)}, \tn + c_j\dt),\quad i=1,\dotsc,s, \\
				\b M\b y^{n+1}&= \b y^n + \dt\sum_{j=1}^{s}b_{j} \b r^p(\b y^{(j)}, \tn + c_j\dt),
			\end{aligned}
		\end{equation}
		where $\b r^p=(r_1^p,\dotsc, r_N^p)^T$ and $\b M^{(i)}=(m^{(i)}_{k\nu})_{1\leq k,\nu\leq N}$ with
		\begin{equation*}
			\begin{aligned}
				m^{(i)}_{kk}&=1+ \dt \sum_{j=1}^{i-1}a_{ij}\left(r_k^d(\b y^{(j)}, t_n + c_j\dt) +\sum_{\nu=1}^Nd_{k\nu}(\b y^{(j)}, t_n + c_j\dt) \right)\frac{1}{\pi_k^{(i)}}, \\
				m^{(i)}_{k\nu}&= -\dt \sum_{j=1}^{i-1}a_{ij}p_{k\nu}(\b y^{(j)}, \tn + c_j\dt)\frac{1}{\pi_\nu^{(i)}}, \quad k\neq \nu
			\end{aligned}
		\end{equation*}
		as well as, using $\b M=(m_{k\nu})_{1\leq k,\nu\leq N}$, 
		\begin{equation*}
			\begin{aligned}
				m_{kk}&= 1+\dt \sum_{j=1}^sb_j\left(r_k^d(\b y^{(j)}, t_n + c_j\dt) +\sum_{\nu=1}^Nd_{k\nu}(\b y^{(j)}, t_n + c_j\dt) \right)\frac{1}{\sigma_k}, \\
				m_{k\nu}&= -\dt \sum_{j=1}^sb_jp_{k\nu}(\b y^{(j)}, \tn + c_j\dt)\frac{1}{\sigma_\nu}, \quad k\neq \nu.
			\end{aligned}
		\end{equation*}
		
	\end{rem}
	\begin{rem}
		We require $\sigma_\nu$ to be independent of $y_\nu^{n+1}$ to ensure that the scheme is positive and linear implicit. To see this, recall that the choice $\sigma_\nu=y_\nu^{n+1}$ and $\pi_\nu^{(i)}=y_\nu^{(i)}$ would lead to the original Runge--Kutta scheme, which is not unconditionally positive. Moreover, if $\sigma_\nu$ would allowed to be a nonlinear function of $y_\nu^{n+1}$ we would have to solve a nonlinear system instead of a linear one to compute $y_\nu^{n+1}$. For the same reason we require $\pi_\nu^{(i)}$ to be independent of $y_\nu^{(i)}$.
	\end{rem}
	The following two lemmas state that MPRK schemes as defined in Definition~\ref{def:MPRKdefn} are indeed unconditionally positive and conservative. 
	Both lemmas are slight generalizations of lemmas from \cite{BDM03, KM18}.
	\begin{lem}\label{lem:MPRK}
		An MPRK scheme \eqref{eq:MPRK_PDRS} applied to a conservative PDS, \ie $\b r=\b 0$, is unconditionally conservative. The same holds for all stage values, that is
		$\sum_{k=1}^N y_k^{(i)}=\sum_{k=1}^Ny_k^n$
		for $i=1,\dots,s$.
	\end{lem}
	\begin{lem}\label{lem:MPRKpos}
		An MPRK scheme \eqref{eq:MPRK_PDRS} is unconditionally positive. 
		The same holds for all the stages of the scheme, this is for all $\Delta t>0$ and $\b y^n>0$ we have $\b y^{(i)}>0$ for $i=1,\dots,s$. In particular, the inverses $(\b M^{(i)})^{-1}, (\b M)^{-1}$ exist and their entries lie in the interval $[0,1]$
	\end{lem}
	
	We also note that this scheme always produces positive approximations, if $\b y^0>\b 0$. However, if it is known that the analytic solution is not positive due to the existence of the rest term $\b r$, then one may consider choosing $\b r^d=\b 0$ and $\b r^p = \b r$ in the MPRK scheme \eqref{eq:MPRK_PDRS}. This essentially means that we drop the non-negativity constrain on $\b r^p$, so that the right-hand sides in \eqref{eq:MPRK_PDRS_matrix} are allowed to be negative, and thus, the stage vectors and iterates of the MPRK scheme are not forced to stay positive anymore. However, the production and destruction terms may have to be adjusted in order to guarantee that the matrices $\b M^{(i)}$ and $\b M$ remain invertible.
	
	\begin{rem}\label{rem:negweights}
		Definition~\ref{def:MPRKdefn} is formulated for non-negative Runge--Kutta parameters.
		But MPRK schemes with negative Runge--Kutta parameters can be devised as well.
		In this case, the weighting of the production and destruction terms which get multiplied by the negative weight must be interchanged. To be precise, the index $\nu$ of the PWDs $\pi^{(i)}_\nu$ and $\sigma_\nu$ in the formula \eqref{eq:MPRK_PDRS} is replaced by the value of the \emph{index function} 
		\begin{equation}\label{eq:indexfun}
			\gamma(\nu,k,x)=\begin{cases}
				\nu, &x\geq 0\\
				k, & x< 0
			\end{cases}
		\end{equation} at $x=a_{ij}$ and $x=b_j$, respectively. Similarly, the index $k$ is replaced by $\gamma(k,\nu,a_{ij})$ and $\gamma(k,\nu,b_j)$ for $\pi^{(i)}_k$ and $\sigma_k$, respectively.
		
		This procedure will ensure the unconditional positivity of the scheme, but one may argue that this has an impact on the necessary requirements to obtain a certain order of accuracy. Fortunately, we will discover that this is not the case in Chapter~\ref{chap:order}. To avoid multiple case distinctions we demand for positive Runge--Kutta parameters in the remainder of this thesis.
	\end{rem}
	
	%
	
	Next, we want to explain in what sense the given definition of MPRK schemes generalizes the existing ones from \cite{KM18,IssuesMPRK}. First, MPRK schemes can be understood as NSARK methods using the splitting of the right-hand side mentioned in Remark~\ref{rem:PDStoAdditive}. Substituting this into \eqref{eq:MPRK_PDRS} and setting $t_j= t_n + c_j\dt$, we see
	\begin{align*} 
		y^{(i)}_k & = y^n_k +\! \dt \sum_{j=1}^{i-1} a_{ij}\!\! \left(\sum_{\substack{\nu=1\\ \nu\neq k}}^N  f_k^{[\nu]}(\byj, t_j) \frac{\yi_\nu}{\pi_{\nu}^{(i)}} + f_k^{[k]}(\byj, t_j)\frac{\yi_k}{\pi_{k}^{(i)}} +f_k^{[N+1]}(\byj, t_j)\!\right) \\
		&=y^n_k + \dt \sum_{j=1}^{i-1}a_{ij}\left( \sum_{\substack{\nu=1}}^N \frac{\yi_\nu}{\pi_{\nu}^{(i)}}  f_k^{[\nu]}(\byj) + f_k^{[N+1]}(\byj, t_j)\right)\\
		& =y^n_k + \dt \sum_{j=1}^{i-1} \sum_{\substack{\nu=1}}^{N+1} a^{[\nu]}_{ij}(\b y^n,\dt)  f_k^{[\nu]}(\byj),\\
		y^{n+1}_k & = y^n_k + \dt \sum_{j=1}^s b_j \left( \sum_{\substack{\nu=1\\ \nu\neq k}}^N  f_k^{[\nu]}(\byj)\frac{y^{n+1}_\nu}{\sigma_\nu}+ f_k^{[k]}(\byj) \frac{y^{n+1}_k}{\sigma_k} +f_k^{[N+1]}(\byj, t_j)\right)\\
		&= y^n_k + \dt \sum_{j=1}^s\sum_{\substack{\nu=1}}^{N+1} b^{[\nu]}_j(\b y^n,\dt)   f_k^{[\nu]}(\byj),
	\end{align*} where the solution-dependent coefficients are given by
	\begin{equation}\label{eq:pertcoeff}
		a_{ij}^{[\nu]}(\b y^n,\dt)=\begin{cases}
			a_{ij}\frac{\yi_\nu}{\pi_{\nu}^{(i)}}, & \nu\leq N,\\
			a_{ij}, & \nu =N+1
		\end{cases} \,\, \text{ and }\,\, b_j^{[\nu]}(\b y^n,\dt)=\begin{cases}
			b_j\frac{y^{n+1}_\nu}{\sigma_\nu}, & \nu\leq N,\\
			b_j, & \nu =N+1.
		\end{cases}
	\end{equation}\label{eq:NS_weightsMPRK}
	This means that the NSWs are \begin{equation}
		\gamma_\nu^{[i]}=\begin{cases}
			\frac{\yi_\nu}{\pi_{\nu}^{(i)}}, & \nu\leq N,\\
			1, & \nu =N+1
		\end{cases} \,\, \text{ and }\,\,\delta_\nu=\begin{cases}
			\frac{y^{n+1}_\nu}{\sigma_\nu}, & \nu\leq N,\\
			1, & \nu =N+1,
		\end{cases}
	\end{equation}
	see \eqref{eq:NSweights}.
	\begin{rem}\label{rem:negButcher}
		In view of the index function \eqref{eq:indexfun}, the NSWs for MPRK schemes based on RK methods with negative entries in the Butcher tableau not only depend on the step size, solution, and splitting of the right-hand side but also vary with its components. Hence, our formulation \eqref{eq:nsark} actually does not capture this case as we used vector notation. However, for the sake of simplicity and the reading flow, we rather discuss this special case in the particular sections than complicating the notation at this point.
	\end{rem}
	If in the context of an MPRK method, constant addends in the right-hand side splitting are treated as rest terms, then $\b F^{[N+1]}$ in Proposition~\ref{prop:nsarkautonom} will be integrated explicitly, which means that the condition $\gamma_{N+1}^{[i]},\delta_{N+1}^{[j]}=1$ is satisfied as this term is not multiplied with a PWD. Hence, with this convention it suffices to study autonomous problems for deriving order conditions. Moreover we are also in the position to apply Theorem~\ref{thm:ConsistencyImpConvergence}, if the production, destruction and rest terms as well as the PWDs are in $\mathcal C^1$ because of the following. The linear systems always possess a unique solution and the implicit function theorem tells us that the resulting incremental map is in $\mathcal C^1$, and hence, locally Lipschitz with respect to its second argument. We will later see that the PWDs fulfill these requirements for the particular MPRK schemes. 
	
	Hereafter, we present schemes for the conservative and autonomous PDS \eqref{eq:PDS}. The formulation for general PDRS is straightforward. In particular, \eqref{eq:MPRK_PDRS} reduces in this case to
	\begin{subequations}\label{eq:MPRK}
		\begin{align} \label{MPRKstage}
			\yi_k & = y^n_k + \dt \sum_{j=1}^{i-1} a_{ij} \sum_{\nu=1}^N \left( p_{k\nu}(\byj) \frac{\yi_\nu}{\pi_\nu^{(i)}} - d_{k\nu}(\byj)\frac{\yi_k}{\pi_k^{(i)}} \right),  \quad i=1,\dotsc,s, \\
			y^{n+1}_k & = y^n_k + \dt \sum_{j=1}^s b_j \sum_{\nu=1}^N \left( p_{k\nu}(\byj)\frac{y^{n+1}_\nu}{\sigma_\nu} - d_{k\nu}(\byj) \frac{y^{n+1}_k}{\sigma_k} \right),\quad  k=1,\dotsc,N.
		\end{align}
	\end{subequations}

	\subsection*{First Order MPRK Scheme}
	Based on the explicit Euler method, the first MPRK method, the so-called modified Patankar Euler (MPE) scheme was developed in \cite{BDM03}. It is proven to be first order accurate when applied to \eqref{eq:PDS} within the same work and reads 
	\begin{equation}\tag{MPE}\label{eq:MPE}
		\begin{aligned}
			y_k^{(1)} =&\, y_k^n,\\ 
			y_k^{n+1} =&\, y_k^n + \Delta t\sum_{\nu=1}^N\left(p_{k\nu}(\b y^{(1)})\frac{y_\nu^{n+1}}{y_\nu^n}-d_{k\nu}(\b y^{(1)})\frac{y_k^{n+1}}{y_k^n}\right),
		\end{aligned}
	\end{equation}
	for $k=1,\dots,N$, that is $\sigma_\nu=y_\nu^n$. Here, $\sigma_\nu$ is obviously a $\mathcal C^1$-mapping.
	\subsection*{Second Order MPRK Schemes}
	The explicit 2-stage RK method based on the Butcher array
	\begin{equation*}
		\begin{aligned}
			\def\arraystretch{1.2}
			\begin{array}{c|ccc}
				0 &  & \\
				\alpha & \alpha &  \\
				\hline
				& 1-\frac{1}{2\alpha} &\frac{1}{2\alpha}
			\end{array}
		\end{aligned}
	\end{equation*}
	is second order accurate. Moreover, the entries of the array are non-negative for $\alpha\geq \frac12$. With that as a starting point, the authors from \cite{KM18} derived a 1-parameter family of second order accurate MPRK schemes using $\pi_\nu^{(2)}=y_\nu^n$ and $\sigma_\nu= (y_\nu^{(2)})^{\frac{1}{\alpha}}(y_\nu^n)^{1-\frac{1}{\alpha}}$ for $i=1,\dotsc,N.$ For simplicity, we again present the resulting MPRK22($\alpha$) scheme for solving \eqref{eq:PDS}, \ie
	\begin{align}
		y_k^{(1)} =&\, y_k^n,\nonumber\\ 
		y_k^{(2)} =&\, y_k^n + \alpha\Delta t\sum_{\nu=1}^N\left(p_{k\nu}(\b y^{(1)})\frac{y_\nu^{(2)}}{y_\nu^n}-d_{k\nu}(\b y^{(1)})\frac{y_k^{(2)}}{y_k^n}\right),\nonumber\\ 
		y_k^{n+1} =&\, y_k^n + \Delta t\sum_{\nu=1}^N\left( \Biggl(\biggl(1-\frac1{2\alpha}\biggr) p_{k\nu}(\b y^{(1)})+\frac1{2\alpha} p_{k\nu}(\b y^{(2)})\Biggr)\frac{y_\nu^{n+1}}{(y_\nu^{(2)})^{\frac{1}{\alpha}}(y_\nu^n)^{1-\frac{1}{\alpha}}}\right.\nonumber\\
		& \left. - \Biggl(\biggl(1-\frac1{2\alpha}\biggr) d_{k\nu}(\b y^{(1)})+ \frac1{2\alpha} d_{k\nu}(\b y^{(2)})\Biggr)\frac{y_k^{n+1}}{(y_k^{(2)})^{\frac{1}{\alpha}}(y_k^n)^{1-\frac{1}{\alpha}}}\right)\tag{MPRK22}\label{eq:MPRK22b}
	\end{align}	
	for $k=1,\dots,N$ with $\alpha\geq\frac12$. Since $y_\nu^{(2)}=y_\nu^{(2)}(\b y^n)$ is in $\mathcal C^1$ due to the implicit function theorem, the same holds for the PWDs.
	
	\subsection*{Third Order MPRK Schemes}
	
	Assuming a non-negative Butcher tableau from an explicit 3-stage RK method, third order MPRK schemes have been constructed in \cite{KM18Order3} for solving \eqref{eq:PDS} using the denominator weights
	\begin{equation}\label{eq:PWDsMPRK43}
		\begin{aligned}
			\pi_\nu^{(2)} =\, &y_\nu^n,\\
			\pi_\nu^{(3)}  =\, &(y_\nu^{(2)})^{\frac{1}{p}}(y_\nu^n)^{1-\frac{1}{p}},\quad p=3a_{21}(a_{31}+a_{32})b_3,\\
			\sigma_k=\, &y_k^n+ \dt \sum_{\nu=1}^N\left( \left( \beta_1p_{k\nu}(\b y^n)+\beta_2p_{k\nu}(\b y^{(2)})\right)\frac{\sigma_\nu}{(y_\nu^{(2)})^{\frac{1}{a_{21}}}(y_\nu^n)^{1-\frac{1}{a_{21}}}}\right. \\
			& - \left.\left( \beta_1d_{k\nu}(\b y^n)+\beta_2d_{k\nu}(\b y^{(2)})\right)\frac{\sigma_k}{(y_k^{(2)})^{\frac{1}{a_{21}}}(y_k^n)^{1-\frac{1}{a_{21}}}}\right)
		\end{aligned}
	\end{equation}
	for $\nu,k=1,\dotsc,N$, $\beta_1=1-\beta_2$ and $\beta_2=\frac{1}{2a_{21}}$. Note, that solving another system of linear equations is necessary to calculate $\bm \sigma=(\sigma_1,\dotsc,\sigma_N)$. Hence, the resulting MPRK scheme may be based on 3-stage RK methods but can be viewed as 4-stage schemes, where we note that $\bm \sigma$ can be computed simultaneously with $\b y^{(3)}$. We also point out that there are no additional right-hand side evaluations required for computing $\bm\sigma$. The final scheme for conservative and autonomous PDS takes the form
	{\allowdisplaybreaks
		\begin{align}
			y^{(1)}_k &= y^n_k,\nonumber\\
			y^{(2)}_k &= y^n_k
			+  a_{21} \dt \sum_{\nu=1}^N\left(
			p_{k\nu}\bigl( \b y^n \bigr) \frac{y^{(2)}_\nu}{y^n_\nu}
			- d_{k\nu}\bigl( \b y^n \bigr) \frac{y^{(2)}_k}{y^n_k}
			\right),
			\nonumber\\
			y^{(3)}_k &= y^n_k
			+\Delta t \sum_{\nu=1}^N
			\Biggl(\left(a_{31} p_{k\nu}\bigl(\b y^n\bigr)+ a_{32} p_{k\nu} \bigl(\b y^{(2)}\bigr) \right)  \frac{  y_\nu^{(3)}
			}{\bigl(y_\nu^{(2)}\bigr)^{\frac1p } \bigl(y_\nu^n\bigr)^{1-\frac1p} }
			\nonumber\\
			& \qquad\qquad\qquad
			-\left(a_{31} d_{k\nu}\bigl(\b y^n\bigr)+ a_{32} d_{k\nu} \bigl(\b y^{(2)}\bigr) \right)  \frac{  y_k^{(3)}
			}{\bigl(y_k^{(2)}\bigr)^{\frac1p } \bigl(y_k^n\bigr)^{1-\frac1p} }
			\Biggr),
			\nonumber\\
			\hphantom{baselineskip}&\nonumber\\
			\sigma_k &= y_k^n + \Delta t \sum_{\nu=1}^N
			\Biggl(\left( \beta_1 p_{k\nu} \bigl( \b y^n\bigr) +\beta_2 p_{k\nu} \bigl(\b y^{(2)}\bigr)  \right) \frac{\sigma_\nu}{\bigl(y_\nu^{(2)} \bigr)^{\frac1q}
				\bigl(y_\nu^n\bigr)^{1-\frac1q}}
			\nonumber\\ & \qquad-
			\left( \beta_1 d_{k\nu} \bigl( \b y^n\bigr) +\beta_2 d_{k\nu} \bigl(\b y^{(2)}\bigr)  \right) \frac{\sigma_k}{\bigl(y_k^{(2)} \bigr)^{\frac1q}
				\bigl(y_k^n\bigr)^{1-\frac1q}}\Biggr),
			\nonumber\\
			y^{n+1}_k &= y^n_k
			+ \dt \sum_{\nu=1}^N \Biggl(
			\left( b_1 p_{k\nu}\bigl( \b y^n \bigr) +b_2p_{k\nu}\bigl( \b y^{(2)} \bigr)
			+ b_3 p_{k\nu}\bigl( \b y^{(3)} \bigr)
			\right) \frac{y^{n+1}_\nu}{\sigma_\nu}
			\nonumber\\&\qquad\qquad\qquad
			- \left( b_1 d_{k\nu}\bigl( \b y^n \bigr) +b_2d_{k\nu}\bigl( \b y^{(2)} \bigr)
			+ b_3 d_{k\nu}\bigl( \b y^{(3)} \bigr)
			\right) \frac{y^{n+1}_k}{\sigma_k}
			\Biggr),\tag{MPRK43}
			\label{eq:MPRK43-family}
	\end{align}}
	where $p=3 a_{21}\left(a_{31}+a_{32} \right)b_3,\; q=a_{21},\;\beta_2=\frac{1}{2a_{21}}$ and $\beta_1= 1-\beta_2$. As before, the PWDs are in $\mathcal C^1$, if the production and destruction terms are.
	\subsubsection{MPRK43($\alpha, \beta$)}
	All entries of the Butcher array
	\begin{equation}\label{array:MPRK43alphabeta}
		\begin{aligned}
			\def\arraystretch{1.2}
			\begin{array}{c|ccc}
				0 &  & & \\
				\alpha & \alpha & & \\
				\beta & \frac{3\alpha\beta (1-\alpha)-\beta^2}{\alpha(2-3\alpha)}& \frac{\beta (\beta-\alpha)}{\alpha(2-3\alpha)}& \\
				\hline
				& 1+\frac{2-3(\alpha+\beta)}{6 \alpha \beta } &\frac{3 \beta-2}{6\alpha (\beta-\alpha)} & \frac{2-3\alpha}{6\beta(\beta-\alpha)}
			\end{array}
		\end{aligned}
	\end{equation}
	with
	\begin{equation}\label{eq:cond:alpha,beta}
		\begin{cases}
			2/3 \leq \beta \leq 3\alpha(1-\alpha)\\
			3\alpha(1-\alpha)\leq\beta \leq 2/3 \\
			\tfrac{3\alpha-2}{6\alpha-3}\leq \beta \leq 2/3
		\end{cases}
		\text{ for }
		\begin{cases}
			1/3 \leq \alpha<\frac23,\\
			2/3 < \alpha<\alpha_0,\\
			\alpha>\alpha_0,
		\end{cases}
	\end{equation}
	and $\alpha_0\approx 0.89255$  are non-negative \cite[Lemma 6]{KM18Order3}, see Figure \ref{fig:RK3_pos} for an illustration of the feasible domain. 
	\begin{figure}[!h]
		\centering
		\includegraphics[width=.7\textwidth]{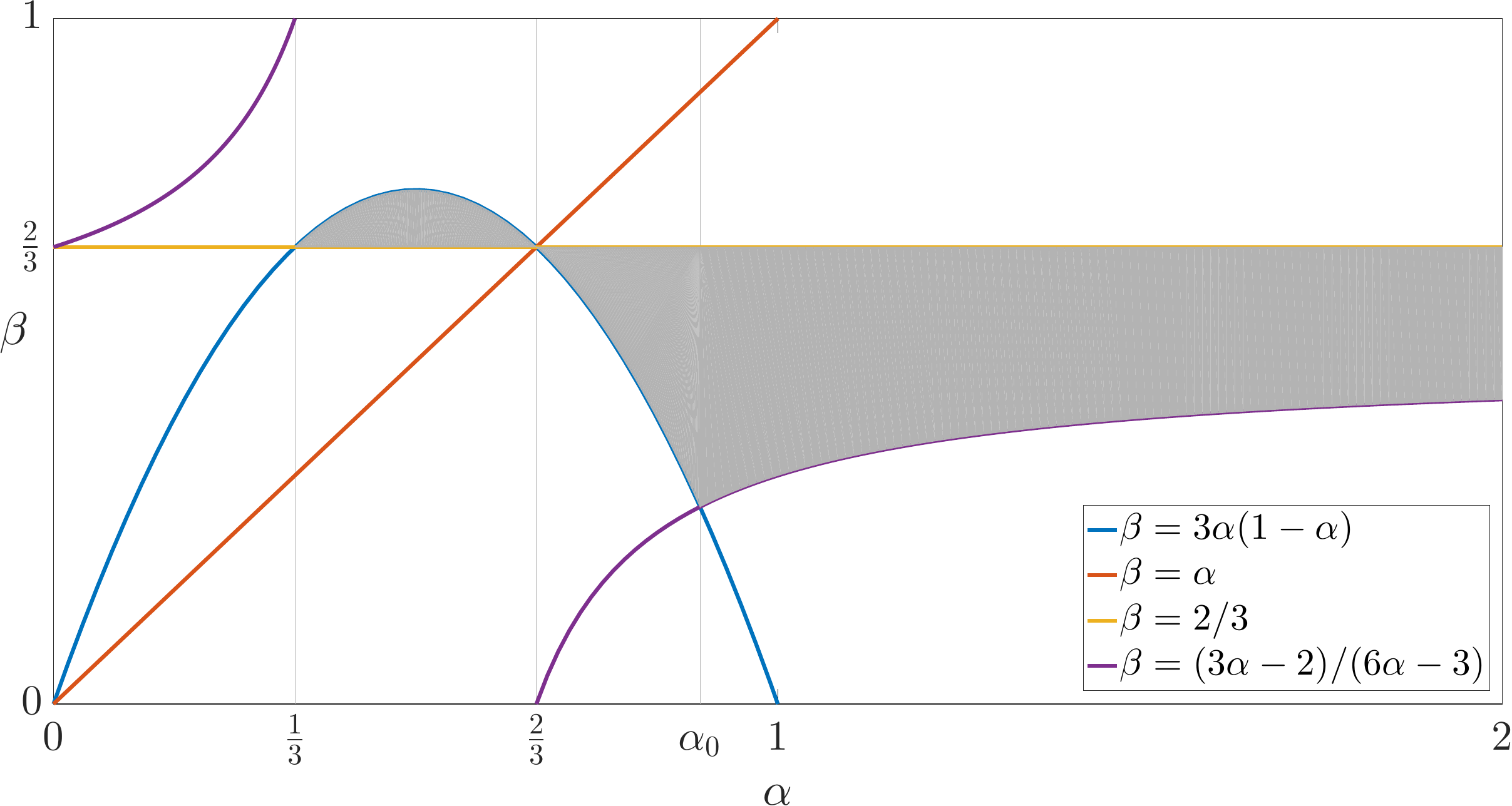}
		\caption{The gray area represents all $(\alpha,\beta)$ pairs which fulfill the conditions \eqref{eq:cond:alpha,beta}, \ie for which the Butcher tableau \eqref{array:MPRK43alphabeta} is non-negative \cite{KM18Order3}.}\label{fig:RK3_pos}
	\end{figure}
	
	The resulting MPRK43$(\alpha,\beta)$ method is determined by \eqref{eq:MPRK43-family} using
	\eqref{array:MPRK43alphabeta} and
	\begin{equation}\label{eq:parametersMPRK43alphabeta}
		\begin{aligned}
			p&=3a_{21}(a_{31}+a_{32})b_3=\alpha\frac{2-3\alpha}{2(\beta-\alpha)},&
			q&=a_{21}=\alpha,\\
			\beta_2&=\frac{1}{2a_{21}}= \frac{1}{2\alpha},& \beta_1&=1-\beta_2=1-\frac{1}{2\alpha}.
		\end{aligned}
	\end{equation}

	\subsubsection{MPRK43($\gamma$)}
	It was also proven in \cite[Lemma 6]{KM18Order3} that all entries of the tableau
	\begin{equation}\label{array:MPRK43gamma}
		\begin{aligned}
			\def\arraystretch{1.2}
			\begin{array}{c|ccc}
				0 &  & & \\
				\frac23 & \frac23 & & \\
				\frac23 & \frac23-\frac{1}{4\gamma}& \frac{1}{4\gamma}& \\
				\hline
				& \frac14 &\frac34-\gamma & \gamma
			\end{array}
		\end{aligned}
	\end{equation}
	are non-negative for $\frac38\leq \gamma \leq \frac34.$ The corresponding MPRK scheme is denoted by MPRK43$(\gamma)$ and can be obtained from \eqref{eq:MPRK43-family} by substituting \eqref{array:MPRK43gamma} and
	\begin{equation}\label{eq:parametersMPRK43gamma}
		\begin{aligned}
			p&=3a_{21}(a_{31}+a_{32})b_3=\frac43\gamma,&
			q&=a_{21}=\frac23,\\
			\beta_2&=\frac{1}{2a_{21}}=\frac34, &\beta_1&=1-\beta_2=\frac14.
		\end{aligned}
	\end{equation}

			\section[Strong-Stability Preserving MPRK]{Strong-Stability Preserving MPRK}
			Strong-stability preserving Runge--Kutta (SSPRK) methods were introduced in \cite{SO88} and developed for the time integration of the semi-discretization of hyperbolic conservation laws. The main idea was to rewrite an explicit RK method into \emph{Shu--Osher form}. With that, the authors in \cite{SO88} present higher order methods that preserve any convex functional bound such as positivity or total variation diminishing (TVD) property whenever the forward Euler method possesses the respective property. To obtain unconditional positivity, strong-stability preserving modified Patankar--Runge--Kutta (SSPMPRK) methods were constructed in \cite{SSPMPRK2} and proven to be of second order. Later, also third order methods were constructed \cite{SSPMPRK3}. Moreover, the schemes are also conservative and there exist analogues of Lemma~\ref{lem:MPRK} and Lemma~\ref{lem:MPRKpos} for these methods.
			
			In order to adapt SSPMPRK methods into our framework of NSARK schemes, we would have to introduce the ARK methods in Shu--Osher form and then consider solution-dependent coefficients. This together with the corresponding generalization of the results from \cite{SSPMPRK2,SSPMPRK3,HIKMS22} along the theory developed in \cite{NSARK} is object to future work. We also want to note here that in \cite{SSPMPRK2,SSPMPRK3}, the SSPMPRK methods were also used as time integrators in the context of reactive Euler equations.
			\paragraph{Second Order SSPMPRK Schemes}
			The second order SSPMPRK scheme for solving \eqref{eq:PDS}, introduced in \cite{SSPMPRK2}, is given by
			\begin{equation}\tag{SSPMPRK2}\label{eq:SSPMPRK2}
				\begin{aligned}
					y_i^{(1)}={}& y_i^n+\beta \dt\left(\sum_{j=1}^Np_{ij}(\b y^n)\frac{y_j^{(1)}}{y_j^n}- \sum_{j=1}^Nd_{ij}(\b y^n)\frac{y_i^{(1)}}{y_i^n}\right),\\
					y_i^{n+1}={}& (1-\alpha)y_i^n+\alpha y_i^{(1)}+\dt\Biggl(\sum_{j=1}^N\left(\beta_{20}p_{ij}(\b y^n)+\beta_{21}p_{ij}(\b y^{(1)})\right)\frac{y_j^{n+1}}{(y_j^n)^{1-s}(y_j^{(1)})^s} \\& - \sum_{j=1}^N\left(\beta_{20}d_{ij}(\b y^n)+\beta_{21}d_{ij}(\b y^{(1)})\right)\frac{y_i^{n+1}}{(y_i^n)^{1-s}(y_i^{(1)})^s}\Biggr),
				\end{aligned}
			\end{equation}
			where $\beta_{20}=1-\frac{1}{2\beta}-\alpha\beta$, $\beta_{21}=\frac{1}{2\beta}$ and $s=\frac{1-\alpha\beta+\alpha\beta^2}{\beta(1-\alpha\beta)}$. Thereby, the free parameters $\alpha$ and $\beta$ are subject to
			\begin{equation}
				0\leq \alpha\leq 1,\quad \beta>0,\quad \alpha\beta+\frac{1}{2\beta}\leq 1.\label{eq:alphabeta_conditions}
			\end{equation}
			We refer to the above scheme as SSPMPRK2($\alpha,\beta$).
			\paragraph{Third Order SSPMPRK Schemes}
			The third order method applied to \eqref{eq:PDS} can be written as
			{\allowdisplaybreaks
				\begin{align}
					y_i^{(1)}=&\alpha_{10}y_i^n+\beta_{10} \dt\left(\sum_{j=1}^Np_{ij}(\b y^n)\frac{y_j^{(1)}}{y_j^n}- \sum_{j=1}^Nd_{ij}(\b y^n)\frac{y_i^{(1)}}{y_i^n}\right),\nonumber\\
					\rho_i=&n_1y_1^{(1)}+n_2y_i^n\left(\frac{y_i^{(1)}}{y_i^n}\right)^2,\nonumber\\
					y_i^{(2)}=&\alpha_{20}y_i^n+\alpha_{21} y_i^{(1)}+\dt\Biggl(\sum_{j=1}^N\left(\beta_{20}p_{ij}(\b y^n)+\beta_{21}p_{ij}(\b y^{(1)})\right)\frac{y_j^{(2)}}{\rho_j}\nonumber\\& - \sum_{j=1}^N\left(\beta_{20}d_{ij}(\b y^n)+\beta_{21}d_{ij}(\b y^{(1)})\right)\frac{y_i^{(2)}}{\rho_i}\Biggr),\nonumber\\
					\gamma_i=&\eta_1y_i^n+\eta_2 y_i^{(1)}+\dt\Biggl(\sum_{j=1}^N\left(\eta_3p_{ij}(\b y^n)+\eta_{4}p_{ij}(\b y^{(1)})\right)\frac{\gamma_j}{(y_j^n)^{1-s}(y_j^{(1)})^s}\\&-\sum_{j=1}^N\left(\eta_3d_{ij}(\b y^n)+\eta_{4}d_{ij}(\b y^{(1)})\right)\frac{\gamma_i}{(y_i^n)^{1-s}(y_i^{(1)})^s}\Biggr),\nonumber\\
					\sigma_i=&\gamma_i+\zeta y_i^n\frac{y_i^{(2)}}{\rho_i},\nonumber\\
					y_i^{n+1}=&\alpha_{30}y_i^n+\alpha_{31} y_i^{(1)}+\alpha_{32}y_i^{(2)}\nonumber\\
					&+\dt\Biggl(\sum_{j=1}^N\left(\beta_{30}p_{ij}(\b y^n)+\beta_{31}p_{ij}(\b y^{(1)})+\beta_{31}p_{ij}(\b y^{(2)})\right)\frac{y_j^{n+1}}{\sigma_j}\\&- \sum_{j=1}^N\left(\beta_{30}d_{ij}(\b y^n)+\beta_{31}d_{ij}(\b y^{(1)})+\beta_{32}d_{ij}(\b y^{(2)})\right)\frac{y_i^{n+1}}{\sigma_i}\Biggr),\tag{SSPMPRK3}\label{eq:SSPMPRK3}
			\end{align}}
			where we use the parameters
			\begin{equation}\label{eq:SSPMPRK3Parameters}
				\begin{aligned}
					\alpha_{10}&=1,&\alpha_{20} &= 9.2600312554031827\cdot10^{-1}, \\\alpha_{21} &= 7.3996874459681783\cdot10^{-2},&
					\alpha_{30} &= 7.0439040373427619\cdot10^{-1},\\
					\alpha_{31} &= 2.0662904223744017\cdot10^{-10},
					&\alpha_{32} &= 2.9560959605909481\cdot10^{-1},\\
					\beta_{10} &= 4.7620819268131703\cdot10^{-1},
					&\beta_{20} &= 7.7545442722396801\cdot10^{-2},\\
					\beta_{21} &= 5.9197500149679749\cdot10^{-1},&
					\beta_{30} &= 2.0044747790361456\cdot10^{-1},\\
					\beta_{31} &= 6.8214380786704851\cdot10^{-10},
					&\beta_{32} &= 5.9121918658514827\cdot10^{-1},\\
					\zeta&=0.62889380778287493358,
					&\eta_1&=0.37110619221712506642 - \eta_2, \\
					\eta_3&=-1.2832127371313151768\eta_2& 	\eta_4&=2.2248760403511226405,\\ &\hphantom{=} + 0.6146025595987523739,&
					\\
					n_1&=0.25690460257320105191, &
					n_2&=1-n_1
				\end{aligned}
			\end{equation}
			in accordance with \cite{SSPMPRK3}.
			Here, $\eta_2$ is a free parameter satisfying $\eta_2\in[0,r_1]$ with $r_1=0.37110619221712506642$, so that we refer to this scheme as SSPMPRK3($\eta_2$). For more details on the parameters we refer to the Maple code in the reproducibility
			repository \cite{repoSSPMPRK}.
			
			\section{Modified Patankar Deferred Correction}
			Arbitrarily high-order conservative and positive modified Patankar Deferred
			Correction schemes (MPDeC) were introduced in \cite{MPDeC} which are based on the Deferred Correction (DeC) approach developed in \cite{dutt2000spectral}. 
			To that end, a time step $[\tn, t_{n+1}]$ is transformed to $[0,1]$ and then divided into $M$ subintervals determined by $0=t_{n,0}<\cdots <t_{n,M}=1$. The idea of the scheme is to mimic the Picard iterations on a discrete level
			as follows. At each subtime step $t_{n,m}$ an approximation $y^{m}$ is calculated.
			An iterative procedure of $K$ correction steps improves the approximation by one
			order of accuracy at each iteration.
			The modified Patankar-trick is introduced inside the basic scheme to guarantee  positivity and conservation of the intermediate approximations.
			
			The MPDeC correction steps
			can be rewritten for $k=1,\dots,K$, $m =1,\dots, M$ and $i=1,\dotsc,N$ as
			\begin{equation}
				\tag{MPDeC}
				\label{eq:explicit_dec_correction}
				y_i^{m,(k)}=y^0_i +\sum_{r=0}^M \theta_r^m \dt  \sum_{j=1}^N
				\left( p_{ij}(y^{r,(k-1)})
				\frac{y^{m,(k)}_{\gamma(j,i, \theta_r^m)}}{y_{\gamma(j,i, \theta_r^m)}^{m,(k-1)}}
				- d_{ij}(y^{r,(k-1)})  \frac{y^{m,(k)}_{\gamma(i,j, \theta_r^m)}}{y_{\gamma(i,j, \theta_r^m)}^{m,(k-1)}} \right),
			\end{equation}
			where $\theta_r^m=\int_{0}^{t_{n,m}}\varphi_r(t)\dd t$ are the \emph{correction weights}, and \[\gamma(j,i,\theta_r^m)=\begin{cases}
				j, & \theta_r^m\geq 0\\
				i, & \theta_r^m< 0
			\end{cases}\] is
			the index function \eqref{eq:indexfun}. Here, $\varphi_r$ is the $r$th Lagrangian polynomial defined by the subtime nodes $\lbrace t_{n,m} \rbrace_{m=0}^M$. As a result of $\theta_r^0=0$, the initial states $y_i^{0,(k)}=y_i^n$
			are identical for any correction $k$.  The new numerical solution at time $\tn+\dt$ is $\b y^{n+1}=\b y^{M,(K)}$.
			\newpage
			\begin{rem}\label{rem:MPDeC_NSARK}
				Formally, MPDeC methods can be interpreted as RK schemes by viewing the correction steps as additional stages. Consequently, MPDeC methods can be written as NSARK schemes. However, similarly to the case discussed in Remark~\ref{rem:negButcher}, the NSWs of MPDeC depend on the components of the vector $\b y^{m,(k)}$ whenever the correction weights are negative, which is already the case for $K>2$. Nevertheless, since the weights are similar to those of MPRK methods we can conclude that the order of accuracy of MPDeC methods is also maintained for non-autonomous PDS and that Theorem~\ref{thm:ConsistencyImpConvergence} can be applied.
			\end{rem}
			
			The order of accuracy of the MPDeC scheme is the minimum between $K$ and the accuracy of the quadrature formula given by the $M$ subtime steps. In view of the existing literature, we will focus on equispaced (\eq) and Gauss--Lobatto (\gl) subtime steps \cite{MPDeC}. To obtain order $p$, a number of $K=p$ iterations is required, while we need  $M=\max\{p-1,1\}$ \eq\, subtime steps or $M=\left\lceil \frac{p}{2} \right \rceil$ \gl\, subtime steps. To indicate the quadrature formula we introduce the notation MPDeCGL($p$) and MPDeCEQ($p$) for MPDeC methods of order $p$ using \gl\, or \eq\, subtime steps, respectively.

			Obviously \eqref{eq:explicit_dec_correction} is due to this iterative process more complicated than the previous schemes, especially since the index function changes productive and destructive part inside the underlying PDS. However, these methods are arbitrary high order, unconditionally positive and conservative. Additionally, they have been applied successfully in the context of the shallow water equations guaranteeing a positive water height \cite{CMOT21}.
			
			\section{Geometric Conservative}\label{sec:GeCo}
			A class of numerical methods that preserve all linear invariants and still guarantee positivity is given by GeCo schemes introduced in \cite{martiradonna2020geco}. These methods fall in the class of  non-standard integrators \cite{mickens1994nonstandard} as they result as  non-standard versions of explicit first and second order Runge--Kutta schemes, where the advancement in time is modulated by a nonlinear functional dependency on the temporal step size and on the  approximation itself. 
			The step size modification thereby guarantees the numerical solution to be unconditionally positive while keeping the accuracy of the underlying method. GeCo schemes are applied to  general biochemical systems \cite{formaggia2011positivity,BBKS2007} 
			\begin{equation}\label{eq:bio}
				\b y'=\b f(\b y,t)=\b S\b r(\b y,t),  \qquad \b y(0)=\b y^0,
			\end{equation}
			where  $\b S\in\mathbb R^{N\times M}$ is the \emph{stoichiometric} matrix with entries $s_{ij}$ for $i=1,\dotsc, N$ and $j=1,\dotsc,M$, and $\b r(\b y)=(r_1(\b y),\dots,r_M(\b y))^T$ is the vector of the \emph{reaction} functions. The following assumptions, stated in \cite{formaggia2011positivity}, assure the well-posedness of the system \eqref{eq:bio} and the positivity of the solutions.
			\begin{enumerate}
				\item For $j=1,\dots,M$ we have $r_j\in \mathcal{C}^0 \left(\mathbb{R}_{\geq 0}^N,\mathbb{R}_{\geq0}\right)$ and $r_j(\cdot,t)$ is locally Lipschitz in $\mathbb{R}^N$, uniformly in $t$.
				\item There holds $\b r(\b y,t)>\b 0$ if $\b y>\b 0$, and     $\b r(\b y,t)=\b 0$ if $\b y=\mathbf 0$ for all $t>0$.
				\item If $s_{ij}<0$, there exists a 
				$q_j\in \mathcal{C}^0 \left(\mathbb{R}_{\geq 0}^N,\mathbb{R}_{\geq0} \right)$
				such that $r_j(\b y,t)=q_j(\b y,t)y_i$.
			\end{enumerate}
			The GeCo methods are of the form
			\begin{equation}\tag{GeCo}\label{eq:GeCoscheme}
				\begin{aligned}
					\byi&=\b y^n+\phi_i(\b y^n,\tn,\dt) \dt\sum_{j=1}^{i-1}a_{ij}\b f(\byj),\quad i=1,\dotsc,s,\\
					\b y^{n+1}&=\b y^n+\phi_{n+1}(\b y^n,\tn,\dt) \dt\sum_{j=1}^sb_j\b f(\byj),
				\end{aligned}
			\end{equation}
			see \cite{martiradonna2020geco}, where we point out that our formulation includes non-autonomous biochemical problems.
			Note that $\phi$ here corresponds to the function $\Phi$ of \cite{martiradonna2020geco} divided by $\dt$, and that the value of $\phi_1$ has no effect since $a_{1j}=0$.
			The idea is to choose the functions  $\phi_i$ and $\phi_{n+1}$
			in a way that guarantees the positivity of the stages
			and the updated solution.  At the same time, these functions
			must be chosen in a way that does not compromise the order
			of accuracy.
			Up to now, only conditions for first and second order GeCo schemes are available. 
			
			To interpret \eqref{eq:GeCoscheme} as a non-standard RK (NSRK) method,
			we absorb the factors $\phi_i, \phi_{n+1}$ into the RK coefficients, which we can write formally in the notation of Section~\ref{sec:NSARK} via the coefficients:
			\begin{equation}\label{eq:Geco_NS}
				a^{[1]}_{ij}(\b y^n,\tn,\dt)=a_{ij}\phi_i(\b y^n,\tn, \dt),\quad b^{[1]}_j(\b y^n,\tn, \dt)=b_j\phi_{n+1}(\b y^n,\tn, \dt)
			\end{equation}
			for  $i,j=1,\dotsc, s.$ This means that the NSWs are $\gamma_i^{[1]}=\phi_i$ and $\delta_i =\phi_{n+1}$.
			\paragraph{First Order GeCo Scheme}For the construction of the NSWs of GeCo methods, the vector field $\b f(\b y,t)=\b S\mathbf r(\b y,t)$ is split into production and destruction parts as
			\begin{equation}\label{eq:biosplitting}
				\b f(\b y,t)=\b f^{[P]}(\b y,t)-\b f^{[D]}(\b y,t),\quad \b f^{[P]}(\b y,t)=\b S^+\b r(\b y,t),\quad \b f^{[D]}(\b y,t)=\b S^-\b r(\b y,t)
			\end{equation}
			with $\b S^+,\b S^-\geq \b 0$.
			The first order GeCo scheme (GeCo1) applied to a general biochemical system \eqref{eq:bio}, \eqref{eq:biosplitting}
			is defined as
			\begin{equation}
				\b y^{n+1}=\b y^n+\dt\varphi\left(\dt \sum_{i=1}^N  \dfrac{f_i^{[D]}(\b y^n,\tn)}{y_i^n}\right)\b f(\b y^n,\tn),\tag{GeCo1}\label{eq:GeCo1scheme}
			\end{equation}
			where the function $\varphi\in \mathcal C^2$ is defined as
			\begin{equation}\label{eq:phi}
				\varphi(x)=\begin{cases}
					\dfrac{1-e^{-x}}{x},& x>0,\\
					1, &x=0.
			\end{cases}   \end{equation}
			In the notation of a general GeCo method, we have \[\phi_{n+1}(\b y^n,\tn ,\dt)=\varphi\left(\dt \sum_{i=1}^N  \dfrac{f_i^{[D]}(\b y^n,\tn)}{y_i^n}\right).\] 
			\begin{rem}\label{rem:GeCononautonomous}
				Even though \eqref{eq:GeCo1scheme} can be interpreted as an NSARK method with $\varphi$ being the NSW, the scheme is not an additive method since the whole right-hand side $\b f$ is weighted by the same factor. Hence, we are not in the position to apply Proposition~\ref{prop:nsarkautonom} directly. However, considering the autonomous problem with $\b F(\b Y)=(\b f(\b Y),1)^T$ and $\b Y=(\b y,t)^T$, one can see from \eqref{eq:GeCo1scheme} that the last component of the method reads \[t_{n+1}=\tn+\dt \varphi\left(\dt \sum_{i=1}^N  \dfrac{f_i^{[D]}(\b y^n,\tn)}{y_i^n}\right),\] which is why it is not clear whether or not the condition for first order from \cite{martiradonna2020geco} is sufficient also for non-autonomous problems. 
			\end{rem}
			We also note that the NSW is in $\mathcal C^1$ whenever $\b f^{[D]}\in \mathcal C^1$, so that we can also apply Theorem~\ref{thm:ConsistencyImpConvergence} to prove the order of convergence.
			\paragraph{Second Order GeCo Scheme}
			The second order GeCo (GeCo2) scheme for a general biochemical system \eqref{eq:bio}, \eqref{eq:biosplitting} is based on Heun's methods and takes the form
			\begin{equation}\tag{GeCo2}\label{eq:GeCo2scheme}
				\begin{aligned}
					\b y^{(1)}&=\b y^n,\\
					\b y^{(2)}&=\b y^n+\dt\varphi\left(\dt \sum_{i=1}^N  \dfrac{f_i^{[D]}(\b y^n,\tn)}{y_i^n}\right)\b f(\b y^n,\tn),\\
					\b y ^{n+1}&=\b y^n + \dfrac{\dt}{2} \varphi\left(\dt \sum_{i=1}^N \dfrac{w_i^+(\b y^n,\tn)}{y^n_i} \right) \left(\mathbf{f}(\b y^n,\tn)+ \mathbf f(\mathbf y^{(2)},\tn+\dt)\right),
				\end{aligned}
			\end{equation}
			where
			\begin{equation*}
				w_i^+(\b y^n,\tn)=\max(0,w_i(\b y^n,\tn)), \quad i=1,\dots,N
			\end{equation*}
			with 
			\begin{equation*}
				\mathbf w(\b y^n,\tn)=2\varphi\left(\dt \sum_{j=1}^N  \dfrac{f_i^{[D]}(\b y^n,\tn)}{y_i^n}\right)\b f(\mathbf y^n,\tn) - \b f(\b y^n,\tn)-\b f(\b y^{(2)},\tn+\dt).
			\end{equation*}
			Since $\varphi$ is in $\mathcal C^1$ we see that $w_i^+$ is the composition of locally Lipschitz continuous mappings if $\b f^{[D]}\in \mathcal C^1$, and hence, itself locally Lipschitz continuous. Thus, we can apply Theorem~\ref{thm:ConsistencyImpConvergence} to deduce the order of convergence. 
			
			Moreover, as discussed in Remark~\ref{rem:GeCononautonomous} for GeCo1, the order conditions derived in \cite{martiradonna2020geco} for GeCo2 hold for autonomous problems and it is not clear if the method is of second order for non-autonomous problems. Investigating this question for GeCo methods will be part of future work.
			\section{Generalized BBKS}
			The generalized BBKS  (gBBKS) schemes, named after the authors Bruggeman, Burchard, Kooi and Sommeijer, were developed in \cite{BBKS2007,BRBM2008, gBBKS} and represent a class of schemes that are unconditionally positive while preserving all linear invariants of the underlying ordinary differential equation $\b y'=\b f(\b y, t)$. Similarly to GeCo methods, the idea is to weight the function $\b f\colon \R^N\times \R\to \R^N$ with a positivity-preserving factor. As a result, gBBKS schemes can also be interpreted as NSRK methods with the positivity-preserving factor being the NSW. So far, first and second order accurate methods have been constructed which we briefly review in the following.
			\paragraph{First Order gBBKS Schemes}
			The first order gBBKS schemes (gBBKS1) can be written as
			\begin{equation}\tag{gBBKS1}\label{eq:gBBKS1Intro}
				\b y^{n+1}=\b y^n + \dt\b f(\b y^n, \tn) \Bigg(\prod_{m\in M^{n}} \frac{y^{n+1}_m}{\sigma^{n}_m}\Bigg)^{\ns r^{n}},
			\end{equation}
			where $r^n,\sigma_m^n>0$ are free parameters, but need to be chosen independently of $\b y^{n+1}$, and
			\begin{equation*}
				M^{n}=\{ m\in \{1,\dotsc,N\} \mid  f_m(\b y^n,\tn)<0\}.
			\end{equation*}
			For instance, the BBKS1 scheme from \cite{BBKS2007,gBBKS} is given by setting $\sigma^n_m=y^n_m$ and $r^n=1$. As discussed for GeCo methods in Remark~\ref{rem:GeCononautonomous}, it is not straightforward to see whether or not the proven first order of \eqref{eq:gBBKS1Intro} is maintained for non-autonomous problems.
			
			Moreover, as the number of factors in the NSW $\Bigg(\prod_{m\in M^{n}} \frac{y^{n+1}_m}{\sigma^{n}_m}\Bigg)^{\ns r^{n}}$ depends on $\b y^n$, further investigation is needed to conclude the order of convergence of the method by means of Theorem~\ref{thm:ConsistencyImpConvergence}. 
			\paragraph{Second Order gBBKS Schemes}
			The second order gBBKS schemes, denoted by gBBKS2($\alpha$), have a free parameter $\alpha\geq \frac12$ and can be written as
			\begin{equation}\tag{gBBKS2}\label{eq:gBBKS2Intro}
				\begin{aligned}
					\b y^{(1)}&=\b y^n,\\
					\b y^{(2)} &= \b y^n+\alpha \dt\b f(\b y^{n},\tn) \Bigg(\prod_{j\in J^{n}} \frac{y^{(2)}_j}{\pi^{n}_j}\Bigg)^{\ns q^{n}},\\
					\b y^{n+1} &= \b y^n+\dt \left(\Big(1-\frac{1}{2\alpha}\Big) \b f(\b y^{n},\tn)+ \frac{1}{2\alpha}\b f(\b y^{(2)},\tn+\alpha \dt)\right)\Bigg(\prod_{m\in M^{n}} \frac{y^{n+1}_m}{\sigma^{n}_m}\Bigg)^{\ns r^{n}}
				\end{aligned}
			\end{equation}
			with  $\pi_j^n,q^n>0$ being free parameters chosen independently of $\b y^{(2)}$, while we require $\sigma_m^n,r^n>0$ to be independent of $\b y^{n+1}$. To give an example, the BBKS2(1) scheme from \cite{BRBM2008, gBBKS} uses $\pi^n_m=\sigma^n_m=y^n_m$ and $q^n=r^n=1$. Moreover, the sets $J^n$ and $M^n$ are given by
			\begin{subequations}
				\begin{align*}
					J^{n}&=\left\{ j\in \{1,\ldots,N\}\mid f_j(\b y^{n},\tn)<0\right\},\\
					M^{n}&=\left\{ m\in \{1,\ldots,N\}\;\Big |\; \Big(1-\frac{1}{2\alpha}\Big) f_m(\b y^{n},\tn)+ \frac{1}{2\alpha}f_m(\b y^{(2)},\tn+\alpha \dt)<0\right\}.
				\end{align*}
			\end{subequations}
			We want to note that $M^n$ always refers to the last step of the corresponding method. As before, the same concerns arise for \eqref{eq:gBBKS2Intro} when it comes to the order of convergence in general and in the case of non-autonomous problems. 
			
			

\chapter{Order Conditions for NSARK Methods}\label{chap:order}
\chaptermark{Order Conditions for NSARK methods}
In this chapter we are interested in deriving order conditions for general NSARK methods. As an application of the upcoming theory, we will reproduce known order conditions for MPRK and GeCo methods from \cite{KM18,KM18Order3,martiradonna2020geco}. Additionally, we present reduced conditions for MPRK and GeCo schemes up to order four.

The main idea is to follow \cite{B16} and to adapt Theorem~\ref{thm:ARKcon} for schemes with solution-dependent coefficients.

\section{Main Result on Order Conditions}
In the appendix, we prove modified versions of theorems from \cite{B16} to demonstrate that for an NSARK scheme we can take the formula for $u$ from \eqref{eq:cond} and replace the constant coefficients with the solution-dependent ones from \eqref{eq:nsark}, i.\,e.\ that the solution-dependent $u=u(\tau,\b y^n,\dt)$ in the case of an NSARK method is given by
\begin{equation}\label{eq:pertcond}
	\begin{aligned}
		u(\tau,\b y^n,\dt)&=\sum_{\nu=1}^N\sum_{i=1}^sb_i^{[\nu]}(\b y^n,\dt) g_i^{[\nu]}(\tau,\b y^n,\dt),\\ 
		g_i^{[\nu]}(\rt[]^{[\mu]},\b y^n,\dt)&=\delta_{\nu\mu},\quad \nu,\mu=1,\dotsc,N, \\
		g_i^{[\nu]}([\tau_1,\dotsc,\tau_l]^{[\mu]},\b y^n,\dt)&=\delta_{\nu\mu}\prod_{j=1}^ld_i(\tau_j,\b y^n,\dt),\quad \nu,\mu=1,\dotsc,N\text{ and } \\
		d_i(\tau,\b y^n,\dt)&=\sum_{\nu=1}^N\sum_{j=1}^sa_{ij}^{[\nu]}(\b y^n,\dt) g_j^{[\nu]}(\tau,\b y^n,\dt).
	\end{aligned}
\end{equation} 
As a result of this claim, we would be in the position to formulate an analogous condition to \eqref{eq:c=1/gamma} for an NSARK method to have an order of at least $p$.

To prove our main result, we introduce in Theorem \ref{thm:main} a generalization of NB-series, in which the coefficients of the series are allowed to depend on $\b y^n$ and $\dt$.  We note that such a series is \emph{not} a Taylor expansion in $\dt$, but instead can be understood as an asymptotic expansion in expressions depending on powers of $\dt$ and the solution-dependent coefficients of the Butcher tableau. As a result of this approach, we do not require at this point any regularity of $a_{ij}^{[\nu]}(\b y^n,\dt)$ or $b^{[\nu]}_j(\b y^n,\dt)$. But for our present purposes the current representation is more convenient. The results in this section are analogous to results in \cite{B16}, and we follow many of the ideas employed therein.
The proofs of the intermediate results can be found in the appendix, so that we directly present and prove the main theorem analogously to Theorem 313B in \cite{B16}.

Moreover, since we have already discussed the circumstances under which the analysis of the convergence order can be reduced to the study of autonomous problems, we will limit ourselves to this case for the sake of simplicity. 
\begin{thm}\label{thm:main}
	Let $d_i$, $g_i^{[\nu]}$ and $u$ be defined as in \eqref{eq:pertcond} for $i=1,\dotsc,s$ and $\nu=1,\dotsc, N$. Suppose that for small enough $\dt$ there exists a solution to the stage equations \eqref{eq:nsark} of the NSARK method,
	that $\Fnu\in \mathcal C^{p+1}$ for $p\in \N$ is Lipschitz continuous,
	and that $a_{ij}^{[\nu]}=\O(1)$ (with respect to $\dt$, as $\dt\to 0$) for all $\nu=1,\dotsc,N$.
	Then
	the stages, stage derivatives and output of the NSARK method can be expressed as
	\begin{subequations}
		\begin{align}
			\byi&=\b y^n+\sum_{\tau\in NT_{p}}\frac{\dt^{\lvert \tau\rvert}}{\sigma(\tau)}d_i(\tau,\b y^n,\dt) \dF(\tau)(\b y^n)+\O(\dt^{p+1}),\label{eq:yithm}\\
			\dt\Fnu(\byi)&=\sum_{\tau\in NT_{p}}\frac{\dt^{\lvert \tau\rvert}}{\sigma(\tau)}g^{[\nu]}_i(\tau,\b y^n,\dt) \dF(\tau)(\b y^n)+\O(\dt^{p+1}),\label{eq:hFlthm}\\
			\b y^{n+1}&=\b y^n+\sum_{\tau\in NT_{p}}\frac{\dt^{\lvert \tau\rvert}}{\sigma(\tau)}u(\tau,\b y^n,\dt) \dF(\tau)(\b y^n)+\O(\dt^{p+1}).\label{eq:expyn+1}
		\end{align}
	\end{subequations} for $i=1,\dotsc, s$ and $\nu=1,\dotsc,N$.
\end{thm}
\begin{proof}
	We follow the idea from \cite[Theorem 313B]{B16}. For approximating the stage $\byi$, define the sequence
	\begin{equation}\label{eq:byim}
		\begin{aligned}
			\byi_{[0]}&=\b y^n,\\
			\byi_{[m]}&=\b y^n+\dt\sum_{j=1}^s\sum_{\nu=1}^Na_{ij}^{[\nu]}(\b y^n,\dt)\Fnu(\byj_{[m-1]}),
		\end{aligned}
	\end{equation}
	where we want to point out that $a_{ij}^{[\nu]}(\b y^n,\dt)$ here only depends on the solution $\b y^n$, the step size $\dt$ and, potentially, the assumed solution to the stage equations, but \emph{not} on the iterates $\byi_{[m]}$.
	
	Next, we demonstrate that for $m\leq p$, this expression for $\byi_{[m]}$ agrees with the expression for $\byi$ from \eqref{eq:yithm} within an error of $\O(\dt^{m+1})$. For $m=0$, this is obvious. By induction we suppose that \[\byi_{[m-1]}=\b y^n+\sum_{\tau\in NT_{m-1}}\frac{\dt^{\lvert \tau\rvert}}{\sigma(\tau)}d_i(\tau,\b y^n,\dt) \dF(\tau)(\b y^n)+\O(\dt^{m}).\] By Lemma \ref{lem:hfnu}, we see that 
	\[\dt\Fnu(\byi_{[m-1]})=\sum_{\tau\in NT_{m}}\frac{\dt^{\lvert \tau\rvert}}{\sigma(\tau)}g^{[\nu]}_i(\tau,\b y^n,\dt) \dF(\tau)(\b y^n)+\O(\dt^{m+1}).\]
	Substituting this into \eqref{eq:byim}, we see from \eqref{eq:pertcond} that
	\begin{equation}\label{eqthm:yim}
		\begin{aligned}
			\byi_{[m]}&=\b y^n+\sum_{\tau\in NT_{m}}\frac{\dt^{\lvert \tau\rvert}}{\sigma(\tau)}\sum_{j=1}^s\sum_{\nu=1}^Na_{ij}^{[\nu]}(\b y^n,\dt)g^{[\nu]}_j(\tau,\b y^n,\dt) \dF(\tau)(\b y^n)+\O(\dt^{m+1}) \\
			&=\b y^n+\sum_{\tau\in NT_{m}}\frac{\dt^{\lvert \tau\rvert}}{\sigma(\tau)}d_i(\tau,\b y^n,\dt) \dF(\tau)(\b y^n)+\O(\dt^{m+1}).
		\end{aligned}
	\end{equation}
	We have shown now that \eqref{eqthm:yim} is true for all $m\leq p$. Indeed, by the same reasoning we have even proven that 
	\[\byi_{[m]}=\b y^n+\sum_{\tau\in NT_{p}}\frac{\dt^{\lvert \tau\rvert}}{\sigma(\tau)}d_i(\tau,\b y^n,\dt) \dF(\tau)(\b y^n)+\O(\dt^{p+1}) \quad\text{ for all $m\geq p$. }\]
	Moreover, for $\dt$ small enough we know that $a_{ij}^{[\nu]}$ is bounded since we assumed $a^{[\nu]}_{ij}(\b y^n,\dt)=\O(1)$ as $\dt\to 0$.
	Together with the Lipschitz continuity of $\Fnu$, we thus conclude that for small enough $\dt$ the iteration \eqref{eq:byim} is a contraction with $\byi=\lim_{m\to \infty}\byi_{[m]}$ being the unique limit. Thus, for $\dt$ small enough and $\epsilon=\dt^{p+1}>0$, there exist $N_\epsilon\in \N$ such that $\Vert \byi_{[m]}-\byi\Vert<\dt^{p+1}$ for all $m\geq N_\epsilon$. Without loss of generality we can choose $N_\epsilon\geq p$, so that we find $m\geq p$. This implies that 
	\[\byi=\byi_{[m]}+\O(\dt^{p+1})=\b y^n+\sum_{\tau\in NT_{p}}\frac{\dt^{\lvert \tau\rvert}}{\sigma(\tau)}d_i(\tau,\b y^n,\dt) \dF(\tau)(\b y^n)+\O(\dt^{p+1}),\] from which
	equation \eqref{eq:yithm} follows. Furthermore, \eqref{eq:hFlthm} then follows from Lemma~\ref{lem:hfnu}. Finally, computing $\b y^{n+1}$ according to \eqref{eq:nsark}, also taking into account equation \eqref{eq:pertcond}, we obtain
	\begin{equation*}
		\begin{aligned}
			\b y^{n+1} &= \b y^n +  \sum_{j=1}^s \sum_{\substack{\nu=1}}^N b^{[\nu]}_j(\b y^n,\dt) \dt\Fnu(\byj)\\
			&=\b y^n+\sum_{\tau\in NT_{p}}\frac{\dt^{\lvert \tau\rvert}}{\sigma(\tau)}\sum_{j=1}^s \sum_{\substack{\nu=1}}^Nb^{[\nu]}_j(\b y^n,\dt)g^{[\nu]}_j(\tau,\b y^n,\dt) \dF(\tau)(\b y^n)+\O(\dt^{p+1})\\
			&=\b y^n+\sum_{\tau\in NT_{p}}\frac{\dt^{\lvert \tau\rvert}}{\sigma(\tau)}u(\tau,\b y^n,\dt) \dF(\tau)(\b y^n)+\O(\dt^{p+1}),
		\end{aligned}
	\end{equation*}
	finishing the proof.
\end{proof}
Note that under the assumptions of this theorem, any solution of the stage equations has the same expansion up to the order $p$. Moreover,  we obtain the following order conditions as a result of this theorem, where the expression $\b A^{[\nu]}(\b y^n,\dt)=\O(1)$ should be understood component-wise and in the limit $\dt\to 0$.\newpage
\begin{cor}\label{cor:orderNSARK}
	Let $u$ be defined as in \eqref{eq:pertcond} and $\Fnu\in \mathcal C^{p+1}$ for $\nu=1,\dotsc,N$ be Lipschitz continuous. Furthermore, let $\b A^{[\nu]}(\b y^n,\dt)=\O(1)$. If the stage equations of the NSARK method possess a solution for small enough $\dt$, then the NSARK scheme \eqref{eq:nsark} is of order at least $p$ if and only if 
	\begin{equation}\label{eq:u=1:gamma}
		u(\tau, \b y^n,\dt)=\frac{1}{\gamma(\tau)}+\O(\dt^{p+1-\lvert \tau \rvert}), \quad \forall \tau\in NT_p.
	\end{equation}
\end{cor}
\begin{cor}\label{cor:suff_cond}
	Under the assumptions of Theorem~\ref{thm:main}, if  \[a_{ij}^{[\nu]}(\b y^n,\dt)=a^{[\nu]}_{ij}+\O(\dt^{p-1}) \qta b_j^{[\nu]}(\b y^n,\dt)=b^{[\nu]}_j+\O(\dt^p),\] for $i,j=1,\dotsc,s$ and $\nu=1,\dotsc,N$, the NSARK method \eqref{eq:nsark}  applied to autonomous problems is of order $p$, if $\b A^{[\nu]}=(a^{[\nu]}_{ij})_{i,j=1,\dotsc,s},$ $\b b^{[\nu]}=(b^{[\nu]}_1,\dotsc,b^{[\nu]}_N)$ define an ARK method of order $p$.
\end{cor}
\begin{proof}
	Inserting the assumptions into \eqref{eq:nsark} yields
	\begin{equation*}
		\begin{aligned}
			\byi & = \b y^n + \dt \sum_{j=1}^s  \sum_{\substack{\nu=1}}^N a^{[\nu]}_{ij}  \fnu(\byj) +\O(\dt^p), \\
			\b y^{n+1} & = \b y^n + \dt \sum_{j=1}^s \sum_{\substack{\nu=1}}^N b^{[\nu]}_j \fnu(\byj) +\O(\dt^{p+1}).
		\end{aligned}
	\end{equation*}
	According to Theorem~\ref{thm:main} and  Lemma~\ref{lem:hfnu} we see 
	\[\dt\Fnu(\byi)=\sum_{\tau\in NT_{p}}\frac{\dt^{\lvert \tau\rvert}}{\sigma(\tau)}g^{[\nu]}_i(\tau) \dF(\tau)(\b y^n)+\O(\dt^{p+1}). \]
	Consequently, \eqref{eq:cond} implies that \[\b y^{n+1}=\b y^n+\sum_{\tau\in NT_{p}}\frac{\dt^{\lvert \tau\rvert}}{\sigma(\tau)}u(\tau) \dF(\tau)(\b y^n)+\O(\dt^{p+1}).\]
	Finally, since the corresponding underlying ARK scheme is of order $p$ the claim follows.
\end{proof}

In order to grasp the condition \eqref{eq:u=1:gamma} from Corollary~\ref{cor:orderNSARK}, we collect the value of $u$ for  all $\tau\in NT_4$ in Table~\ref{tab:NSARKcond}.
\begin{table}[!h]
	\centering
	{\def\arraystretch{2}\tabcolsep=10pt	\begin{tabular}{|>{\centering\arraybackslash}m{0.13\linewidth}|>{\centering\arraybackslash}m{.05\linewidth}|>{\centering\arraybackslash}m{.6\linewidth}|}
			\hline
			$\tau$          &  $\gamma(\tau)$ & $u(\tau,\b y^n,\dt)$ \\ \hline
			$\rt[]^{[\mu]}$ &   1                       &     $\sum_{i=1}^s b_i^{[\mu]}(\b y^n,\dt)$               \\ \hline
			$\ct[\hphantom{.}^{[\mu]}[\hphantom{.}^{[\nu]}]]$	&   2                        &   $\sum_{i,j=1}^s b_i^{[\mu]}(\b y^n,\dt) a_{ij}^{[\nu]}(\b y^n,\dt)$                 \\ \hline
			$	\ct[\hphantom{.}^{[\mu]}[\hphantom{.}^{[\nu]}[\hphantom{.}^{[\xi]}]]]$              &      6        &     $\sum_{i,j,k=1}^s b_i^{[\mu]}(\b y^n,\dt) a_{ij}^{[\nu]}(\b y^n,\dt)a_{jk}^{[\xi]}(\b y^n,\dt)$               \\ \hline
			$\ct[\hphantom{.}^{[\mu]}[\hphantom{.}^{[\nu]}][\hphantom{.}^{[\xi]}]]$	          &         3     &   $\sum_{i,j,k=1}^s b_i^{[\mu]}(\b y^n,\dt) a_{ij}^{[\nu]}(\b y^n,\dt)a_{ik}^{[\xi]}(\b y^n,\dt)$                 \\ \hline
			$	\ct[\hphantom{.}^{[\mu]}[\hphantom{.}^{[\nu]}[\hphantom{.}^{[\xi]}[\hphantom{.}^{[\eta]}]]]]$              &  24            &    $	\smashoperator{\sum_{i,j,k,l=1}}^s b_i^{[\mu]}(\b y^n,\dt) a_{ij}^{[\nu]}(\b y^n,\dt)a_{jk}^{[\xi]}(\b y^n,\dt)a_{kl}^{[\eta]}(\b y^n,\dt)$                \\ \hline
			$\ct[\hphantom{.}^{[\mu]}[\hphantom{.}^{[\eta]}][\hphantom{.}^{[\nu]}][\hphantom{.}^{[\xi]}]]$	             &  4            &     $		\smashoperator{\sum_{i,j,k,l=1}}^s b_i^{[\mu]}(\b y^n,\dt)a_{il}^{[\eta]}(\b y^n,\dt) a_{ij}^{[\nu]}(\b y^n,\dt)a_{ik}^{[\xi]}(\b y^n,\dt)$               \\ \hline
			$	\ct[\hphantom{.}^{[\mu]}[\hphantom{.}^{[\xi]}[\hphantom{.}^{[\eta]}]][\hphantom{.}^{[\nu]}]]$              &   8           &    $	\smashoperator{\sum_{i,j,k,l=1}}^s b_i^{[\mu]}(\b y^n,\dt)a_{il}^{[\nu]}(\b y^n,\dt) a_{ij}^{[\xi]}(\b y^n,\dt)a_{jk}^{[\eta]}(\b y^n,\dt)$                \\ \hline
			$	\ct[\hphantom{.}^{[\mu]}[\hphantom{.}^{[\nu]}[\hphantom{.}^{[\xi]}][\hphantom{.}^{[\eta]}]]]$              &   12           &    $	\smashoperator{\sum_{i,j,k,l=1}}^s b_i^{[\mu]}(\b y^n,\dt) a_{ij}^{[\nu]}(\b y^n,\dt)a_{jk}^{[\xi]}(\b y^n,\dt)a_{jl}^{[\eta]}(\b y^n,\dt)$                \\ \hline
	\end{tabular}}\caption{Density $\gamma$ from \eqref{eq:sigmagamma} and value of $u$ from \eqref{eq:pertcond} for $\tau\in NT_4$.}\label{tab:NSARKcond}
\end{table}
\begin{rem}
	Using Corollary \ref{cor:orderNSARK} and Table~\ref{tab:NSARKcond}, the condition for $p=1$ reads
	\begin{equation}\label{eq:condp=1}
		\begin{aligned}
			\sum_{i=1}^s b_i^{[\mu]}(\b y^n,\dt)=1 +\O(\dt),&& \mu=1,\dotsc,N.
		\end{aligned}
	\end{equation}
	For $p=2$ we find the conditions
	\begin{equation}\label{eq:condp=2}
		\begin{aligned}
			\sum_{i=1}^s b_i^{[\mu]}(\b y^n,\dt)&=1 +\O(\dt^2),& \mu&=1,\dotsc,N,\\
			\sum_{i,j=1}^s b_i^{[\mu]}(\b y^n,\dt) a_{ij}^{[\nu]}(\b y^n,\dt)&=\frac12 +\O(\dt),& \mu,\nu&=1,\dotsc,N,\\
		\end{aligned}
	\end{equation}
	and for $p=3$ we obtain
	\begin{equation}\label{eq:condp=3}
		\begin{aligned}
			\sum_{i=1}^s b_i^{[\mu]}(\b y^n,\dt)&=1 +\O(\dt^3), &\mu&=1,\dotsc,N,\\
			\sum_{i,j=1}^s b_i^{[\mu]}(\b y^n,\dt) a_{ij}^{[\nu]}(\b y^n,\dt)&=\frac12 +\O(\dt^2), &\mu,\nu&=1,\dotsc,N,\\
			\sum_{i,j,k=1}^s b_i^{[\mu]}(\b y^n,\dt) a_{ij}^{[\nu]}(\b y^n,\dt)a_{ik}^{[\xi]}(\b y^n,\dt)&=\frac13 +\O(\dt), &\mu,\nu,\xi&=1,\dotsc,N,\\
			\sum_{i,j,k=1}^s b_i^{[\mu]}(\b y^n,\dt) a_{ij}^{[\nu]}(\b y^n,\dt)a_{jk}^{[\xi]}(\b y^n,\dt)&=\frac16 +\O(\dt), &\mu,\nu,\xi&=1,\dotsc,N.
		\end{aligned}
	\end{equation}
	As we derive also 4th order conditions for GeCo and MPRK schemes in the next sections, we also present the general conditions for $p=4$ reading
	\begin{equation}\label{eq:condp=4}
		\begin{aligned}
			\sum_{i=1}^s b_i^{[\mu]}(\b y^n,\dt)&=1 +\O(\dt^4),\\
			\sum_{i,j=1}^s b_i^{[\mu]}(\b y^n,\dt) a_{ij}^{[\nu]}(\b y^n,\dt)&=\frac12 +\O(\dt^3),\\
			\sum_{i,j,k=1}^s b_i^{[\mu]}(\b y^n,\dt) a_{ij}^{[\nu]}(\b y^n,\dt)a_{ik}^{[\xi]}(\b y^n,\dt)&=\frac13 +\O(\dt^2), \\
			\sum_{i,j,k=1}^s b_i^{[\mu]}(\b y^n,\dt) a_{ij}^{[\nu]}(\b y^n,\dt)a_{jk}^{[\xi]}(\b y^n,\dt)&=\frac16 +\O(\dt^2), \\
			\sum_{i,j,k,l=1}^s b_i^{[\mu]}(\b y^n,\dt)a_{il}^{[\nu]}(\b y^n,\dt) a_{ij}^{[\xi]}(\b y^n,\dt)a_{jk}^{[\eta]}(\b y^n,\dt)&=\frac18 +\O(\dt), \\
			\sum_{i,j,k,l=1}^s b_i^{[\mu]}(\b y^n,\dt)a_{il}^{[\eta]}(\b y^n,\dt) a_{ij}^{[\nu]}(\b y^n,\dt)a_{ik}^{[\xi]}(\b y^n,\dt)&=\frac14 +\O(\dt), \\
			\sum_{i,j,k,l=1}^s b_i^{[\mu]}(\b y^n,\dt) a_{ij}^{[\nu]}(\b y^n,\dt)a_{jk}^{[\xi]}(\b y^n,\dt)a_{kl}^{[\eta]}(\b y^n,\dt)&=\frac{1}{4!} +\O(\dt), \\
			\sum_{i,j,k,l=1}^s b_i^{[\mu]}(\b y^n,\dt) a_{ij}^{[\nu]}(\b y^n,\dt)a_{jk}^{[\xi]}(\b y^n,\dt)a_{jl}^{[\eta]}(\b y^n,\dt)&=\frac{1}{12} +\O(\dt)
		\end{aligned}
	\end{equation}
	for $\mu,\nu,\xi,\eta=1,\dotsc,N$.
\end{rem}

\subsection{Application to Geometric Conservative Methods}
In this section we derive the known order conditions for Geometric Conservative (GeCo) schemes \cite{martiradonna2020geco} and present for the first time order conditions for 3rd and 4th order. 
As GeCo schemes are NSRK methods we can interpret them formally as NSARK methods in order to use Corollary \ref{cor:orderNSARK}.
The resulting order conditions can easily be simplified somewhat, using the fact that the original
coefficients $a_{ij}, b_j$ satisfy traditional RK order conditions. In view of \eqref{eq:Geco_NS}, the first condition is
\[
\sum_{i=1}^s b_i \phi_{n+1}(\b y^n,\dt) = 1 + \O(\dt^{p}),
\]
which implies simply $\phi_{n+1}(\b y^n,\dt) = 1 + \O(\dt^p)$.  This turns out
to allow us to neglect the factor $\phi_{n+1}$ in all the remaining
order conditions.  For instance, the next condition is
\[
\sum_{i=1}^s b_i c_i \phi_{n+1}(\b y^n,\dt) \phi_i(\b y^n,\dt) = \frac 12 + \O(\dt^{p-1}),
\]
which is equivalent to 
\[\sum_{i=1}^s b_i c_i \phi_i(\b y^n,\dt) =\frac12 +\O(\dt^{p-1}).\]
With more work, we can use these conditions to obtain direct conditions
on the functions $\phi$ for specific cases of $s$ and $p$, as demonstrated
in the following theorem.
\begin{thm}
	Let $\b A,\b b$ be the coefficients of an explicit RK scheme of order $p$ with $s$ stages satisfying $\sum_{j=1}^s a_{ij}=c_i$. Assume $\phi_i(\b y^n,\dt)=\O(1)$ as $\dt\to 0$ for $i=2,\dotsc,s$ and that $\b f\in \mathcal C^{p+1}$ is Lipschitz continuous. Then
	\begin{enumerate}
		\item if $p=1$, \eqref{eq:GeCoscheme}  is of order $1$ if and only if $\phi_{n+1}(\b y^n,\dt)=1+\O(\dt)$.
		\item if $p=s=2$, \eqref{eq:GeCoscheme} is of order $2$ if and only if $\phi_2(\b y^n,\dt)=1+\O(\dt)$ and $ \phi_{n+1}(\b y^n,\dt)=1+\O(\dt^2)$.
		\item if $p=s=3$, \eqref{eq:GeCoscheme} is of order 3 if and only if    \begin{align*}
			\phi_{n+1}(\b y^n,\dt)&=1 +\O(\dt^3),\\
			\sum_{i=2}^3 b_ic_i\phi_i(\b y^n,\dt)&=\frac 1 2 +\O(\dt^2),\\
			\phi_i(\b y^n,\dt)&=1 +\O(\dt),\quad i=2,3.
		\end{align*}
		\item if $p=s=4$, \eqref{eq:GeCoscheme} is of order $4$ if and only if    \begin{align*}
			\phi_{n+1}(\b y^n,\dt)&=1 +\O(\dt^4),\\
			\sum_{i=2}^4 b_ic_i\phi_i(\b y^n,\dt)&=\frac 1 2 +\O(\dt^3),\\
			\phi_i(\b y^n,\dt)&=1 +\O(\dt^2),\quad i=2,3,4.
		\end{align*}
	\end{enumerate}
\end{thm}
\begin{proof} First of all, the assumptions of Theorem \ref{thm:main} and Corollary \ref{cor:orderNSARK} are met. Thus, we can use the order conditions \eqref{eq:condp=1} to \eqref{eq:condp=4} as a basis of this proof.
	\begin{enumerate}
		\item Substituting $\sum_{i=1}^sb_i=1$ into \eqref{eq:condp=1} yields $\phi_{n+1}(\b y^n,\dt)=1+\O(\dt)$.
		\item Using $\sum_{i=1}^sb_i=1$ now in \eqref{eq:condp=2} together with $\sum_{j=1}^s a_{ij}=c_i$, the order conditions reduce to 
		\[\phi_{n+1}(\b y^n,\dt)=1+\O(\dt^2),\]
		and
		\[\sum_{i=1}^s b_i\phi_{n+1}(\b y^n,\dt)c_i\phi_i(\b y^n,\dt)=\frac12 +\O(\dt).\]
		The latter condition can be further simplified to 
		\[b_2c_2\phi_2(\b y^n,\dt)=\frac12 +\O(\dt),\] since $s=2$ and $\phi_{n+1}(\b y^n,\dt)=1+\O(\dt^2)$. As $b_2c_2=\frac12$, this means that 
		\[ \phi_2(\b y^n,\dt)=1+\O(\dt).\]
		\item Similar as before we obtain from \eqref{eq:condp=3} the simplified conditions
		\begin{align*}
			\phi_{n+1}(\b y^n,\dt)&=1 +\O(\dt^3),\\
			\sum_{i=2}^3 b_ic_i\phi_i(\b y^n,\dt)&=\frac12 +\O(\dt^2),\\
			\sum_{i=2}^3 b_ic_i^2(\phi_i(\b y^n,\dt))^2&=\frac13 +\O(\dt), \\
			\sum_{i,j=2}^3 b_i a_{ij}c_j\phi_i(\b y^n,\dt)\phi_j(\b y^n,\dt)&=\frac16 +\O(\dt), 
		\end{align*}
		which by Lemma \ref{lem:equivalent} with $N=1$, $\gamma^{(i)}_1=\phi_i(\b y^n,\dt)$ and $\delta_1=\phi_{n+1}(\b y^n,\dt)$ are equivalent to
		the conditions stated in the Theorem.
			\item As in the previous part, the conditions \eqref{eq:condp=4} are simplified resulting in \eqref{eq:MPRKcondp=4} with $N=1$, $\gamma^{(i)}_1$ and $\delta_1$ as before. These conditions are then reduced by Lemma~\ref{lem:equivalent4} resulting in the conditions given in this theorem.
		\end{enumerate}
	\end{proof}
	With this result, we have shown that the conditions from \cite[Theorem~1]{martiradonna2020geco} are also necessary. Moreover, we provided the very first necessary and sufficient order conditions for the construction of 3rd and 4th order GeCo schemes.
	\subsection{Application to Modified Patankar--Runge--Kutta Methods}
	As we have discussed in Section~\ref{sec:MPRK}, modified Patankar--Runge--Kutta methods (MPRK) originally were constructed for conservative and positive PDS of the form \eqref{eq:PDS}. Moreover, we concluded in that section that we may assume without loss of generality that the PDS is autonomous.

	Until now, sufficient and necessary order conditions for MPRK schemes only up to order three were constructed and in the context of autonomous PDS. However, these order conditions are also valid in the context of a non-autonomous PDRS, as the NSWs are either the same as in the PDS case or equal to $1$, see \eqref{eq:NS_weightsMPRK}. 
	
	Thus, in order to obtain order conditions for even higher order MPRK schemes in the context of a PDRS, we actually can restrict to autonomous PDS where the solution-dependent Butcher tableau is determined by
	\begin{equation*}
		a_{ij}^{[\nu]}(\b y^n,\dt)=a_{ij}\frac{\yi_\nu}{\pi_{\nu}^{(i)}} \quad \text{ and } \quad b_j^{[\nu]}(\b y^n,\dt)=b_j\frac{y^{n+1}_\nu}{\sigma_\nu}.
	\end{equation*}
	Note that, since $a_{1j}=0$, the value of
	$\pi_\nu^{(1)}$ has no effect.
	In order to apply Theorem~\ref{thm:main} and Corollary~\ref{cor:orderNSARK}, we show in the next lemma that the stages are uniquely determined for any $\dt\geq 0$ and that $\b A^{[\nu]}(\b y^n,\dt)=\O(1)$. The key observation to prove this is that $\pi_\nu^{(i)},\sigma_\nu$ are positive even for $\dt = 0$ by definition. Moreover, we assume that the PWDs are continuous functions of $\b y^n$ and the stages, that is $\pi_{\nu}^{(i)}=\pi_{\nu}^{(i)}(\b y^n,\b y^{(1)},\dotsc,\b y^{(i-1)})$ and $\sigma_\nu=\sigma_\nu(\b y^n,\b y^{(1)},\dotsc,\b y^{(s)})$, which is fulfilled by the PWDs introduced so far.   Also, as we will apply the lemma in the context of the local error analysis, we may start with some $\b y^n$ representing the exact solution at a given time level $t_n$.
	\begin{lem}\label{lem:aij=O(1)}  An MPRK scheme \eqref{eq:MPRK} with a given positive vector $\b y^n$ has uniquely determined stages and satisfies $\byi,\b y^{n+1}=\O(1)$ as $\dt\to 0$. Moreover, if $p_{k\nu},d_{k\nu}\in\mathcal C$ and $\pi_\nu^{(i)}(\b y^n,\b y^{(1)},\dotsc,\b y^{(i-1)}),\sigma_\nu(\b y^n,\b y^{(1)},\dotsc,\b y^{(s)})>0$ are continuous functions of $\b y^n$ and the stages, then $\pi_\nu^{(i)},\sigma_\nu=\O(1)$ for $i,\nu=1,\dotsc,N$.
		In addition, the modified Butcher coefficients from \eqref{eq:pertcoeff} satisfy $\b A^{[\nu]}(\b y^n,\dt)=\O(1)$ and $\b b^{[\nu]}(\b y^n,\dt)=\O(1)$.
	\end{lem}
	\begin{proof}
		According to Lemma~\ref{lem:MPRKpos}, there exist unique matrices $\mathbf M^{(i)}$, $i=1,\dotsc,s$ and $\mathbf{M}$, with inverses in $\O(1)$ as $\dt\to 0$, such that the stage vectors satisfy the equations $\byi=\left(\mathbf M^{(i)}\right)^{-1}\b y^n=\O(1)$ and $\b y^{n+1}=\mathbf{M}^{-1}\b y^n=\O(1)$. Now, since $p_{k\nu},d_{k\nu},\pi_\nu^{(i)},\sigma_\nu\in\mathcal C$, we conclude by induction over $i$ that the stage vectors are continuous functions of $\dt$ themselves by pointing out that every entry in $\mathbf M^{(i)},\mathbf M$ is a continuous function of $\dt$. Hence, even $\pi_\nu^{(i)}$ and $\sigma_\nu$ are continuous functions of $\dt$, so that we conclude $\pi_\nu^{(i)}=\O(1)$ and $\sigma_{\nu}=\O(1)$ as $\dt\to 0$. Since even $\pi_\nu^{(i)},\sigma_\nu>0$ for $\dt=0$, we deduce from the continuity in $\dt$ that there is a positive lower bound also for $\dt>0$ small enough. This gives us $\tfrac{\yi_\nu}{\pi_\nu^{(i)}}=\O(1)$ and  $\tfrac{y_\nu^{n+1}}{\sigma_{\nu}}=\O(1)$, from which the claim follows. 
	\end{proof}
	Using this lemma and Corollary~\ref{cor:orderNSARK} we are able to provide necessary and sufficient conditions for arbitrary high order NSARK schemes, to which MPRK methods belong. However, those conditions are in general implicit, since the NSWs depend on the stages.  In the next two subsections, for specific classes of MPRK methods we reformulate these conditions to be explicit.
	\begin{rem}\label{rem:negativeMPRK}
		At this point we should discuss the situation mentioned in Remark~\ref{rem:negButcher}. To prove an analogue of Corollary~\ref{cor:orderNSARK} for MPRK schemes based on a Butcher tableau with partially negative entries, we first note that Remark~\ref{rem:compute_u} tells us that using the index function \eqref{eq:indexfun} to switch the PWDs corresponds to switching the colors in the corresponding N-tree. Now, since the condition \eqref{eq:u=1:gamma} needs to be satisfied for all colored  N-trees in $NT_p$, the order conditions for MPRK schemes do not depend on the sign of the Butcher tableau.
	\end{rem}

	\subsubsection{Order Conditions for MPRK Schemes from the Literature}
	In this subsection we focus on reformulating the order conditions from our theory deriving the sufficient and necessary conditions from the literature, that is the conditions up to order three. 
	Note that, as discussed in Remark \ref{rem:negativeMPRK}, the order conditions for an MPRK scheme do not depend on the sign of the entries of the Butcher tableau. In the following, we thus assume without loss of generality that $\b A,\b b\geq \bm 0$, so that we can use the representation \eqref{eq:MPRK} of the MPRK scheme. Furthermore, we assume throughout this section that $\Fnu$ is Lipschitz continuous for all $\nu=1,\dotsc,N$ and $\sum_{j=1}^s a_{ij}=c_i$.
	
	First we give a lemma that we will repeatedly use throughout this section without further notice.
	\begin{lem}
		For given scalars $x,y$ with $y\neq 0$ the identity
		\[\frac{x+\O(\dt^p)}{y+\O(\dt^p)}=\frac{x}{y}+\O(\dt^p)\] holds true.
	\end{lem}
	\begin{proof}
		Let $f,g=\O(\dt^p)$. Then, as $y\neq 0$, we find
		\[ \frac{x+f(\dt)}{y+g(\dt)}-\frac{x}{y}= \frac{yf(\dt)-xg(\dt)}{y(y+g(\dt))}=\frac{f(\dt)}{y+g(\dt)}-\frac{xg(\dt)}{y(y+g(\dt))}.\]
		Now since $\lim_{\dt\to 0}g(\dt)=0$, the denominators of both fractions on the right-hand side tend to constants as $\dt\to 0$. By definition, we know $\limsup_{\dt \to 0}\frac{f(\dt)}{\dt^p}<\infty$ and  $\limsup_{\dt \to 0}\frac{g(\dt)}{\dt^p}<\infty$. Thus, the claim follows from 
		\[\limsup_{\dt \to 0}\frac{\frac{x+f(\dt)}{y+g(\dt)}-\frac{x}{y}}{\dt^p}\leq\limsup_{\dt \to 0}\frac{f(\dt)}{(y+g(\dt))\dt^p}+\left\lvert\limsup_{\dt \to 0}\frac{xg(\dt)}{y(y+g(\dt))\dt^p}\right\rvert <\infty.\qedhere\]
	\end{proof}
	
	To formulate the conditions up to the order $p=3$, we observe from the general conditions \eqref{eq:condp=1}, \eqref{eq:condp=2} and \eqref{eq:condp=3} that we should expand $a_{ij}^{[\nu]}(\b y^n,\dt)$ up to an error of $\O(\dt^2)$. As we will see, it suffices for our current purposes to assume $\frac{\yi_\nu}{\pi_\nu^{(i)}}=1+\O(\dt)$ for deriving these expansions.
	\begin{lem}\label{lem:stageweights}
		Let $a_{ij}^{[\nu]}(\b y^n,\dt)=a_{ij}\frac{\yi_\nu}{\pi_\nu^{(i)}}$, where $\byi$ is the $i$th stage of an MPRK method \eqref{eq:MPRK}. Moreover, let $\Fnu\in \mathcal C^{2}$ for $\nu=1,\dotsc,N$, and $\b f=\sum_{\nu=1}^N\Fnu$ be the right-hand side of \eqref{ivp}. If \[\frac{\yi_\nu}{\pi_\nu^{(i)}}=1+\O(\dt), \quad \nu=1,\dots,N,\] then
		\begin{equation*}
			\yi_\nu=y^n_\nu+ \dt c_if_\nu(\b y^n)+\O(\dt^2), \quad \nu=1,\dots,N,
		\end{equation*}
		and in particular,
		\begin{equation*}
			a_{ij}^{[\nu]}(\b y^n,\dt)=a_{ij} \frac{y^n_\nu+ \dt c_if_\nu(\b y^n)}{\pi_\nu^{(i)}}+\O(\dt^2),\quad \nu=1,\dots,N.
		\end{equation*}
	\end{lem}
	\begin{proof}
		The conditions for applying Theorem \ref{thm:main} with $k=1$ are met due to Lemma~\ref{lem:aij=O(1)}, so that we can use the expansion of the stages to obtain
		\begin{equation}\label{proofeq:aijnu}
			\yi_\nu=y^n_\nu+ \dt\sum_{\mu=1}^Nd_i(\rt[]^{[\mu]},\b y^n,\dt)(\dF(\rt[]^{[\mu]})(\b y^n))_\nu+\O(\dt^2),\quad \nu=1,\dots,N.
		\end{equation}
		Next, we substitute our assumption for the Patankar weights into $a^{[\mu]}_{ij}(\b y^n,\dt)$ to receive \[a^{[\mu]}_{ij}(\b y^n,\dt)=a_{ij}\frac{\yi_\mu}{\pi_\mu^{(i)}}=a_{ij}+\O(\dt), \quad \mu=1,\dots,N\] and find
		\begin{equation*}
			\begin{aligned}
				d_i(\rt[]^{[\mu]},\b y^n,\dt)&=\sum_{\nu=1}^N\sum_{j=1}^sa^{[\nu]}_{ij}(\b y^n,\dt)g_j^{[\nu]}(\rt[]^{[\mu]},\b y^n,\dt)=\sum_{j=1}^sa^{[\mu]}_{ij}(\b y^n,\dt)\\
				&=\sum_{j=1}^sa_{ij}+\O(\dt)=c_i+\O(\dt).
			\end{aligned}
		\end{equation*}
		Finally, the claim follows from $\sum_{\mu=1}^N\dF(\rt[]^{[\mu]})(\b y^n)=\sum_{\mu=1}^N\Fmu(\b y^n)=\b f(\b y^n)$.
	\end{proof}
	Another helpful result for deriving the known order conditions from the literature is the following.
	\begin{lem}\label{lem:sigma}
		Let $\b A,\b b,\b c$ describe an explicit $s$-stage Runge--Kutta method of at least order $p$ for some $p\in \N$. Consider the corresponding MPRK scheme \eqref{eq:MPRK} and assume $\Fnu\in \mathcal{C}^{p+1}$ for $\nu=1,\dotsc,N$. If the MPRK method is of order $p$, then
		\[\sigma_\mu=(\NB_{p-1}(\tfrac1\gamma,\b y^n))_\mu+\O(\dt^{p}),\quad \mu=1,\dotsc, N. \]
		This means that $\sigma_\mu$ defines an embedded method of order $p-1$.
	\end{lem}
	\begin{proof}
		The MPRK scheme is of order $p$, \ie $\b y^{n+1}=\NB_p(\tfrac{1}{\gamma},\b y^n)+\O(\dt^{p+1}).$ Next, according to Lemma~\ref{lem:aij=O(1)} we can apply Corollary~\ref{cor:orderNSARK} to see \[ u(\rt[]^{[\mu]},\b y^n,\dt)=\sum_{j=1}^s b_j\frac{y^{n+1}_\mu}{\sigma_\mu}=1 +\O(\dt^p)\] for $ \mu=1,\dotsc,N,$ which is equivalent to $\frac{y^{n+1}_\mu}{\sigma_\mu}=1 +\O(\dt^p)$ for $ \mu=1,\dotsc,N$ as $\sum_{j=1}^sb_j=1$. Using Lemma~\ref{lem:aij=O(1)} once again we find $\sigma_\mu=\O(1)$ yielding\[\sigma_\mu=y^{n+1}_\mu+\O(\dt^p)=(\NB_{p-1}(\tfrac1\gamma,\b y^n))_\mu+\O(\dt^{p}),\quad \mu=1,\dotsc, N.\]
	\end{proof}
	We are now in the position to derive the known order conditions from \cite{KM18,KM18Order3} for MPRK schemes up to order 3.
	\begin{thm}\label{thm:MPRKp=1} Let $\Fnu\in \mathcal C^2$ for $\nu=1,\dotsc,N$ and $\b A,\b b,\b c$ describe an explicit $s$-stage Runge--Kutta method of order at least 1.
		The corresponding MPRK scheme \eqref{eq:MPRK} is of order at least 1 if and only if 
		\begin{equation}\label{eq:MPRKcondp=1}
			\sigma_\mu=y^n_\mu+\O(\dt), \quad \mu=1,\dotsc,N.
		\end{equation}
	\end{thm}
	\begin{proof}
		The condition \eqref{eq:condp=1} for an MPRK scheme to be at least of order $p=1$ reads
		\begin{equation*}
			\begin{aligned}
				u(\rt[]^{[\mu]},\b y^n,\dt)=\sum_{j=1}^s b_j\frac{y^{n+1}_\mu}{\sigma_\mu}=1 +\O(\dt),&& \mu=1,\dotsc,N,
			\end{aligned}
		\end{equation*}
		which can be simplified to $ u(\rt[]^{[\mu]},\b y^n,\dt)=\frac{y^{n+1}_\mu}{\sigma_\mu}=1+\O(\dt)$  for $\mu=1,\dotsc,N$ as $\sum_{j=1}^sb_j=1$. From Lemma~\ref{lem:sigma} the condition $\sigma_\mu=y^n_\mu+\O(\dt)$ for $\mu=1,\dotsc,N$ can be deduced.
		
		Now let \eqref{eq:MPRKcondp=1} be satisfied. It follows immediately from Lemma \ref{lem:aij=O(1)} and \eqref{eq:expyn+1} that $y^{n+1}_\nu=y^n_\nu+\O(\dt).$ Comparing with \eqref{eq:MPRKcondp=1}, this gives us \[u(\rt[]^{[\mu]},\b y^n,\dt)=\frac{y^{n+1}_\mu}{\sigma_\mu}=1+\O(\dt)\]  for $\mu=1,\dotsc,N$ proving that \eqref{eq:MPRKcondp=1} is sufficient and necessary.
	\end{proof}\newpage
	\begin{thm}\label{thm:MPRKp=2}Let $\Fnu\in \mathcal C^3$ for $\nu=1,\dotsc,N$ and $\b A,\b b,\b c$ describe an explicit $2$-stage Runge--Kutta method of order two.
		Then the corresponding MPRK scheme \eqref{eq:MPRK} is of order two if and only if  
		\begin{subequations}\label{eq:MPRKcondp=2}
			\begin{align}
				\pi_\nu^{(2)}&= y^n_\nu+\O(\dt), && \nu=1,\dotsc,N,\label{eq:MPRKp=2a}\\
				\sigma_\mu&=(\NB_1(\tfrac{1}{\gamma},\b y^n))_\mu+\O(\dt^2),&& \mu=1,\dotsc,N.\label{eq:MPRKp=2b}
			\end{align}
		\end{subequations}
	\end{thm}
	\begin{proof}
		First we reduce the necessary and sufficient conditions for $p=2$ from \eqref{eq:condp=2}, which state
		\begin{equation*}
			\begin{aligned}
				u(\rt[]^{[\mu]},\b y^n,\dt)=\sum_{i=1}^s b_i\frac{y^{n+1}_\mu}{\sigma_\mu}&=1 +\O(\dt^2),& \mu&=1,\dotsc,N,\\
				u([\rt[]^{[\nu]}]^{[\mu]},\b y^n,\dt)=\sum_{i,j=1}^s b_i\frac{y^{n+1}_\mu}{\sigma_\mu} a_{ij}\frac{\yi_\nu}{\pi_\nu^{(i)}}&=\frac12 +\O(\dt),& \mu,\nu&=1,\dotsc,N.\\
			\end{aligned}
		\end{equation*}
		Since $\sum_{i=1}^sb_i=1$, the first equation can be reduced to $\frac{y^{n+1}_\mu}{\sigma_\mu}=1 +\O(\dt^2)$ for $\mu=1,\dotsc,N$. Plugging this information into the second condition and using $\sum_{j=1}^sa_{ij}=c_i$, we end up with the condition
		$\sum_{i=1}^s b_ic_i\frac{\yi_\nu}{\pi_\nu^{(i)}}=\frac12 +\O(\dt)$ for $\nu=1,\dotsc,N.$
		Since we assumed $s=2$, we can use $c_1=0$ to obtain the equivalent conditions
		\begin{align*}
			\frac{y^{n+1}_\mu}{\sigma_\mu}&=1 +\O(\dt^2),\\
			\frac{y^{(2)}_\nu}{\pi_\nu^{(2)}}&=1+\O(\dt)
		\end{align*}
		for $\mu,\nu=1,\dotsc,N.$
		
		To prove the claim, first assume that the MPRK scheme has order $p=2$. Then Lemma \ref{lem:stageweights} and Lemma~\ref{lem:sigma} yield the conditions from \eqref{eq:MPRKcondp=2}.
		
		Now let \eqref{eq:MPRKcondp=2} be satisfied. Using \eqref{eq:MPRKp=2a} and Lemma \ref{lem:aij=O(1)} together with the expansion \eqref{eq:yithm} of the stages we see $y^{(2)}_\nu=y^n_\nu+\O(\dt)$, and hence, \[\frac{y^{(2)}_\nu}{\pi_\nu^{(2)}}=1+\O(\dt).\] Moreover, with \eqref{eq:MPRKp=2b} we can apply Theorem \ref{thm:MPRKp=1} to find that the scheme is at least first order accurate, that is $y^{n+1}_\nu=(\NB_1(\frac{1}{\gamma},\b y^n))_\nu+\O(\dt^2)$. Comparing with \eqref{eq:MPRKp=2b} we end up with $\frac{y^{n+1}_\nu}{\sigma_{\nu}}=1+\O(\dt^2).$
	\end{proof}
	As one can see, the stage-dependent conditions for second order are
	\begin{align*}
		\frac{y^{n+1}_\mu}{\sigma_\mu}&=1 +\O(\dt^2),\\
		\frac{y^{(2)}_\nu}{\pi_\nu^{(2)}}&=1+\O(\dt)
	\end{align*}
	and were reformulated in the above Theorem.
	Similarly, the simplified conditions from \eqref{eq:MPRKcondp=3} with $\gamma^{(i)}_\nu=\frac{\yi_\nu}{\pi_\nu^{(i)}}$ and $\delta_\mu= \frac{y^{n+1}_\mu}{\sigma_\mu}$ are, by means of Lemma \ref{lem:equivalent}, equivalent to 
	{\allowdisplaybreaks
		\begin{subequations}\label{eq:KM18Order3}
			\begin{align}
				\frac{y^{n+1}_\mu}{\sigma_\mu}&=1 +\O(\dt^3), &\mu&=1,\dotsc,N, \label{eq:KM18Order3a}\\
				\sum_{i=2}^3 b_ic_i\frac{\yi_\nu}{\pi_\nu^{(i)}}&=\frac12 +\O(\dt^2), &\nu&=1,\dotsc,N,\label{eq:KM18Order3b}\\
				\frac{\yi_\nu}{\pi_\nu^{(i)}}&=1+\O(\dt),&\nu&=1,\dotsc,N,\quad i=2,3.\label{eq:KM18Order3c}
			\end{align}
	\end{subequations} }
	The next theorem decodes these conditions reformulating them in an explicit form.
	\begin{thm}\label{thm:MPRKp=3}
		Let $\b A,\b b,\b c$ describe an explicit 3-stage RK scheme and let $\Fnu\in \mathcal C^4$ for $\nu=1,\dotsc,N$. Then the corresponding MPRK scheme \eqref{eq:MPRK} is at least of order $p=s=3$ if and only if  
		\begin{subequations}\label{eq:KM18Order3condlit}
			\begin{align}
				\sigma_\mu&=(\NB_2(\tfrac{1}{\gamma},\b y^n))_\mu+\O(\dt^3),&& \mu=1,\dotsc,N,\label{eq:KM18Order3condlitc}\\
				\sum_{i=2}^3b_ic_i\frac{y^n_\nu+ \dt c_if_\nu(\b y^n)}{\pi_\nu^{(i)}}&=\frac12+\O(\dt^2),&& \nu=1,\dots,N,\label{eq:KM18Order3condlitb}\\
				\pi_\nu^{(i)}&=y^n_\nu+\O(\dt), && \nu=1,\dotsc,N, && i=2,3.\label{eq:KM18Order3condlita}       
			\end{align}
		\end{subequations}   
	\end{thm}
	\begin{proof}
		We now show that the conditions \eqref{eq:KM18Order3} are equivalent to \eqref{eq:KM18Order3condlit}. First, assuming \eqref{eq:KM18Order3} is fulfilled, the MPRK scheme is of order 3. Thus, Lemma \ref{lem:sigma} implies \eqref{eq:KM18Order3condlitc}. Finally, with \eqref{eq:KM18Order3c} we are in the position to apply Lemma \ref{lem:stageweights}, which, together with \eqref{eq:KM18Order3b}, yield the conditions \eqref{eq:KM18Order3condlita} and \eqref{eq:KM18Order3condlitb}.
		
		Let's now suppose that \eqref{eq:KM18Order3condlit} holds. The condition \eqref{eq:KM18Order3c} follows from \eqref{eq:KM18Order3condlita} and the expansion \eqref{eq:yithm} for the stages. Having derived \eqref{eq:KM18Order3c}, we can apply Lemma \ref{lem:stageweights} to obtain 
		\begin{equation*}
			\yi_\nu=y^n_\nu+ \dt c_if_\nu(\b y^n)+\O(\dt^2), \quad \nu=1,\dots,N.
		\end{equation*}
		Together with \eqref{eq:KM18Order3condlitb} we can thus conclude \eqref{eq:KM18Order3b}. Therefore, it remains to deduce condition \eqref{eq:KM18Order3a}. 
		
		First of all, \eqref{eq:KM18Order3condlitc} and Theorem \ref{thm:MPRKp=1} imply that the MPRK scheme is of order at least $1$, which means that $\b y^{n+1}=\NB_1(\frac1\gamma,\b y^n)+\O(\dt^2).$ Comparing with \eqref{eq:KM18Order3condlitc}, we see \[\frac{y^{n+1}_\mu}{\sigma_\mu}=1 +\O(\dt^2),\quad \mu=1,\dotsc,N.\] Moreover, since we have already shown \eqref{eq:KM18Order3c}, we can now verify that condition \eqref{eq:condp=2} is fulfilled which means that the MPRK scheme is even second order accurate. Therefore, we find $\b y^{n+1}=\NB_2(\frac1\gamma,\b y^n)+\O(\dt^3)$, so that a comparison with \eqref{eq:KM18Order3condlitc} gives us \eqref{eq:KM18Order3a}.
	\end{proof}
	We have now derived all known order conditions for MPRK schemes from the literature and even proved that they are valid for MPRK schemes based on $\b A$ and $\b b$ with negative entries, see Remark~\ref{rem:negativeMPRK}.

	\subsubsection{Reduced Order Conditions for 4th Order MPRK Methods}
	The main idea in deriving the known conditions for 3rd order MPRK schemes was to use Lemma \ref{lem:equivalent} for reducing the order conditions \eqref{eq:condp=3} and then substituting the expansions for the stages to obtain conditions depending only on $\b y^n$ and $\dt$. 
	
	Similarly, we are in the position to derive conditions for 4th order by first using Lemma \ref{lem:equivalent4} to reduce the order conditions \eqref{eq:condp=4}. Now, in order to eliminate the dependency of the conditions on the stages, we need to expand $a_{ij}^{[\nu]}(\b y^n,\dt)$ up to an error of $\O(\dt^3)$ giving an analogue to Lemma \ref{lem:stageweights}. However, since equation \eqref{eq:MPRKcondp'=4} in Lemma \ref{lem:equivalent4} gives us $\gamma^{(i)}_\nu=1+\O(\dt^2)$, we will see that it suffices to prove the following lemma assuming $\frac{\yi_\nu}{\pi_\nu^{(i)}}=1+\O(\dt^2)$.
	\begin{lem}\label{lem:stageweights4}
		Let $a_{ij}^{[\nu]}(\b y^n,\dt)=a_{ij}\frac{\yi_\nu}{\pi_\nu^{(i)}}$, where $\byi$ is the $i$th stage of an MPRK method \eqref{eq:MPRK}. Moreover, let $\Fnu\in \mathcal C^{3}$ for $\nu=1,\dotsc,N$, and $\b f=\sum_{\nu=1}^N\Fnu$ be the right-hand side of \eqref{ivp}. If \[\frac{\yi_\nu}{\pi_\nu^{(i)}}=1+\O(\dt^2), \quad \nu=1,\dots,N,\] then
		\begin{equation*}
			\yi_\nu=y^n_\nu+ \dt c_if_\nu(\b y^n)+\frac12\dt^2\sum_{k=1}^sa_{ik}c_k(\D\b f(\b y^n)\b f(\b y^n))_\nu+\O(\dt^3), \quad \nu=1,\dots,N,
		\end{equation*}
		and in particular,
		\begin{equation*}
			a_{ij}^{[\nu]}(\b y^n,\dt)=a_{ij} \frac{y^n_\nu+ \dt c_if_\nu(\b y^n)+\frac12\dt^2\sum_{k=1}^sa_{ik}c_k(\D\b f(\b y^n)\b f(\b y^n))_\nu}{\pi_\nu^{(i)}}+\O(\dt^3)
		\end{equation*}
		for $\nu=1,\dotsc,N$.
	\end{lem}
	\begin{proof}
		The conditions for applying Theorem \ref{thm:main} with $k=2$ are met due to Lemma~\ref{lem:aij=O(1)}, so that we can use the expansion of the stages to obtain 
		\begin{equation}\label{proofeq:aijnu4}
			\begin{aligned}
				\yi_\nu=&y^n_\nu+ \dt\sum_{\mu=1}^Nd_i(\rt[]^{[\mu]},\b y^n,\dt)(\dF(\rt[]^{[\mu]})(\b y^n))_\nu\\&+\frac12\dt^2\sum_{\mu,\eta=1}^Nd_i\left(\left[\rt[]^{[\eta]}\right]^{[\mu]},\b y^n,\dt\right)\left(\dF\left(\left[\rt[]^{[\eta]}\right]^{[\mu]}\right)(\b y^n)\right)_\nu+\O(\dt^3)
			\end{aligned}
		\end{equation}
		for $\nu=1,\dots,N$.
		Moreover, we know that
		\begin{equation}\label{eq:proofaij}
			a^{[\mu]}_{ij}(\b y^n,\dt)=a_{ij}\frac{\yi_\mu}{\pi_\mu^{(i)}}=a_{ij}+\O(\dt^2), \quad \mu=1,\dots,N
		\end{equation}
		as well as 
		\begin{equation}\label{eq:proofdi}
			d_i(\rt[]^{[\mu]},\b y^n,\dt)=\sum_{j=1}^sa^{[\mu]}_{ij}(\b y^n,\dt)=\sum_{j=1}^sa_{ij}+\O(\dt^2)=c_i+\O(\dt^2)
		\end{equation}
		by following the lines of the proof of Lemma \ref{lem:stageweights}. Moreover, in that proof we have already seen that $\sum_{\mu=1}^N\dF(\rt[]^{[\mu]})(\b y^n)=\b f(\b y^n)$ so that we obtain the intermediate result
		\begin{equation*}
			\begin{aligned}
				\yi_\nu=&y^n_\nu+ \dt c_if_\nu(\b y^n)\\&+\frac12\dt^2\sum_{\mu,\eta=1}^Nd_i\left(\left[\rt[]^{[\eta]}\right]^{[\mu]},\b y^n,\dt\right)\left(\dF\left(\left[\rt[]^{[\eta]}\right]^{[\mu]}\right)(\b y^n)\right)_\nu+\O(\dt^3).
			\end{aligned}
		\end{equation*} 
		Turning to the coefficient of $\dt^2$, we first point out that, according to \eqref{eq:pertcond}, we have
		\begin{equation*}
			\begin{aligned}
				d_i\left(\left[\rt[]^{[\eta]}\right]^{[\mu]},\b y^n,\dt\right)&=\sum_{\nu=1}^N\sum_{j=1}^sa^{[\nu]}_{ij}(\b y^n,\dt)g_j^{[\nu]}\left(\left[\rt[]^{[\eta]}\right]^{[\mu]},\b y^n,\dt\right)\\&=\sum_{j=1}^sa^{[\mu]}_{ij}(\b y^n,\dt)d_j\left(\rt[]^{[\eta]},\b y^n,\dt\right)=\sum_{j=1}^sa_{ij}c_j+\O(\dt^2),
			\end{aligned}
		\end{equation*}
		where we used \eqref{eq:proofaij} and \eqref{eq:proofdi}.
		Finally, using \eqref{eq:elemdiff} we obtain
		\begin{align*}
			\sum_{\mu,\eta=1}^N\dF\left(\left[\rt[]^{[\eta]}\right]^{[\mu]}\right)(\b y^n)&=\sum_{\mu,\eta=1}^N\sum_{i_1=1}^d\partial_{i_1}\Fmu(\b y^n)\dF_{i_1}(\rt[]^{[\eta]})(\b y)\\
			&=\sum_{\mu=1}^N\D\Fmu(\b y^n)\sum_{\eta=1}^N\b f^{[\eta]}(\b y^n)=\D\b f(\b y^n)\b f(\b y^n).
		\end{align*}
		The claim follows after substituting these equations into \eqref{proofeq:aijnu4}.
	\end{proof}
	With that lemma we now derive sufficient and necessary conditions for 4th order MPRK schemes. 
	
	The simplified conditions for an MPRK method of order 4 are given by \eqref{eq:MPRKcondp=4} with $\gamma^{(i)}_\nu=\frac{\yi_\nu}{\pi_\nu^{(i)}}$ and $\delta_\mu= \frac{y^{n+1}_\mu}{\sigma_\mu}$. Using Lemma \ref{lem:equivalent4} these conditions are equivalent to 
	\begin{subequations}\label{eq:mprk4}
		\begin{align}
			\frac{y^{n+1}_\mu}{\sigma_\mu}&=1 +\O(\dt^4), &\mu&=1,\dotsc,N, \label{eq:mprk4a}\\
			\sum_{i=2}^4 b_ic_i\frac{\yi_\nu}{\pi_\nu^{(i)}}&=\frac12 +\O(\dt^3), &\nu&=1,\dotsc,N,\label{eq:mprk4b}\\
			\frac{\yi_\nu}{\pi_\nu^{(i)}}&=1+\O(\dt^2),&\nu&=1,\dotsc,N,\quad i=2,3,4.\label{eq:mprk4c}
		\end{align}
	\end{subequations} 
	However, these conditions again depend on the stages. The next theorem gives us equivalent conditions depending only on $\b y^n$ and $\dt$.
	\newpage
	\begin{thm}\label{thm:MPRKp=4}
		Let $\b A,\b b,\b c$ describe an explicit 4-stage RK scheme of order 4  with $\sum_{j=1}^s a_{ij}=c_i$, and let $\Fnu\in \mathcal C^5$ for $\nu=1,\dotsc,N$. Then the corresponding MPRK scheme \eqref{eq:MPRK} is at least of order $p=s=4$ if and only if  for $\mu,\nu=1,\dotsc,N$ and $i=2,3,4$ we have
		\begin{subequations}\label{eq:MPRK4condlit}
			\begin{align}
				\sigma_\mu&=(\NB_3(\tfrac{1}{\gamma},\b y^n))_\mu+\O(\dt^4),\label{eq:MPRK4condlitc}\\
				\sum_{i=2}^4b_ic_i&\frac{y^n_\nu+ \dt c_if_\nu(\b y^n)+\frac12\dt^2\sum_{k=1}^4a_{ik}c_k(\D\b f(\b y^n)\b f(\b y^n))_\nu}{\pi_\nu^{(i)}}=\frac12+\O(\dt^3),\label{eq:MPRK4condlitb}\\
				\pi_\nu^{(i)}&=y^n_\nu+ \dt c_if_\nu(\b y^n)+\O(\dt^2).\label{eq:MPRK4condlita}       
			\end{align}
		\end{subequations}   
	\end{thm}
	\begin{proof}
		We start by assuming that \eqref{eq:mprk4} is fulfilled and note that this part works along the same lines as in Theorem \ref{thm:MPRKp=3}. Nevertheless, we present it here for the sake of completeness. 
		
		Now, since \eqref{eq:mprk4} holds, the MPRK scheme is of order 4. Thus, Lemma \ref{lem:sigma} implies \eqref{eq:MPRK4condlitc}. Finally, with \eqref{eq:mprk4c} we are in the position to apply Lemma \ref{lem:stageweights4}, which, together with \eqref{eq:mprk4b}, yield the conditions \eqref{eq:MPRK4condlita} and \eqref{eq:MPRK4condlitb}.
		
		Let's now suppose that \eqref{eq:MPRK4condlit} holds. Using the expansion \eqref{eq:yithm} for the stages, we first observe with \eqref{eq:MPRK4condlita} that $\frac{\yi_\nu}{\pi_\nu^{(i)}}=1+\O(\dt)$. Applying Lemma~\ref{lem:stageweights}, we see $\yi_\nu=y^n_\nu+\dt c_if_\nu(\b y^n)+\O(\dt^2)$, and thus, comparing with \eqref{eq:MPRK4condlita}, we derived \eqref{eq:mprk4c}. As a result, we can now apply Lemma \ref{lem:stageweights4} to obtain 
		\begin{equation*}
			\yi_\nu=y^n_\nu+ \dt c_if_\nu(\b y^n)+\frac12\dt^2\sum_{k=1}^4a_{ik}c_k(\D\b f(\b y^n)\b f(\b y^n))_\nu+\O(\dt^3), \quad \nu=1,\dots,N.
		\end{equation*}
		As a direct consequence of this and \eqref{eq:MPRK4condlitb}, we thus  conclude \eqref{eq:mprk4b}. Therefore, it remains to deduce condition \eqref{eq:mprk4a}. First of all, \eqref{eq:MPRK4condlitc} and Theorem \ref{thm:MPRKp=1} imply that the MPRK scheme is of order at least $1$, i.\,e.\ $\b y^{n+1}=\NB_1(\frac1\gamma,\b y^n)+\O(\dt^2).$ Comparing with \eqref{eq:MPRK4condlitc}, we see \[\frac{y^{n+1}_\mu}{\sigma_\mu}=1 +\O(\dt^2),\quad \mu=1,\dotsc,N.\] Moreover, since we have already shown \eqref{eq:mprk4c}, we can now verify that condition \eqref{eq:condp=2} is fulfilled which means that the MPRK scheme is even second order accura$\frac{\yi_\nu}{\pi_\nu^{(i)}}=1+\O(\dt^3)$te. Therefore, we find $\b y^{n+1}=\NB_2(\frac1\gamma,\b y^n)+\O(\dt^3)$, so that a comparison with \eqref{eq:MPRK4condlitc} gives us 
		\[\frac{y^{n+1}_\mu}{\sigma_\mu}=1 +\O(\dt^3),\quad \mu=1,\dotsc,N.\] Finally, using this and \eqref{eq:mprk4c} once again, we even fulfill the conditions \eqref{eq:condp=3} proving the 3rd order accuracy of the scheme, that is $\b y^{n+1}=\NB_3(\frac1\gamma,\b y^n)+\O(\dt^4)$. Comparing a last time with \eqref{eq:MPRK4condlitc} gives us \eqref{eq:mprk4a}.
	\end{proof}
	With this proof, we obtain for the first time necessary and sufficient order conditions
	for 4th order MPRK methods. A first intuitive, yet rather expensive way of achieving $4$th order would be to use lower order MPRK methods for the computation of the PWDs. In particular, we propose the following method based on the classical Runge--Kutta method described by the Butcher tableau
	\begin{equation*}
		\begin{aligned}
			\def\arraystretch{1.2}
			\begin{array}{c|cccc}
				0 &  & & & \\
				\frac12 & \frac12 & & &\\
				\frac12 &0 &\frac12 & &\\
				1 &0 & 0 &1 &\\
				\hline
				& \frac16 &\frac13&\frac13 & \frac16
			\end{array}
		\end{aligned}
	\end{equation*}
	as a proof of concept scheme. We know that the PWD $\bm\sigma$ needs to be a third order approximation to $\b y^{n+1}$, for which we use the MPRK43($0.5,0.75$) method derived in \cite{KM18Order3}. Within this method, there is a second order scheme embedded, which we denote by $\hat{\bm\sigma}$ and use to compute $\bm \pi^{(i)}$ for $i=2,3,4$ using $c_i\dt$ as a time step, resulting in . Now, according to Corollary~\ref{cor:suff_cond} the overall method is of order 4. 
	
	The third order scheme  returns $\hat{\b y}$ and consists of solving $4$ linear systems, where $\hat{\bm\sigma}$ requires solving $2$ systems. However, as $c_4=1$, we can actually use $\bm\pi^{(4)}=\hat{\bm\sigma}$, and since $c_2=c_3$ we can use $\bm \pi^{(2)}=\bm \pi^{(3)}$. Finally, the MP trick applied to the classical RK method also adds $4$ linear systems to our list.  Altogether $\bm \pi^{(2)}$ and $\bm \pi^{(3)}$ yield a total of 2 linear systems, $\bm \pi^{(4)}=\hat{\bm\sigma}$ and $\bm\sigma=\hat{\b y}$ need the solution of $4$ linear systems and the MP approach applied to the classical RK scheme results in another $4$ linear systems giving us a total of $10$ stages and linear systems to solve. The optimal amount of linear systems of course would be $4$ and to reduce the number of linear systems to be solved will be part of my future work. An indication that this is possible is given by MPDeC methods where fourth order is obtained by $7$ stages for Gauss--Lobatto nodes. Nevertheless, our first attempt has as many stages as MPDeCEQ(4). 
	
	The experimental order of convergence of our first fourth order MPRK method, denoted by MPRKord4, is verified in Figure~\ref{fig:MPRK4EOC}, where the linear system 
	\begin{equation}\label{eq:EOC_testprob}
		\b y'(t)=\vec{-5 & \hphantom{-}1\\ \hphantom{-}5& -1}\b y(t), \quad \b y(0)=\vec{0.9\\0.1}
	\end{equation}
	is solved on $[0,1.75]$ as suggested in \cite{KM18Order3}. We plot the error of the numerical solution at the final time $\tend=1.75$, where the reference solution was computed with the \MATLAB ODE solver \code{ode45} using \code{RelTol = 1e-13} and \code{AbsTol = 1e-13}.
	\begin{figure}[!h]
		\centering
		\includegraphics[width = 0.5\textwidth]{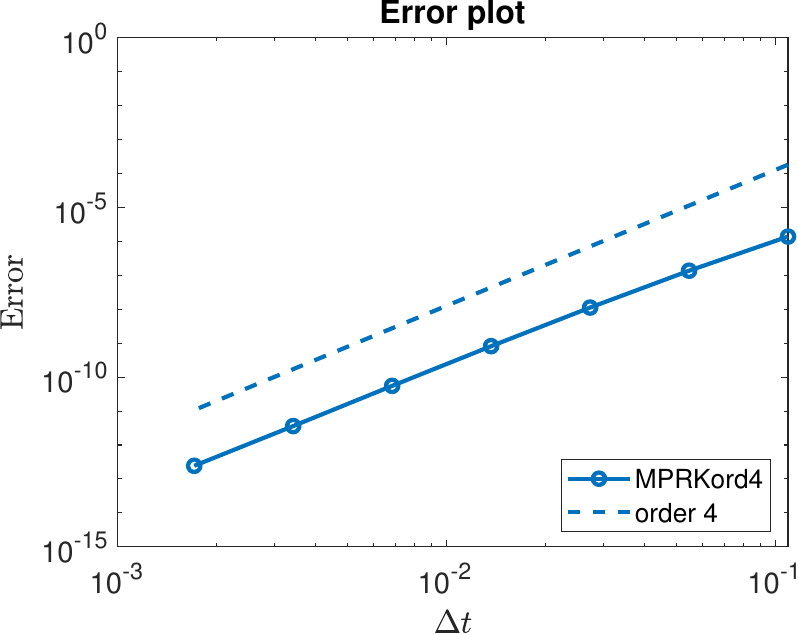}\caption{Error plot of MPRKord4 applied to \eqref{eq:EOC_testprob}. The error was computed at $\tend=1.75$ using \code{ode45} as a reference solution.}\label{fig:MPRK4EOC}
	\end{figure}

\chapter{Stability Theory}\label{chap:stab}
\section{Test Equations and Properties}
As discussed in Section~\ref{sec:stabRK} and Section~\ref{sec:intro_dyn_sys}, we are interested in the stability properties of the positivity-preserving methods reviewed in Chapter~\ref{chap:NumSchemes} when applied to positive linear systems of ordinary differential equations $\b y'=\bA\b y$. Before we formulate assumptions on the system matrix $\bA=(\lambda_{ij})_{1\leq i,j\leq N}$, we introduce the  algebraic multiplicity $\mu_\bA(\lambda)$ of the eigenvalue $\lambda\in \sigma(\bA)$ as well as the corresponding geometric multiplicity $\gamma_\bA(\lambda)$, where \[\sigma(\bA)\tm \Cminus=\{z\in \C\mid \re(z)\leq 0\}\] denotes the spectrum of $\bA$. 

In view of Theorem~\ref{Thm:_Asym_und_Instabil} for the hyperbolic case, we are particularly interested in problems possessing linear invariants such as conservativity. As mentioned in Section~\ref{sec:PDRS}, the presence of $k>0$ linear invariants means that there exist linearly independent vectors $\b n_1,\dotsc,\b n_k\in \R^N\setminus\{\b 0\}$ such that $\b n_i^T\b y(t)=\b n_i^T\b y^0$ for all $t\geq 0$, or equivalently $\b n_i^T\bA=\b 0$ for $i=1,\dotsc,k$. 
Note that the existence of $k$ linear invariants is given if and only if $k=\dim(\ker(\bA^T))=\dim(\ker(\bA))$. 
The presence of $k$ linear invariants means that $\gamma_\bA(0)=k$, so that we consider in the following systems of the form
\begin{equation}\label{eq:PDS_Sys}
	\b y'=\bA\b y,\quad \bA\neq\b 0,\quad  \bA-\diag(\bA)\geq \b 0,\quad \mu_\bA(0)=\gamma_\bA(0)=k,
\end{equation}
together with the initial condition
\begin{equation}
	\b y(0)=\b y^0>\b 0, \label{eq:IC}
\end{equation}
where $\diag(\bA)$ denotes the diagonal of $\bA$. In particular, $\bA-\diag(\bA)\geq \b 0$ means that $\bA$ is a so-called \emph{Metzler} matrix \cite{Luen79}, which is sufficient and necessary to guarantee the positivity of the analytic solution. Moreover, in the presence of linear invariants, the conditions $\mu_\bA(0)=\gamma_\bA(0)$ and $\sigma(\bA)\tm \Cminus$ are necessary for the stability of steady states of $\b y'=\bA\b y$, see Theorem~\ref{Thm:StabLin}. To give an example, the IVP
\begin{align}
	\b y'(t)=\begin{pmatrix}
		-a &\hphantom{-}b\\ \hphantom{-}a &-b
	\end{pmatrix}\b y(t),\quad\b y(0)=\b y^0\in \R^2_{> 0},\label{PDS}
\end{align} 
with $a,b\geq 0$ and $a+b>0$ describes all nontrivial positive and conservative linear problems in $N=2$. To include also non-conservative systems with a linear invariant we may consider
\begin{align}
	\b y'(t)=\begin{pmatrix}
		-a &\hphantom{-}bc\\ \hphantom{-}ac &-b
	\end{pmatrix}\b y(t),\label{PDS_test}
\end{align} 
where $c>0$.

We want to note that if $k=0$, then the only steady state is $\b y^*=\b 0$. As we discuss in the following remark, this steady state is then asymptotically stable.

\begin{rem}\label{rem:Aneg}
	First, we want to mention that at least one diagonal element of $\bA$ is negative. Otherwise we find $\diag(\bA)\geq \b 0$, and hence, $\bA\geq \b 0$. Then, due to  $\mu_\bA(0)=\gamma_\bA(0)=k$ and $\bA\neq \b 0$ we find that $k<N$, and thus, there exists a nonzero eigenvalue of $\bA$. Therefore, $\bA$ is not similar to a strictly upper triangular matrix. Utilizing a generalization of the Perron--Frobenius Theorem \cite[Theorem 2.20]{V00} yields that $\bA$ possesses a positive eigenvalue contradicting $\sigma(\bA)\tm \Cminus$. This means that $\bA$ is a so-called proper Metzler Matrix, i.\,e.\ a Metzler matrix $\bA$ with at least one negative diagonal element. Consequently, \cite[Theorem 10, Corollary 11]{StabMetz} yields \[\sigma(\bA)\tm\mathcal B= \left\{ z\in \C \, \Big| \, \abs{z-r}\leq \abs r, r=\min_{j=1,\dotsc N} \lambda_{jj}\right\},\] where $r<0$ follows since $\bA$ is a proper Metzler matrix. Thus, we obtain $\re(\lambda)<0$ as well as $\arg(\lambda)\in(\tfrac\pi2,\tfrac32\pi)$ for all $0\neq\lambda\in\sigma(\bA)$. Finally, if $k=0$ holds in \eqref{eq:PDS_Sys}, it follows from Theorem~\ref{Thm:StabLin} that $\b y^*=\b 0$ is asymptotically stable.
\end{rem}

Now, since we want to generalize $A$-stability, we may speak of stable methods rather than stating that all steady states become stable fixed points. A precise definition for positivity-preserving methods is given in the following.
\begin{defn}\label{Def:uncondstab} Let \eqref{eq:PDS_Sys}, \eqref{eq:IC} with $k>0$ fulfill the requirements for the application of a given one-step method with generating map $\b g\colon \R^N_{>0}\to \R^N_{>0}$.
	\begin{itemize}
		\item The one-step method is called \emph{conditionally stable}, if there exists a $c>0$ such that any $\b y^*\in \ker(\bA)\cap \R^N_{>0}$ is a Lyapunov stable fixed point of $\b g$ for all $0<\dt< c$.
		\item If the method is conditionally stable and $c>0$ can be chosen arbitrarily large, we call the method \emph{unconditionally stable}.  
		\item If all $\b y^*\in \ker(\bA)\cap \R^N_{>0}$ are unstable fixed points of $\b g$, we call the method \emph{unstable}.
	\end{itemize} 
\end{defn}
\begin{rem}
	In the above definition it is assumed that the one-step method can be applied to the system to \eqref{eq:PDS_Sys}, \eqref{eq:IC}. For conservative schemes this requires $\b 1\in \ker(\bA^T)$. 
	It is also worth mentioning that not being conditionally stable implies instability if the method is linear. However, this does not need to be true for nonlinear one-step methods as we will discuss later. 
\end{rem}

	%
		%


\section{Main Theorem for Stability}

In this section we provide a theorem for the investigation of stability, defined in Section~\ref{sec:intro_dyn_sys}, of the numerical methods from Chapter~\ref{chap:NumSchemes} applied to stable positive linear systems  \eqref{eq:PDS_Sys} with $k>0$. 

As a consequence of the presence of linear invariants, $0$ is always an eigenvalue of $\bA$ which implies the existence of nontrivial steady state solutions $\b y^*$. For every reasonable  time integration scheme $\b y^{n+1}=\b g(\b y^n)$, these steady state solutions have to be fixed points. The common way to study the stability of a fixed point $\b y^*$ of $\b g$ is to compute the eigenvalues of the Jacobian $\b D\b g(\b y^*)$. It is well-known that the fixed point $\b y^*$ is asymptotically stable if the spectral radius $\rho$ of the Jacobian satisfies $\rho(\b D\b g(\b y^*))<1$, see Theorem~\ref{Thm:_Asym_und_Instabil}. Unfortunately, the existence of linear invariants leads to non-hyperbolic fixed points $\b y^*$ of the numerical scheme, \ie the Jacobian $\b D\b g(\b y^*)$ has at least one eigenvalue $\lambda$ with $\abs\lambda = 1$. 

If the time integration scheme applied to \eqref{eq:PDS_Sys} results in a linear iteration
\[\b y^{n+1}=\bm R(\dt,\bA)\b y^n,\]
as is the case for Runge--Kutta schemes, the stability of the non-hyperbolic fixed point $\b y^*$ is again fully determined by the eigenvalues of the Jacobian \[\b D\b g(\b y^*)=\bm R(\dt,\bA)\] as discussed in Remark~\ref{rem:RKstab_dynsys}. 

Unfortunately, the application of higher-order positivity-preserving schemes to the linear system \eqref{eq:PDS_Sys}
results in a nonlinear iteration of the form 
\begin{equation*}
	\b y^{n+1}=\bm R(\dt, \bA,\b y^n)\b y^n,
\end{equation*}
see \cite{OH17} for an illustrative example. 
For such iterations the stability is not fully determined by the eigenvalues of the Jacobian, see for instance Example~\ref{exmp:nonhyper}.
Hence, the stability analysis of these numerical methods requires the investigation of non-hyperbolic fixed points of a nonlinear iteration. This is significantly more demanding compared to the linear case.

One way to study the stability of non-hyperbolic fixed points of nonlinear iterations is the center manifold theory from \cite{mccracken1976hopf,carr1982,iooss1979}, reviewed in Section~\ref{sec:stab_dyn_sys}. This theory states that the stability of a non-hyperbolic fixed point can be determined by studying the iteration on a lower-dimensional invariant manifold, the center manifold.

To avoid the application of the center manifold theory to each positivity-preserving scheme separately, we present a theorem which provides sufficient conditions for the stability of all such methods. Thereby, the main assumption of this new theorem published in \cite{izgin2022stability} is 
that the fixed points of the nonlinear iteration form a linear subspace of $\R^N$. This is a reasonable requirement due to the fact that the steady states of the underlying differential equation \eqref{eq:PDS_Sys} also form a linear subspace of dimension $k>0$, whenever $k$ linear invariants are present.
The theorem contains two main statements. First, the existence of $k$ linear invariants implies that $\lambda=1$ is an eigenvalue of the Jacobian $\b D\b g(\b y^*)$ of multiplicity at least $k$ and the non-hyperbolic fixed point $\b y^*$ is stable, if the remaining $N-k$ eigenvalues have absolute value less than one. Second, if the numerical scheme preservers all $k$ linear invariants, then the iterates locally converge to the unique steady state of the initial value problem \eqref{eq:PDS_Sys}, \eqref{eq:IC}.
Furthermore, it is worth mentioning that the new theorem can directly be used for the stability analysis of time integration schemes in the context of nonlinear systems of differential equations as we will discuss in Remark~\ref{Rem:NonlinearAppl}.

In addition, we want to emphasize at this point that it is not sufficient to assess the stability of a higher-order positivity-preserving scheme in terms of a linear system of the form  
\begin{equation}\label{eq:2x2Dahlquist}
	\b y'=\Vec{
		\lambda & 0\\-\lambda&0} \b y, \quad \b y(0)=\b y^0>\b 0,\quad \lambda\in \R^-,
\end{equation}
which can be seen as a adaptation of Dahlquist's equation
\begin{equation*}
	y'=\lambda y,\quad \lambda\in \C^-,
\end{equation*}
originally introduced in \cite{D63}, to linear conservative systems.
One example for this fact is given in \cite{IKM2122}, where the so-called MPRK22ncs($\alpha$) schemes are investigated. These methods differ from original MPRK schemes in the non-conservative stages (ncs). To be precise, the stages are only treated with the Patankar-trick for guaranteeing unconditional positivity while the modification is only applied to the last step. The total method is still conservative and positive, however the linear systems for the stages are easier to solve. Now, these methods are proven to be $L_0$-stable in the following sense. Applied to the conservative system \eqref{eq:2x2Dahlquist}
the state variable $y_1^n$ satisfies $y_1^{n+1}=R(\dt \lambda)y_1^n$ with 
\begin{equation*}
	R(z)=\frac{(1-\alpha z)^{1-\frac{1}{\alpha}}}{(1-\alpha z)^{1-\frac{1}{\alpha}}-z\bigl(1-(\alpha-\frac12)z\bigr)},
\end{equation*}
so that $\lim_{z\to -\infty}R(z)=0$ and $\lvert R(z)\rvert\leq1$ for all $z\leq0$ and $\alpha\geq \frac12$. In total this means that the first component represents the behavior of the numerical scheme applied to the Dahlquist equation for $\lambda\in \R^-$ and satisfies all conditions for a scheme to be $L_0$-stable, see \cite{TGA96}.
Nevertheless, in \cite{IKM2122} it is proved that MPRK22ncs($\alpha$) face severe time step restrictions for $\alpha<1$ in order to be stable when applied to a general two--dimensional linear positive and conservative system \eqref{PDS},
which was also used in \cite{IKM21} for studying the linearization of MPRK22 schemes.
Hence, to understand the stability behavior of such nonlinear schemes, one should directly investigate the system \eqref{eq:PDS_Sys}.

\subsection*{Main Result for Stability}
In this subsection we make use of the center manifold theory to investigate the stability of fixed points $\b y^*$ of a numerical scheme $\b y^{n+1}=\b g(\b y^n)$ with $\b g\from D\to D$ and $D\tm \R^N$. To that end, we assume that there exists a neighborhood $\mathcal D\tm D$ of $\b y^*$ such that $\b g\big|_\mathcal{D}\in \mathcal C^1$ has first derivatives that are Lipschitz continuous on $\mathcal{D}$, so we can apply Theorem~\ref{Thm:Ex_CM}. Based on this assumption, Theorem~\ref{Thm_MPRK_stabil} below yields a sufficient condition for the Lyapunov stability of $\b y^*$ based on the eigenvalues of the corresponding Jacobian $\b D\b g(\b y^*)$. If $\b g$ in addition conserves all linear invariants of $\bA$ from \eqref{eq:PDS_Sys}, \ie $\b n^T\b g(\b y)=\b n^T\b y$ for all $\b y\in D$ whenever $\b n^T\bA=\b 0$, then Theorem~\ref{Thm_MPRK_stabil} also states that the numerical scheme locally converges towards the unique steady state $\b y^*$ of \eqref{eq:PDS_Sys}, \eqref{eq:IC}.


For a compact notation we introduce the matrix
\begin{equation}\label{eq:N}
	\b N=\begin{pmatrix}
		\b n_1^T\\
		\vdots\\
		\b n_k^T
	\end{pmatrix}\in \R^{k\times N} 
\end{equation}
with $\b n_1,\dotsc,\b n_k$ being a basis of $\ker(\bA^T)$ as well as the set
\begin{equation}
	H=\{\b y\in \R^N\mid \b N\b y=\b N\b y^*\}\label{eq:H}
\end{equation}
and point out that for $\b y\in H\cap D$ we have $\b g(\b y)\in H\cap D$, if and only if $\b g$ conserves all linear invariants.

\begin{thm}\label{Thm_MPRK_stabil}
	Let $\bA\in \R^{N\times N}$ be such that $\ker(\bA)=\Span(\b v_1,\dotsc,\b v_k)$ represents a $k$-dimensional subspace of $\R^N$ with $k>0$. Also, let $\b y^*\in \ker(\bA)$ be a fixed point of $\b g\from D\to D$ where $D\tm \R^N$ contains a neighborhood $\mathcal D$ of $\b y^*$. Moreover, let any element of $C=\ker(\bA)\cap \mathcal D$ be a fixed point of $\b g$ and suppose that $\b g\big|_\mathcal{D}\in \mathcal C^1$ as well as that the first derivatives of $\b g$ are Lipschitz continuous on $\mathcal{D}$. Then $\b D\b g(\b y^*)\b v_i=\b v_i$ for $i=1,\dotsc, k$ and the following statements hold.
	\begin{enumerate}
		\item\label{it:Thma} If the remaining $N-k$ eigenvalues of $\b D\b g(\b y^*)$ have absolute values smaller than $1$, then $\b y^*$ is stable.\label{It:Thm_Stab_a}
		\item\label{it:Thmb} Let $H$ be defined by \eqref{eq:H} and $\b g$ conserve all linear invariants, which means that $\b g(\b y)\in H\cap D$ for all $\b y\in H\cap D$. If additionally the assumption of \ref{It:Thm_Stab_a} is satisfied, then there exists a $\delta>0$ such that $\b y^0\in H\cap D$ and $\norm{\b y^0-\b y^*}<\delta$ imply $\b y^n\to \b y^*$ as $n\to \infty$.
	\end{enumerate}
\end{thm}
Before we prove the above theorem we want to emphasize in the next remark that its application is not restricted to linear systems of differential equations \eqref{eq:PDS_Sys}. 
\begin{rem}\label{Rem:NonlinearAppl}
	Let us consider a general system of autonomous ordinary differential equations $\b y'=\b f(\b y)\in \R^N$ with $k>0$ linear invariants determined by $\b n_1,\dotsc,\b n_k$ and a $k$--dimensional subspace $\mathcal V=\Span(\b v_1,\dotsc,\b v_k)\tm\{\b y\in \R^N\mid \b f(\b y)=\b 0\}$. In the following, we construct a matrix $\bA$ such that $\ker(\bA)=\mathcal V$ as well as $\ker(\bA^T)=\Span(\b n_1,\dotsc,\b n_k)$, and thus are in the position to apply Theorem~\ref{Thm_MPRK_stabil}.
	
	As $\bA$ is uniquely determined by its operation on a basis of $\R^N$ we first set $\bA\b v_i=\b 0$  for $i=1,\dotsc,k$  so that $\ker(\bA)=\mathcal V$ is satisfied. To find an expression for $\im(\bA)$ we make use  of  $\im(\bA)=(\ker(\bA^T))^\perp=(\Span(\b n_1,\dotsc,\b n_k))^\perp$. Using the matrix notation \eqref{eq:N}, this means that $\b s\in \im(\bA)$ if and only if $\b N\b s=\b 0$, or equivalently $\b s\in \ker(\b N)$. Since $\dim(\ker(\b N))=N-k$, there exist linearly independent vectors $\b s_1,\dotsc,\b s_{N-k}$ with $\im(\bA)=\Span(\b s_1,\dotsc,\b s_{N-k})$, and hence, there exist linearly independent vectors $\b w_1,\dotsc,\b w_{N-k}$ such that $\bA\b w_i=\b s_i$ for $i=1,\dotsc, N-k$. As a consequence, setting $\mathcal{W}=\Span(\b w_1,\dotsc,\b w_{N-k})\tm \R^N$ yields $\mathcal{V}\oplus\mathcal{W}=\R^N$.
	Altogether, $\mathcal V$ and $\mathcal W$ uniquely determine  the matrix $\bA$ satisfying $\ker(\bA)=\mathcal V$ and $\ker(\bA^T)=\Span(\b n_1,\dotsc,\b n_k)$. Hence, Theorem \ref{Thm_MPRK_stabil} is not restricted to linear systems.
\end{rem}

\begin{proof}[Proof of Theorem \ref{Thm_MPRK_stabil}]
	First, we show $\b D\b g(\b y^*)\b v_i=\b v_i$ for $i=1,\dotsc,k$. Since $\b g$ is differentiable in $\b y^*\in \mathcal D$ the directional derivatives $\partial_\b v \b g(\b y^*)=\b D\b g(\b y^*)\b v$ exist for all directions $\b v\in\R^N$ and for $i=1,\dotsc,k$ we find
	\[\b D\b g(\b y^*)\b v_i=\partial_{\b v_i}\b g(\b y^*)=\lim_{h\to 0}\frac1h\bigl(\b g(\b y^*+h\b v_i)-\b g(\b y^*)\bigr). \]
	For $\lvert h\rvert$ small enough, we see that $\b y^*+h\b v_i\in C$ because of the following. First of all $\b y^*+h\b v_i\in\ker(\bA)$ holds for all $h\in\R$, so that we have to show that $\b y^*+h\b v_i\in\mathcal D$ for $\lvert h\rvert$ small enough. Since $\b y^*\in \mathcal{D}$, there exists a $\gamma>0$ such that the open ball $B_\gamma(\b y^*)$ with center $\b y^*$ and radius $\gamma$ satisfies $B_\gamma(\b y^*)\tm \mathcal D$. Choosing $\lvert h\rvert<\frac{\gamma}{\Vert \b v_i\Vert}$ we find \[ \Vert \b y^*+h\b v_i-\b y^*\Vert \leq \lvert h\rvert \Vert \b v_i\Vert <\gamma,\] such that $\b y^*+h\b v_i\in\ker(\bA)\cap B_\gamma(\b y^*)\tm \ker(\bA)\cap\mathcal D=C$ is a fixed point of $\b g$. Hence,
	\[\b D\b g(\b y^*)\b v_i=\lim_{h\to 0}\frac1h\bigl(\b y^*+h\b v_i-\b y^*\bigr)=\b v_i, \]
	which shows that $\b v_i$ is an eigenvector of $\b D\b g(\b y^*)$ with associated eigenvalue $1$.
	Thus, the spectrum of $\b D\b g(\b y^*)$ contains the eigenvalue $1$ with a multiplicity of at least $k$.
	
	\begin{enumerate} 
		\item\label{item:stability}
		We now assume that the remaining $N-k$ eigenvalues  of $\b D\b g(\b y^*)$ have absolute values smaller than 1.
		Next we introduce the matrix of generalized eigenvectors $\b S$ where the first $k$ columns are given by the basis vectors $\b v_1,\dotsc, \b v_k$ of $\ker(\bA)$. Thus, we obtain
		\begin{equation}\label{eq:DG_digaonal}
			\b S^{-1}\b D\b g(\b y^*)\b S=\b J
		\end{equation}
		with the Jordan normal form $\b J$ of $\b D\b g(\b y^*)$. We want to point out that the upper left $k\times k$ block of $\b J$ is the identity matrix, since the $k$ basis vectors $\b v_1,\dotsc, \b v_k$ of $\ker(\bA)$ are eigenvectors with associated eigenvalue $1$.
		
		We want to use the Theorem~\ref{Thm:Ex_CM}~\ref{item:existence} in combination with Theorem~\ref{Thm:Stab_CM} to conclude that $\b y^*$ is a stable fixed point.
		The theorems require a map $\b G$ of form \eqref{Form_of_G}, which shall be obtained from $\b g$ by means of an affine linear transformation.
		We consider the affine transformation \[\b T\from\R^N\to\R^N,\quad \b y\mapsto\b w=\b T(\b y)=\b S^{-1}(\b y-\b y^*),\] where
		the inverse transformation $\b T^{-1}$ is given by $\b T^{-1}(\b w)=\b S\b w+\b y^*$.
		By construction, $\ker(\bA)$ is mapped onto the subspace spanned by the first $k$ unit vectors $\b e_1,\dotsc,\b e_k$ of $\R^N$, as for $\b y^*=\sum_{i=1}^kt_i\b v_i\in \ker(\bA)$ we find
		\begin{equation*}
			\begin{aligned}
				\b T\left(\sum_{i=1}^kr_i\b v_i\right)&=\b S^{-1}\left(\sum_{i=1}^kr_i\b v_i-\b y^*\right)=\b S^{-1}\left(\sum_{i=1}^k(r_i-t_i)\b v_i\right)\\
				&=\sum_{i=1}^k(r_i-t_i)\b S^{-1}\b v_i=\sum_{i=1}^k(r_i-t_i)\b e_i
			\end{aligned}
		\end{equation*}
		for arbitrary choices of $r_1,\dotsc,r_k\in \R$.
		In particular, $\b y^*$ is mapped to the origin.

		In order to use Theorem \ref{Thm:Ex_CM}, we have  to define an appropriate $\mathcal C^1$-map $\b G\from \mathcal M\to \R^N$. Therefore we define $\mathcal M=\b T(\mathcal D)$ which is a neighborhood of the origin since $\b T$ is an invertible affine linear map. In particular, we use
		\begin{equation}
			\b G\from \b T(\mathcal D)\to \R^N,\quad \b G(\b w) = \b T(\b g(\b T^{-1}(\b w)))\label{eq:TgTinv}
		\end{equation}
		and observe that the origin is a fixed point of $\b G$. To represent $\b G$ in the form \eqref{Form_of_G}, we use $\b g(\b y^*)=\b y^*$ and write $\b g$ as
		\begin{equation}
			\begin{aligned}
				\b g(\b y)&=\b g(\b y^*)+\b D\b g(\b y^*)(\b y-\b y^*)+\b R(\b y)\\&=\b y^*+\b D\b g(\b y^*)(\b y-\b y^*)+\b R(\b y),\label{yn+1_Lagrange_remainder}
			\end{aligned}
		\end{equation}
		where the remainder $\b R(\b y)$ can be written as
		\begin{align*}
			\b R(\b y)=\b g(\b y)-\b y^*-\b D\b g(\b y^*)(\b y-\b y^*).
		\end{align*}
		In particular, we have
		\begin{equation}
			{\b R}(\b y^*)=\b 0,\quad
			\b D {\b R}(\b y^*)=\b 0.\label{R(0)=0,DR(0)=0}
		\end{equation}
		By inserting \eqref{yn+1_Lagrange_remainder} in \eqref{eq:TgTinv} we obtain
		\begin{equation*}
			\begin{aligned}
				\b G(\b w) &= \b S^{-1}\bigl(\b D\b g(\b y^*)(\b T^{-1}(\b w)-\b y^*)+\b R(\b T^{-1}(\b w))\bigr)\\
				&=\b S^{-1}\b D\b g(\b y^*)\b S\b w + \b S^{-1}\b R(\b T^{-1}(\b w))
			\end{aligned}
		\end{equation*}
		and using \eqref{eq:DG_digaonal} yields
		\begin{equation}\label{eq:G}
			\b G(\b w) = \b J\b w + \b S^{-1}\b R(\b T^{-1}(\b w))=\begin{pmatrix}
				\b I &\\
				& \bm R 
			\end{pmatrix}\b w + \b S^{-1}\b R(\b T^{-1}(\b w)),
		\end{equation}
		where $\b I\in \R^{k\times k}$ and $\bm R\in \R^{(N-k) \times (N-k)}$ and $\rho(\bm R)<1$ as $N-k$ eigenvalues of $\b D\b g(\b y^*)$ have absolute values smaller than $1$.
		Setting $\b w=(\b w_1,\b w_2)^T$ with $\b w_1\in \R^{k}$, $\b w_2\in \R^{N-k}$ and $(\b w_1,\b w_2)\in \b T(\mathcal D)$, \eqref{eq:G} can be rewritten as
		\begin{equation}\label{eq:G_form}
			\b G(\b w_1,\b w_2)=\Vec{\b U \b w_1+\b u(\b w_1,\b w_2)\\\b V \b w_2+\b v(\b w_1,\b w_2)}
		\end{equation} 
		with
		\begin{equation}\label{eq:UVuv}
			\begin{aligned}
				\b U&=\b I,  & \b u(\b w_1,\b w_2)&=\bigl(\b S^{-1}{\b R}(\b T^{-1}(\b w_1,\b w_2))\bigr)_{1:k},\\
				\b V&= \bm R, & \b v(\b w_1,\b w_2)&=\bigl(\b S^{-1}{\b R}(\b T^{-1}(\b w_1,\b w_2))\bigr)_{k+1:N,}
			\end{aligned}
		\end{equation}
		where we defined $\b v_{l:m}=(v_l,\dotsc,v_m)^T$ for a vector $\b v$ and $l\leq m$.
		Each eigenvalue of $\b U$ has absolute value 1 and those of $\b V$ have absolute values smaller than $1$. Furthermore, utilizing $\b T^{-1}(\b 0,\b 0)=\b y^*$ we conclude from \eqref{R(0)=0,DR(0)=0} that
		$\b u(\b 0,\b 0) = \b v(\b 0,\b 0)=\b 0$, since $\b R(\b y^*)=\b 0$. Moreover, we have $\b D \b u(\b 0,\b 0) = \b D \b v(\b 0,\b 0)=\b 0$, since $\b D\b R(\b y^*)=\b 0$.
		Altogether this demonstrates that \eqref{eq:G} is of form \eqref{Form_of_G}, which is necessary for applying the center manifold theory.

		Now, the center manifold theorem~\ref{Thm:Ex_CM}~\ref{item:existence} states that for some $\epsilon>0$ there exists a $\mathcal{C}^1$ function $\b h\from\R^k\to~ \R^{N-k}$ with $\b h(\b 0)=\b 0$ and $\b D\b h(\b 0)=\b 0$, such that $(\b w_1^1,\b w_2^1)^T=\b G(\b w_1^0,\b h(\b w_1^0))$  implies $\b w_2^1=\b h(\b w_1^1)$ for $\norm{\b w_1^0} ,\norm{\b w_1^1}<\epsilon$. 
		
		In the following we make use of the fact that the center manifold is given by 
		\begin{equation}
			\{(\b w_1,\b w_2)\in\R^N\mid \b w_2=\b 0,\ \Vert \b w_1\Vert <\epsilon\},\label{eq:CM}
		\end{equation}
		i.\,e.\ $\b h(\b w_1)=\b 0$, for a sufficiently small $\epsilon>0$, which can be shown with Theorem~\ref{Thm:Comp_func_h}.
		The function $\bm\Phi\from\R^k\to\R^{N-k}$, $\bm\Phi(\b w_1)=\b 0$ satisfies $\bm\Phi(\b 0)=\b 0$ and $\b D\bm\Phi(\b 0)=\b 0$. In order to compute $\b h$ we first prove that all points $(\b w_1,\b 0)\in \b T(\mathcal D)$ are fixed points of $\b G$. Note, that points $(\b w_1,\b 0)\in \b T(\mathcal D)$ even satisfy
		\begin{equation*}
			\begin{aligned}
				\b T^{-1}(\b w_1,\b 0)&=\b T^{-1}\left(\sum_{i=1}^k(\b w_1)_i\b e_i\right)=\sum_{i=1}^k(\b w_1)_i\b S\b e_i+\b y^*\\
				&=\sum_{i=1}^k(\b w_1)_i\b v_i+\b y^*\in \mathcal D\cap \ker(\bA)=C.
			\end{aligned}
		\end{equation*}
		Hence, we find 
		\begin{equation}\label{eq:G(w1,0)=(w1,0)}
			\b G(\b w_1,\b 0)=  \b T\left(\b g\left(\b T^{-1}\left(\b w_1,\b 0\right)\right)\right)
			=\b T\left(\b T^{-1}\left(\b w_1,\b 0\right)\right)=(\b w_1,\b 0)^T.
		\end{equation}
		Thus, it follows that
		\begin{equation*}
			\begin{aligned}
				\bm\Phi(\b U \b w_1+\b u(\b w_1,\bm\Phi(\b w_1)) &- \left(\b V \bm\Phi(\b w_1)+\b v(\b w_1,\bm\Phi(\b w_1))\right)\\
				&\overset{\eqref{eq:G_form}}{=}-(\b G(\b w_1,\b 0))_{k+1:N}=\b 0.
			\end{aligned}
		\end{equation*}
		By Theorem~\ref{Thm:Comp_func_h}, $\bm\Phi$ is an approximation of $\b h$ for any order $q>1$. Thus, 
		\begin{equation*}
			\b h(\b w_1)=\bm\Phi(\b w_1)=\b 0 \text{ for }\norm{\b w_1}<\epsilon.
		\end{equation*}
		To investigate the stability of $\b y^*$, we can now consider the map
		\begin{equation*}
			\mathcal  G(\b w_1)=\b U \b w_1 + \b u(\b w_1,\b h(\b w_1))=\b U \b w_1 + \b u(\b w_1,\b 0)
		\end{equation*}
		for $\norm{\b w_1}<\epsilon$, where $\b U$ and $\b u$ are given in \eqref{eq:UVuv}.
		According to Theorem~\ref{Thm:Stab_CM}, the fixed point $\b 0\in\R^N$ of $\b G$ is stable, if the fixed point $\b 0\in\R^k$ is a stable fixed point of $\mathcal G$.
		From \eqref{eq:G(w1,0)=(w1,0)} we see
		\begin{equation*}
			\mathcal G(\b w_1)=\left(\b G(\b w_1,\b 0)\right)_{1:k}=\b w_1,
		\end{equation*}
		which implies $\b w_1^{n}=\mathcal G(\b w_1^{n-1})=\b w_1^0$ for all $n\in\N$ and every $\b w_1^0$ with $\norm{\b w_1^0}<\epsilon$. Consequently, for every $\widetilde\epsilon>0$ we define $\widetilde\delta = \min\{\widetilde\epsilon,\epsilon\}$ to obtain that $\norm{\b w_1^0}<\widetilde\delta$ implies $\norm{\b w_1^n}=\norm{\b w_1^0}<\widetilde\delta\leq \widetilde\epsilon$. Thus, $\b 0\in \R^k$ is a stable fixed point of $\mathcal G$ in the sense of Definition~\ref{Def_Lyapunov_Diskr}~\ref{def:stab}.
		Furthermore, by Theorem~\ref{Thm:Stab_CM} the fixed point $\b 0\in\R^N$ of $\b G$ is stable as well.
		
		As a last step, we show that the above conclusions imply that $\b y^*$ is a stable fixed point of $\b g$.
		We know that $\b 0$ is a stable fixed point of the iteration scheme $\b w^{n+1}=\b G(\b w^n)$, that is for every $\epsilon_w>0$ exists $\delta_w>0$ such that $\norm{\b w^0}<\delta_w$ implies $\norm{\b w^n}<\epsilon_w$. Now, let $\epsilon>0$ be arbitrary, we define $\epsilon_w=\epsilon/\norm{\b S}$ and  $\delta=\delta_w/\norm{\b S^{-1}}$. Hence, if $\norm{\b y^0-\b y^*}<\delta$, then \[\norm{\b w^0}=\norm{\b T(\b y^0)}=\norm{\b S^{-1}(\b y^0-\b y^*)}\leq \norm{\b S^{-1}}\norm{\b y^0-\b y^*}<\norm{\b S^{-1}}\delta=\delta_w\] and consequently $\norm{\b w^n}<\epsilon_w$. Furthermore, $\b w^n=\b T(\b y^n)=\b S^{-1}(\b y^n-\b y^*)$ is equivalent to $\b S\b w^n=\b y^n-\b y^*$ and hence,  $\norm{\b y^n-\b y^*}\leq\norm{\b S}\norm{\b w^n}<\norm{\b S}\epsilon_w=\epsilon$. Thus, we have shown that $\b y^*$ is a stable fixed point of the iteration scheme $\b y^{n+1}=\b g(\b y^n)$.
		%

		\item
		
		Recall from \eqref{eq:H} that $H=\{ \b y\in \R^N\mid \b N\b y=\b N\b y^*\}$ and let $\b y^0\in H\cap D$, where $\b N$ is given by \eqref{eq:N}. Note, that $\dim(H)=N-k$ as $\b N$ has rank $k$, and $\b y^n\in H$ for all $n\in \N_0$ since $\b g(\b y)\in H$ for all $\b y\in H\cap D$. Moreover, for all $\b y\in H$ we find \[(\b y-\b y^*)\perp \ker(\bA^T)=\Span(\b n_1,\dotsc,\b n_k)\] since $\b N(\b y-\b y^*)=\b N\b y^*-\b N\b y^*=\b 0$. Hence $\b y^n-\b y^* \in (\ker(\bA^T))^\perp=\im(\bA)$ for all $n\in \N_0$.  We now want to show that the last $N-k$ column vectors of the invertible matrix $\b S=(\b v_1 \dotsc \b v_k \b v_{k+1} \dotsc \b v_N)$ of generalized eigenvectors associated with $\b D\b g(\b y^*)$, see  \eqref{eq:DG_digaonal}, form a basis of $\im(\bA)$. Since $\b g$ conserves all linear invariants we observe
		\begin{equation*}
			\begin{aligned}
				\b n_i^T\b D\b g(\b y^*)\b v&=\lim_{h\to 0}\frac{1}{h}\Bigl(\b n_i^T\b g(\b y^*+h\b v)-\b n_i^T\b g(\b y^*)\Bigr)\\&=\lim_{h\to 0}\frac{1}{h}\Bigl(\b n_i^T(\b y^*+h\b v)-\b n_i^T\b y^*\Bigr)=\b n_i^T\b v
			\end{aligned}
		\end{equation*}
		for all $\b v\in \R^N$, and in particular we find
		\begin{equation}
			\b n_i^T(\b D\b g(\b y^*)-\lambda\b I)\b v = \b n_i^T\b D\b g(\b y^*)\b v - \lambda\b n_i^T\b v=(1-\lambda)\b n_i^T\b v.\label{eq:n_i(Dg(y*)-lambdaI)v}
		\end{equation}
		If $\b v$ is a generalized eigenvector of $\b D\b g(\b y^*)$ corresponding to an eigenvalue $\lambda\neq 1$, so that \[(\b D\b g(\b y^*)-\lambda\b I)^m\b v =\b 0\] is satisfied for some $m\in \N$, it follows from \eqref{eq:n_i(Dg(y*)-lambdaI)v} that
		\begin{equation*}0=\b n_i^T(\b D\b g(\b y^*)-\lambda\b I)^m\b v=(1-\lambda)\b n_i^T(\b D\b g(\b y^*)-\lambda\b I)^{m-1}\b v=(1-\lambda)^m\b n_i^T\b v,
		\end{equation*}
		which implies $\b n_i^T\b v=0$ as $\lambda\neq 1$. Hence, all generalized eigenvectors $\b v$ corresponding to an eigenvalue $\lambda\neq 1$ are elements of $(\ker(\bA^T))^\perp=\im(\bA)$. Now note that $\b v_{k+1},\dotsc,\b v_N$ are $N-k$ generalized eigenvectors corresponding to eigenvalues of absolute value smaller than 1. Finally, since \[\dim(\im(\bA))=N-\dim(\ker(\bA))=N-k,\]
		the vectors $\b v_{k+1},\dotsc,\b v_N$ form a basis of $\im(\bA)$. Since \[\b y^n-\b y^*\in \im(\bA)=\Span(\b v_{k+1},\dotsc,\b v_N),\] there exist coefficients $\gamma^n_i\in \R$ such that for all $n\in \N_0$ we can write
		\begin{align}
			\b y^n-\b y^*=\sum_{i=k+1}^N \gamma^n_i\b v_i.\label{eq:yn-y*}
		\end{align}
		
		In order to prove the local convergence of the iterates $\b y^n$ to $\b y^*$ we  investigate the local convergence of $\b w^n$ to the origin. According to Theorem~\ref{Thm:Ex_CM}~\ref{item:attractivity}
		the distance of the iterates $\b w^n\in\R^N$ from \ref{item:stability} to the center manifold given in \eqref{eq:CM} tends to zero for $n\to \infty$, if the iterates stay within a certain neighborhood of the origin. More precisely, this means that the sequence  $(\b w^n)_{n\in \N_0}$ approaches \[\{(\b w_1,\b w_2)\in \R^N\mid \norm{ \b w_1}<\epsilon,\b w_2=\b 0\} = \Span(\b e_1,\dotsc,\b e_k)\cap B_\epsilon(\b 0)\] for $n\to \infty$, if $\norm{ \b w^n}<\epsilon$ for $i=1,\dotsc, N$ and all $n\in \N_0$, where $\epsilon>0$ is sufficiently small. Now, since the origin is a stable fixed point of $\b G$, as shown in \ref{item:stability}, there exists $\widetilde{\delta}>0$ such that $\norm{\b w^0}<\widetilde{\delta}$ implies $\norm{\b w^n}<\epsilon$ for all $n \in\N_0$. Assuming $\norm{\b w^0}<\widetilde{\delta}$, we can conclude 
		\begin{equation}\label{eq:limwn}
			\lim_{n\to\infty}\b w^n\in \Span(\b e_1,\dotsc,\b e_k).
		\end{equation}
		Furthermore, from \eqref{eq:yn-y*} it follows
		\begin{equation*}
			\begin{aligned}
				\b w^n=\b T(\b y^n)&=\b S^{-1}(\b y^n-\b y^*)=\b S^{-1}\biggl(\sum_{i=k+1}^N \gamma^n_i\b v_i\biggr)\\&=\sum_{i=k+1}^N \gamma^n_i\b S^{-1}\b v_i=\sum_{i=k+1}^N\gamma^n_i\b e_i.
			\end{aligned}
		\end{equation*}
		In particular this means $\b w^n\in \Span(\b e_{k+1},\dots,\b e_N)$, and hence, in combination with \eqref{eq:limwn} one obtains \[\lim_{n\to\infty}\b w^n\in \Span(\b e_1,\dotsc,\b e_k)\cap\Span(\b e_{k+1},\dotsc,\b e_N)=\{\b 0\},\] i.\,e.\ $\lim_{n\to\infty}\b w^n=\b 0$.  Due to the transformation $\b T$ this is equivalent to $\lim_{n\to\infty}\b y^n=\b y^*$ for $\b y^0\in H\cap D$ satisfying $\norm{\b y^0-\b y^*}<\delta=\widetilde{\delta}/\norm{\b S^{-1}}$ since then \[\norm{\b w^0}=\norm{\b T(\b y^0)}=\norm{\b S^{-1}(\b y^0-\b y^*)}\leq \norm{\b S^{-1}}\norm{\b y^0-\b y^*}<\widetilde{\delta}\] follows.
	\end{enumerate}
\end{proof}
\begin{rem}\label{rem:th2.9}
	The novel theorem presented here is a generalization of \cite[Theorem 2.9]{IKM2122} and improves its statements considerably. First of all, \cite[Theorem 2.9]{IKM2122} is restricted to systems of size $2\times 2$, whereas here we consider the general $N\times N$ case.
	Second, \cite[Theorem 2.9]{IKM2122} is restricted to conservative numerical schemes, whereas the novel theorem can be applied to general iteration maps $\b g:D\to D$.
	Third, \cite[Theorem 2.9]{IKM2122} does not make clear that the stability of the non-hyperbolic fixed point requires less assumptions than the local convergence towards it.
	In the theorem presented above, on the other hand, it becomes evident that the preservation of linear invariants is not at all necessary to guarantee the stability of the fixed point. Therefore, the theorem can be applied to study the stability of methods that do not preserve all linear invariants. 
	Moreover, the new theorem is formulated with less restrictive assumptions on the regularity of the map generating the numerical approximations.
\end{rem}
\begin{rem}\label{rem:C2->C1}
	As a final remark, we note that if $\b g\in \mathcal{C}^2$, which is also assumed in \cite{IKM2122}, we may choose $\mathcal D\tm D$ in such a way that $\overline{\mathcal D}\tm D$. As a result the second derivatives are bounded on the compact set $\overline{\mathcal D}$, so that the first derivatives are Lipschitz continuous due to the mean value theorem. Therefore, $\b g$ restricted to $\mathcal D$ is a $\mathcal C^1$-map with Lipschitz continuous derivatives. For more details, see for example \cite[Remark 8.12 (b)]{AE08}.
	
	Moreover, we want to mention that Theorem~\ref{Thm_MPRK_stabil} recently was applied to analyze the stability properties of MPRK22($\alpha$) when applied to a nonlinear systems of ordinary differential equations, see \cite{IKMnonlin22}.
\end{rem}

\section{A Necessary Condition for Non-Oscillatory Schemes}
 In this section, we investigate the connection between oscillations \cite{IssuesMPRK} and the stability theory above for $N=2$. To that end, we first rewrite all 2--dimensional linear systems of ODEs that are positive and conservative, \ie \eqref{PDS}, with a change of variables, as the following IVP
\begin{equation}\label{eq:Testprob}
	\begin{cases}
		\b y'(t)=\bA_\theta\b y(t),\\
		\b y(0)=\b y^0>\b 0,
	\end{cases}\quad \bA_\theta=\begin{pmatrix*}[r]
		-\theta &1-\theta\\ \theta & -(1-\theta)
	\end{pmatrix*},\quad \theta\in(0,1),
\end{equation}
where this can be seen as PDS, with $p_{12}=d_{21}=(1-\theta)y_2$, $d_{12}=p_{21}=\theta y_1$ and all other entries zero. Let us also consider a one-step numerical method whose iterates are generated by a map $\b g$, i.\,e.\ $\b y^{n+1} = \b g(\b y^n)$. Note that $\b g$ might be given implicitly.

We first describe oscillations for 2--dimensional linear ODEs through the solution and the steady state. It is known that the exact solution does not overshoot the steady state, so that we require the same from the numerical approximation.
\begin{defn}
	\begin{enumerate}
		\item\label{it:defovershoota} A method is \emph{not overshooting} the steady state of \eqref{eq:Testprob} if $y_2^1<\theta$ and $y_1^1>1-\theta$ for any given initial state $\b y^0=(1-\epsilon,\epsilon)^T$ with $\epsilon<\theta$, while when $\epsilon>\theta$ the method is \emph{not overshooting} the steady state if $y_2^1>\theta$ and $y_1^1<1-\theta$.
		\item Otherwise the method is said to be \emph{overshooting} the steady state of \eqref{eq:Testprob}.
	\end{enumerate}	
\end{defn}
The following theorem extends the results from \cite{IKM2122} to statements regarding oscillatory behavior. To apply the corresponding theory, we assume $\b g$ to have the same properties as in \cite[Theorem 2.9]{IKM2122}.
\begin{thm}\label{Thm:osci}
	Let any positive steady state of \eqref{eq:Testprob} be a fixed point of a map $\b g\in \mathcal{C}^2(\R^2_{>0})$. In addition, let the iterates generated by $\b y^{n+1}=\b g(\b y^n)$ satisfy $\lVert \b y^{n+1} \rVert_1=\lVert \b y^n \rVert_1$ for all $n\in \N_0$. Finally, let $\b y^*$ be the unique positive steady state of \eqref{eq:Testprob}.
	
	Then, the spectrum of the Jacobian $\b D\b g(\b y^*)$ is $\sigma(\b D\b g(\b y^*))=\{1,R\}$ with $R\in\R$. Furthermore, if $R<0$, then the method generated by $\b g$ is overshooting the steady state of \eqref{eq:Testprob}.
\end{thm}
\begin{proof}
	Throughout this proof, we use $\b e_1=(1,0)^T$, $\b e_2=(0,1)^T$ to denote the standard unit vectors as well as the notation $\bby=(1,-1)^T$. In the proof of \cite[Theorem 2.9]{IKM2122}, it is shown that $\b D\b g(\b y^*)\b y^*=\b y^*$ and $\b D\b g(\b y^*)\bby=R\bby$ with $R\in \R$, which means that the matrix of eigenvectors
	\begin{equation}\label{eq:Smatrix}
		\b S=(\b y^*\;\;  \bby)
	\end{equation}
	is invertible since $\bby$ cannot be a multiple of the positive vector $\b y^*$. 
	%
	Along the lines of Theorem~\ref{Thm_MPRK_stabil}, we construct a map
	\begin{equation*}
		\b G\from \b T(\R_{>0}^2)\to \b T(\R_{>0}^2),\quad \b G(\b w) = \b T(\b g(\b T^{-1}(\b w)))
	\end{equation*}
	by means of a transformation $\b T(\b y)=\b S^{-1}(\b y-\b y^*)$. To see that the method defined by $\b g$ is overshooting $\b y^*$, we show that the transformed method given by the map $\b G$	is overshooting the transformed steady state which is $\b w^*=\b 0$.
	As demonstrated in \cite[Theorem 2.9]{IKM2122}, $\b y^0$ is transformed onto the $w_2$-axis and due to the conservation of the map $\b g$, it is proven that $\b G(\b w^0)\in\Span(\b w^0)$ for $\b w^0=(0,w_2^0)^T$. Moreover, 
	\begin{equation}\label{eq:Gproof}
		\b G(\b w) = \diag(1,R)\b w + \b S^{-1}\bar{\b R}(\b T^{-1}(\b w))
	\end{equation}
	holds, where $\bar{\bm R}$ denotes the Lagrangian remainder
	\begin{align}
		(\bar{\bm R}(\b y))_i=\frac{1}{2}(\b y-\b y^*)^T\b H g_i(\b y^*+c_i(\b y-\b y^*))(\b y-\b y^*),\quad i=1,2\label{Lagrange_remainder}
	\end{align}
	for some $c_i\in(0,1)$ depending on $\b y$ and $\b y^*$, and where $\b H g_i$ are the Hessian matrices of $g_i$ for $i=1,2$.
	We consider from now on the iterates given by
	\begin{equation*}
		\b w^{n+1}=\begin{pmatrix}
			1 &0\\
			0& R
		\end{pmatrix}\b w^n+\b S^{-1}\bar{\b R}(\b T^{-1}(\b w^n)),\quad \b w^0=(0,w^0_2)^T.
	\end{equation*}
	Here, using $\b S^{-1}=(\widetilde s_{ij})_{i,j=1,2}$ and $w_1^n=0$ it follows from \eqref{eq:Gproof} that
	\begin{equation*}
		(\b S^{-1}\bar{\bm R}(\b T^{-1}(\b w^0)))_1=0
	\end{equation*}
	since $(\b G(\b w))_1=w_1$. Furthermore,
	\begin{equation}\label{eq:remainder}
		\begin{aligned}
			(\b S^{-1}\bar{\bm R}(\b T^{-1}(\b w^0)))_2=&\frac{1}{2} \sum_{i=1}^2 \widetilde s_{2i}(\b T^{-1}(\b w^0)-\b y^*)^T\b H g_i(\bxi^0_i)(\b T^{-1}(\b w^0)-\b y^*)\\
			=&\frac{1}{2}\sum_{i=1}^2 \widetilde s_{2i}(w_2^0\b S\b e_2)^T\b H g_i(\bxi^0_i)(w_2^0\b S\b e_2)\\
			=&\frac{1}{2}\sum_{i=1}^2 \widetilde s_{2i}(w_2^0\bby)^T\b H g_i(\bxi^0_i)(w_2^0\bby)\\
			=&C(\bxi^0_1,\bxi^0_2)\cdot (w_2^0)^2,
		\end{aligned}
	\end{equation}
	where $\bxi^0_i=\b y^*+c_i^0(\b y^0-\b y^*)$ and $c^0_i\in (0,1)$. Also note that the mapping \mbox{$C\colon \R^2\times \R^2\to\R$} depends on the entries of the Hessians as well as $\b S^{-1}$. 
	
	We now prove that the method defined by $\b G$ is overshooting $\b w^*=\b 0$ by proving the existence of $w_2^0\in \R$ such that $\sgn(w_2^1)\neq\sgn(w_2^0)$. 
	We set \[L=\left\{\b y\in \R^2\Big| \exists s\in\left[-\tfrac{y_1^*}{2},\tfrac{y_2^*}{2}\right]: \b y=\b y^*+s\bby\right\}\tm\R^2_{>0}\] and observe that there exists a $K>0$ such that $\sup_{\bxi\in L\times L}\{\lvert C(\bxi_1,\bxi_2)\rvert\}\leq  K<\infty$ since $\b g\in \mathcal C^2$ has bounded second derivatives on the compact set $L$. 
	
	Next, we restrict to $\b w^0$ satisfying $\lvert w_2^0\rvert< \min\left\{\tfrac{y_1^*}{2},\tfrac{y_2^*}{2},\frac{\lvert R\rvert}{K}\right\}$. As a result, $\b w^0=w_2^0\b e_2$ yields $\b y^0=\b T^{-1}(\b w^0)=\b S\b w^0+\b y^*=w_2^0\bby+\b y^*\in L$, which means that 
	\[\bxi^0_i=\b y^*+c^0_i(\b y^0-\b y^*)=\b y^*+c^0_iw_2^0\bby\in L\]
	for $i=1,2$.
	Now, according to \eqref{eq:remainder}, we have
	\begin{equation}
		w^{1}_2=Rw_2^0+C(\bxi^0_1,\bxi^0_2)\cdot(w_2^0)^2=(R+C(\bxi^0_1,\bxi^0_2) w_2^0)w_2^0\label{w2(n+1)}
	\end{equation}
	as well as
	\begin{equation}
		C(\bxi^0_1,\bxi^0_2) w_2^0\leq \lvert C(\bxi^0_1,\bxi^0_2)\rvert \lvert w_2^0\rvert <\lvert C(\bxi^0_1,\bxi^0_2)\rvert\frac{\lvert R\rvert}{K}\leq \lvert R\rvert.\label{cn w2n<R delta}
	\end{equation}
	Because of $R<0$, the inequality \eqref{cn w2n<R delta} turns into the statement
	\begin{equation*}
		R+C(\bxi^0_1,\bxi^0_2)w_2^n<0,
	\end{equation*}
	and thus, $\sgn(w_2^1)\neq\sgn(w_2^0)$ due to \eqref{w2(n+1)}. This proves that the method defined by $\b G$ is overshooting $\b w^*$ and consequently, the method with iterates given by the map $\b g$ is overshooting $\b y^*$.	
\end{proof}
\section{Lyapunov Stability Analysis}\label{sec:lyapunovAnalysis}
This section is devoted to the investigation of the numerical methods from Chapter~\ref{chap:NumSchemes} by means of the stability Theorem~\ref{Thm:_Asym_und_Instabil} and Theorem~\ref{Thm_MPRK_stabil}. We note that all schemes from Chapter~\ref{chap:NumSchemes} preserve positive steady states with the same arguments as in \cite{HIKMS22} or \cite[Proposition 2.3]{IssuesMPRK}. Thus, in order to apply Theorem \ref{Thm:_Asym_und_Instabil} and part \ref{it:Thma} of Theorem \ref{Thm_MPRK_stabil}, we need to prove a certain regularity and compute the eigenvalues of the Jacobian of $\b g$ evaluated at some steady state $\b y^*\in \ker(\bA)\cap D^\circ$ according to \eqref{eq:FormularJacobian}. However, to use also part \ref{it:Thmb} of Theorem \ref{Thm_MPRK_stabil}, we need to prove that $\b g$ additionally conserves all linear invariants.

In the case where the mapping $\b g$ satisfying $\b y^{n+1}=\b g(\b y^n)$ is implicitly given we compute $\b D\b g(\b y^*)$ as described in \cite{IKM2122, HIKMS22,IOE22StabMP} by introducing functions $\bm \Phi_i$ and several auxiliary Jacobians.
The functions $\bm \Phi_i$ arise from rearranging the equations for the $s$ stages and the updating step of the numerical method leading to
\begin{equation}\label{eq:Schemes:Phi_k}
	\begin{aligned}
		\b 0&=\bm \Phi_i(\b y^n,\b y^{(1)}(\b y^n),\dotsc, \b y^{(i)}(\b y^n)),\quad i=1,\dotsc,s\\
		\b 0&=\bm \Phi_{n+1}(\b y^n,\b y^{(1)}(\b y^n),\dotsc, \b y^{(s)}(\b y^n),\b g(\b y^n)).
	\end{aligned}
\end{equation}
Note that $\bm \Phi_i(\b x_0,\dotsc,\b x_i)$ is a function of $i+1$ vector-valued variables while $\bm \Phi_{n+1}(\b x_0,\dotsc,\b x_s,\b y)$ depends on $s+2$ variables. We will find that $\bm \Phi_k,\bm\Phi_{n+1}$ are in $\mathcal C^1$ for all schemes from Chapter~\ref{chap:NumSchemes}, so that we may define 
\begin{equation}\label{eq:jacobians1}
	\begin{aligned}
		\b D_{n}\bm\Phi_i&=\frac{\partial}{\partial \b x_0}\bm\Phi_i, &	\b D_{l}\bm\Phi_i&=\frac{\partial}{\partial \b x_l}\bm\Phi_i,
	\end{aligned}
\end{equation}
for $i,l=1,\dotsc,s$ with $l\leq i$, and
\begin{equation}\label{eq:jacobians2}
	\begin{aligned}
		\b D_{n}\bm\Phi_{n+1}&=\frac{\partial}{\partial \b x_0}\bm\Phi_{n+1}, &\b D_{l}\bm\Phi_{n+1}&=\frac{\partial}{\partial \b x_l}\bm\Phi_{n+1},& &\b D_{n+1}\bm\Phi_{n+1}=\frac{\partial}{\partial \b y}\bm\Phi_{n+1}
	\end{aligned}
\end{equation}
for $l=1,\dotsc, s$. Besides for GeCo and gBBKS, we will even be able to show that $\bm \Phi_i,\bm\Phi_{n+1}$ are in $\mathcal C^2$, and by means of the implicit function theorem, $\b g\in \mathcal C^1$ has locally Lipschitz first derivatives. For GeCo and gBBKS more effort is needed to justify the application of Theorem~\ref{Thm_MPRK_stabil}.

Moreover, we introduce operators $\b D_k^*$ indicating the evaluation of the corresponding auxiliary Jacobian at $\b y^*,\b y^{(1)}(\b y^*)$ et cetera, e.\,g.\ \[\b D^*_n\bm \Phi_2=\b D_n\bm \Phi_2(\b y^*,\b y^{(1)}(\b y^*),\b y^{(2)}(\b y^*)).\] As we interpret $\b y^{(i)}=\b y^{(i)}(\b y^n)$ we also introduce the Jacobian 
\begin{equation*}
	\b D^*\b y^{(i)}=\b D\b y^{(i)}(\b y^*).
\end{equation*}
With that we can derive a formula for computing $\b D\b g(\b y^*)$, where $\b y^{n+1}=\b g(\b y^n)$ is the unique solution to \eqref{eq:Schemes:Phi_k}. 
The chain rule yields
\begin{equation*}
	\begin{aligned}
		\b 0&=\b D^*_{n}\bm\Phi_i+\sum_{l=1}^{i}\b D_{l}^*\bm \Phi_i \b D^*\b y^{(l)},\quad i=1,\dotsc, s,\\
		\b 0&=\b D^*_{n}\bm\Phi_{n+1}+\sum_{l=1}^{s}\b D_{l}^*\bm \Phi_{n+1} \b D^*\b y^{(l)}+ \b D_{n+1}^*\bm\Phi_{n+1}\b D\b g(\b y^*),
	\end{aligned}
\end{equation*}
which can be rewritten to
\begin{equation}\label{eq:FormularJacobian}
	\begin{aligned}
		\b D^*\b y^{(i)}&=-\left(\b D^*_{i}\bm\Phi_i\right)^{-1}\left(\b D^*_{n}\bm\Phi_i+\sum_{l=1}^{i-1}\b D_{l}^*\bm \Phi_i \b D^*\b y^{(l)}\right),\quad i=1,\dotsc, s,\\
		\b D\b g(\b y^*)&=-\left(\b D^*_{n+1}\bm\Phi_{n+1}\right)^{-1}\left(\b D^*_{n}\bm\Phi_{n+1}+\sum_{l=1}^{s}\b D_{l}^*\bm \Phi_{n+1} \b D^*\b y^{(l)}\right),
	\end{aligned}
\end{equation}
if all occurring inverses exist. Also, in order to avoid long formulas in the following, we may omit to write the functions $\bm\Phi_i$ together with all their arguments.

Since we already discussed the case of Runge--Kutta methods in Section~\ref{sec:intro_dyn_sys} we start analyzing MPRK schemes. To that end, we will use the notation of MPRK as an NSARK method. In contrast, all other MP methods presented in Chapter~\ref{chap:NumSchemes} will be analyzed directly because of the following. First, the NSWs of gBBKS and GeCo are not in $\mathcal C^1$. Also, as discussed in Remark~\ref{rem:MPDeC_NSARK}, MPDeC methods can be understood as NSARK methods with potentially negative Butcher tableau entries. This is why we will focus in this work on the ansatz followed in \cite{IOE22StabMP}. Moreover, SSPMPRK methods do not fit into the form of an NSARK method.  Nonetheless, their analysis using ARK methods in Shu--Osher form will be part of my future research.
\subsection{Modified Patankar--Runge--Kutta}\label{sec:stab_MPRK}
It turns out to be convenient to derive the stability properties of MPRK methods using the notation of NSARK schemes. However, we thereby restrict to non-negative Butcher tableaux, \ie we use the vector notation \eqref{eq:NSARK_autonom}.  Moreover, we derive the Jacobian of the NSARK method in a more general context since Theorem~\ref{Thm_MPRK_stabil} is not restricted only to linear systems. In particular, let us consider  $\b y'=\b f(\b y)=\b F(\b y)\b y$, where $\b F(\b y)\in \R^{N\times N}$ consists of the columns $\ \b F^1(\b y),\dotsc, \b F^N(\b y)$. Hence, $\b F=\sum_{\nu=1}^N\b F^\nu\b e_\nu^T$, where $\b e_\nu$ is the $\nu$th column unit vector in $\R^N$. Moreover, this gives $\b f(\b y)=\sum_{\nu=1}^N\b F^\nu(\b y) y_\nu=\sum_{\nu=1}^N\fnu(\b y)$ with \[\fnu(\b y)=\b F^\nu(\b y) y_\nu.\] Note that in the case of the linear system \eqref{eq:PDS_Sys}, we have $\b F(\b y)= \bA$. We restrict to conservative problems, which means that we will assume that $\b 1^T\b f(\b y)=\b 0$ for all $\b y$ in the domain of $\b f$. With that we reproduce the results from the literature \cite{IKM2122,izgin2022stability,IOE22StabMP}. The generalization to non-conservative problems is then straightforward.

Since MPRK methods are linear implicit and based on explicit RK schemes, the stage equation for $\byi$ depends only $\b y^n,\dotsc,\byi$. even more, we have $\b y^n=\b y^{(1)}$, however, in order to keep the notation, we will not substitute this directly into the stage equations, $\bm\Phi_i$ or $\bm \Phi_{n+1}$. 

In \cite{AGKM_Oliver} it was assumed that the PWDs only depend on the $\nu$th component of the stages, \ie \begin{equation} \label{eq:pik_sigma}
	\pi_\nu^{(i)}=\pi_\nu^{(i)}(y_\nu^n,y_\nu^{(1)},\dotsc, y_\nu^{(i-1)}) \qta \sigma_\nu=\sigma_\nu(y_\nu^n, y_\nu^{(1)},\dotsc, y_\nu^{(s)}),
\end{equation}
which includes the PWDs presented in \cite{KM18,KM18Order3}. Thus we will assume this as well for our analysis. Moreover, the NSWs \[\gamma_\nu^{[i]}=\frac{\yi_\nu}{\pi_{\nu}^{(i)}}\qta\delta_\nu=\frac{y^{n+1}_\nu}{\sigma_\nu},\] see \eqref{eq:pertcoeff}, will be understood as functions of the stages in the following. Furthermore, we will assume that 
\begin{equation} \label{eq:pik_sigma(y*)}
	\pi_\nu^{(i)}(y_\nu^*,y_\nu^*,\dotsc, y_\nu^*)=y_\nu^* \qta \sigma_\nu(y_\nu^*,y_\nu^*,\dotsc, y_\nu^*)=y_\nu^*,
\end{equation}
for any steady state $\b y^*$ of the ODE, which is also fulfilled by the MPRK schemes presented so far. Therefore,
\begin{equation}\label{eq:gamma_delta*=1}
	\gamma_\nu^{[i]}(y_\nu^*,y_\nu^*,\dotsc, y_\nu^*,y_\nu^*)=1\qta \delta_\nu(y_\nu^*,y_\nu^*,\dotsc, y_\nu^*,y_\nu^*)=1.
\end{equation}
Altogether, the mappings $\bm \Phi_i,\bm\Phi_{n+1}$ of the NSARK method \eqref{eq:nsark} are 
\begin{equation}\label{eq:PhiNSARK}
	\begin{aligned}
		\bm \Phi_i&=\b y^n+\dt \sum_{j=1}^{i-1}\sum_{\nu=1}^Na_{ij}\gamma_\nu^{[i]}(y_\nu^n,y_\nu^{(1)},\dotsc, y_\nu^{(i)})\fnu(\b y^{(j)}) - \byi,\\
		\bm \Phi_{n+1}&=\b y^n+\dt \sum_{j=1}^s\sum_{\nu=1}^Nb_j\delta_\nu(y_\nu^n,y_\nu^{(1)},\dotsc, y_\nu^{(s)},y_\nu^{n+1})\fnu(\b y^{(j)}) - \b y^{n+1}.
	\end{aligned}
\end{equation}
Now, in the special case of $\b f(\b y)=\bA\b y$, we note that
\[\sum_{\nu=1}^Nx_\nu \fnu(\b y^{(j)})=\b F(\b y^{(j)})\cdot \vec{x_1y_1^{(j)}\\\vdots\\ x_Ny_N^{(j)}}=\bA \cdot \vec{x_1y_1^{(j)}\\\vdots\\ x_Ny_N^{(j)}} \]
for any values $x_1,\dotsc,x_N$. Substituting this information into \eqref{eq:PhiNSARK}, we observe that MPRK schemes preserve all linear invariants. Moreover, due to \eqref{eq:gamma_delta*=1} we see that the MPRK methods are steady state preserving as already mentioned.

Moreover, the maps $\bm \Phi_i$ and $\bm \Phi_{n+1}$ are in $\mathcal C^2$ for positive arguments, and as defined in \eqref{eq:PhiNSARK}, vanish for the argument $(\b y^n,\b y^{(1)}(\b y^n),\dotsc, \b y^{(i)}(\b y^n))$, and $(\b y^n,\b y^{(1)}(\b y^n),\dotsc,\b y^{(s)}(\b y^n),\b g(\b y^n))$, respectively. And since the computation of $\b y^{n+1}$ requires only the solution of linear systems which possess always a unique solution for any $\b y^n>\b 0$, the function $\b g$ is also a $C^2$-map. According to Remark~\ref{rem:C2->C1} and Theorem~\ref{Thm_MPRK_stabil} we thus only need to compute the eigenvalues of the Jacobian of $\b g$ to investigate the stability of MPRK schemes. The upcoming lemma is a first step towards this goal. \newpage
\begin{lem}\label{lem:DPhiNSARK}
	Assume $\b 1^T\b f=\b 0$, and that \eqref{eq:pik_sigma} and \eqref{eq:pik_sigma(y*)} hold with $\pi_\nu^{(i)},\sigma_\nu\in\mathcal C^1$, and let $\fnu\in \mathcal C^1$. Then $\b y^n=\b y^*$ implies $\byi=\b y^*$ and $\b y^{n+1}=\b y^*$ for any positive steady state $\b y^*$, and the maps $\bm\Phi_i$ and $\bm\Phi_{n+1}$ from \eqref{eq:PhiNSARK} satisfy
	{\allowdisplaybreaks
		\begin{align*}
			\b D^*_k\bm \Phi_i		&=\begin{cases}
				\b I-\dt c_i \b F(\b y^*)\b D_n^*\bm\pi^{(i)}		, & k=n,\\
				-\dt c_i \b F(\b y^*)\b D_k^*\bm\pi^{(i)}+\dt a_{ik}\b D\b f(\b y^*)		, & k=1,\dotsc, i-1,\\
				\dt c_i\b F(\b y^*)	-\b I	, &k=i,
			\end{cases}\\
			\b D^*_l\bm \Phi_{n+1}&=\begin{cases}
				\b I-\dt \b F(\b y^*)\b D_n^*\bm\sigma		, & l=n,\\
				-\dt \b F(\b y^*)\b D_l^*\bm\sigma+\dt b_l\b D\b f(\b y^*)		, & l=1,\dotsc, s,\\
				\dt\b F(\b y^*)	-\b I	, &l=n+1,
			\end{cases}
	\end{align*}}
	where $\bm \pi^{(i)}=(\pi_1^{(i)},\dotsc,\pi_N^{(i)})^T$ and  $\bm \sigma=(\sigma_1,\dotsc,\sigma_N)^T$.
\end{lem}
\begin{proof} Let $\delta_{m,l}$ denote the Kronecker delta. For $i=1,\dotsc,s$, straightforward calculations yield
	\begin{equation*}
		\begin{aligned}
			\b D^*_k\bm \Phi_i&=\begin{cases}
				\b I+\dt \sum_{j=1}^{i-1}\sum_{\nu=1}^Na_{ij}\fnu(\b y^*)\nabla_k^*\gamma_\nu^{[i]}, &k=n, \\
				\dt \sum_{j=1}^{i-1}\sum_{\nu=1}^Na_{ij}\fnu(\b y^*)\nabla_k^*\gamma_\nu^{[i]}+\dt a_{ik}\b D\b f(\b y^*), &k=1,\dotsc,i-1,\\
				\dt \sum_{j=1}^{i-1}\sum_{\nu=1}^Na_{ij}\fnu(\b y^*)\nabla_k^*\gamma_\nu^{[i]}-\b I, &k=i,
			\end{cases}\\
			\b D^*_l\bm \Phi_{n+1}&=\begin{cases}
				\b I+\dt \sum_{j=1}^{s}\sum_{\nu=1}^Nb_j\fnu(\b y^*)\nabla_l^*\delta_\nu, &l=n,\\
				\dt \sum_{j=1}^{s}\sum_{\nu=1}^Nb_j\fnu(\b y^*)\nabla_l^*\delta_\nu+\dt b_l\b D\b f(\b y^*), &l=1,\dots,s,\\
				\dt \sum_{j=1}^{s}\sum_{\nu=1}^Nb_j\fnu(\b y^*)\nabla_l^*\delta_\nu-\b I, &l=n+1,
			\end{cases}
		\end{aligned}
	\end{equation*}
	where
	\begin{equation*}
		\begin{aligned}
			\nabla_k^*\gamma_\nu^{[i]} &=\begin{cases}
				\frac{1}{y_\nu ^*}\b e_\nu^T, & k=i,\\
				-\frac{1}{y_\nu ^*}\nabla_k^*\pi_\nu^{(i)}, & k\neq i,
			\end{cases} \\
			\nabla_l^*\delta_\nu&=\begin{cases}
				\frac{1}{y_\nu ^*}\b e_\nu^T, & l=n+1,\\
				-\frac{1}{y_\nu ^*}\nabla_l^*\sigma_\nu	, & l\neq n+1.
			\end{cases}
		\end{aligned}
	\end{equation*}
	Using  \[\sum_{\nu=1}^N\fnu(\b y^*)\frac{1}{y_\nu ^*}\b e_\nu^T=\sum_{\nu=1}^N\b F^\nu(\b y^*)\b e_\nu^T=\b F(\b y^*),\] we end up with
	\begin{equation*}
		\begin{aligned}
			\b D^*_k\bm \Phi_i&=\begin{cases}\b I-\dt\sum_{j=1}^{i-1}\sum_{\nu=1}^Na_{ij}\fnu(\b y^*)\frac{1}{y_\nu ^*}\nabla_n^*\pi_\nu^{(i)}, & k=n,\\
				-\dt\sum_{j=1}^{i-1}\sum_{\nu=1}^Na_{ij}\fnu(\b y^*)\frac{1}{y_\nu ^*}\nabla_k^*\pi_\nu^{(i)}+\dt a_{ik}\b D\b f(\b y^*)	, & k=1,\dotsc, i-1,\\
				\dt\sum_{j=1}^{i-1}\sum_{\nu=1}^Na_{ij}\fnu(\b y^*)	\frac{1}{y_\nu ^*}\b e_\nu^T-\b I	, &k=i,		
			\end{cases}\\
			&=\begin{cases}
				\b I-\dt c_i \b F(\b y^*)\b D_n^*\bm\pi^{(i)}		, & k=n,\\
				-\dt c_i \b F(\b y^*)\b D_k^*\bm\pi^{(i)}+\dt a_{ik}\b D\b f(\b y^*)		, & k=1,\dotsc, i-1,\\
				\dt c_i\b F(\b y^*)	-\b I	, &k=i.
			\end{cases}
		\end{aligned}
	\end{equation*}
	Analogously, we obtain
	\begin{equation*}
		\begin{aligned}
			\b D^*_l\bm \Phi_{n+1}&=\begin{cases}
				\b I-\dt \b F(\b y^*)\b D_n^*\bm\sigma		, & l=n,\\
				-\dt \b F(\b y^*)\b D_l^*\bm\sigma+\dt b_l\b D\b f(\b y^*)		, & l=1,\dotsc, s,\\
				\dt\b F(\b y^*)	-\b I	, &l=n+1. \quad\quad \quad\qedhere
			\end{cases}
		\end{aligned}
	\end{equation*} 
\end{proof}
For linear conservative systems we thus obtain the following from \eqref{eq:FormularJacobian}.
\begin{thm}\label{thm:DgNSARK}
	In the situation of Lemma~\ref{lem:DPhiNSARK}, the Jacobian of the generating map $\b g$ of \eqref{eq:MPRK_PDRS} applied to a conservative problem  $\b y'=\bA\b y$ with $\sigma(\bA)\tm \overline{\C^-}$ reads
	\begin{equation}\label{eq:Dg(y*)NSARK}
		\begin{aligned}
			\b D^*\b y^{(i)}\!&=\left(\b I-\dt c_i\bA\right)^{-1}\!\left(\!\b I-\dt c_i \bA\b D_n^*\bm\pi^{(i)}- \dt\bA\sum_{l=1}^{i-1}\left(c_i\b D_l^*\bm \pi^{(i)}-a_{il}\b I\right) \b D^*\b y^{(l)}\!\right)\!\!,\\
			\b D\b g(\b y^*)\!&=\left(\b I-\dt\bA\right)^{-1}\left(	\b I-\dt \bA\b D_n^*\bm\sigma-\dt \bA\sum_{l=1}^{s}\left(\b D_l^*\bm\sigma-b_l\b I\right) \b D^*\b y^{(l)}\right),
		\end{aligned}
	\end{equation}
	where $i=1,\dotsc,s$.
\end{thm}
\begin{proof}
	The inverses exist since $\sigma(\dt c_i\bA-\b I)\tm \C^-$ for all $c_i\geq 0$. The rest follows from \eqref{eq:FormularJacobian} and Lemma~\ref{lem:DPhiNSARK}.
\end{proof}
\begin{rem}
	Our framework opens the door to a comprehensive approach of investigating even PDRS, since negative rest terms $\b r^d$ are weighted like destruction terms and positive rest terms $\b r^p$ are not modified. Hence, already at this point we may also consider PDRS with $\b r^p=\b 0$ and $\b r^d>\b 0$ and investigate the asymptotic stability of the origin using the same stability function as for $\b r=\b 0$. In the case of $\b r^p>\b 0$, one may revisit the proof of Lemma~\ref{lem:DPhiNSARK} adjusting the appearing Jacobians of the PWDs. The analysis of Patankar--Runge--Kutta methods would then also be available since production terms can formally be treated as positive rest terms. However, this together with the corresponding analyses of the stability functions and numerical experiments is beyond this work.
\end{rem}
In the following we replicate the results from \cite{IKM2122,izgin2022stability,IOE22StabMP} using this new framework. 
\paragraph{MPE}
The MPE method for conservative and autonomous PDS can be found in \eqref{eq:MPE}. Since the first stage equals $\b y^n$ and we also have $\bm\sigma=\b y^n$, we find from \eqref{eq:Dg(y*)NSARK} that
\begin{equation*}
	R(z)=\frac{1-z+z}{1-z}=\frac{1}{1-z}.
\end{equation*}
Hence, the MPE method has the same stability function as the implicit Euler scheme. As a consequence of that and Theorem~\ref{Thm_MPRK_stabil} we obtain the following results.
\begin{cor}\label{Cor:MPEstab}
	The MPE method is unconditionally stable in the sense of Definition~\ref{Def:uncondstab}.
\end{cor}
\begin{cor}\label{Cor:MPEstab1}
	Let $\b y^*$ be the unique steady state of the initial value problem \eqref{eq:PDS_Sys}, \eqref{eq:IC}  with $\b 1\in \ker(\bA^T)$. Then there exists a $\delta >0$ such that $\Vert\b y^0-\b y^*\Vert<\delta$ implies the convergence of the iterates of MPE towards $\b y^*$ as $n\to \infty$ for all $\dt >0$.
\end{cor}
\paragraph{MPRK22($\alpha$)}
The second order MPRK method for a conservative and autonomous PDS is given in \eqref{eq:MPRK22b}. Here, we have $\b y^{(1)}=\b y^n$, $\bm\pi^{(2)}=\b y^n$ and $\sigma_\nu=(y_\nu^{(2)})^{\frac{1}{\alpha}}(y_\nu^n)^{1-\frac{1}{\alpha}}$. Hence, due to $c_1=0$, we find $\b D^*\b y^{(1)}=\b I$ and
\[\b D_n^*\bm\pi^{(2)}=\b I,\quad \b D_1^*\bm\pi^{(2)}=\b 0,\quad \b D_n^*\bm\sigma=\left(1-\frac{1}{\alpha}\right)\b I,\quad \b D_1^*\bm\sigma=\b 0,\quad \b D_2^*\bm\sigma=\frac{1}{\alpha}\b I. \]
Since $\b D\b g(\b y^*)$ is a rational function of $\bA$ and the identity matrix $\b I$, any eigenvector of $\bA$ with the eigenvalue $\lambda$ is consequently an eigenvector of $\b D\b g(\b y^*)$.
Therefore, using \eqref{eq:Dg(y*)NSARK} we see $\sigma(\b D\b g(\b y^*))=\{R(\dt\lambda)\mid \lambda\in \sigma(\bA)\}$, where
\begin{equation}\label{eq:R_MPRK}
	\begin{aligned}
		R(z)&=\frac{1-z\left(1-\frac{1}{\alpha}\right)-z\left(0-b_1+(\frac1\alpha-b_2)\frac{1-c_2 z-z(0+a_{21})}{1-c_2 z}\right)}{1-z}\\
		&=\frac{1-z\left(1-\frac{1}{\alpha}\right)-z\left(-1+\frac{1}{2\alpha}+(\frac1\alpha-\frac{1}{2\alpha})\frac{1-2\alpha z}{1-\alpha z}\right)}{1-z}=\frac{-z^2-2\alpha z+2}{2(1-\alpha z)(1-z)}.
	\end{aligned}
\end{equation}
\begin{prop}\label{Prop:MPRK22}
	The stability function $R(z)=\frac{-z^2-2\alpha z+2}{2(1-\alpha z)(1-z)}$ from \eqref{eq:R_MPRK} with $\alpha> \frac{1}{2}$ satisfies $R(0)=1$ and $\lvert R(z)\rvert<1$ for all $z\in \Cminus\setminus\{0\}$. For $\alpha=\frac12$ we have $\lvert R(z)\rvert<1$ for all $z$ with $\operatorname{Re}(z)<0$ and $\lvert R(z)\rvert =1$, if $\operatorname{Re}(z)=0$. 
\end{prop}
\begin{proof}
	We first investigate $\lvert R(z)\rvert$ for $z=\ii y$ and $y\in \R$. 
	A small calculation reveals that the numerator of $\lvert R(z)\rvert^2$ can be written as
	\begin{equation}\label{eq:numerator_R(z)}
		\lvert -z^2-2\alpha z+2\rvert^2=\lvert y^2+2+(-2\alpha y)\ii\rvert^2=(y^2+2)^2+4\alpha^2 y^2=y^4+4y^2(1+\alpha^2)+4.
	\end{equation}
	Performing a similar calculation for the denominator of $\lvert R(z)\rvert^2$ we find
	\begin{equation*}
		\begin{aligned}
			\lvert 2(1-\alpha z)(1-z)\rvert^2&=\lvert 2\alpha z^2-2z(1+\alpha) +2\rvert^2=\lvert -2\alpha y^2+2 +(- 2y(1+\alpha))\ii\rvert^2\\
			&=(-2\alpha y^2+2)^2+4y^2(1+\alpha)^2						=4\alpha^2y^4+4y^2(1+\alpha^2)+4.
		\end{aligned}
	\end{equation*}
	Using \eqref{eq:numerator_R(z)} and $\alpha= \frac12$ we see that $\lvert R(z)\rvert =  1$ on the imaginary axis, and if $\alpha>\frac12$ we find $\lvert R(\ii y)\rvert< 1$ for all $y\neq 0$.
	
	Next we note that $R$ is a holomorphic function which is defined for all $z\in \overline{\C^-}$. Since $R$ is rational we can apply the Phragmén--Lindelöf principle \cite{SS03,T39} on the union of the origin and $\C^-$ and conclude that $\lvert R(z)\rvert\leq 1$ for all $z\in \overline{\C^-}$. Furthermore, since $R$ is not constant, we conclude from the maximum modulus principle that there exist no $z_0\in\C^-$ with $\abs{R(z_0)}=1$, or equivalently, $\abs{R(z_0)}<1$ holds for all $z_0$ with $\operatorname{Re}(z_0)<0$.
	
\end{proof}

As a direct consequence of the application of Theorem \ref{Thm_MPRK_stabil} in combination with Proposition~\ref{Prop:MPRK22} we obtain the following two corollaries, where we note that all nonzero eigenvalues of $\bA$ from \eqref{eq:PDS_Sys} have a negative real part, see Remark~\ref{rem:Aneg}.

\begin{cor}\label{Cor:MPRKstab}
	The MPRK22($\alpha$) scheme is unconditionally stable for all $\alpha \geq \frac12$.
\end{cor}
\begin{cor}\label{Cor:MPRKstab1}
	Let  $\b y^*$ be the unique steady state of the initial value problem \eqref{eq:PDS_Sys}, \eqref{eq:IC} with $\b 1\in \ker(\bA^T)$. Then there exists a $\delta >0$ such that $\Vert\b y^0-\b y^*\Vert<\delta$ implies the convergence of the iterates of MPRK22($\alpha$) towards $\b y^*$ as $n\to \infty$ for all $\dt >0$ and $\alpha\geq\frac12$.
\end{cor}

\begin{rem} \label{rem:Phragmen}
	We note that as long as the stability function $R(z)=\frac{N(z)}{D(z)}$ with polynomials $N,D$ satisfying $\deg(N)\leq\deg(D)$ and $D(z)\neq 0$ for all $z\in \Cminus$, we can conclude $\lvert R(z)\rvert<1$ for $\re(z)<0$ whenever $\lvert R(z)\rvert\leq1$ holds on the imaginary axis with the same reasoning as in the proof of Proposition \ref{Prop:MPRK22}.
	
	Moreover, we point out that the Phragmén--Lindelöf principle can also be applied to different sectors $S_{(\varphi_1,\varphi_2)}=\left\{z\in \Cminus\mid \arg(z)\in(\varphi_1,\varphi_2) \right\}$ of $\overline{\C^-}$. 
\end{rem}
\paragraph{MPRK43}
We consider the two families of third order MPRK schemes presented in Chapter~\ref{chap:NumSchemes}. The PWDs can be found in \eqref{eq:PWDsMPRK43}, where $\bm \sigma$ is given only implicitly. However, interpreting $\bm \sigma=\b y^{(4)}$ and introducing $a_{41}=\beta_1$ and $a_{42}=\beta_2$ as well as $\pi_\nu^{(4)}=(y_\nu^{(2)})^{\frac{1}{a_{21}}}(y_\nu^n)^{1-\frac{1}{a_{21}}}$, we can use the derived formula from Theorem~\ref{thm:DgNSARK}. To that end, we note that the nonzero auxiliary Jacobians are
\begin{equation*}
	\begin{aligned}
		\b D_n^*\bm\pi^{(2)}&=\b I,\quad \b D_n^*\bm\pi^{(3)}=\left(1-\tfrac1p\right)\b I,\quad \b D_2^*\bm\pi^{(3)}=\tfrac1p\b I,\\
		\b D_n^*\bm\pi^{(4)}&=\left(1-\tfrac{1}{a_{21}}\right)\b I,\quad \b D_2^*\bm\pi^{(4)}=\tfrac{1}{a_{21}}\b I,\quad \b D_4^*\bm\sigma=\b I.
	\end{aligned}
\end{equation*}
As a result of \eqref{eq:Dg(y*)NSARK}, the stability function is
\begin{equation}\label{eq:R_MPRK43}
	\begin{aligned}
		R(z)=&\frac{1-z\left(-b_1-b_2  \frac{1-zc_2+za_{21}}{1-zc_2}  -b_3  \frac{1-zc_3\left(1-\frac1p\right)-z\left(-a_{31} +\left(c_3\frac1p-a_{32}\right)\frac{1-zc_2+za_{21}}{1-zc_2} \right)}{1-zc_3} \right) }{1-z}\\
		&+\frac{-z\frac{1-zc_4\left(1-\frac{1}{a_{21}}\right)-z\left(-a_{41}+\left(\frac{c_4}{a_{21}}-a_{42}\right)\frac{1-zc_2+za_{21}}{1-zc_2}\right)}{1-zc_4}}{1-z}\\
		=&\frac{1+z\left(b_1+ \frac{b_2 }{1-zc_2}  +b_3  \frac{1+zc_3\left(\frac1p-1\right)+z\left(a_{31} +\frac{a_{32}-\frac{c_3}{p}}{1-zc_2} \right)}{1-zc_3}  \right) }{1-z}\\
		&	-\frac{z\frac{1+zc_4\left(\frac{1}{a_{21}}-1\right)+z\left(a_{41}+\frac{a_{42}-\frac{c_4}{a_{21}}}{1-zc_2}\right)}{1-zc_4}}{1-z}.
	\end{aligned}
\end{equation}
\paragraph{MPRK43($\alpha,\beta$)}
In \cite{IOE22StabMP} the stability function of MPRK43$(\alpha,\beta$) was computed using a different approach. Unfortunately, there is a typo in the stability function on page 2328: Instead of writing "$-\frac{\beta}{p}$" as suggested in equation (33) on the same page, it is written "$-\frac{\beta}{q}$". This typo undermines all claims that are based on it. The stability function actually is
\begin{equation}\label{eq:Stab_fun_MPRK43(alpha,beta)}
	\begin{aligned}
		R(z)=&\frac{1+b_1 z+\frac{b_2 z}{1-\alpha z} +\frac{b_3 z\left(1+z\left( \left(\frac1p-1\right)\beta +  a_{31}\right)+\frac{z\left(-\bm{\frac{\beta}{p}} +a_{32}\right)}{1-\alpha z}\right)}{1-\beta z}}{1-z}\\
		&-	\frac{\frac{z\left(1+z\left(\frac1q -1+\beta_1\right)+\frac{z(-\frac1q+\beta_2)}{1-\alpha z} \right)}{1-z}}{1-z}\\
		=&\frac{((\frac12-\beta) \alpha-\frac16) z^{4}+((\frac12-\beta) \alpha+\frac12 \beta+\frac16) z^{3}+(( \beta+1) \alpha+ \beta-\frac12) z^{2}}{ (z-1)^{2} (z \beta-1) (\alpha z-1)}\\
		&	+\frac{-(1+ \alpha+\beta) z+1}{(z-1)^{2} (z \beta-1) (\alpha z-1)},
	\end{aligned}
\end{equation}
which can also be obtained within our framework by substituting \eqref{array:MPRK43alphabeta} and \eqref{eq:parametersMPRK43alphabeta} into \eqref{eq:R_MPRK43}.  It is thus the purpose of this subsection to correct and to extend the results from \cite{IOE22StabMP} concerning this stability analysis.

Since different cases for different $(\alpha,\beta)$ pairs need to be distinguished, see  \eqref{eq:cond:alpha,beta}, the analysis of the stability function \eqref{eq:Stab_fun_MPRK43(alpha,beta)} is more involved. Moreover, we will find out that the method is not unconditionally stable for all feasible parameters. In order to give an insight in the stability properties, we investigate the stability function numerically.
We first rewrite \eqref{eq:Stab_fun_MPRK43(alpha,beta)} to
\begin{equation*}
	R(z)=\frac{\sum_{j=0}^4n_jz^j}{\sum_{j=0}^4d_jz^j},
\end{equation*}
where
\begin{equation*}
	\begin{aligned}
		n_0&=1, &&&d_0&=1,\\
		n_1&=-(1+\alpha+\beta),&&&d_1&=-(\alpha+\beta+2),\\
		n_2&=( \beta+1) \alpha+ \beta-\tfrac12, &&& 	d_2&=(\beta+2)\alpha+2\beta+1,
	\end{aligned}
\end{equation*}
\begin{equation*}
	\begin{aligned}
		n_3&=(\tfrac12-\beta) \alpha+\tfrac12 \beta+\tfrac16,&&&d_3&=-(2\beta+1)\alpha-\beta,\\
		n_4&=(\tfrac12-\beta) \alpha-\tfrac16,&&&d_4&=\alpha\beta.
	\end{aligned}
\end{equation*}
In what follows, we investigate the polynomial $p_{\alpha,\beta,\varphi}(r)$ from Lemma~\ref{lem:stabconditionR(rexp(phi))} satisfying
\[\lvert R(re^{\ii\varphi})\rvert <1 \Longleftrightarrow p_{\alpha,\beta,\varphi}(r)<0\qta  \lvert R(re^{\ii\varphi})\rvert >1 \Longleftrightarrow p_{\alpha,\beta,\varphi}(r)>0, \]
which means that MPRK43($\alpha,\beta$) is unconditionally stable, if $ p_{\alpha,\beta,\frac\pi2}(r)<0$ for all $r>0$. Moreover, if $ p_{\alpha,\beta,\frac\pi2}(r)>0$ for some $r>0$, then the method cannot be unconditionally stable. For instance, observing $ p_{\frac13,\frac23,\frac\pi2}(r)=-\frac{11}{81} r^{6}-\frac{11}{36} r^{4}<0$ for all $r>0$, we find that MPRK$(\frac13,\frac23)$ is unconditionally stable.

We want to note that for any other feasible pair $(\alpha,\beta)$, see Figure~\ref{fig:RK3_pos}, the degree of $ p_{\alpha,\beta,\varphi}$ is $8$. To see this, we point out that the leading coefficient
\begin{equation}\label{eq:lc(p)}
	n_4^2-d_4^2=-\alpha^{2} \beta^{2}+(-\alpha \beta+\tfrac{1}{2} \alpha-\tfrac{1}{6})^{2}=-(\alpha-\tfrac{1}{3}) \alpha \beta+(\tfrac{\alpha}{2}-\tfrac{1}{6})^{2}
\end{equation}
vanishes for $(\alpha,\beta)=(\tfrac13,\tfrac23)$ and $(\alpha,\tfrac{3 \alpha-1}{12 \alpha})$. However, since $\tfrac{3 \alpha-1}{12 \alpha}< \tfrac{3 \alpha}{12 \alpha}=\tfrac14$ and all feasible values of $\beta$ lie in $[0.25,0.75]$, the pair $(\alpha,\tfrac{3 \alpha-1}{12 \alpha})$ does not lie in the feasible domain.

Altogether, the question of unconditional stability can be answered if the polynomial $ p_{\alpha,\beta,\frac\pi2}$ has no positive root because of the following. 
Suppose that all roots are non-positive. Since the coefficient of $\beta$ in the leading coefficient \eqref{eq:lc(p)} of $ p_{\alpha,\beta,\varphi}(r)$ is negative for $\alpha>\frac13$, we find due to $\beta\in[0.25,0.75]$, that
\[n_4^2-d_4^2\leq -(\alpha-\tfrac{1}{3}) \alpha \tfrac14+(\tfrac{\alpha}{2}-\tfrac{1}{6})^{2}=\tfrac{1}{36}-\tfrac{\alpha}{12}<0.\]
This means that $ \lim_{r\to\infty}p_{\alpha,\beta,\varphi}(r)=-\infty$. Finally, as we assumed that there are no positive roots, this implies that the polynomial is negative for all $r>0$.

For the investigation of the remaining parameter combinations, we create a grid for $(\alpha,\beta)\in [\frac13,2]\times[0.25,0.75]$ with a resolution of $102^2$ points in a unit square $[0,1]^2$. We chose this resolution to sample the domain for $\beta$ of length $0.5$ with about 100 points. In particular, we use 102 points so that the set of sampled pairs includes the combinations $(0.5,0.75)$ and $(1,0.5)$ which are used in the literature \cite{KM18Order3}.
Given a pair $(\alpha,\beta)$ from the grid, we also sample $\varphi_j=\frac{j}{1000}\pi,$ $j=0,1,\dotsc,1000$ and determine the smallest value of $j$ such that $ p_{\alpha,\beta,\varphi_j}$ has no positive root by using Sturm's Theorem \cite[Theorem 8.8.14]{C03}. The corresponding value
\begin{equation}
	\theta_\text{num}=\min_{j=0,1,\dotsc,1000}\{2(\pi-\varphi_j)\mid p_{\alpha,\beta,\varphi_j}(r)=0\Longrightarrow r\leq 0\}\label{eq:theta}
\end{equation} represents a lower bound for the opening angle $\theta$ of the stability domain of MPRK43($\alpha,\beta$).  Moreover, if there exist a  simple positive root, or with odd multiplicity, we know that the polynomial $p_{\alpha,\beta,\frac\pi2}$ will become positive within a neighborhood of that root, and hence, the method cannot be unconditional stable. In this case we have the  the error estimate $\theta_\text{num}\leq\theta<\theta_\text{num}+\frac{\pi}{500}$ since
\[2(\pi-\varphi_j)-2(\pi-\varphi_{j+1})=\frac{\pi}{500}. \] 
Note that in the case of $\theta_\text{num}\geq\pi$  the related method is unconditionally stable. The plot of $\theta_\text{num}$ can be found in Figure~\ref{fig:maxstab}, noting that the MPRK43$(\alpha,\beta)$ scheme is not defined for $\alpha=\beta$.
\begin{figure}
	\centering
	\includegraphics[width=0.95\textwidth]{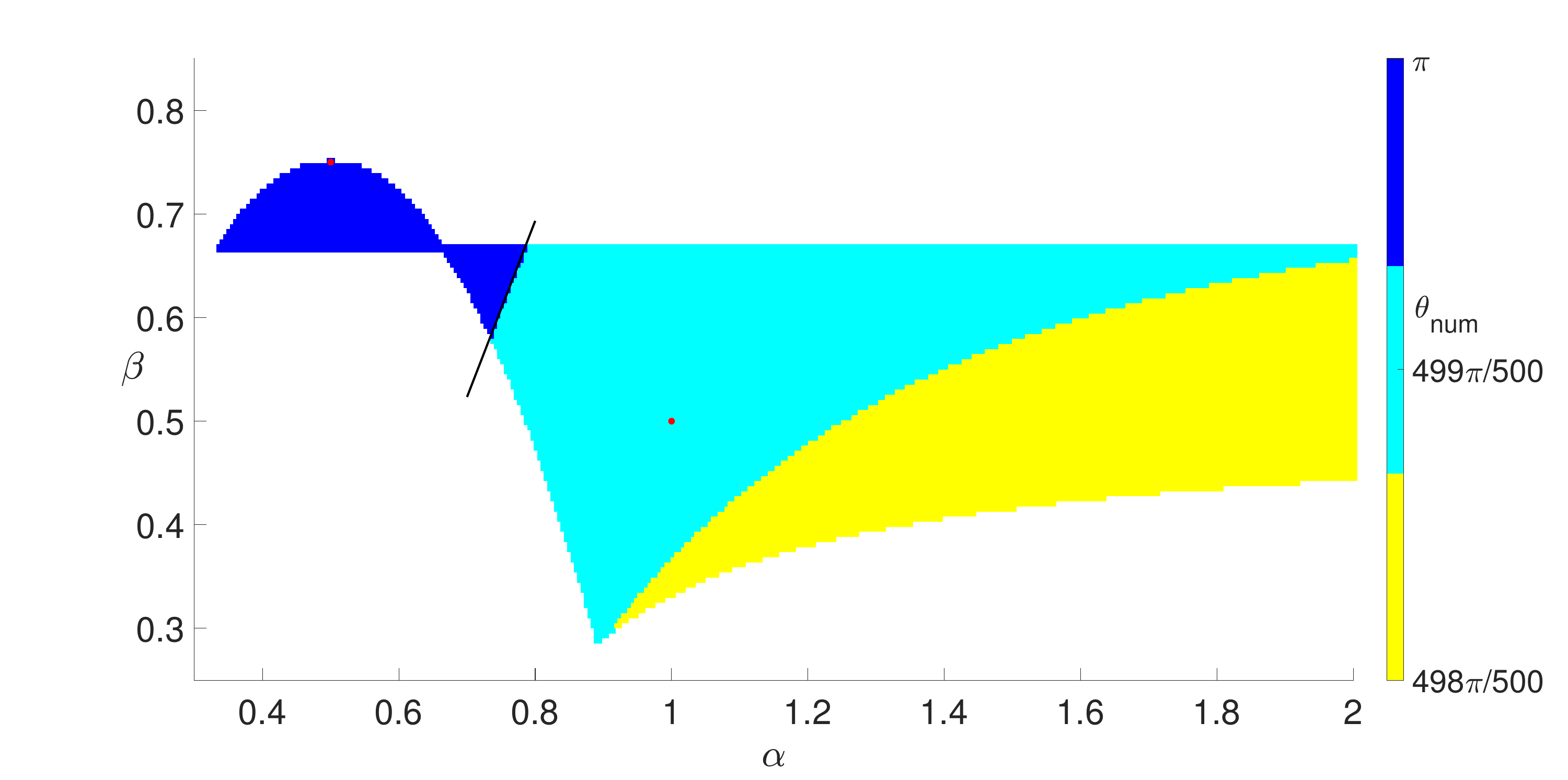}
	\caption{Plot of $\theta_\text{num}$, see \eqref{eq:theta}, estimating the opening angle of the stability domain of MPRK43($\alpha,\beta$). Pairs $(\alpha,\beta)$ colored in dark blue belong to $\theta_\text{num}=\pi$ and yield an unconditionally stable MPRK43($\alpha,\beta$) method. The black line is given by $\beta=\frac{17}{10}\alpha-\frac23$ and its intersection with the feasible domain lies in the dark blue segment. The red markers indicate the position of the pairs $(0.5,0.75)$ with an opening angle of at least $\theta_\text{num}=\pi$, and $(1,0.5)$ with an opening angle of at least $\theta_\text{num}=\pi-\frac{\pi}{500}=\frac{499}{500}\pi$.}\label{fig:maxstab}
\end{figure}

As one can see, several parameter combinations $(\alpha,\beta)$ are unconditionally stable and the smallest $\theta_\text{num}$ observed is $\frac{498}{500}\pi$. In particular the pair corresponding to $\alpha=\frac12$ and $\beta=\frac34$ already used in the literature \cite{KM18Order3} is now proved to be unconditionally stable. This cannot be said with certainty about the combination $\alpha=1$ and $\beta=\frac12$ as the corresponding value of $\theta_\text{num}$ is $\frac{499}{500}\pi$. Indeed, computing the roots of $p_{1,0.5,\frac\pi2}$ we find a simple positive root, proving that the corresponding method is not unconditionally stable.
\paragraph{MPRK43($\gamma$)}
Substituting \eqref{array:MPRK43gamma} and \eqref{eq:parametersMPRK43gamma} into \eqref{eq:R_MPRK43}, we obtain the stability function
\begin{equation}\label{eq:Stab_fun_MPRK43(gamma)}
	R(z)=\frac{-5z^4 + 7z^3 + 23z^2 - 42z + 18}{2(2z-3)^2(z-1)^2}.
\end{equation}
Note that the stability function is independent of the parameter $\gamma$, so that the following investigation is valid for all $\frac38\leq \gamma\leq \frac34.$ \newpage
\begin{prop}
	Let $R$ be defined by \eqref{eq:Stab_fun_MPRK43(gamma)}. Then $\lvert R(z)\rvert<1$ holds true for all \mbox{$z\in \Cminus\setminus\{0\}$} and $R(0)=1$.
\end{prop}
\begin{proof}
	A straightforward calculation yields
	\begin{equation*}
		R(z)=\frac{-\frac{5}{18}z^4 + \frac{7}{18}z^3 + \frac{23}{18}z^2 - \frac{42}{18}z + 1}{(\frac23z-1)^2(z-1)^2}=\frac{\sum_{j=0}^4n_jz^j}{\sum_{j=0}^4d_jz^j},
	\end{equation*}
	where
	\begin{equation*}
		\begin{aligned}
			n_0&=1,&
			n_1&=- \frac{7}{3},&
			n_2&=\frac{23}{18},&
			n_3&=\frac{7}{18},&
			n_4&=-\frac{5}{18},\\
			d_0&=1,&
			d_1&=-\frac{10}{3},&
			d_2&=\frac{37}{9},&
			d_3&=-\frac{20 }{9},&
			d_4&=\frac49.
		\end{aligned}
	\end{equation*}
	Hence $R(0)=1$ and the polynomial $p_\varphi(r)$ from Lemma~\ref{lem:stabconditionR(rexp(phi))} satisfies
	\begin{equation*}
		p_{\frac\pi2}(r)=-\frac{13}{108}r^8-\frac{137}{324}r^6-\frac{1}{12}r^4<0
	\end{equation*} for all $r>0$. Hence, it follows that $\lvert R(\ii y)\rvert <1$ for all $y\neq 0$. The claim then follows from Remark~\ref{rem:Phragmen}.
\end{proof}

From this result, we can conclude as a direct consequence of Theorem \ref{Thm_MPRK_stabil} the following statements.
\begin{cor}\label{Cor:MPRK43gamma}
	\begin{enumerate}
		\item The MPRK($\gamma$) method is unconditionally stable for all $\frac{3}{8} \leq \gamma \leq \frac{3}{4}$.
		\item If $\b y^*$ is  the unique steady state of the initial value problem \eqref{eq:PDS_Sys}, \eqref{eq:IC} with $\b 1\in \ker(\bA^T)$, then there exists a $\delta >0$ such that $\Vert\b y^0-\b y^*\Vert<\delta$ implies the convergence of the iterates of MPRK($\gamma$) towards $\b y^*$ for  all $\dt>0$  and $\frac{3}{8} \leq \gamma \leq \frac{3}{4}$.
	\end{enumerate}
\end{cor}
\subsection{Strong-Stability Preserving Modified Patankar--Runge--Kutta}\label{sec:stab_SSPMPRK}
As SSPMPRK schemes from \cite{SSPMPRK2,SSPMPRK3} are only constructed for positive and conservative PDS, we assume that the linear test equation \eqref{eq:PDS_Sys} is conservative, \ie $\b 1\in\ker(\bA^T)$. Since $\bA$ is a Metzler matrix, the test equation can be rewritten as a positive and conservative PDS with $p_{ij}(\b y) = d_{ji}(\b y) = \lambda_{ij}y_j$ for $i\neq j$ and $p_{ii}=d_{ii}=0$. 
Moreover, from $\b 1\in\ker(\bA^T)$, one can easily derive $\sum_{j=1}^N \lambda_{ji} = 0$ and thus obtain
\begin{equation}\label{eq:-sum(dij)}
	-\sum_{\substack{j=1}}^Nd_{ij}(\b y) = -\sum_{\substack{j=1\\j\neq i}}^N\lambda_{ji}y_i = \lambda_{ii}y_i,
\end{equation}
which will be used in the following to write the SSPMPRK schemes in the matrix-vector notation. \newpage
\paragraph{SSPMPRK2($\alpha,\beta$)}
When applied to a conservative system \eqref{eq:PDS_Sys}, the terms $p_{ij}$ and $d_{ij}$ fulfill \eqref{eq:-sum(dij)}. 
As a consequence, the scheme \eqref{eq:SSPMPRK2} can be rewritten as
\begin{equation}\label{eq:SSPMPRK2Matrixvector}
	\begin{aligned}
		\b 0=&\bm \Phi_1(\b y^n,\b y^{(1)}) = \b y^n+\beta \dt\bA\b y^{(1)}-\b y^{(1)},\\
		\b 0=&\bm \Phi_{n+1}(\b y^n,\b y^{(1)},\b y^{n+1}) = (1-\alpha)\b y^n+\alpha\b y^{(1)}\\&+\dt\bA\diag(\b y^{n+1})(\diag(\b y^{(1)}))^{-s}(\diag(\b y^{n}))^{s-1}(\beta_{20}\b y^n+\beta_{21}\b y^{(1)})-\b y^{n+1},\\
	\end{aligned}
\end{equation}
where we use the notation $(\diag(\b y))_{ij}=\delta_{ij}y_i$ with the Kronecker delta $\delta_{ij}$ as well as $((\diag(\b y))^{x})_{ij}=\delta_{ij}y_i^{x}$ for $x\in\R$.  Furthermore, $\b y^{(1)}=\b y^{(1)}(\b y^n)$ and $\b y^{n+1}=\b g(\b y^n)$ defined by \eqref{eq:SSPMPRK2Matrixvector} are functions of $\b y^n$. In order to apply Theorem~\ref{Thm:_Asym_und_Instabil} and Theorem~\ref{Thm_MPRK_stabil}, we have to investigate the map $\b g$ with respect to its smoothness as well as steady state and linear invariants preservation. 

First of all, we show that $\b g\in \mathcal C^2$ and then use Remark \ref{rem:C2->C1} in order to see that the first derivatives are Lipschitz continuous on an appropriately chosen neighborhood $\mathcal D$ of $\b y^*$. 

Indeed, the maps $\bm \Phi_1\colon\R^N_{>0}\times \R^N_{>0}\to \R^N$ and $\bm \Phi_{n+1}\colon\R^N_{>0}\times \R^N_{>0}\times \R^N_{>0}\to \R^N$ are in $\mathcal C^2$ for the same reasons as for MPRK schemes, which means that $\b g$ is also a $\mathcal C^2$-map.

Next, we show that any positive steady state of \eqref{eq:PDS_Sys} is a fixed point of $\b g$. To see this, we want to mention that $\b y^n=\b y^{(1)}=\b y^{n+1}=\b y^*$ is a solution to the system of equations \eqref{eq:SSPMPRK2Matrixvector} due to $\bA\b y^*=\b 0$. Since the solution for given $\b y^n$ is unique, we conclude that $\b y^n=\b y^*$ implies $\b y^{(1)}=\b y^{n+1}=\b y^*$, i.\,e.\ $\b g(\b y^*)=\b y^*$.

Moreover, $\b g$ conserves all linear invariants since $\b n^T\bA=\b 0$ and \eqref{eq:SSPMPRK2Matrixvector} imply
\begin{equation*}
	\begin{aligned}
		\b n^T\b g(\b y^n)&=\b n^T\b y^{n+1}=(1-\alpha)\b n^T\b y^n+\alpha\b n^T\b y^{(1)}+\b 0\\
		&=(1-\alpha)\b n^T\b y^n+\alpha\b n^T(\b y^{n}+\beta \dt\bA \b y^{(1)})=\b n^T\b y^n.
	\end{aligned}
\end{equation*}
Therefore, the map $\b g\colon \R^N_{>0}\to \R^N_{>0}$ meets the assumptions of Theorem \ref{Thm_MPRK_stabil}, so that we now focus on computing the Jacobian $\b D\b g(\b y^*)$ according to \eqref{eq:FormularJacobian}. In particular, we have
\begin{equation}\label{eq:Formula_Dg(y*)}
	\b D\b g(\b y^*)=-(\b D^*_{n+1}\bm \Phi_{n+1})^{-1}\left(	\b D^*_n\bm \Phi_{n+1}+ \b D^*_1\bm \Phi_{n+1}\b D^*\b y^{(1)}\right),
\end{equation}
where
\begin{equation}\label{eq:Formula_D*y^1}
	\b D^*\b y^{(1)}=-\left(\b D_1^*\bm\Phi_1\right)^{-1}\b D_n^*\bm\Phi_1,
\end{equation}
if $\b D_1^*\bm \Phi_1$ is invertible. Hence, we have to compute several auxiliary Jacobians in order to calculate $\b D\b g(\b y^*)$ and we start with
\begin{equation*}
	\begin{aligned}
		\b D^*_n\bm \Phi_1=\b I \qta \b D^*_1\bm \Phi_1=\beta \dt\bA-\b I.
	\end{aligned}
\end{equation*}
Note that $\beta>0$ and $\sigma(\bA)\tm \Cminus$, which implies that $\b D_1^*\bm \Phi_1$ is nonsingular. Thus, we can use \eqref{eq:Formula_D*y^1} and find
\begin{align*}
	\b D^*\b y^{(1)}=-(\beta \dt\bA-\b I)^{-1}\cdot \b I=(\b I-\beta \dt\bA)^{-1}.
\end{align*}
Next, we compute $\b D^*_n\bm\Phi_{n+1}$ and $\b D^*_1\bm\Phi_{n+1}$. To that end, we first define \[\b f(\b y^n,\b y^{(1)})=\diag(\b y^n)^k(\beta_{20}\b y^n+\beta_{21}\b y^{(1)})\] for some $k\in \R$ and get
\begin{equation}\label{eq:ProductRuleDiag}
	\begin{aligned}
		\left(\b D_n^*\b f\right)_{ij}&= \partial_{y_j^n}\left((y_i^n)^k(\beta_{20}y_i^n+\beta_{21}y_i^{(1)}))\right)\Big|_{\b y^n=\b y^*}\\&=\delta_{ij}\left(k(y_i^*)^{k-1}(\beta_{20}+\beta_{21})y_i^*+(y_i^*)^k\beta_{20}\right)\\&=\left(\diag(\b y^*)^{k}\right)_{ij}(k(\beta_{20}+\beta_{21})+\beta_{20}),
	\end{aligned}
\end{equation}
where we have used the fact that $\b y^{(1)}(\b y^*)=\b y^*$. Similarly, defining \[\b u(\b y^n,\b y^{(1)})=\diag(\b y^{(1)})^k(\beta_{20}\b y^n+\beta_{21}\b y^{(1)}),\] we obtain
\begin{equation}
	\b D_1^*\b u=\diag(\b y^*)^{k}(k(\beta_{20}+\beta_{21})+\beta_{21}).\label{eq:D_1u}
\end{equation}
In order to apply the formulae \eqref{eq:ProductRuleDiag} and \eqref{eq:D_1u} to compute  $\b D^*_n\bm\Phi_{n+1}$ and $\b D^*_1\bm\Phi_{n+1}$, we also make use of the fact that diagonal matrices commute, so that we end up with
\begin{align*}
	\b D^*_n\bm \Phi_{n+1}&=(1-\alpha)\b I+\dt\bA((s-1)(1-\alpha\beta)+\beta_{20}),\\
	\b D^*_1\bm \Phi_{n+1}&=\alpha \b I+\dt\bA(-s(1-\alpha \beta)+\beta_{21}),
\end{align*}
where we have exploited $\beta_{20}+\beta_{21}=1-\alpha\beta$.
Finally, to compute $\b D_{n+1}^*\bm\Phi_{n+1}$ we rewrite \eqref{eq:SSPMPRK2Matrixvector} utilizing $\diag(\b v)\b w=\diag(\b w)\b v$ to get
\begin{equation}
	\begin{aligned}\label{eq:trick}
		\bm \Phi_{n+1}=&(1-\alpha)\b y^n+\alpha\b y^{(1)}\\&+\dt\bA\diag(\beta_{20}\b y^n+\beta_{21}\b y^{(1)})(\diag(\b y^{(1)}))^{-s}(\diag(\b y^{n}))^{s-1}\b y^{n+1}-\b y^{n+1}.
	\end{aligned}
\end{equation}
From this, it is easy to see that
\begin{equation*}
	\begin{aligned}
		\b D^*_{n+1}\bm \Phi_{n+1}=(1-\alpha\beta) \dt\bA-\b I
	\end{aligned}
\end{equation*}
which is a nonsingular matrix since $\sigma(\bA)\tm \overline{\C^-}$ and $1-\alpha\beta\geq \frac{1}{2\beta}>0$, see \eqref{eq:alphabeta_conditions}.
Finally, we introduce the expressions for the auxiliary Jacobians into the formula \eqref{eq:Formula_Dg(y*)} resulting in
\begin{equation*}
	\begin{aligned}
		\b D\b g(\b y^*)=(\b I-(1-\alpha\beta)\dt\bA&)^{-1}\Bigl((1-\alpha)\b I+\dt\bA((s-1)(1-\alpha\beta)+\beta_{20})\\&+\left(\alpha \b I+\dt\bA(-s(1-\alpha \beta)+\beta_{21})\right)(\b I-\beta \dt\bA)^{-1}\Bigr).
	\end{aligned}
\end{equation*}
Since $\b D\b g(\b y^*)$ is a rational function of $\bA$ and the identity matrix $\b I$, we find $\sigma(\b D\b g(\b y^*))=\{R(\dt\lambda)\mid \lambda\in \sigma(\bA)\}$, where
\begin{equation*}
	\begin{aligned}
		R(z)=\frac{1 - \alpha + z((s - 1)(1-\alpha\beta ) + \beta_{20}) + \frac{\alpha + z(-s(1-\alpha\beta) + \beta_{21})}{1-\beta z}}{1 - (1-\alpha\beta)z}.
	\end{aligned}
\end{equation*}
From 
\begin{equation*}
	\beta_{20} = 1 - \frac{1}{2\beta} - \alpha\beta,\quad  \beta_{21} = \frac{1}{2\beta}\qta s = \frac{\alpha\beta^2 - \alpha\beta + 1}{\beta(1-\alpha\beta)}
\end{equation*}
elementary computations lead to
\begin{equation*}
	R(z)=\frac{-2 + (2\alpha\beta^2 - 2\alpha\beta + 1)z^2 - 2\beta(\alpha - 1)z}{2(1 + (\alpha\beta - 1)z)(\beta z - 1)}.
\end{equation*}
In summary, we obtain the following proposition.
\begin{prop}\label{Prop:SSPMPRK2_Dg(y*)}
	Let $\b g\from\R^N_{>0}\to\R^N_{>0}$ be the map given by the application of the second order SSPMPRK family \eqref{eq:SSPMPRK2} to the differential equation \eqref{eq:PDS_Sys} with $\bm 1\in \ker(\bA^T)$.
	Then any $\b y^*\in \ker(\bA)\cap \R^N_{>0}$ is a fixed point of $\b g$ and $\b g\in \mathcal{C}^2(\R^N_{>0},\R^N_{>0})$, where the first derivatives of $\b g$ are Lipschitz continuous in an appropriate neighborhood of $\b y^*$. Moreover, all linear invariants are conserved and an eigenvalue $\lambda$ of $\bA$ corresponds to the eigenvalue $R(\dt\lambda)$ of the Jacobian of $\b g$ where
	\begin{equation}\label{eq:StabfunSSPMPRK2}
		R(z)=\frac{-2 + (2\alpha\beta^2 - 2\alpha\beta + 1)z^2 - 2\beta(\alpha - 1)z}{2(1 + (\alpha\beta - 1)z)(\beta z - 1)}.
	\end{equation}
\end{prop}
By this proposition, the SSPMPRK2($\alpha,\beta$) scheme satisfies all preconditions in order to apply Theorem \ref{Thm_MPRK_stabil}. Thus, we have to analyze the stability function $R$.
\begin{prop}\label{Prop:Stab_SSPMPRK2}
	Let $R$ be defined by \eqref{eq:StabfunSSPMPRK2} with $\alpha,\beta$ satisfying \eqref{eq:alphabeta_conditions}.
	\begin{enumerate}
		\item\label{item:a_propSSP2}  For any $\alpha>\frac{1}{2\beta}$, the set $\{z\in \Cminus\mid \lvert R(z)\rvert \leq 1\}$ is bounded.
		\item \label{item:b_propSSP2}	For all $\alpha<\frac{1}{2\beta}$ with $(\alpha,\beta)\neq(0,\frac12)$ we have $\lvert R(z)\rvert <1$ for all $z\in \Cminus\setminus\{0\}$. 
		\item\label{item:c_propSSP2} For $\alpha=\frac{1}{2\beta}$ or $(\alpha,\beta)=(0,\frac12)$ the relation $\lvert R(z)\rvert <1$ is true for all $z$ with $\re(z)<0$, and $\abs{R(z)}=1$ holds whenever $\re(z)=0$.
	\end{enumerate}
	
\end{prop}
\begin{proof}
	For proving part \ref{item:a_propSSP2}, we consider \eqref{eq:StabfunSSPMPRK2} with $z=re^{\ii \varphi}\in \Cminus\setminus\{0\}$ which yields
	\[\lim_{r\to\infty} R(z)=\frac{2\alpha\beta^2 - 2\alpha\beta + 1}{2\beta(\alpha\beta-1)}=\frac{2\alpha\beta^2 - 2\alpha\beta + 1}{2\alpha\beta^2-2\beta}.\]
	Note that for $\alpha=\frac{1}{2\beta}$, we obtain $\lim_{r\to\infty} R(z)=\frac{\beta-1+1}{\beta-2\beta}=-1$. Finally, it is straightforward to verify
	\begin{equation*}
		\begin{aligned}
			\partial_\alpha\left(\lim_{r\to\infty} R(z)\right)&=	\partial_\alpha\left(\frac{2\alpha\beta^2 - 2\alpha\beta + 1}{2\beta(\alpha\beta-1)}\right)\\
			&=\frac{2\beta(\beta-1)2\beta(\alpha\beta-1)-(2\alpha\beta(\beta-1) + 1)2\beta^2}{4\beta^2(\alpha\beta-1)^2}\\
			&=\frac{1-2\beta}{2(\alpha\beta-1)^2}<0,
		\end{aligned}
	\end{equation*}
	since $\beta\geq \frac12$. Therefore $\lim_{r\to\infty} R(z)$ decreases with increasing $\alpha$. As a result, for any $\alpha>\frac{1}{2\beta}$, we find $\lim_{r\to\infty} R(z)<-1$ and thus, there exists $z^*\in \Cminus$ so that $\lvert R(z^*)\rvert>1$. Indeed, the set $\{z\in \Cminus\mid \lvert R(z)\rvert \leq 1\}$ is bounded, as we find $\lvert R(z^*)\rvert>1$ for any $z^*\in \Cminus$ with $\lvert z^*\rvert$ large enough.
	
	We now focus on the derivation of the remaining statements, we
	investigate $\lvert R(z)\rvert$ first on the imaginary axis.
	A technical but elementary computation  for $z = \ii b$, with $b \in \R$, yields
	\[\lvert R(\ii b)\rvert^2=\frac{1+b^4(\alpha\beta^2-\alpha\beta+\frac12)^2+b^2(1+(\alpha^2+1)\beta^2-2\alpha\beta)}{(1+(\alpha\beta-1)^2b^2)(\beta^2b^2+1)}.\]
	Subtracting the denominator from the numerator leads to the expression
	\begin{equation}\label{eq:expProveSSP2}
		-(2\alpha\beta-2\beta-1)(2\alpha\beta-1)(2\beta-1)b^4.
	\end{equation}
	With respect to statement \ref{item:b_propSSP2}, we consider $\alpha<\frac{1}{2\beta}$ and $\beta>\frac12$, as $\beta=\frac12$ implies $\alpha=0$ due to equation \eqref{eq:alphabeta_conditions}. It follows that $2\beta-1>0$. Due to $\alpha<\frac{1}{2\beta}$, we see $2\alpha\beta <1$ and $2\alpha\beta-2\beta-1<1-2\beta-1<0$, so that the whole product \eqref{eq:expProveSSP2} becomes negative, whenever $z=\ii b\neq 0$. This is equivalent to $\lvert R(z)\rvert<1$ on the imaginary axis without the origin.
	Using Remark~\ref{rem:Phragmen}, we see that
	$\abs{R(z)}<1$ holds for all $z\in \Cminus\setminus\{0\}$.
	
	The assertion \ref{item:c_propSSP2} can be proved in a similar way using \eqref{eq:expProveSSP2}. Indeed, in the case of $\alpha=\frac{1}{2\beta}$ or $(\alpha,\beta)=(0,\frac12)$, the product \eqref{eq:expProveSSP2} vanishes proving $\abs{R(z)}=1$ on the imaginary axis.
	Once again taking advantage of the Phragmén--Lindelöf principle one can conclude $\abs{R(z)}<1$ in $\C^-$.
\end{proof}
As a result we obtain the following corollaries that are a direct consequence of the application of Theorem~\ref{Thm:_Asym_und_Instabil} and Theorem \ref{Thm_MPRK_stabil}, as well as Remark~\ref{rem:Aneg}.
\begin{cor}\label{Cor:SSPMPRK2stab}
	Let $\b y^*$ be a positive steady state of the differential equation \eqref{eq:PDS_Sys} with $\b 1\in \ker(\bA^T)$. Then $\b y^*$ is a fixed point of the SSPMPRK2($\alpha,\beta$) scheme and the following holds:
	\begin{enumerate}
		\item
		For any $\alpha>\frac{1}{2\beta}$, the stability region of the SSPMPRK2($\alpha,\beta$) method is bounded.
		\item  For all $\alpha\leq\frac{1}{2\beta}$, the SPPMPRK22($\alpha,\beta$) scheme is unconditionally stable.
	\end{enumerate}
\end{cor}
\begin{cor}\label{Cor:SSPMPRK2stab1}
	Let the unique steady state $\b y^*$ of the initial value problem \eqref{eq:PDS_Sys}, \eqref{eq:IC} be positive and $\b 1\in \ker(\bA^T)$. Then there exists a $\delta >0$ such that $\Vert\b y^0-\b y^*\Vert<\delta$ implies the convergence of the iterates of the SSPMPRK2($\alpha,\beta$) scheme towards $\b y^*$ for all $\dt>0$, if $\alpha \leq\frac{1}{2\beta}$. For $\alpha >\frac{1}{2\beta}$, the method is conditionally stable.
\end{cor}
In order to illustrate the consequences of Corollary \ref{Cor:SSPMPRK2stab} consider Figure \ref{Fig:alphabetagraph}, where due to \eqref{eq:alphabeta_conditions} all permitted pairs of $(\alpha,\beta)$ with $\beta\leq 5$ lie between the $\beta$-axis and the black curve. The blue graph is determined by $\alpha=\frac{1}{2\beta}$, and thus, separates pairs of parameters associated with unconditionally stable methods from those with bounded stability domains. As an example, here we will consider the red rectangular with vertices $(0.2,3)$, $(0.2,3.5)$, $(0.24,3)$ and $(0.24,3.5)$, which is located in that critical region, so that we further analyze the corresponding choices of parameters with the help of Figure \ref{Fig:StabregionSSPMPRK2}, where we plot the corresponding stability regions. One can observe that the chosen pairs of parameters from Figure~\ref{Fig:alphabetagraph} that are closer to the blue graph are associated with a larger stability domain. The smallest stability region among the examples from Figure~\ref{Fig:StabregionSSPMPRK2} are associated with the $(\alpha,\beta)$ pair at the top right corner of the red rectangular from Figure~\ref{Fig:alphabetagraph}.
\begin{figure}[!h]
	\centering
	\includegraphics[width=0.9\textwidth]{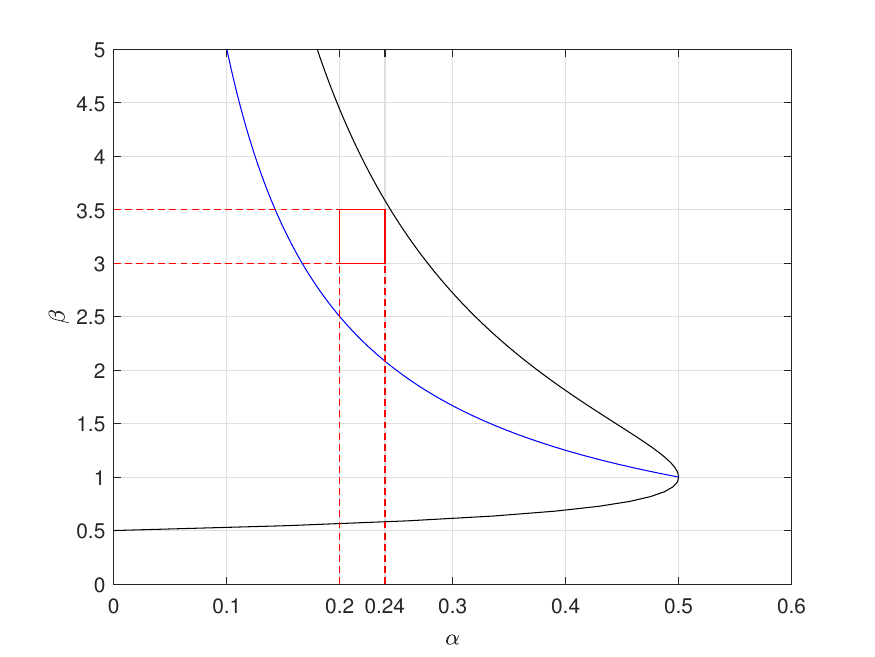}
	\caption{The black curve is implicitly given by the function $\alpha(\beta)=\frac{1-\frac{1}{2\beta}}{\beta}$. The blue graph is determined by the equation $\alpha=\frac{1}{2\beta}$ and the red rectangular possesses the vertices $(\alpha,\beta)$ with $(0.2,3)$, $(0.2,3.5)$, $(0.24,3)$ and $(0.24,3.5)$ which lie between the black and blue curve.}\label{Fig:alphabetagraph}
\end{figure}
\begin{figure}[h!]
	\begin{subfigure}[c]{0.495\textwidth}
		\includegraphics[width=\textwidth]{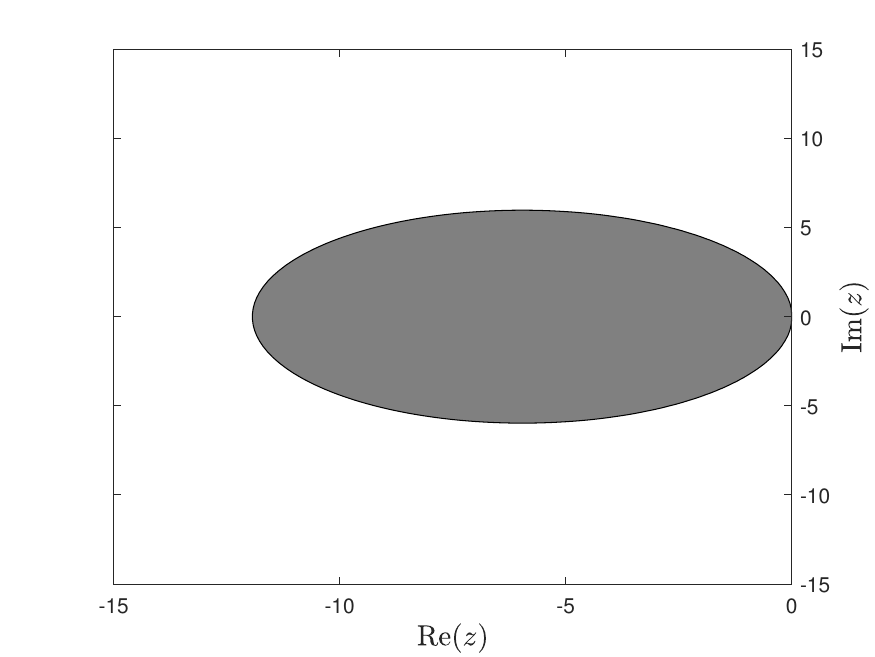}
		\subcaption{$(\alpha,\beta)=(0.2,3)$}
	\end{subfigure}
	\begin{subfigure}[c]{0.495\textwidth}
		\includegraphics[width=\textwidth]{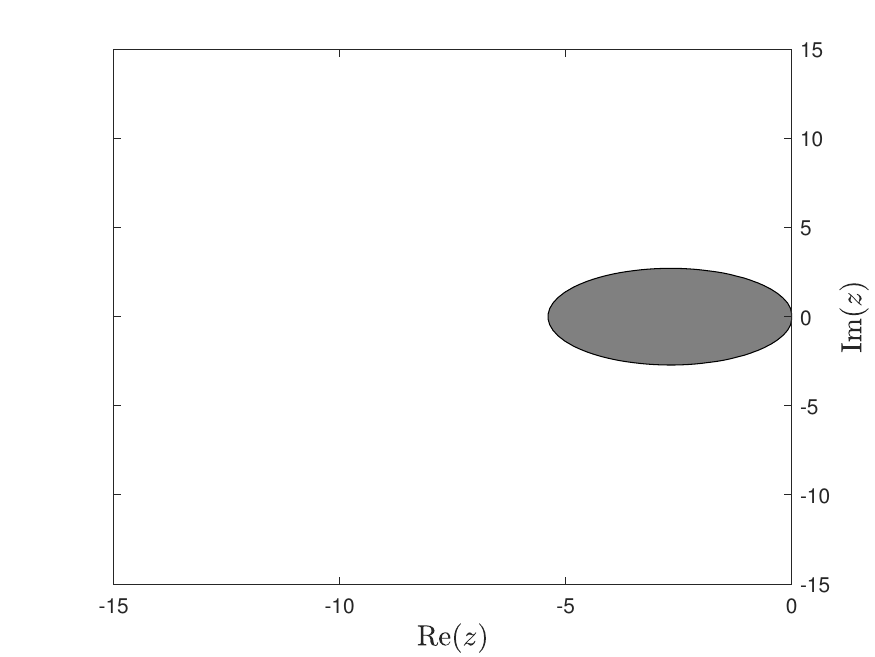}
		\subcaption{$(\alpha,\beta)=(0.24,3)$}
	\end{subfigure}\\
	\begin{subfigure}[t]{0.495\textwidth}
		\includegraphics[width=\textwidth]{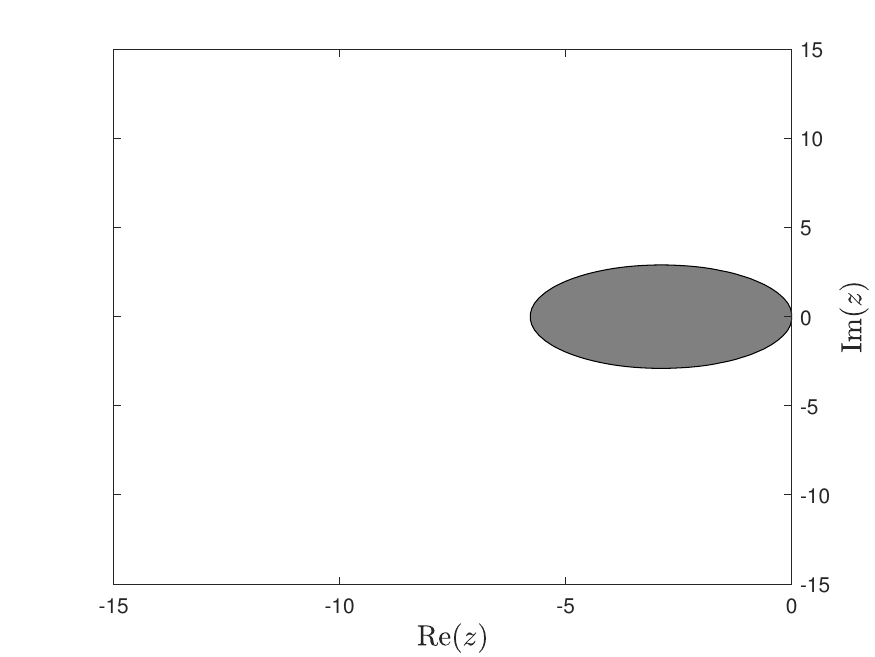}
		\subcaption{$(\alpha,\beta)=(0.2,3.5)$}
	\end{subfigure}
	\begin{subfigure}[t]{0.495\textwidth}
		\includegraphics[width=\textwidth]{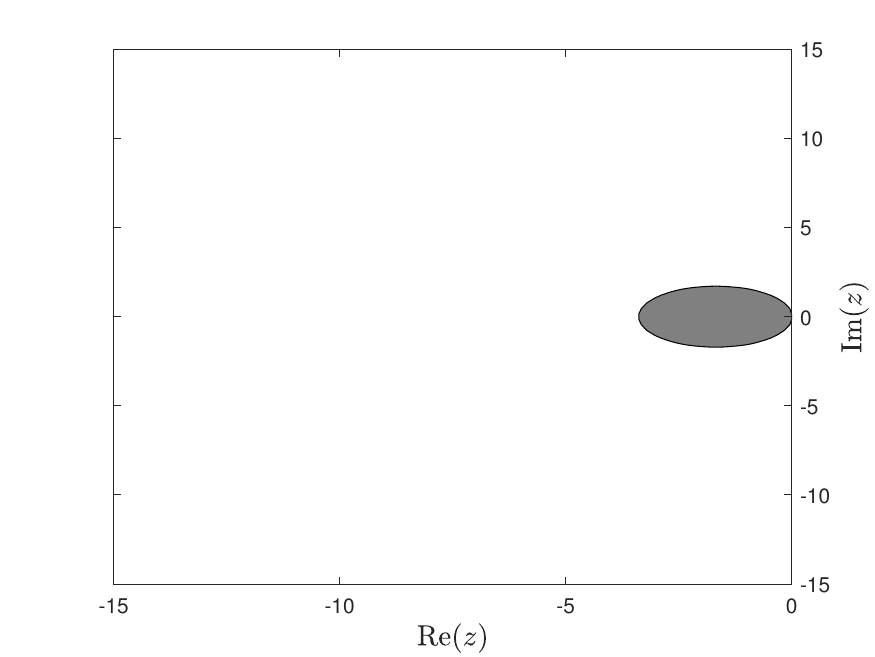}
		\subcaption{$(\alpha,\beta)=(0.24,3.5)$}
	\end{subfigure}
	\caption{Different stability domains of the SSPMPRK2($\alpha,\beta$) method are plotted for $(\alpha,\beta)$ associated with the corners of the red rectangular from  Figure \ref{Fig:alphabetagraph}.}\label{Fig:StabregionSSPMPRK2}
\end{figure}
\paragraph{SSPMPRK3($\eta_2)$}

As the first step, we apply \eqref{eq:SSPMPRK3} to the linear test problem \eqref{eq:PDS_Sys}, assuming conservativity, and rewrite it in the
matrix-vector notation. For this, we again make use of equation \eqref{eq:-sum(dij)} and the fact that the production and destruction terms are linear, which results in
{\allowdisplaybreaks
	\begin{align}
		\b 0=&\bm \Phi_1(\b y^n,\b y^{(1)}) = \alpha_{10}\b y^n+\beta_{10} \dt\bA\b y^{(1)}-\b y^{(1)},\nonumber\\
		\b 0=&\bm \Phi_{\bm\rho}(\b y^n,\b y^{(1)},\bm \rho) = n_1\b y^{(1)}+n_2(\diag(\b y^{(1)}))^{2}(\diag(\b y^{n}))^{-1}\bm 1-\bm \rho,\nonumber\\
		\b 0=&\bm \Phi_2(\b y^n,\b y^{(1)},\bm\rho,\b y^{(2)}) = \alpha_{20}\b y^n+\alpha_{21}\b y^{(1)}\nonumber\\&+\dt\bA\diag(\b y^{(2)})(\diag(\bm \rho))^{-1}(\beta_{20}\b y^n+\beta_{21}\b y^{(1)})-\b y^{(2)},\nonumber\\
		\b 0=&\bm \Phi_{\bm\gamma}(\b y^n,\b y^{(1)},\bm\gamma) = \eta_1\b y^n+\eta_2\b y^{(1)}\nonumber\\ &+\dt\bA\diag(\bm\gamma)(\diag(\b y^n))^{s-1}(\diag(\b y^{(1)}))^{-s}(\eta_{3}\b y^n+\eta_{4}\b y^{(1)})-\bm\gamma,\nonumber\\
		\b 0=&\bm \Phi_{\bm \sigma}(\b y^n,\b y^{(2)},\bm \rho,\bm\gamma,\bm \sigma) = \bm\gamma+\zeta(\diag(\b y^{(2)}))(\diag(\bm \rho))^{-1}\b y^n-\bm \sigma,\nonumber\\
		\b 0=&\bm \Phi_{n+1}(\b y^n,\b y^{(1)},\bm\rho,\b y^{(2)},\b y^{n+1})= \alpha_{30}\b y^n+\alpha_{31}\b y^{(1)}\nonumber\\&+\alpha_{32}\b y^{(2)}+\dt\bA\diag(\b y^{n+1})(\diag(\bm \sigma)^{-1}(\beta_{30}\b y^n+\beta_{31}\b y^{(1)}+\beta_{32}\b y^{(2)})-\b y^{n+1},\label{eq:SSPMPRK3MatrixVector}
\end{align}}
where we omitted to write the arguments as functions of $\b y^n$. Moreover, the parameter $s$ is determined by \cite[Eq. (3.19)]{SSPMPRK3}, also see \cite{repoSSPMPRK} for the details of the computation. 

Now, we could  introduce  three stages $\b y^{(3)}$, $\b y^{(4)}$ and $\b y^{(5)}$ for quantities $\bm\rho$, $\bm\gamma $ and $\bm \sigma$ in order to keep the notation from \eqref{eq:FormularJacobian} as we did for MPRK43 schemes. However, this might be more confusing at this point. We instead introduce auxiliary Jacobians $\b D_{\bm\sigma}$ etc.\ in the same manner as for \eqref{eq:FormularJacobian}.

We want to point out that all functions from above are $\mathcal C^2$-maps for positive arguments. Thus, the map $\b g$, which is determined by solving linear systems, is in $\mathcal C^2$. Due to Remark \ref{rem:C2->C1}, the first derivatives are Lipschitz continuous for a sufficiently small neighborhood of $\b y^*$. 

Also, we can prove $\b g(\b v)=\b v$ for all $\b v\in \ker(\bA)\cap \R^N_{>0}$ as follows. We know that $\bm \Phi_1(\b y^*,\b y^*)=\b 0$, and hence,  $\b y^n=\b y^*$ implies $\b y^{(1)}=\b y^*$  as $\b y^{(1)}$ is uniquely determined by $\b y^n$. Analogously, we conclude $\bm \rho(\b y^*)=\b y^*$ as $n_1+n_2=1$. As a consequence, we conclude from $a_{20}+a_{21}=1$ at double precision that also $\b y^{(2)}(\b y^*)=\b y^*$. However, $\bm\gamma(\b y^*)=(\eta_1+\eta_2)\b y^*$, from which it follows that $\bm \sigma(\b y^*)=(\eta_1+\eta_2)\b y^*+\zeta\b y^*=\b y^*$ since $\eta_1+\eta_2=1-\zeta$. Finally $\b y^{n+1}(\b y^*)=\b g(\b y^*)=\b y^*$ follows because $\sum_{i=0}^2\alpha_{3i}=1$ is true at double precision. 

In the following we use $a_{20}+a_{21}=1$, $\sum_{i=0}^2\alpha_{3i}=1$ and $\alpha_{10}=1$ as well as the values of the functions evaluated at $\b y^*$ without further notice. 

Moreover, we can observe that $\b g$ conserves all linear invariants as follows. First, $\b n^T\bA=\b 0$ implies \[\b n^T\b y^{(1)}=\alpha_{10}\b n^T \b y^n+\beta_{10}\dt\b n^T\bA\b y^{(1)}=\b n^T\b y^n.\]
As a consequence, we obtain
\[\b n^T\b y^{(2)}=\alpha_{20}\b n^T \b y^n+\alpha_{21}\b n^T\b y^{(1)}+\b 0=(\alpha_{20}+\alpha_{21})\b n^T\b y^n=\b n^T\b y^n.\]
Altogether, we find that $\b g$ is linear invariants preserving due to
\[\b n^T\b g(\b y^n)=\b n^T\b y^{n+1}=\sum_{i=0}^2\alpha_{3i}\b n^T\b y^n+\b 0=\b n^T\b y^n.\] Hence, also in the third order case, the map $\b g$ satisfies all conditions for applying Theorem \ref{Thm:_Asym_und_Instabil} and Theorem \ref{Thm_MPRK_stabil}. Therefore, we are now interested in computing the Jacobian of $\b g$, which can be done by using the same techniques as for the second order SSPMPRK scheme. Since we use a slightly different notation, let us recall the formula for $\b D\b g(\b y^*)$.
Using the chain rule for the last equation of \eqref{eq:SSPMPRK3MatrixVector} and solving for $\b D\b g(\b y^*)$ yield
\begin{equation}\label{eq:Dg(y*)_Formula_SSP3}
	\begin{aligned}
		\b D\b g(\b y^*)=-(\b D^*_{n+1}\bm \Phi_{n+1})^{-1}(&\b D^*_n\bm \Phi_{n+1}+\b D^*_1\bm \Phi_{n+1}\b D^*\b y^{(1)}+\b D^*_2\bm \Phi_{n+1}\b D^*\b y^{(2)}
		\\&+\b D^*_{\sigma}\bm \Phi_{n+1}\b D^*\bm \sigma),
	\end{aligned}
\end{equation}
if $(\b D^*_{n+1}\bm \Phi_{n+1})^{-1}$ exists. Hence, we need formulae for $\b D^*\b y^{(1)}, \b D^*\b y^{(2)}$ and $\b D^*\bm \sigma$. We use the same strategies as for the second order scheme and obtain by means of the chain rule of the corresponding equation in \eqref{eq:SSPMPRK3MatrixVector} the formulae
\begin{equation}\label{eq:Dy-Dsigma}
	\begin{aligned}
		\b D^*\b y^{(1)}&=-(\b D_1^*\bm\Phi_1)^{-1}\b D_n^*\bm\Phi_1,\\
		\b D^*\b y^{(2)}&=-(\b D^*_2\bm \Phi_2)^{-1}(\b D^*_n\bm \Phi_2+\b D^*_1\bm \Phi_2\b D^*\b y^{(1)}+\b D^*_{\rho}\bm \Phi_2\b D^*\bm\rho),\\
		\b D^*\bm\sigma&=-(\b D^*_{\sigma}\bm \Phi_{\bm \sigma})^{-1}(\b D^*_n\bm \Phi_{\bm \sigma}+\b D^*_2\bm \Phi_{\bm \sigma}\b D^*\b y^{(2)}+\b D^*_{\rho}\bm \Phi_{\bm \sigma}\b D^*\bm\rho+\b D^*_{\bm \gamma}\bm \Phi_{\bm \sigma}\b D^*\bm\gamma),
	\end{aligned}
\end{equation}
provided that the inverses exist. However, to compute the last two Jacobians, we now require to have knowledge about $\b D^*\bm \rho$ and $\b D^*\bm\gamma$. These Jacobians can be obtained by
\begin{equation}\label{eq:Drho,Da}
	\begin{aligned}
		\b D^*\bm \rho&=-(\b D_{\bm\rho}^*\Phi_{\bm\rho})^{-1}(\b D^*_n\bm \Phi_{\bm\rho}+\b D^*_1\bm \Phi_{\bm\rho}\b D^*\b y^{(1)}),\\
		\b D^*\bm\gamma&=-(\b D^*_{\bm\gamma}\bm \Phi_{\bm\gamma})^{-1}(\b D^*_n\bm \Phi_{\bm\gamma}+\b D^*_1\bm \Phi_{\bm\gamma}\b D^*\b y^{(1)}),
	\end{aligned}
\end{equation}
if the expressions are defined. Starting off with the calculation of $\b D^*\b y^{(1)}$, we obtain
\begin{equation*}
	\begin{aligned}
		\b D^*_n\bm \Phi_1=\alpha_{10}\b I, \quad \b D^*_1\bm \Phi_1=\beta_{10} \dt\bA-\b I.
	\end{aligned}
\end{equation*}
Since $\beta_{10}>0$ we can use \eqref{eq:Dy-Dsigma} to conclude that
\begin{equation*}
	\b D^*\b y^{(1)}=-(\beta_{10} \dt\bA-\b I)^{-1}\cdot \alpha_{10}\b I=(\b I-\beta_{10} \dt\bA)^{-1}
\end{equation*}
is defined.
Next we focus on $\b D^*\bm \rho$ so that we can compute $\b D^*\b y^{(2)}$ afterwards. For this, we use again that diagonal matrices commute and that $\diag(\b v)\b w=\diag(\b w)\b v$ holds. Hence, we find
\begin{equation*}
	\begin{aligned}
		\b D^*_n\bm \Phi_{\bm\rho}&=-n_2\b I, \quad \b D^*_1\bm \Phi_{\bm\rho}=(n_1+2n_2)\b I, \quad \b D_{\bm\rho}^*\Phi_{\bm\rho}=-\b I,
	\end{aligned}
\end{equation*}
and due to \eqref{eq:Drho,Da},
\begin{equation*}
	\b D^*\bm \rho=-n_2\b I+(n_1+2n_2)(\b I-\beta_{10} \dt\bA)^{-1}.
\end{equation*}
The computation of the following Jacobians requires the same technique as described in equations \eqref{eq:ProductRuleDiag} and \eqref{eq:trick}, from which we get
\begin{equation*}
	\begin{aligned}
		\b D^*_n\bm \Phi_2&=\alpha_{20}\b I+\beta_{20}\dt\bA, & \b D^*_1\bm \Phi_2&=\alpha_{21} \b I+\beta_{21}\dt\bA,\\ \b D^*_{\rho}\bm \Phi_2&=-(\beta_{20}+\beta_{21})\dt\bA, & \b D^*_2\bm \Phi_2&=(\beta_{20}+\beta_{21}) \dt\bA-\b I,
	\end{aligned}
\end{equation*}
respectively.
Since $\beta_{20}+\beta_{21}>0$ the inverse of $\b D^*_2\bm \Phi_2$ exists, and thus, $\b D^*\b y^{(2)}$ is formally given by \eqref{eq:Dy-Dsigma}.

Next, we need $\b D^*\bm\gamma$ in order to find $\b D^*\bm \sigma$. Exploiting once again the ideas from \eqref{eq:ProductRuleDiag} and \eqref{eq:trick}, we obtain with $\bm\gamma(\b y^*)=(\eta_1+\eta_2)\b y^*$ the Jacobians
{\allowdisplaybreaks
	\begin{align*}
		\b D^*_n\bm \Phi_{\bm\gamma}&=\eta_1\b I+(\eta_1+\eta_2)\dt\bA((s-1)(\eta_3+\eta_4)+\eta_3),\\
		\b D^*_1\bm \Phi_{\bm \gamma}&=\eta_2 \b I+(\eta_1+\eta_2)\dt\bA(-s(\eta_3+\eta_4)+\eta_4),\\ \b D^*_{\bm\gamma}\bm \Phi_{\bm\gamma}&=(\eta_3+\eta_4) \dt\bA-\b I,
\end{align*}}
where $\b D^*_{\bm \gamma}\bm \Phi_{\bm\gamma}$ is nonsingular since $\eta_3+\eta_4>0$. Hence, with \eqref{eq:Drho,Da} even the Jacobian $\b D^*\bm\gamma$ can be determined.

Computing
\begin{equation*}
	\begin{aligned}
		\b D^*_n\bm \Phi_{\bm \sigma}&=\zeta\b I, \quad \b D^*_2\bm \Phi_{\bm \sigma}=\zeta \b I,\quad \b D^*_{\rho}\bm \Phi_{\bm \sigma}=-\zeta\b I, \quad \b D^*_{\bm\gamma}\bm \Phi_{\bm \sigma}=\b I,\quad \b D^*_{\bm \sigma}\bm \Phi_{\bm \sigma}=-\b I,
	\end{aligned}
\end{equation*}
we are able to obtain $\b D^*\bm \sigma$ from \eqref{eq:Dy-Dsigma}. Finally, the remaining Jacobians are given by
\begin{equation*}
	\begin{aligned}
		\b D^*_n\bm \Phi_{n+1}&=\alpha_{30}\b I+\beta_{30}\dt\bA, & \b D^*_1\bm \Phi_{n+1}&=\alpha_{31} \b I+\beta_{31}\dt\bA,\\
		\b D^*_2\bm \Phi_{n+1}&=\alpha_{32}\b I+\beta_{32}\dt\bA, & \b D^*_{\bm \sigma}\bm \Phi_{n+1}&=-\dt\bA\sum_{i=0}^2\beta_{3i},\\ \b D^*_{n+1}\bm \Phi_{n+1}&= \dt\bA\sum_{i=0}^2\beta_{3i}-\b I
	\end{aligned}
\end{equation*} with $\sum_{i=0}^2\beta_{3i}>0$, so that we are now in the position to compute $\b D\b g(\b y^*)$ using \eqref{eq:Dg(y*)_Formula_SSP3}. As all the matrices occurring within the expressions of the Jacobians above are either the identity matrix $\b I$ or the system matrix $\bA$ from \eqref{eq:PDS_Sys}, the stability function for the third order SSPMPRK scheme can easily be computed by calculating $\b D\b g(\b y^*)$ and substituting $\dt\bA$ by $\dt\lambda=z$, so that we end up with the stability function $R(\dt\lambda)=R(z)$ that reads
\begin{equation}\label{eq:StabfunSSPMPRK3}
	\begin{aligned}
		R(z)=&\frac{1}{1-z\sum_{i=0}^2\beta_{3i}}\Biggl[\alpha_{30}+\beta_{30}z+\frac{\alpha_{31}+\beta_{31}z}{1-\beta_{10}z}+(\alpha_{32}+\beta_{32}z)P(z)\\&-z\sum_{i=0}^2\beta_{3i}\Biggl(\zeta+\zeta P(z)-\zeta\left(\frac{n_1+2n_2}{1-\beta_{10}z}-n_2\right)\\
		&+\frac{1}{1-(\eta_3+\eta_4)z}\Biggl(\eta_1+(\eta_1+\eta_2)z\Bigl((s-1)(\eta_3+\eta_4)+\eta_3\Bigr)
		\\&+\frac{\eta_2+(\eta_1+\eta_2)z\bigl(-s(\eta_3+\eta_4)+\eta_4\bigr)}{1-\beta_{10}z}\Biggr)\Biggr) \Biggr],\\
		P(z)=&\frac{\alpha_{20}+\beta_{20}z+\frac{\alpha_{21}+\beta_{21}z}{1-\beta_{10}z}
			-(\beta_{20}+\beta_{21})z\left(\frac{n_1+2n_2}{1-\beta_{10}z}-n_2\right)}{1-(\beta_{20}+\beta_{21})z}.
	\end{aligned}
\end{equation}
Before a detailed investigation of the stability function $R$, we summarize the above derived results by means of the following proposition.
\begin{prop}\label{Prop:SSPMPRK3_Dg(y*)}
	Let $\b g\from\R^N_{>0}\to\R^N_{>0}$ be the generating map of SSPMPRK3($\eta_2$) when applied to the differential equation \eqref{eq:PDS_Sys} with $\bm 1\in \ker(\bA^T)$.
	Then any $\b y^*\in \ker(\bA)\cap \R^N_{>0}$ is a fixed point of $\b g\in \mathcal{C}^2(\R^N_{>0},\R^N_{>0})$, where the first derivatives of $\b g$ are Lipschitz continuous in an appropriate neighborhood of $\b y^*$. Moreover, all linear invariants are conserved and an eigenvalue $\lambda$ of $\bA$ corresponds to the eigenvalue $R(\dt\lambda)$ of the Jacobian of $\b g$ where $R$ is defined in \eqref{eq:StabfunSSPMPRK3}
	and the parameters are given in \eqref{eq:SSPMPRK3Parameters}.
\end{prop}
Next, we will prove that the third order SSPMPRK scheme possesses stable fixed points for all $\eta_2\in[0,r_1]$ when applied to the test equation.
\begin{prop}
	The stability function $R(z)$ of the third order SSPMPRK scheme satisfies $R(0)=1$ and $\lvert R(z)\rvert <1$ for all $z\in \Cminus\setminus\{0\}$ up to double precision.
\end{prop}
\begin{proof}
	It is straightforward to see that $R(0)=\alpha_{30}+\alpha_{31}+\alpha_{32}(\alpha_{20}+\alpha_{21})$ holds true. Up to double precision, we obtain $\alpha_{20}+\alpha_{21}=1$ and $\alpha_{30}+\alpha_{31}+\alpha_{32}=1$, so that $R(0)=1$. Also, as $\alpha_{ij},\beta_{ij},\eta_3+\eta_4>0$, see \eqref{eq:SSPMPRK3Parameters}, no poles of $R$ are located in $\Cminus$.
	Furthermore, by a technical calculation we can rewrite $R$ to receive
	\begin{equation*}
		R(z)=\frac{\sum_{j=0}^4n_jz^j}{\sum_{j=0}^4d_jz^j},
	\end{equation*}
	where, for $\eta_2\in [0,r_1]\tm [0,\frac12)$ the coefficients are given by 
	\begin{equation*}
		\begin{aligned}
			n_0=&\frac{0.47620819268131705757\eta_2 - 1.0537480911094115481}{0.47620819268131703\eta_2 - 1.0537480911094114871},\\
			n_1=&\frac{-3.1507612671062001337\eta_2 + 3.9798736646158920698}{0.47620819268131703\eta_2 - 1.0537480911094114871}\\
			&+\frac{0.61107641837494959323\eta_2^2}{0.47620819268131703\eta_2 - 1.0537480911094114871},\\
			n_2=&\frac{2.4343280828365809236\eta_2 - 2.5818776483048969774}{0.47620819268131703\eta_2 - 1.0537480911094114871}\\
			&+\frac{-0.57282016379130601724\eta_2^2}{0.47620819268131703\eta_2 - 1.0537480911094114871},\\
			n_3=&\frac{0.6536869584417787153\eta_2 - 0.81355615989342266462}{0.47620819268131703\eta_2 - 1.0537480911094114871}\\
			&+\frac{ - 0.1292603911580354457\eta_2^2}{0.47620819268131703\eta_2 - 1.0537480911094114871},\\
			n_4=&\frac{-0.59499575916146815582\eta_2 + 0.63887056198975790458 }{0.47620819268131703\eta_2 - 1.0537480911094114871}\\
			&+\frac{0.1384128438067575936\eta_2^2}{0.47620819268131703\eta_2 - 1.0537480911094114871},\\
			d_0=&1,\\
			d_1=&-4.7768739020212929733 + 1.2832127371313151768\eta_2,\\
			d_2=&6.7270587897458664634 - 2.4860903284764154151\eta_2,\\
			d_3=&-3.7332290665687486456 + 1.5730472371819288192\eta_2,\\
			d_4=& 0.71670702950202557445-0.32389312216150656420\eta_2,
		\end{aligned}
	\end{equation*}
	where $n_0=1$ at double precision, see \cite{repoSSPMPRK}. We want to mention here, that these values were computed with \emph{Maple~2021} and $\operatorname{Digits}=20$, which means that 20 digits were used when making calculations with software floating-point numbers.
	
	We investigate the polynomial $p_{\frac\pi2}(r)$ from Lemma~\ref{lem:stabconditionR(rexp(phi))} with $\eta_2$ being a parameter.
	At double precision, we obtain $n_0=1$, so that $n_0^2-1=0$, \ie $p_{\frac\pi2}(0)=0$. Next, our strategy is to prove that all nonzero coefficients of $r^k$, in the following denoted by $c_k$ are negative.
	
	For $\eta_2\leq r_1< \frac12$, it suffices for our argument to round to three decimal places in the following expressions, which can be reproduced using the Maple repository \cite{repoSSPMPRK} and read
	{\allowdisplaybreaks
		\begin{align*}
			c_8&\approx\frac{4.410(-0.168\eta_2^2+0.271\eta_2+0.046\eta_2^3-0.162-0.005\eta_2^4)}{(\eta_2-2.213)^2},\\
			c_6&\approx\frac{4.410(-0.790\eta_2^2 + 1.310\eta_2+0.210\eta_2^3 - 0.808 - 0.021\eta_2^4)}{(\eta_2 - 2.213)^2},\\
			c_4&\approx\frac{4.410(0.556\eta_2 - 0.442 + 0.032\eta_2^3 - 0.232\eta_2^2)}{(\eta_2 - 2.231)^2},\\
			10^{14}c_2&\approx\frac{\eta_2(\eta_2 - 1 - 0.2\eta_2^2 + 0.03\eta_2^3)}{(0.476\eta_2 - 1.054)^2}.
	\end{align*}}
	First of all, the denominators occurring in any of the above $c_k$ are positive. Also, positive terms in the numerator are multiplied with powers of $\eta_2<\frac12$ and thus are smaller than the absolute value of the corresponding constant, which is always negative. This holds true even if the rounding error is taken into account, i.\,e.\ after adding $10^{-2}$ to positive terms and subtracting it from negative expressions. This proves that $c_k<0$, and thus, $\abs{R(\ii y)}<1$ for all $y\in \R\setminus\{0\}$.
	
	Finally, we can conclude even $\lvert R(z)\rvert<1$ for all $z\in \overline{\C^-}\setminus\{0\}$ by  means of Remark~\ref{rem:Phragmen}.	
\end{proof}
As an immediate consequence of this proposition in combination with Theorem~\ref{Thm:_Asym_und_Instabil} and Theorem~\ref{Thm_MPRK_stabil}, we obtain the following results.
\begin{cor}\label{Cor:SSPMPRK3stab}
	The SSPMPRK3($\eta_2$) scheme is unconditionally stable for all $\eta_2\in [0,r_1]$, where $r_1\approx0.37$.
\end{cor}
\begin{cor}\label{Cor:SSPMPRK3stab1}
	Let $\b y^*$ be the unique steady state of the initial value problem \eqref{eq:PDS_Sys}, \eqref{eq:IC} with $\b 1\in \ker(\bA^T)$. Then there exists a $\delta >0$ such that $\Vert\b y^0-\b y^*\Vert<\delta$ implies the convergence of the iterates of of SSPMPRK3($\eta_2$) towards $\b y^*$ for all $\dt>0$ and $\eta_2\in [0,r_1]$.
\end{cor}
\newpage
\subsection{Modified Patankar Deferred Correction}\label{sec:stab_MPDeC}
In this subsection we investigate \eqref{eq:explicit_dec_correction}. Since the index function $\gamma$ depends on the sign of $\theta_r^m$, we introduce the nonnegative part $\theta_{m,+}=\max\{0,\theta_r^m\}$ and nonpositive part $\theta_{m,-}=\min\{0,\theta_r^m\}$. It is worth mentioning that 
\begin{equation*}
	\begin{aligned}
		\theta_{r,\pm}^m=\frac{\theta_r^m\pm\lvert \theta_r^m\rvert}{2}
	\end{aligned}
\end{equation*}
and
\[\theta_r^m=\begin{cases}\theta_{r,-}^m,& \theta_r^m<0,\\
	\theta_{r,+}^m,& \theta_r^m\geq 0\end{cases}\]
as well as $\theta_{r,-}^m+\theta_{r,+}^m=\theta_r^m$.
With that, we split the sum appearing in \eqref{eq:explicit_dec_correction} into two sums containing $\theta_{r,+}^m$ and $\theta_{r,-}^m$, respectively. For the separated sums, we know the value of $\gamma(j,i,\theta_r^m)$ so that we introduce the positive part 
\begin{equation}\label{eq:prk}
	\begin{aligned}
		\b p^{r,(k)}(\b y^{r,(k-1)},\b y^{m,(k-1)},\b y^{m,(k)})
		=\bA\diag(\b y^{m,(k)})\left(\diag(\b y^{m,(k-1)})\right)^{-1} \b y^{r,(k-1)}
	\end{aligned}
\end{equation}
analogously as we did for SSPMPRK, as well as  the negative part $\b n^{r,(k)}$ given by
\begin{equation}\label{eq:nrkpij}
	\begin{aligned}
		n_i^{r,(k)}(\b y^{r,(k-1)},\b y^{m,(k-1)},\b y^{m,(k)})=\sum_{j=1}^N \Biggl(& p_{ij}(\b y^{r,(k-1)})
		\frac{y^{m,(k)}_{i}}{y_{i}^{m,(k-1)}}
		\\&- d_{ij}(\b y^{r,(k-1)})  \frac{y^{m,(k)}_{j}}{y_{j}^{m,(k-1)}} \Biggr)
	\end{aligned}
\end{equation}
for $i=1,\dotsc, N$, $r=0,\dotsc, M$ and $k=1,\dotsc,K$. Using $p_{ij}(\b y)=d_{ji}(\b y)=\lambda_{ij}y_j$ for $i\neq j$ and $p_{ii}(\b y)=d_{ii}(\b y)=0$ this can be rewritten as
\begin{equation}\label{eq:nrk}
	n_i^{r,(k)}(\b y^{r,(k-1)},\b y^{m,(k-1)},\b y^{m,(k)})=	\frac{y^{m,(k)}_{i}}{y_{i}^{m,(k-1)}}\sum_{\substack{j=1\\j\neq i}}^N \lambda_{ij}  y_j^{r,(k-1)}
	- y_i^{r,(k-1)}  \sum_{\substack{j=1\\j\neq i}}^N \lambda_{ji}\frac{y^{m,(k)}_{j}}{y_{j}^{m,(k-1)}}.
\end{equation}
Utilizing these vector fields, the iterates from \eqref{eq:explicit_dec_correction} satisfy
\begin{equation}\label{eq:Phis_MPDeC}
	\begin{aligned}
		\b 0 &=\bm \Phi_k^m(\b y^n,\b y^{1,(k-1)},\dotsc,\b y^{M,(k-1)},\b y^{m,(k)})\\
		\bm \Phi_k^m &=\b y^{m,(k)}-\b y^n-\sum_{r=0}^M \theta_{r,+}^m \dt\b p^{r,(k)}(\b y^{r,(k-1)},\b y^{m,(k-1)},\b y^{m,(k)})\\ &\hphantom{=\b y^{m,(k)}-\b y^n}-\sum_{r=0}^M \theta_{r,-}^m \dt\b n^{r,(k)}(\b y^{r,(k-1)},\b y^{m,(k-1)},\b y^{m,(k)})
	\end{aligned}
\end{equation}
for $k=1,\dotsc,K$ and $m=1,\dotsc,M$. 
Furthermore, analogously to the auxiliary Jacobians introduced in \eqref{eq:jacobians1} and \eqref{eq:jacobians2}, we write $\b D^*_x\bm \Phi_k^m$ to represent the Jacobian with respect to the entries of the vector $\b y^x$ for some $x$, evaluated at $(\b y^*,\dotsc, \b y^{m,(k)}(\b y^*))$. Finally, we introduce similar notations for the auxiliary Jacobians of $\b p^{r,(k)}$ and $\b n^{r,(k)}$ with respect to $\b y^x$. 

Also note that  MPDeC schemes are steady state preserving as plugging in $\b y^{m,(k)}=\b y^n=\b y^*\in \ker(\bA)$ into \eqref{eq:explicit_dec_correction} yields a true statement. Hence, $\b y^{m,(k)}(\b y^*)=\b y^*$ for all $k=1,\dotsc,K$ and $m=1,\dotsc, M$.

The next theorem summarizes further properties of the method and provides us a formula for the computation of $\b D\b g(\b y^*)$.
\begin{thm}\label{thm:mpdec}
	Let $\b g:\R^N_{>0}\to\R^N_{>0}$, implicitly given by the solution of \eqref{eq:Phis_MPDeC}, be the generating map of the MPDeC iterates when applied to \eqref{eq:PDS_Sys} with $\b 1\in \ker(\bA^T)$. Furthermore, let $\b y^*\in \ker(\bA)\cap \R^N_{>0}$ be a steady state of \eqref{eq:PDS_Sys}.
	
	Then, $\b g\in \mathcal C^2$ and the Jacobian of $\b g$ evaluated at $\b y^*$ is given by
	\begin{equation}\label{eq:Dg_MPDeC}
		\begin{aligned}
			\b D\b g(\b y^*)&=\b D^*_{n}\b y^{M,(K)},\\
			\b D^*_{n}\b y^{m,(k)}&=-(\b D^*_{m,(k)}\bm \Phi_k^m)^{-1}\left(\b D^*_{n}\bm\Phi_k^m+(1-\delta_{k1})\sum_{r=1}^M\b D^*_{r,(k-1)}\bm \Phi_k^m\b D^*_{n}\b y^{r,(k-1)}\right)
		\end{aligned}
	\end{equation}
	for $m=1,\dotsc,M$ and $k=1,\dotsc, K$. Thereby, $\delta_{ij}$ is the Kronecker delta and 
	\begin{equation}\label{eq:D*y0Phikm}
		\begin{aligned}
			\b D^*_{n}\bm \Phi^m_k= \begin{cases}
				-\left(\b I+\dt(\bA +\diag(\b y^*)\bA^T(\diag(\b y^*))^{-1}) \sum_{r=0}^M \theta_{r,-}^m\right), &k= 1,\\
				-(\b I+\theta_0^m\dt\bA), &k> 1,
			\end{cases}
		\end{aligned}
	\end{equation}
	as well as 
	\begin{equation}\label{eq:D*l,(s)Phikm}
		\begin{aligned}
			\b D^*_{l,(s)}\bm \Phi^m_k=\begin{cases}
				-\theta_l^m\dt\bA,\\
				\sum_{\substack{r=0}}^M\theta_{r,+}^m\dt\bA-\sum_{\substack{r=0}}^M\theta_{r,-}^m\dt(\diag(\b y^*)\bA^T(\diag(\b y^*))^{-1})-\theta_{m}^m\dt \bA, \\
				\b I-\sum_{\substack{r=0}}^M\theta_{r,+}^m\dt\bA+\sum_{\substack{r=0}}^M\theta_{r,-}^m\dt\diag(\b y^*)\bA^T(\diag(\b y^*))^{-1}, .
			\end{cases}
		\end{aligned}
	\end{equation}
	for
	\[s=\begin{cases}
		k- 1>0,\, 0<l\neq m,\\
		k-1>0,\, l=m,\\
		s=k,\, l=m,
	\end{cases}\] respectively.
\end{thm} 
\begin{proof}
	Since the $\theta_r^m$ are fixed for a given scheme, the functions $\bm\Phi_k^m$ are in $\mathcal C^2$ and as a consequence of solving only linear systems, the map $\b g$ is also in $\mathcal C^2$. Furthermore, the formula \eqref{eq:Dg_MPDeC} follows analogously to \eqref{eq:FormularJacobian}, where we want to point out that the sum appearing in \eqref{eq:Dg_MPDeC} is multiplied with $0$ for $k=1$ since $\b y^{r,(k-1)}=\b y^n$ in this case. Hence, we only have to prove the formulae \eqref{eq:D*y0Phikm} and \eqref{eq:D*l,(s)Phikm}. For this, we compute the Jacobians of each addend of the sums in \eqref{eq:Phis_MPDeC} separately by considering \eqref{eq:prk} and \eqref{eq:nrk}.
	
	Let us start proving \eqref{eq:D*y0Phikm}, first considering $k=1$.  From \eqref{eq:prk}  and $\b y^{s,(0)}=\b y^n$ for all $s=0,\dotsc,M$ it follows that \[\b p^{r,(1)}(\b y^n,\b y^n,\b y^{m,(1)})=\b p^{r,(1)}(\b y^n,\b y^{m,(1)})=\bA\b y^{m,(1)}\] 
	and hence, $\b D^*_{n}\b p^{r,(1)}=\b 0$. 
	Moreover, \eqref{eq:nrk} for $k=1$ yields 
	\[ n_i^{r,(1)}(\b y^{n}, \b y^{n},\b y^{m,(1)})=	n_i^{r,(1)}(\b y^{n},\b y^{m,(1)})=	\frac{y^{m,(1)}_{i}}{y_{i}^{n}}\sum_{\substack{j=1\\j\neq i}}^N \lambda_{ij}  y_j^{n}
	- y_i^{n}  \sum_{\substack{j=1\\j\neq i}}^N \lambda_{ji}\frac{y^{m,(1)}_{j}}{y_{j}^{n}}.\]
	Hence, using $\bm 1\in \ker(\bA^T)$, we obtain
	\begin{equation}\label{eq:2aii}
		\begin{aligned}
			\frac{\partial}{\partial y_i^n} n_i^{r,(1)}(\b y^*,\b y^*)&= -\frac{1}{y_{i}^*}\sum_{\substack{j=1\\j\neq i}}^N \lambda_{ij}  y_j^{*}
			- \sum_{\substack{j=1\\j\neq i}}^N \lambda_{ji} =\frac{1}{y_{i}^*}\left(-\sum_{\substack{j=1\\j\neq i}}^N \lambda_{ij}  y_j^{*}
			+\lambda_{ii}y^{*}_{i}\right)\\
			&=\frac{1}{y_{i}^*}\Biggl(-\underbrace{\sum_{\substack{j=1}}^N \lambda_{ij}  y_j^{*}}_{(\bA\b y^*)_i=0}
			+2\lambda_{ii}y^{*}_{i}\Biggr)=2\lambda_{ii},
		\end{aligned}
	\end{equation}
	and for $q\neq i$ we find
	\begin{equation*}
		\begin{aligned}
			\frac{\partial}{\partial y_q^n} n_i^{r,(1)}(\b y^*,\b y^*)&= \lambda_{iq}
			+ \lambda_{qi}\frac{y^{*}_{i}}{y_{q}^{*}}. 
		\end{aligned}
	\end{equation*} 
	Altogether, we obtain 
	\begin{equation}\label{eq:A+diagAdiag}
		\b D^*_{n}\b n^{r,(1)}=\bA+\vec{
			\lambda_{11} & \lambda_{21}\tfrac{y_1^*}{y_2^*} &\dots & \lambda_{N1}\tfrac{y_1^*}{y_N^*} \\
			\lambda_{12}\tfrac{y_2^*}{y_1^*} & \ddots\hphantom{\tfrac{y_2^*}{y_1^*}}& & \vdots \\
			\vdots & & \ddots & \vdots\\ 
			\lambda_{1N}\tfrac{y_N^*}{y_1^*}& \dots &\dots & \lambda_{NN}
		}=\bA +\diag(\b y^*)\bA^T(\diag(\b y^*))^{-1},
	\end{equation}
	and thus,
	\begin{equation*}
		\b D^*_{n}\bm \Phi^m_1=-\left(\b I+(\bA +\diag(\b y^*)\bA^T(\diag(\b y^*))^{-1}) \Delta t\sum_{r=0}^M \theta_{r,-}^m\right).
	\end{equation*}
	Next, for $k> 1$ it follows from \eqref{eq:prk} that 
	\[\b D^*_{n}\b p^{r,(k)}=\delta_{r0} \bA. \]
	Similarly,  $\b D^*_{n}\b n^{r,(k)}=\b 0$ if $r\neq 0$. Furthermore,
	\[	n_i^{0,(k)}(\b y^{n},\b y^{m,(k-1)},\b y^{m,(k)})=	\frac{y^{m,(k)}_{i}}{y_{i}^{m,(k-1)}}\sum_{\substack{j=1\\j\neq i}}^N \lambda_{ij}  y_j^{n}
	- y_i^{n}  \sum_{\substack{j=1\\j\neq i}}^N \lambda_{ji}\frac{y^{m,(k)}_{j}}{y_{j}^{m,(k-1)}}\]
	yields
	\begin{equation*}
		\begin{aligned}
			\frac{\partial}{\partial y_i^n} n_i^{0,(k)}(\b y^*,\b y^*,\b y^*)&=- \sum_{\substack{j=1\\j\neq i}}^N \lambda_{ji}=\lambda_{ii}\\
			\frac{\partial}{\partial y_q^n} n_i^{0,(k)}(\b y^*,\b y^*,\b y^*)&=\lambda_{iq}, \quad i\neq q,
		\end{aligned}
	\end{equation*}
	so that $\b D^*_{n}\b n^{r,(k)}=\delta_{r0}\bA$. This results in
	\begin{equation*}
		\b D^*_{n}\bm \Phi^m_k=-\left(\b I+ (\theta_{0,-}^m+\theta_{0,+}^m)\Delta t\bA\right)=-\left(\b I+ \theta_{0}^m\Delta t\bA\right),
	\end{equation*}
	proving  \eqref{eq:D*y0Phikm}. 
	
	To derive \eqref{eq:D*l,(s)Phikm} consider first the case $s=k-1>0$ and $0<l\neq m$. From \eqref{eq:prk} it follows immediately that
	\begin{equation*}
		\b D^*_{l,(k-1)}\b p^{r,(k)}=\delta_{rl}\bA.
	\end{equation*}
	Moreover, \eqref{eq:nrk} yields
	\begin{equation*}
		\begin{aligned}
			\frac{\partial}{\partial y_i^{l,(k-1)}} n_i^{r,(k)}(\b y^*,\b y^*,\b y^*)&=-\delta_{rl} \sum_{\substack{j=1\\j\neq i}}^N \lambda_{ji}=\delta_{rl}\lambda_{ii},\\
			\frac{\partial}{\partial y_q^{l,(k-1)}} n_i^{r,(k)}(\b y^*,\b y^*,\b y^*)&=\delta_{rl}\lambda_{iq}, \quad i\neq q,
		\end{aligned}
	\end{equation*} 
	which means that $\b D^*_{l,(k-1)}\b n^{r,(k)}=\delta_{rl}\bA$ for $l\neq m$. In total \eqref{eq:Phis_MPDeC} gives us
	\begin{equation*}
		\b D^*_{l,(k-1)}\bm \Phi_k^m=-\dt (\theta_{l,+}^m+\theta_{l,-}^m)\bA=-\dt\theta_l^m\bA.
	\end{equation*}
	
	Next, we investigate the case of $s=k-1>0$ and $l=m$. Using once again $\diag(\b v)\b w=\diag(\b w)\b v$  and \eqref{eq:prk}, we obtain
	\begin{equation*}
		\begin{aligned}
			\b D^*_{m,(k-1)}\b p^{r,(k)}&=\b D^*_{m,(k-1)}\left(\bA\diag(\b y^{m,(k)})\diag\left(\b y^{m,(k-1)}\right)^{-1} \b y^{r,(k-1)}\right)\\&=-(1-\delta_{rm})\bA.
		\end{aligned}
	\end{equation*}
	Furthermore, recalling \eqref{eq:nrk}, \ie
	\begin{equation*}
		n_i^{r,(k)}(\b y^{r,(k-1)},\b y^{m,(k-1)},\b y^{m,(k)})=	\frac{y^{m,(k)}_{i}}{y_{i}^{m,(k-1)}}\sum_{\substack{j=1\\j\neq i}}^N \lambda_{ij}  y_j^{r,(k-1)}
		- y_i^{r,(k-1)}  \sum_{\substack{j=1\\j\neq i}}^N \lambda_{ji}\frac{y^{m,(k)}_{j}}{y_{j}^{m,(k-1)}},
	\end{equation*}
	we also distinguish between $r=m$ and $r\neq m$. In the first case we observe $n_i^{m,(k)}=n_i^{m,(k)}(\b y^{m,(k-1)},\b y^{m,(k)})$ and 
	\begin{equation*}
		\begin{aligned}
			\frac{\partial}{\partial y_i^{m,(k-1)}} n_i^{m,(k)}(\b y^*,\b y^*)&=-\frac{1}{y_i^*} \sum_{\substack{j=1\\j\neq i}}^N \lambda_{ij}y_j^*-\sum_{\substack{j=1\\j\neq i}}^N \lambda_{ji}\overset{\eqref{eq:2aii}}{=}2\lambda_{ii},\\
			\frac{\partial}{\partial y_q^{m,(k-1)}} n_i^{m,(k)}(\b y^*,\b y^*)&=\lambda_{iq}+\lambda_{qi}\frac{y_i^*}{y_q^*}\overset{\eqref{eq:A+diagAdiag}}{=}(\bA+\diag(\b y^*)\bA^T(\diag(\b y^*))^{-1})_{iq},
		\end{aligned}
	\end{equation*}
	for $i\neq q$,	which means that $	\b D^*_{m,(k-1)}\b n^{m,(k)}=\bA +\diag(\b y^*)\bA^T(\diag(\b y^*))^{-1}$. Turning to the case $r\neq m$, we find
	\begin{equation*}
		\begin{aligned}
			\frac{\partial}{\partial y_i^{m,(k-1)}} n_i^{r,(k)}(\b y^*,\b y^*,\b y^*)&=-\frac{1}{y_i^*} \sum_{\substack{j=1\\j\neq i}}^N \lambda_{ij}y_j^*\overset{\eqref{eq:2aii}}{=}\lambda_{ii},\\
			\frac{\partial}{\partial y_q^{m,(k-1)}} n_i^{r,(k)}(\b y^*,\b y^*,\b y^*)&=\lambda_{qi}\frac{y_i^*}{y_q^*}\overset{\eqref{eq:A+diagAdiag}}{=}(\diag(\b y^*)\bA^T(\diag(\b y^*))^{-1})_{iq}, \quad i\neq q,
		\end{aligned}
	\end{equation*}
	resulting in  $	\b D^*_{m,(k-1)}\b n^{r,(k)}=\diag(\b y^*)\bA^T(\diag(\b y^*))^{-1}$ for $r\neq m$. Altogether, we thus end up with
	\begin{equation*}
		\b D^*_{m,(k-1)}\bm \Phi^m_k=\sum_{\substack{r=0}}^M\theta_{r,+}^m\dt\bA-\sum_{\substack{r=0}}^M\theta_{r,-}^m\dt(\diag(\b y^*)\bA^T(\diag(\b y^*))^{-1})-\theta_{m}^m\dt \bA.
	\end{equation*}
	Finally, we have to consider the case $s=k$ and $l=m$, i.\,e.\ we have to compute $\b D^*_{m,(k)}\bm \Phi^m_k$.  Using $\diag(\b v)\b w=\diag(\b w)\b v$  and \eqref{eq:prk} once again we see that 
	\[\b D^*_{m,(k)}\b p^{r,(k)}=\bA. \]
	Furthermore, we obtain
	\begin{equation*}
		\begin{aligned}
			\frac{\partial}{\partial y_i^{m,(k)}} n_i^{r,(k)}(\b y^*,\b y^*,\b y^*)&=\frac{1}{y_i^*} \sum_{\substack{j=1\\j\neq i}}^N \lambda_{ij}y_j^*\overset{\eqref{eq:2aii}}{=}-\lambda_{ii},\\
			\frac{\partial}{\partial y_q^{m,(k)}} n_i^{r,(k)}(\b y^*,\b y^*,\b y^*)&=-\lambda_{qi}\frac{y_i^*}{y_q^*}\overset{\eqref{eq:A+diagAdiag}}{=}-(\diag(\b y^*)\bA^T(\diag(\b y^*))^{-1})_{iq}, \quad i\neq q,
		\end{aligned}
	\end{equation*}
	resulting in
	\begin{equation*}
		\b D^*_{m,(k)}\bm \Phi^m_k=\b I-\sum_{\substack{r=0}}^M\theta_{r,+}^m\dt\bA+\sum_{\substack{r=0}}^M\theta_{r,-}^m\dt\diag(\b y^*)\bA^T(\diag(\b y^*))^{-1}.
	\end{equation*}
	With this, we have finally proven Theorem \ref{thm:mpdec}.
	
\end{proof}
Focusing on Gauss--Lobatto nodes, a higher-order quadrature rule is applied\footnote{The $L^2$ operator inside the DeC framework is based on a collocation method with Lobatto nodes (also known as the RK Lobatto III A method). }. Here, we  use $M= \lceil\tfrac{K}{2}\rceil$ subintervals and K=p corrections. 
Recall that we denoted the $p$th order MPDeC method by MPDeC$(p)$ and indicated \gl\! and \eq\! nodes by using MPDeCGL and MPDeCEQ, respectively.
Note that MPDeC(1) is equivalent to the modified Patankar--Euler scheme  and MPDeC(2) is equivalent to MPRK22(1) for both, \gl\! and \eq\! nodes.

Due to $\b y^{n+1}=\b y^{M,(K)}$, MPDeC conserves all linear invariants, if $\theta_r^M\geq 0$ for all $r=0,\dotsc,M$ since in this case the index function yields $\gamma(j,i,\theta_r^M)=j$ and \eqref{eq:explicit_dec_correction} can be written as 
\[\b y^{n+1}-\b y^n-\sum_{r=0}^M\theta_r^M\dt \bA \diag(\b y^{n+1})(\diag(\b y^{M,(K-1)})^{-1}\b y^{r,(K-1)}=\b 0,\]
which means that $\b n^T\b y^{n+1}=\b n^T\b y^n$ for all $\b n\in \ker(\bA^T)$. Indeed, for equispaced nodes, $\theta_r^M$ with $r=0,\dotsc,M$ are the weights of the closed Newton--Cotes formulas for integrals over $I=[0,1]$. Hence, a negative $\theta_r^M$ occurs for the first time at $M=7$, \ie with MPDeCEQ$(8)$. In this case, we also have to consider $\b n^{r,(K)}(\b y^{r,(K-1)},\b y^{M,(K-1)},\b y^{n+1})$ given in \eqref{eq:nrkpij}, resulting in
\begin{equation*}
	\begin{aligned}
		\b n^T\b n^{r,(K)}
		&=\sum_{i,j=1}^Nn_i\frac{y_i^{n+1}}{y_i^{M,(K-1)}}p_{ij}(\b y^{r,(K-1)})- \sum_{i,j=1}^Nn_j\frac{y_i^{n+1}}{y_i^{M,(K-1)}}p_{ij}(\b y^{r,(K-1)}),
	\end{aligned}
\end{equation*}
where we switched indices and used $d_{ij}=p_{ji}$. We observe that $\b n^T\b n^{r,(k)}$ does not need to vanish for $\b n\notin\Span(\bm 1)$, so that the preservation of all linear invariants can not be guaranteed anymore for arbitrary systems $\b y'=\bA\b y$ and MPDeC($p$) with equispaced nodes with $p\geq 8$. However, as the system matrix of \eqref{eq:PDS_Sys} satisfies additional properties, additional research is required to answer this question.

Moreover, in the case of Gauss--Lobatto nodes, the values $2\theta_r^M$ for $r=0,\dotsc,M$ equal the weights of the corresponding Gauss--Lobatto quadrature, which are always positive. This gives us the following result.
\begin{lem}
	The MPDeCGL methods conserve all linear invariants when applied to \eqref{eq:PDS_Sys}.
\end{lem}
\begin{rem}\label{rem:Mpdec}
	From Theorem \ref{thm:mpdec}, we see that the Jacobian in general depends on $\b y^*$, if there exist negative correction weights $\theta_r^m$, \ie not being conditional stable does not necessarily result in instability in this case. For equispaced or Gauss--Lobatto points negative correction weights already occur for $K>2$. Hence, to study the stability of MPDeC schemes applied to general linear systems, one needs to locate the eigenvalues of the Jacobian, which possibly depend on $\b y^*$ themselves. Such an analysis is outside the scope of this work, which is why we will focus on the following class of problems.
	
	If $\bA$ is normal, then  $\bA$ and $\bA^T$ share the same eigenvectors and the corresponding eigenvalues are the complex conjugate of each other. Since $\bm 1\in \ker(\bA^T)$ this means that even $\bm1 \in \ker(\bA)$. Hence, we may discuss the stability of $\b y^*=\bm 1$. Then, we find
	\[\sigma\left(r_1(\bA)+r_2\left(\diag(\b y^*)\bA^T(\diag(\b y^*))^{-1}\right)\right)=\{r_1(\lambda)+r_2(\bar\lambda)\mid \lambda\in \sigma(\bA)\} \]
	for any rational maps $r_1,r_2$, which means that the spectrum of the Jacobian of the map $\b g$ generating  the MPDeC iterates can be written only in terms of the eigenvalues of $\bA$. Using \eqref{eq:Dg_MPDeC}, \eqref{eq:D*y0Phikm} and \eqref{eq:D*l,(s)Phikm}, the stability function $R_p$ of MPDeC($p$) for normal matrices $\bA$ and $\b y^*=\bm 1$ can be computed recursively by
	\begin{equation}\label{eq:Rmpdec}
		\begin{aligned}
			R^{m,(1)}(z)=&\frac{1+(z+\bar{z})\sum_{r=0}^M\theta_{r,-}^m}{1-\left(z \sum_{r=0}^M \theta_{r,+}^m-\bar{z}\sum_{r=0}^M \theta_{r,-}^m\right)},\\
			R^{m,(\hat k)}(z)=&\frac{1+\theta_0^mz+z\displaystyle\sum_{\substack{r=1\\r\neq m}}^M\theta_r^mR^{r,(\hat k-1)}(z)}{1-\left(z \sum_{r=0}^M \theta_{r,+}^m-\bar{z}\sum_{r=0}^M \theta_{r,-}^m\right)}\\
			&-\frac{\left( z \sum_{r=0}^M \theta_{r,+}^m-\bar{z} \sum_{r=0}^M \theta_{r,-}^m-z\theta_m^m\right)R^{m,(\hat k-1)}(z)}{1-\left(z \sum_{r=0}^M \theta_{r,+}^m-\bar{z}\sum_{r=0}^M \theta_{r,-}^m\right)},\\
			R_p(z)=&R^{M,(K)}(z),
		\end{aligned}
	\end{equation}
	for $\hat k=2,\dotsc,K$ and $m=1,\dotsc,M$. Note that if $\bA$ is symmetric it is also normal and we obtain $\sigma(\bA)\tm \R$, so that one can further simplify \eqref{eq:Rmpdec}	using $\theta^m_{r,+}+\theta^m_{r,-}=\theta_r^m$ to receive
	\begin{equation}\label{eq:Rmpdecreduced}
		\begin{aligned}
			R^{m,(1)}(z)&=\frac{1+2z\sum_{r=0}^M\theta_{r,-}^m}{1-z \sum_{r=0}^M \abs{\theta_{r}^m}},\\
			R^{m,(\hat k)}(z)&=\frac{1+\theta_0^mz+z\displaystyle\sum_{\substack{r=1\\r\neq m}}^M\theta_r^mR^{r,(\hat k-1)}(z)-z\left(\sum_{r=0}^M \abs{\theta_{r}^m}-\theta_m^m\right)R^{m,(\hat k-1)}(z)}{1-z \sum_{r=0}^M \abs{\theta_{r}^m}},\\
			R_p(z)&=R^{M,(K)}(z).
		\end{aligned}
	\end{equation} 
	It is also worth mentioning that for the system matrix  
	\begin{equation}\label{eq:A2x2}
		\bA=\begin{pmatrix*}[r]
			-a & b\\a&-b
		\end{pmatrix*}
	\end{equation} with $a,b>0$, used in \cite{IKM2122,IssuesMPRK}, we find that $\ker(\bA)=\Span(\b y^*)$ with $\tfrac{y_1^*}{y_2^*}=\tfrac{b}{a}$, and thus
	\[\diag(\b y^*)\bA^T(\diag(\b y^*))^{-1}=\begin{pmatrix*}[r]
		-a & a\frac{b}{a}\\
		b\frac{a}{b}&-b
	\end{pmatrix*}=\bA,\]
	so that the stability function $R_p$ in this case is also given by \eqref{eq:Rmpdecreduced}. 
\end{rem}
Deferred Correction schemes are described by an iterative process which can be compared with classical RK schemes with more stages \cite{abgrall2022relaxation}. As MPDeC and DeC share the same amount of stages, we thus know that MPDeCEQ(3) contains 5 stages. Furthermore, MPDeCEQ(4) has already 10 stages inside, resulting in rational function with polynomial degree 10 in the numerator and denominator.
Using Gauss--Lobatto nodes decreases the number of stages. For an MPDeCGL(4) we would end up with seven stages. 

Similarly as we did for MPRK43($\alpha,\beta$) we can estimate the stability region for MPDeC schemes. It turns out that MPDeCEQ$(p)$ and MPDeCGL$(p)$ are unconditionally stable for $p=1,2$, as they coincide with MPE and MPRK22(1), respectively. For $p=3,\dotsc,8$ we collect the lower bounds $\theta_{\text{num}}$ for the opening angle $\theta$ of the stability domain for normal system matrices $\bA$ in Table~\ref{tab:mpdecstab}. Thereby we show that $\theta_{\text{num}}$ actually satisfies the error bound $\theta_\text{num}\leq\theta<\theta_\text{num}+\frac{\pi}{500}$ by adding $\frac{\pi}{500}$ to $\theta_{\text{num}}$ and demonstrating that the absolute value of the stability function then exceeds $1$ for some $r=\lvert z\rvert$,  see Figure~\ref{fig:mpdecstab}. 
\begin{table}[!h]
	\centering
	{\def\arraystretch{1.5}\tabcolsep=5pt	\begin{tabular}{|c|c|c|}\hline
			$p$ & MPDeCEQ($p$) & MPDeCGL($p$) \\ \hline 
			1 & $\pi$ & $\pi$\\
			2 & $\pi$ & $\pi$\\
			3& $\frac{498}{500}\pi$ & $ \frac{498}{500}\pi$\\
			4& $\frac{496}{500}\pi$ & $\frac{496}{500}\pi$ \\
			5&$\frac{495}{500}\pi$ & $\frac{495}{500}\pi$ \\
			6 & $\frac{495}{500}\pi$ &  $\frac{494}{500}\pi$\\
			7 & $\frac{494}{500}\pi$ &  $\frac{494}{500}\pi$\\
			8 & $\frac{495}{500}\pi$& $\frac{494}{500}\pi$ \\ 
			\hline
	\end{tabular}}
	\caption{Estimate $\theta_{\text{num}}$ from \eqref{eq:theta} of the opening angle $\theta$ of the stability domain of MPDeC($p$) for problems \eqref{eq:PDS_Sys} with a normal system matrix. We have $\theta_\text{num}\leq\theta<\theta_\text{num}+\frac{\pi}{500}$, see Figure~\ref{fig:mpdecstab}.}\label{tab:mpdecstab} 
\end{table}
\begin{figure}[!h]
	\centering
	\begin{subfigure}[t]{0.495\textwidth}
		\includegraphics[width=\textwidth]{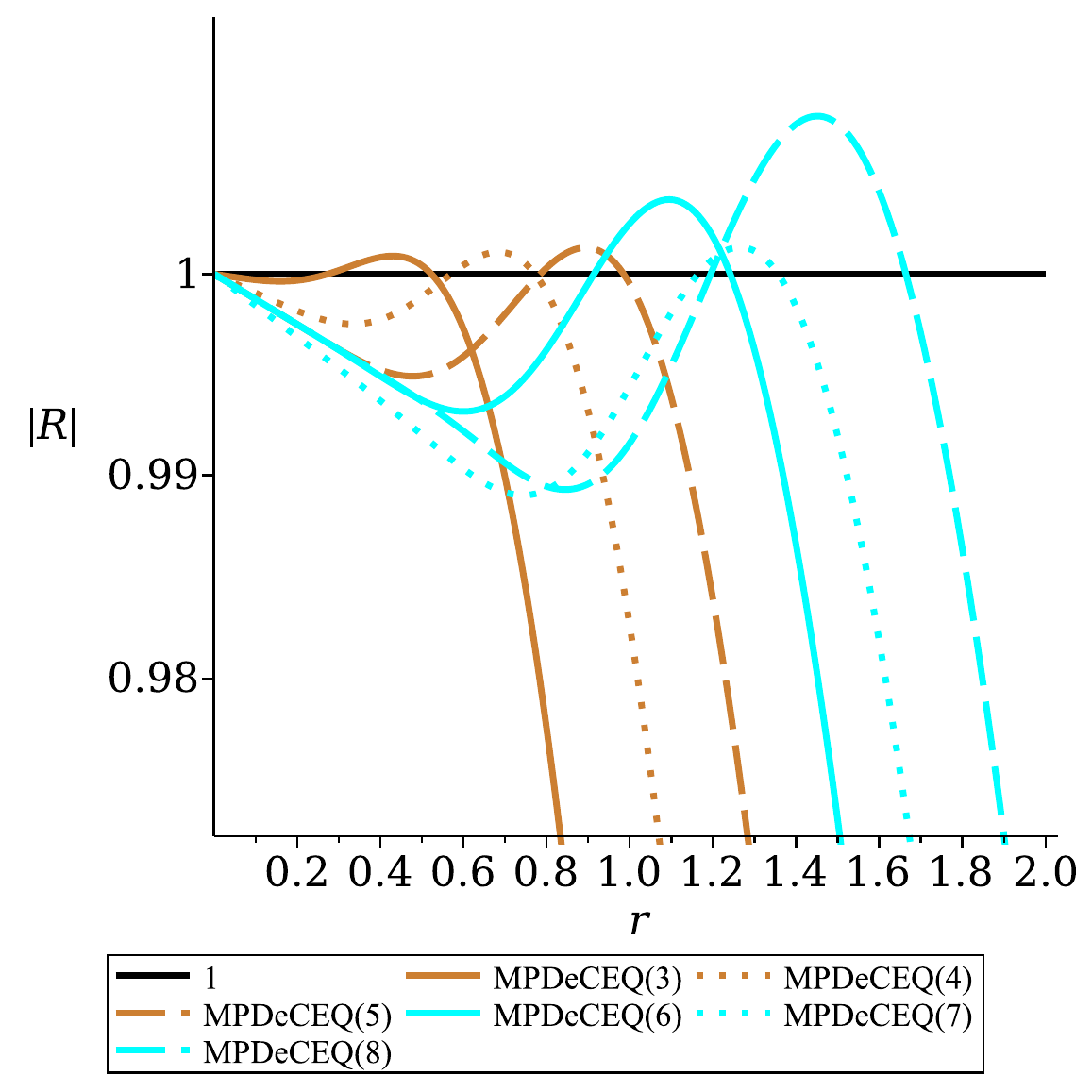}
		\subcaption{Equispaced nodes}
	\end{subfigure}
	\begin{subfigure}[t]{0.495\textwidth}%
		\includegraphics[width=\textwidth]{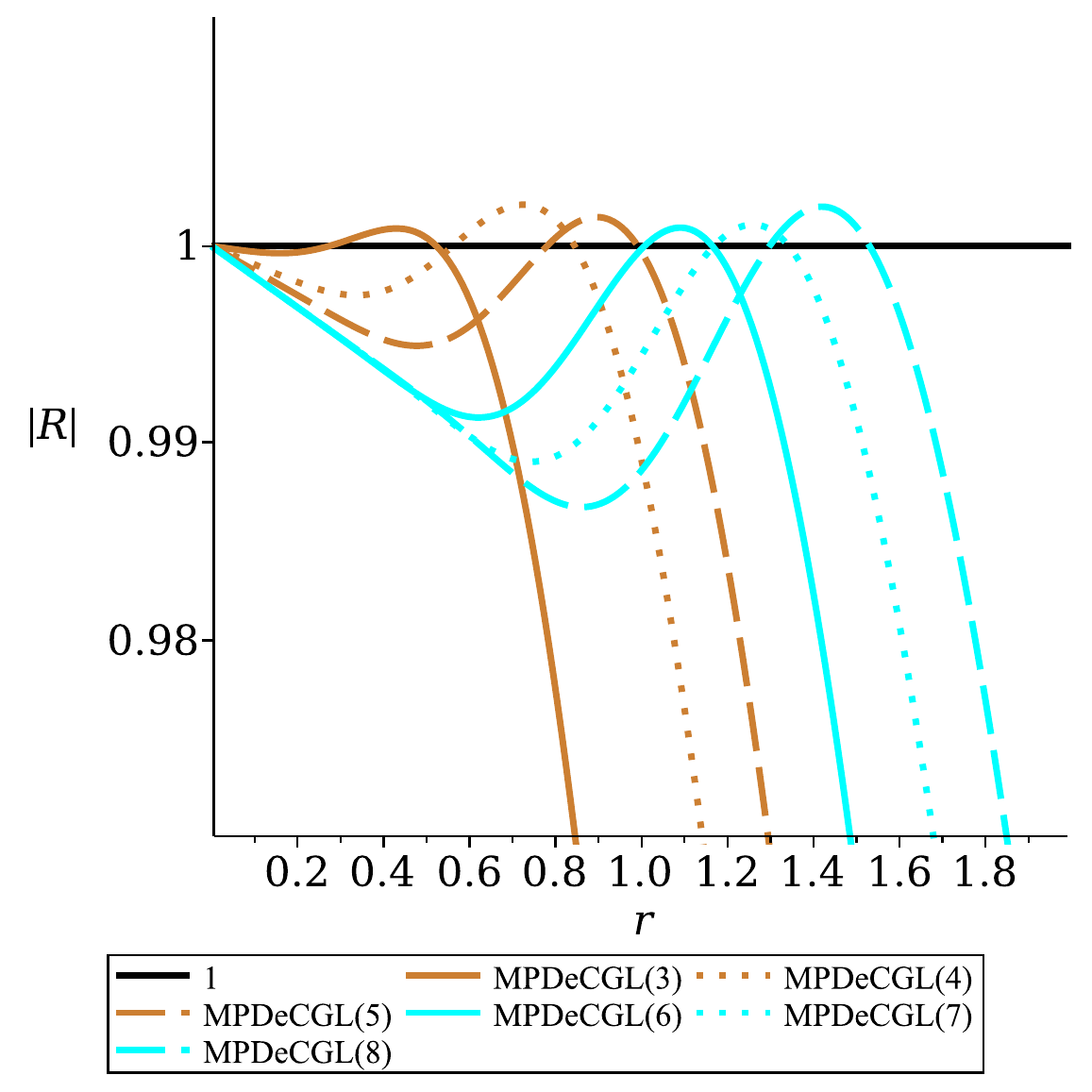}
		\subcaption{Gauss--Lobatto nodes}
	\end{subfigure}
	\caption{Plots of $\lvert R(re^{\ii \varphi})\rvert$ over $r$, where $\varphi=\pi-\frac{\theta_{\text{num}}+\frac{\pi}{500}}{2}$ corresponds to the opening angle $\theta_{\text{num}}+\frac{\pi}{500}$ with $\theta_{\text{num}}$ from Table~\ref{tab:mpdecstab}.}
	\label{fig:mpdecstab}
\end{figure}

To give a first insight in the stability properties of MPDeC methods of order higher than $8$, we analyze the reduced stability function \eqref{eq:Rmpdecreduced}. In both cases described in Remark \ref{rem:Mpdec}, the eigenvalues of $\bA$ leading to \eqref{eq:Rmpdecreduced} are real, which is why we
present the absolute value of the stability function over real $z$ in Figure~\ref{fig:stability}. 
To obtain a stable scheme, the absolute value of $R(z)$ has to be always smaller than one (the black line). In Figure~\ref{fig:GL}, we investigate MPDeC from order 4 to 14 using Gauss--Lobatto points. As can be recognized all MPDeC methods are stable using Gauss--Lobatto points. In Figure~\ref{fig:ESP}, the stability functions of MPDeC schemes from 4th to 14th order are depicted for equispaced nodes. Here, we recognize that MPDeCEQ(12) and MPDeCEQ(14) are unstable but MPDeCEQ(13) is stable.  However, this is not surprising since already for classical DeC methods using equidistant points has been problematic for high-order methods, cf.  \cite{dutt2000spectral,han2021dec,MPDeC, IssuesMPRK} and references therein. The reason for this is related with classical interpolation theory where it is known that equidistant points may lead to Runge's phenomenon. However, we would like to point out that our investigation also supports the numerical investigation in \cite{IssuesMPRK} where problems in MPDeCEQ have been recognized. 

\begin{figure}[tb]
	\centering 
	\begin{subfigure}[b]{0.495\textwidth}
		\includegraphics[width=\textwidth]{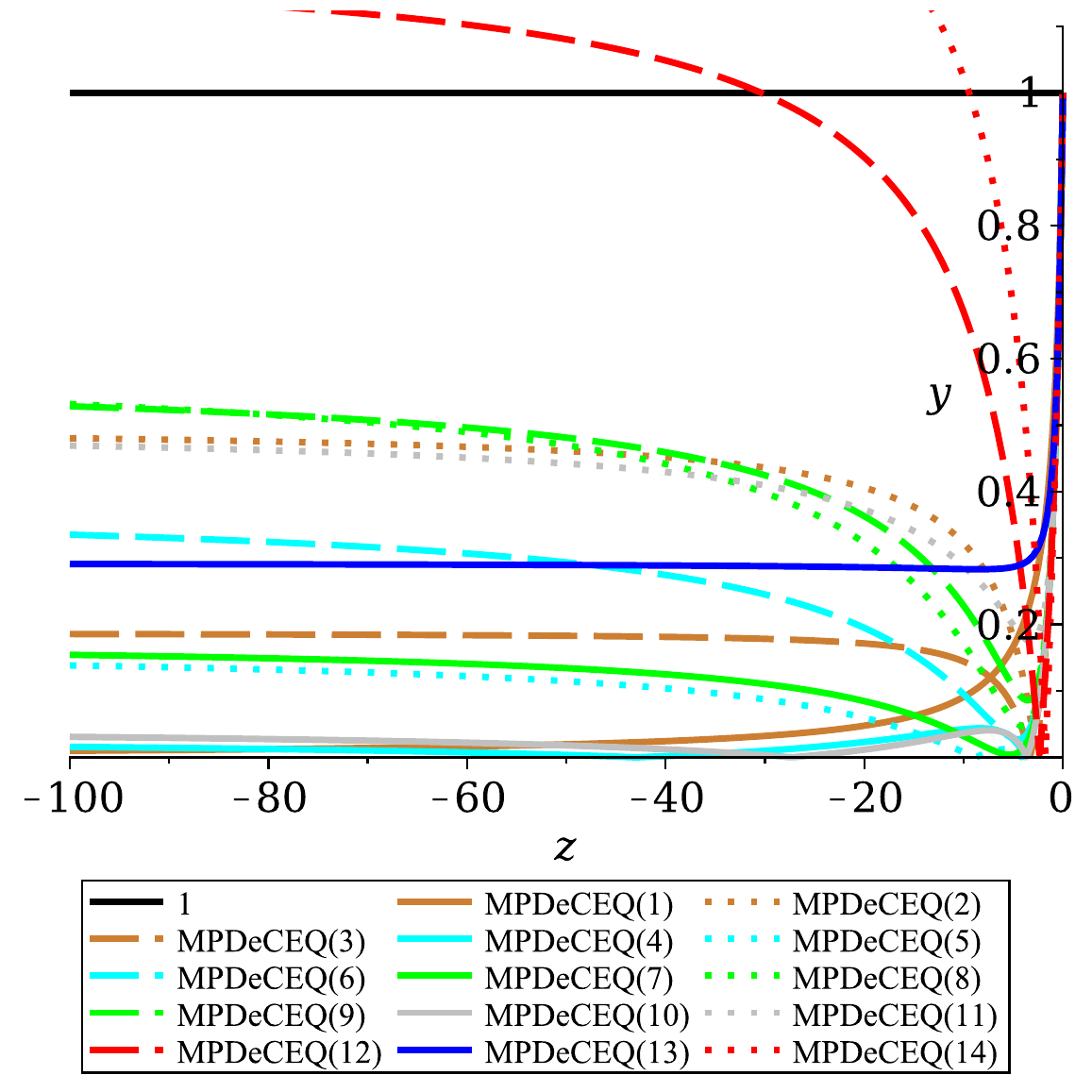} 
		\caption{Equispaced points}
		\label{fig:ESP}
	\end{subfigure}%
	\begin{subfigure}[b]{0.495\textwidth}
		\includegraphics[width=\textwidth]{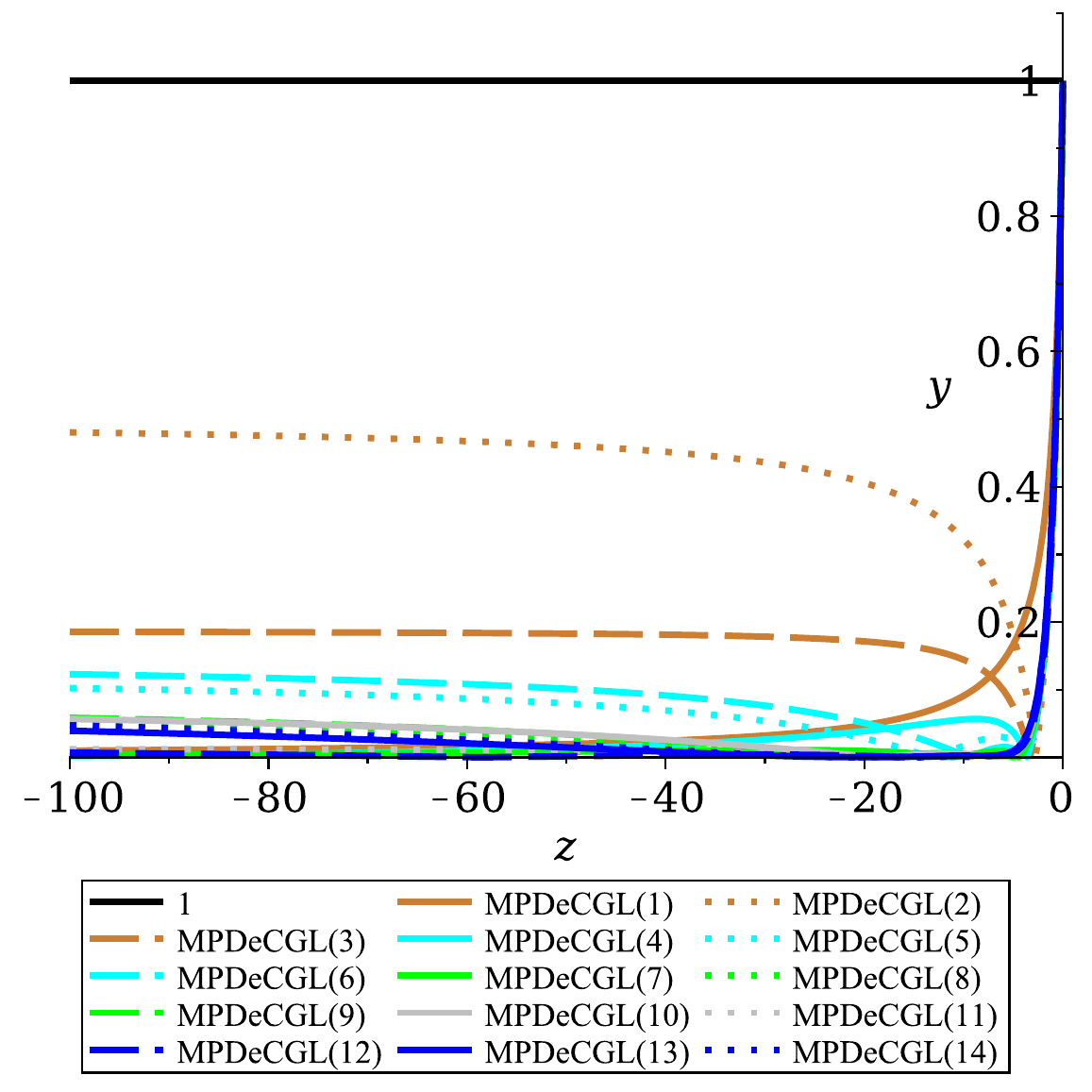} 
		\caption{Gauss--Lobatto points}
		\label{fig:GL}
	\end{subfigure}%
	\caption{Absolute value of the stability function over $z\leq 0$.}
	\label{fig:stability}
\end{figure}  
\newpage
\subsection{Geometric Conservative}\label{sec:stab_GeCo}

	We will start analyzing GeCo1 applied to a general positive linear test problem with stable steady states. Turning to GeCo2, we prove that already for the $2\times 2$ system \eqref{PDS_test}, \eqref{eq:IC} the stability domain of GeCo2 is bounded. Approaches for the analysis of GeCo2 for $N\times N$ systems will be discussed at the end of the respective subsection.
	
	\subsubsection{Stability of GeCo1}
	In this subsection, we investigate the stability properties of GeCo1, see \eqref{eq:GeCo1scheme}, 
	
	To that end, we first rewrite $\b y'=\bA\b y$ as a bio chemical system of the form
	\begin{equation}\label{eq:split_A}
		\b y'=	\bA\b y=\b f^{[P]}(\b y)-\b f^{[D]}(\b y)=\b S^+\b y-\b S^-\b y
	\end{equation}
	with $\b S^+,\b S^-\geq \b 0$. Since $\bA$ is a Metzler matrix, $\b S^-=(s^-_{ij})_{i,j=1,\dotsc,N}$ is a diagonal matrix and $f_i^{[D]}(\b y)=s_{ii}^-y_i$. Moreover, Remark \ref{rem:Aneg} states that at least one diagonal element of $\bA$ is negative, which results in 
	\begin{equation}
		\tr(\b S^-)>0.\label{eq:tr(S-)>0}
	\end{equation}
	With this in mind, let us recall the function $\varphi$ from \eqref{eq:phi}, that is
	\begin{equation*}
		\varphi(x)=\begin{cases}\frac{1-e^{-x}}{x},& x>0,\\
			1, &x=0
		\end{cases}
	\end{equation*} 
	and write the GeCo1 scheme \eqref{eq:GeCo1scheme} applied to \eqref{eq:split_A} as
	\begin{equation*}
		\begin{aligned}
			\mathbf{g}(\mathbf{y}^{n})=\mathbf{y}^{n+1} &=  \mathbf{y}^{n} + \dt \varphi\left(\dt\sum_{i=1}^N \dfrac{f_i^{[D]}(\b y)}{y_i^{n}} \right) \bA\mathbf{y}^{n}\\&=  \mathbf{y}^{n} + \dt\varphi(\dt\tr(\b S^-)) \bA\mathbf{y}^{n}.
	\end{aligned}\end{equation*}
	Due to \eqref{eq:tr(S-)>0}, the GeCo1 scheme can be rewritten as
	\begin{equation}
		\b y^{n+1}=\b g(\b y^n)=(\b I+\Phi(\dt)\bA)\b y^n,\quad \Phi(\dt)=\dt\varphi(\dt\tr(\b S^-))=\tfrac{1-e^{-\dt\tr(\b S^-)}}{\tr(\b S^-)}.\label{eq:ThmGeco1scheme}
	\end{equation}
	It is worth mentioning that this reasoning holds for all $k=\dim(\ker(\bA))\geq 0$. Also note that steady states of \eqref{eq:PDS_Sys} become fixed points of $\b g$, and that $\b g\in \mathcal C^\infty$ conserves all linear invariants, if there are any. Hence, we are in the position to apply Theorem~\ref{Thm_MPRK_stabil}. It is worth noting that the eigenvalues of the Jacobian of the GeCo1 map $\b g$ in general not only depend on $\dt\lambda$, but also on the trace of $\b S^-$.
	Nevertheless, we are able to prove that in the case of GeCo1, the remaining $N-k$ eigenvalues of $\b D\b g(\b y^*)$ lie inside the unit circle, resulting in the following theorem.
	\begin{thm}\label{Thm:GeCo1}
		If $k>0$ then	the steady state 
		$\b y^*$ of \eqref{eq:PDS_Sys}, \eqref{eq:IC} is a stable fixed point of GeCo1 for all $\Delta t>0$. Furthermore, there exists a $\delta >0$ such that $\Vert\b y^0-\b y^*\Vert<\delta$ implies the convergence of the iterates towards $\b y^*$ for all $\Delta t>0$.
		If $k=0$, then $\b y^*=\b 0$ is an asymptotically stable fixed point of GeCo1 for all $\dt>0$.
	\end{thm}
	\begin{proof}
		The Jacobian $\b D\b g(\b y^*)$ reads
		\begin{equation*}
			\b D\b g(\b y^*)=\b I+\Phi(\dt)\bA=\b I+\frac{1-e^{-\dt\tr(\b S^-)}}{\tr(\b S^-)}\bA
		\end{equation*} 
		and its eigenvalues are \[\mu=1+\Phi(\dt)\lambda\] with $\lambda\in \sigma(\bA)$.
		Hereby, we see that 
		in the case of $k>0$, any 
		$\b v\in \ker(\bA)\setminus\{\b 0\}$ is an eigenvector of the Jacobian $\b D\b g(\b y^*)$ with an associated eigenvalue of $1$.
		
		In order to investigate the location of the remaining $N-k$ eigenvalues of the Jacobian 
		for $k\geq 0$,
		we first numerate the distinct and nonzero eigenvalues of $\bA$ from \eqref{eq:PDS_Sys} by $\lambda_1,\dotsc,\lambda_m$.
		Now, the corresponding eigenvalues $\mu_i=1+\Phi(\dt)\lambda_i$ with $i=1,\dotsc,m$ lie inside the unit circle if and only if
		\[\lvert 1+\Phi(\dt)\lambda_i \rvert^2<1,\quad i=1,\dotsc,m,\]
		which can be written as
		\[(\re(\Phi(\dt)\lambda_i)+1)^2+\im (\Phi(\dt)\lambda_i)^2<1,\quad i=1,\dotsc,m,\]
		or equivalently,
		\[2\Phi(\dt)\re(\lambda_i)+\Phi(\dt)^2\lvert \lambda_i\rvert^2<0,\quad  i=1,\dotsc,m.\]
		Dividing by $\Phi(\dt)>0$ and exploiting $\sigma(\bA)\tm \Cminus$ gives
		\[\Phi(\dt)< -\frac{2\re(\lambda_i)}{\lvert \lambda_i\rvert^2}=\frac{2\lvert \re(\lambda_i)\rvert}{\lvert \lambda_i\rvert^2},\quad  i=1,\dotsc,m.\]
		Introducing 	
		\begin{equation}\label{eq:M_eig}
			M=\min_{i=1,\dotsc,m}\left\{2\tfrac{\lvert\re(\lambda_i)\rvert}{\lvert\lambda_i\rvert^2}\right\},
		\end{equation} we end up with the equivalent condition
		\[ \Phi(\dt)<M.\]
		Hence, after plugging in $\Phi(\dt)=\tfrac{1-e^{-\dt\tr(\b S^-)}}{\tr(\b S^-)}$,
		we multiply with its denominator $\tr(\b S^-)>0$, see \eqref{eq:tr(S-)>0}, to get
		\begin{equation*}
			\lvert \mu_i \rvert^2<1,\quad i=1,\dotsc,m
		\end{equation*}
		if and only if
		\begin{equation*}
			1-e^{-\dt\tr(\b S^-)}<M \tr(\b S^-).
		\end{equation*}
		Now, if $M \tr(\b S^-)\geq 1$, then \[M\tr(\b S^-)\geq 1>1-e^{-\dt\tr(\b S^-)}\] is true for all $\dt>0$, and hence, the remaining eigenvalues of $\b D\b g(\b y^*)$ associated with nonzero eigenvalues of $\bA$ lie inside the unit circle. We now aim to prove that 
		\[M \tr(\b S^-)\geq 1\]
		is indeed the case. 
		
		Due to \cite[Theorem 10, Corollary 11]{StabMetz} it holds that \[\sigma(\bA)\tm\mathcal B= \left\{ z\in \C \, \Big| \, \abs{z-r}\leq \abs r, r=\min_{j=1,\dotsc, N} \lambda_{jj}\right\}.\] Since $\bA$ is a proper Metzler matrix, see Remark \ref{rem:Aneg}, there exists an $l\in\{1,\dotsc,N\}$ such that
		\begin{equation}\label{eq:r<0}
			r=\min_{j=1,\dotsc,N}\lambda_{jj}=\lambda_{ll}<0.	
		\end{equation}
		Even more, we know $\re(\lambda)<0$ as well as $\arg(\lambda)\in(\tfrac\pi2,\tfrac32\pi)$ for all $0\neq\lambda\in\sigma(\bA)$.
		For any given $\lambda\in \sigma(\bA)\setminus\{0\}$, we define $\alpha=\pi-\arg(\lambda)\in(-\tfrac\pi2,\tfrac{\pi}{2})$, so that 
		\[ \cos(\alpha)=\frac{\abs{\re(\lambda)}}{\abs{\lambda}}\neq 0.\]
		Next, we choose $\theta<0$ satisfying
		\[\abs\lambda=\cos(\alpha)\abs\theta.\]
		A sketch of this geometry can be found in Figure \ref{Fig:proof}.
		\begin{figure}[!h]
			\centering
			\begin{scaletikzpicturetowidth}{0.4\textwidth}
				\begin{tikzpicture}[scale=\tikzscale]
					\tkzInit[xmin=-5,xmax=0,ymin=-3,ymax=3]
					\draw [-stealth](-7,0) -- (0.5,0) node[below]{$\re$};
					\draw [-stealth](0,-3.25) -- (0,3.25) node[right]{$\im$};
					\tkzDefPoint[label=below left:{$\theta$}](-4,-0){P}
					\tkzDefPoint[label=below right:{$2 r $}](-6.5,-0){B}
					\tkzDefPoint[label=below right:{$ r $}](-3.25,-0){X}
					\tkzDefPoint(0,0){Q}
					\tkzDefPoint(-2,0){D}
					\tkzDefPoint[label=above left:{$\lambda$}](-1,3^0.5){R}
					\tkzDrawSegments(P,Q Q,R R,P)
					\tkzDefPoint[label=below:{$\re(\lambda)$}](-1,0){N}

					\tkzMarkRightAngle[fill=lightgray](Q,N,R)
					\tkzMarkRightAngle[fill=lightgray](P,R,Q)
					\tkzDrawPoints[color=black](P,Q,R,N,B,X)
					\tkzDrawSegment(R,N)
					\pic [draw, -, "$\alpha$", angle eccentricity=0.65] {angle = R--Q--P};
					\tkzDrawCircles(X,B)
				\end{tikzpicture}
			\end{scaletikzpicturetowidth}
			\caption{Sketch of the geometric setup for $\arg(\lambda)\in(\tfrac\pi2,\pi)$.}\label{Fig:proof}
		\end{figure}
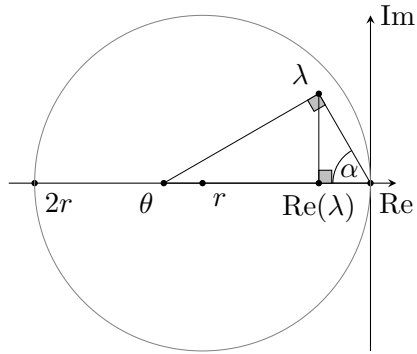
		With this, equation \eqref{eq:M_eig} becomes
		\[ M=\min_{i=1,\dotsc,m}\left\{2\tfrac{\lvert\re(\lambda_i)\rvert}{\lvert\lambda_i\rvert^2}\right\}=\min_{i=1,\dotsc,m}\left\{2\tfrac{\cos(\alpha_i)}{\lvert\lambda_i\rvert}\right\}=\min_{i=1,\dotsc,m}\left\{\tfrac{2}{\lvert \theta_i\rvert}\right\}.\]
		Moreover, with Thales's Theorem we can conclude that even $\theta_i\in \R^-$ is contained in $\mathcal B$, and thus, satisfies \[\abs{\theta_i}\leq 2\abs{\min_{j=1,\dotsc,N}\lambda_{jj}}.\]
		
		From \eqref{eq:r<0} it thus follows that
		\[M=\min_{i=1,\dotsc,m}\left\{\tfrac{2}{\lvert \theta_i\rvert}\right\}\geq\frac{2}{2\abs{\min_{j=1,\dotsc,N}\lambda_{jj}}}=\frac{1}{\abs{\lambda_{ll}}}. \]
		Additionally, setting $S=\{j\in\{1,\dotsc,N\}\mid \lambda_{jj}<0\}$ we find
		\[\tr(\b S^-)=-\sum_{j\in S}^N\lambda_{jj}=\sum_{j\in S}^N\lvert \lambda_{jj}\rvert \geq \abs{\lambda_{ll}},\]
		and thus
		\[M\tr(\b S^-)\geq\frac{1}{\abs{\lambda_{ll}}}\lvert \lambda_{ll}\rvert= 1,\]
		which finishes the proof as we have also proven that $\rho(\b D\b g(\b 0))<1$ in the case of $k=0$.
	\end{proof}
	With this theorem, a stability result for the GeCo1 scheme is provided for the first time. Having proved the unconditional stability of all fixed points of GeCo1 associated with steady states of the general $N\times N$ system of differential equations \eqref{eq:PDS_Sys}, we can conclude that GeCo1 mimics the stability behavior of the analytic solution close to a steady state solution for any chosen time step size $\dt>0$.	Whether or not this already suggests that the explicit GeCo1 scheme is even capable of solving stiff problems will be discussed in Section~\ref{subsec:GeCostiff}.
	
	\subsubsection{Stability of GeCo2}\label{subsec:StabGeCo}
	In this subsection we aim to prove that \eqref{eq:GeCo2scheme} applied to \eqref{PDS_test} can be described by a $\mathcal{C}^1$-map $\b g$ using Lemma \ref{Lem:diff} from the appendix, and to compute the spectrum of the corresponding Jacobian. To prove that the partial derivatives are even locally Lipschitz continuous, we use Lemma \ref{Lem:locallyLip} from the appendix. However, we will also prove that $\b g\notin \mathcal C^2$ for any neighborhood of $\b y^*$, extending the work \cite{gecostab}. This underlines the benefits discussed in Remark~\ref{rem:th2.9} on the stability theorems from \cite{IKM2122} and Theorem~\ref{Thm_MPRK_stabil}, published in \cite{izgin2022stability}.
	
	Let us investigate the GeCo2 scheme applied to \eqref{eq:IC}, \eqref{PDS_test} with \[\bA=\underbrace{\Vec{0 & bc\\a & 0}}_{=\mathbf S^+} - \underbrace{\Vec{ac & 0\\0 & b}}_{=\mathbf S^-} \] and $\mathbf r(\b y)=\b y$. This means that $\b f^{[D]}(\b y)=\mathbf S^-\mathbf r(\b y)=\vec{acy_1\\b y_2}$, and hence the GeCo2 scheme \eqref{eq:GeCo2scheme} reads
	\begin{equation}\label{eq:yGeco2}
		\begin{aligned}
			\displaystyle \mathbf{y}^{(2)}&=\mathbf{y}^{n}+\dt\varphi(\dt \tr(\b S^-)) \bA\mathbf{y}^{n}\\
			\displaystyle \mathbf{y}^{n+1} &= 
			\mathbf{y}^{n} + \dfrac{1}{2} \dt\varphi\left(\dt\left(\dfrac{w_1^+(\b y^n)}{y_1^n}+\dfrac{w_2^+(\b y^n)}{y_2^n}\right)\right)  \bA\left(\mathbf{y}^{n} +\mathbf{y}^{(2)}\right)\\
			&=\mathbf{y}^{n}\! +\! \dfrac{1}{2} \dt\varphi\!\left(\dt\!\left(\dfrac{w_1^+(\b y^n)}{y_1^n}\! +\!\dfrac{w_2^+(\b y^n)}{y_2^n}\right)\right)\!  \bA\left(2\mathbf{y}^{n}\! +\!\dt\varphi(\dt\tr(\b S^-)) \bA\mathbf{y}^{n}\right)\!,
		\end{aligned}
	\end{equation}
	where $	w_i^+(\b y^n)=\max(0,w_i(\b y^n))$ for $i=1,2$ and
	\begin{equation}
		\begin{aligned}\label{eq:w_MatrixForm}
			\mathbf w(\b y^n)&=2\varphi(\dt\tr(\b S^-))\bA\mathbf y^n - \bA\mathbf{y}^n-\bA\mathbf{y}^{(2)}\\&=\left(2 \varphi(\dt\tr(\b S^-))\bA-2\bA-\bA^2\dt\varphi(\dt\tr(\b S^-))\right)\b y^n.
		\end{aligned}
	\end{equation}
	
	We formulate a helpful lemma to understand some properties of $\b w$ and to express equation \eqref{eq:yGeco2} with $\b w$ rather than $w_1^+$ and $w_2^+$.
	\begin{lem}\label{Lem:w}
		The map $\b w$ from \eqref{eq:w_MatrixForm} with $\bA$ from \eqref{PDS_test} satisfies $w_1=-\tfrac1cw_2$, and we have
		\[w_1(\b y^n)\begin{cases}>0, &y^n_1>\tfrac{b}{a}y^n_2, \\
			=0, &y^n_1=\tfrac{b}{a}y^n_2,\\
			<0, & y^n_1<\tfrac{b}{a}y^n_2.\end{cases} \]
	\end{lem}
	\begin{proof}
		First note that $y^n_1=\tfrac{b}{a}y^n_2$ is equivalent to $\b y^n\in \ker(\bA)$, and thus \eqref{eq:w_MatrixForm} yields $\b w(\b y^n)=\b 0$.
		
		Next, we focus on finding conditions for $\b y^n$ so that $w_1(\b y^n)>0$. For this, it is worth mentioning that for every $\b y^n>\b 0$, there exists a unique $\b y^*\in \ker(\bA)\cap \R^2_{>0}$ satisfying  $\b n^T\b y^n=\b n^T\b y^*$ with $\b n=(1,c)^T$, see \cite[Lemma 2.8]{IKM2122}. Hence, since $\b y^*>\b 0$ and $\bby=(1,-\tfrac1c)^T$ are linearly independent, there exists a unique $s^n\in\R$ such that $\b y^n=\b y^*+s^n\bby$. Also note that $\bA\bby=\lambda\bby$ with $\lambda=-(ac+b)<0$ and 
		\begin{equation}\label{eq:sn}
			s^n\begin{cases}
				>0, &y^n_1>\tfrac{b}{a}y^n_2,\\
				=0, &y^n_1=\tfrac{b}{a}y^n_2,\\
				<0, &y^n_1<\tfrac{b}{a}y^n_2.
			\end{cases}
		\end{equation}
		Thus, the linearity of $\b w$ and \eqref{eq:phi} lead to
		\begin{equation}\label{eq:w} 
			\begin{aligned}
				\b w(\b y^n)&=\b w(\b y^*)+\b w(s^n\bby)=\left(2 \varphi(-\dt\lambda)\lambda-2\lambda-\lambda^2\dt\varphi(-\dt\lambda)\right)s^n\bby\\
				&=\frac{1}{\dt}\left(2(1-e^{\dt\lambda}) -2\dt\lambda-\dt\lambda(1-e^{\dt\lambda})\right)s^n\bby\\
				&=\frac{1}{\dt}\left(2 -3\dt\lambda+ e^{\dt\lambda}(\dt\lambda-2)\right)s^n\bby.
			\end{aligned}
		\end{equation}
		Furthermore, introducing the function
		\[p(z)=-3z+2-e^z(2-z),\]
		we can rewrite \eqref{eq:w} to get
		\begin{equation}\label{eq:w_final}
			\b w(\b y^n)=\frac{1}{\dt}p(\dt\lambda)s^n\bby.
		\end{equation}
		Now, the first derivative of $p$ satisfies
		\[p'(z)=-(3+e^z(1-z))<0\]
		for all $z\leq0.$ Hence, the function $p$ is strictly decreasing for $z\leq 0$ and satisfies $p(0)=0$ proving that $p(\lambda\dt)>0$ for all $\dt>0$. 
		Therefore, with \eqref{eq:sn} it follows that $w_1(\b y^n)>0$ if $y_1^n>\frac{b}{a}y_2^n$. Similarly,  $w_1(\b y^n)<0$ holds if $y_1^n<\frac{b}{a}y_2^n$. Finally, note that  \eqref{eq:w_final} implies $w_1(\b y^n)=-\tfrac1cw_2(\b y^n)$.
	\end{proof} 
	
	As a consequence of this lemma we simplify \eqref{eq:yGeco2} by introducing \[H\colon \R^2_{>0}\to \R_{>0}\] with 
	\begin{equation*}
		\begin{aligned}
			H(\b x)&=\dt \varphi\left(\dt\left(\dfrac{w_1^+(\b x)}{x_1}+\dfrac{w_2^+(\b x)}{x_2}\right)\right)  =\begin{cases}\widetilde H_1(\b x),& x_1>\frac{b}{a} x_2,\\
				\dt,& x_1= \frac{b}{a} x_2,\\
				\widetilde H_2(\b x),& x_1< \frac{b}{a} x_2,
			\end{cases}\\
			\widetilde H_i(\b x)&=\dfrac{1-e^{-\dt \tfrac{w_i(\b x)}{x_i}}}{\tfrac{w_i(\b x)}{x_i}},\quad i=1,2
		\end{aligned}
	\end{equation*}
	and point out that $H$ is continuous, since $\varphi$ from \eqref{eq:phi} is in $ \mathcal C^2$ and $\b w\in \mathcal C^\infty$. As a result of Lemma \ref{Lem:w}, we even know that \[\widetilde H_i\in \mathcal C^\infty(\R^2_{>0}\setminus\ker(\bA))\] for $i=1,2$.
	
	The map $\b g$ defining the iterates of the GeCo2 scheme when applied to \eqref{eq:IC}, \eqref{PDS_test} is given by \eqref{eq:yGeco2} and can be written as
	\begin{equation*}
		\begin{aligned}
			\b g(\b x)&=\b x+ \dfrac{1}{2} H(\b x)  \bA\left(2\b x+\dt\varphi(\Delta \tr(\b S^-))\bA \b x\right).
		\end{aligned}
	\end{equation*}
	Introducing $\b G(\b x)=\bA\b x H(\b x)$ we obtain
	\begin{equation}\label{eq:gGeCo2}
		\b g(\b x)=\b x+\b G(\b x)+\frac12\dt\varphi(\dt\tr(\b S^-))\bA\b G(\b x).
	\end{equation}
	The following theorem uses this representation of $\b g$ to analyze the stability properties of GeCo2.\newpage
	\begin{thm}\label{Thm:geco2}
		Let $\b g$, given by \eqref{eq:gGeCo2}, be the generating function of the GeCo2 iterates $\b y^n$ when applied to \eqref{PDS_test}, \eqref{eq:IC}. Further, let $\b y^*>\b 0$ be a steady state solution of \eqref{PDS_test}. 
		\begin{enumerate}
			\item\label{itThm:geco} The map $\b g\in\mathcal C^1(\mathcal D)$ has Lipschitz continuous derivatives on a sufficiently small neighborhood $\mathcal D$ of $\b y^*$.
			Moreover, the stability function reads
			\begin{equation}\label{eq:stabfun_GeCo2}
				R(z)=1+z+\frac12 z^2\varphi(\dt\tr(\b S^-)).
			\end{equation} If $\lvert R(-\dt(ac+b))\rvert <1$, then $\b y^*$ is stable and there exists a $\delta>0$ such that $\vec{1\\c}^T\b y^0=\vec{1\\c}^T\b y^*$ and $\norm{\b y^0-\b y^*}<\delta$ imply $\lim_{n\to \infty}\b y^n=\b y^*$.  If $\lvert R(-\dt(ac+b))\rvert >1$, then $\b y^*$ is an unstable fixed point of GeCo2.
			\item\label{itThm:geco_c2} There holds $\b g\notin \mathcal C^2$ in any neighborhood of $\b y^*$.
		\end{enumerate}
	\end{thm}
	\begin{proof}
		\begin{enumerate}
			\item	We demonstrate that all assumptions of Theorem~\ref{Thm_MPRK_stabil} and Theorem~\ref{Thm:_Asym_und_Instabil} are fulfilled.
			
			From part \ref{it:alemdiff} of Lemma \ref{Lem:diff} from the appendix with $\mathbf \Phi(\b x)=\bA\b x$ and $\Psi(\b x)=H(\b x)$, it follows that the partial derivatives of $\b G(\b x)=\bA\b x H(\b x)$ on $\ker(\bA)$ exist and that $\b D\b G(\b x_0)=\Psi(\b x_0)\bA=\dt\bA$ holds for $i=1,2$ and all $\b x_0\in \ker(\bA)$.
			As a result of \eqref{eq:gGeCo2} we obtain
			\begin{equation*}
				\b D\b g(\b y^*)=\b I+\dt\bA+\frac{1}{2}(\dt)^2\varphi(\dt\tr(\b S^-))\bA^2
			\end{equation*}
			and the eigenvalues are given by $1$ and $R(-\dt(ac+b))$, where
			\begin{equation*}
				R(z)=1+z+\frac12 z^2\varphi(-\dt\tr(\b S^-)).
			\end{equation*}
			In total, we can write
			\begin{equation}\label{eq:DGGeCo2}
				\b D\b G(\b x)=\b B(\b x)+\b C(\b x)
			\end{equation}
			with
			\[\b B(\b x)=\bA H(\b x)\qta \b C(\b x)=\bA\b x\cdot\begin{cases}\nabla \widetilde H_1(\b x),& x_1>\frac{b}{a}x_2,\\
				\b 0^T ,& x_1=\frac{b}{a}x_2,\\
				\nabla \widetilde H_2(\b x),& x_1<\frac{b}{a}x_2. \end{cases} \] 
			Note that, if each entry of $\b B=(b_{ij})_{i,j=1,2}$ and $\b C=(c_{ij})_{i,j=1,2}$ satisfies the assumptions of Lemma \ref{Lem:locallyLip} from the appendix, we can conclude that $\b G\in \mathcal C^1(\mathcal D)$ in a sufficiently small neighborhood $\mathcal D$ of $\b y^*$ and that the first derivatives are Lipschitz continuous on $\mathcal{D}$. As a direct consequence of \eqref{eq:gGeCo2}, the same would then hold true for $\b g$.
			
			Now we show that the entries $b_{ij}$ and $c_{ij}$ of of the matrices $\b B$ and $\b C$ satisfy the assumptions of Lemma~\ref{Lem:locallyLip} from the appendix, that is 
			\begin{enumerate}
				\item $b_{ij}$ and  $c_{ij}$ are continuous on $\R^2_{>0}$,
				\item  $b_{ij}$ and  $c_{ij}$ are constant on $\ker(\bA)$,
				\item  $b_{ij}$ and  $c_{ij}$ are in $\mathcal C^1$ on $\R^2_{>0}\setminus\ker(\bA)$ and
				\item $\lim_{\b x\to\b x_0}\nabla b_{ij}(\b x)$ and $\lim_{\b x\to\b x_0}\nabla c_{ij}(\b x)$ exist for all $\b x_0\in \ker(\bA)\cap \R^2_{>0}$
			\end{enumerate}
			for $i,j\in\{1,2\}$. First, note that $\b B$ and $\b C$ are constant on $\ker(\bA)$, and due to $\widetilde H_k\in \mathcal C^2$ for $k=1,2$,  we find that each entry of the two matrices is continuously differentiable on $\R^2_{>0}\setminus\ker(\bA)$. Even more, since $H$ is continuous we know that $b_{ij}\in \mathcal C(\R^2_{>0})$ for  $i,j\in\{1,2\}$. 
			
			We want to point out that if $\lim_{\b x\to\b x_0}\nabla \widetilde H_k(\b x)$ exists, this proves the continuity of $c_{ij}$ as well as that $\lim_{\b x\to\b x_0}\nabla b_{ij}(\b x)$ exists for all $i,j\in \{1,2\}$.
			Furthermore, for $\b x\notin \ker(\bA)$ we find
			\begin{align*}
				\nabla c_{ij}(\b x)&= \nabla(\bA\b x\nabla \widetilde H_k(\b x))_{ij}=\nabla((\bA\b x)_i\tfrac{\partial}{\partial x_j}\widetilde H_k(\b x))\\
				&=(\bA\b e_i)^T \tfrac{\partial}{\partial x_j}\widetilde H_k(\b x)+(\bA\b x)_i \nabla (\tfrac{\partial}{\partial x_j}\widetilde H_k(\b x)).
			\end{align*}
			Thus, $\lim_{\b x\to\b x_0}\nabla c_{ij}(\b x)$ exists, if both limits, $\lim_{\b x\to\b x_0}\nabla \widetilde H_k(\b x)$ as well as $\lim_{\b x\to\b x_0}\nabla (\tfrac{\partial}{\partial x_j}\widetilde H_k(\b x))$ exist for $i,j,k\in\{1,2\}$. 
			To see that both limits exist for $\b x_0\in \ker(\bA)\cap \R^2_{>0}$, we introduce $\Phi(z)=\frac{1-e^{-\dt z}}{z}$, so that \[ \widetilde H_k(\b x)=\Phi(\tfrac{w_k(\b x)}{x_k}).\] 
			Hence, we have  $\Phi\in \mathcal C^2(\R\setminus\{0\})$ and
			\begin{equation}
				\begin{aligned}\label{eq:NablaH_i(x)}
					\nabla\widetilde H_k(\b x)=&\Phi'(\tfrac{w_k(\b x)}{x_k})\left(\frac{\nabla w_k(\b x)}{x_k}+w_k(\b x)\nabla\left(\frac{1}{x_k}\right) \right),\\
					\nabla\tfrac{\partial\widetilde H_k(\b x)}{\partial x_j} =&\Phi''(\tfrac{w_k(\b x)}{x_k})\left(\frac{\nabla w_k(\b x)}{x_k}+w_k(\b x)\nabla \left(\frac{1}{x_k}\right) \right)\!\left(\frac{\tfrac{\partial w_k(\b x)}{\partial x_j} }{x_k}+w_k(\b x)\tfrac{\partial\left(\frac{1}{x_k}\right)}{\partial x_j} \right)\\
					&+\Phi'(\tfrac{w_k(\b x)}{x_k})\nabla\left(\frac{\tfrac{\partial}{\partial x_j} w_k(\b x)}{x_k}+w_k(\b x)\tfrac{\partial}{\partial x_j}\left(\frac{1}{x_k}\right) \right).
				\end{aligned}
			\end{equation}
			As $\lim_{\b x\to\b x_0}\tfrac{w_k(\b x)}{x_k}=0$ for $\b x_0\in \ker(\bA)\cap \R^2_{>0}$, see \eqref{eq:w_MatrixForm}, we are interested in the limits of the first two derivatives of $\Phi$ at $z=0$. By l'Hospital's rule, a straightforward calculation yields
			\begin{equation}
				\lim_{z\to 0}\Phi'(z)=-\frac{(\dt)^2}{2} \qta \lim_{z\to 0}\Phi''(z)=\frac{(\dt)^3}{3}.\label{eq:limitPhi'}
			\end{equation}
			In addition, due to \eqref{eq:w_MatrixForm}, we know that $\nabla w_k$ is a constant function for $k=1,2$, which means that
			\begin{align*}
				\nabla\left(\frac{\frac{\partial}{\partial x_j} w_k(\b x)}{x_k}+w_k(\b x)\frac{\partial}{\partial x_j}\left(\frac{1}{x_k}\right) \right)=&\frac{\partial}{\partial x_j} w_k(\b x)\nabla\left(\frac{1}{x_k}\right)\\&+\nabla w_k(\b x)\frac{\partial}{\partial x_j}\left(\frac{1}{x_k}\right)\\&+ w_k(\b x)\nabla\left(\frac{\partial}{\partial x_j}\left(\frac{1}{x_k}\right)\right).
			\end{align*}
			It thus follows from \eqref{eq:NablaH_i(x)} and \eqref{eq:limitPhi'} that  $\lim_{\b x\to\b x_0}\nabla \widetilde H_k(\b x)$ as well as $\lim_{\b x\to\b x_0}\nabla (\tfrac{\partial}{\partial x_j}\widetilde H_k(\b x))$ exist for all $i,j,k\in\{1,2\}$ and each $\b x_0\in \ker(\bA)\cap \R^2_{>0}$.  This finishes this part of the proof.
			\item First note that \eqref{eq:gGeCo2} implies that
			\[\b g(\b x)=\b x+\b B\b G(\b x) \]
			with 
			\[\b B=\b I+\frac12\dt\varphi(\dt(ac+b))\bA.\]
			Hence, $\b g\in \mathcal C^2$ if and only if $\b B\b G\in \mathcal C^2$.
			In general, \eqref{eq:DGGeCo2} can be expressed by
			\begin{equation*}
				\b D\b G(\b x)=\bA H(\b x)+\bA\b x\begin{cases}\nabla \widetilde H_1(\b x),& x_1>\frac{b}{a}x_2,\\
					\b c(\b x) ,& x_1=\frac{b}{a}x_2,\\
					\nabla \widetilde H_2(\b x),& x_1<\frac{b}{a}x_2 \end{cases}
			\end{equation*}
			with an arbitrary function $\b c\colon \R^2_{>0}\to\R^2_{>0}$.  Our strategy is to use Lemma~\ref{Lem:diff} to conclude that the first partial derivative of the first column of \[\b B\bA\b x\begin{cases}\nabla \widetilde H_1(\b x),& x_1>\frac{b}{a}x_2,\\
				\b c(\b x) ,& x_1=\frac{b}{a}x_2,\\
				\nabla \widetilde H_2(\b x),& x_1<\frac{b}{a}x_2 \end{cases}\] does not exist at $\b y^*$. To that end, we prove that $\b T(\b x)=\b x$ satisfies \[\partial_j\b T(\b x)=\b e_j\notin\ker(\b B\bA),\] and that  \begin{equation}\label{neq:limNablaHi}
				\lim_{\b x\to\b y^*}\nabla \widetilde H_1(\b x)\neq\lim_{\b x\to\b y^*} \nabla \widetilde H_2(\b x),
			\end{equation}
			which shows the claim independently of $\b c(\b y^*)$. Indeed, the matrix $\b B$ satisfies $\b B\b y^*=\b y^*$ and $\b B\bby=\mu \bby$, where $\bby=(1,-1)^T$ and 
			\[\mu=1-\frac12\varphi(\dt(ac+b))\dt(ac+b). \]
			Using $z=-\dt(ac+b)<0$ we have
			\[\mu=1+\frac12(1-e^z)=\frac32-\frac12e^z >0,\]
			which means that $\b B$ is invertible, and hence $\ker(\b B\bA)=\ker(\bA)$. Therefore, we obtain $\partial_j\b T(\b x)\notin \ker(\bA)=\ker(\b B\bA)$.
			
			For proving \eqref{neq:limNablaHi} we use \eqref{eq:limitPhi'}, \eqref{eq:NablaH_i(x)} and $\b w(\b y^*)=\b 0$ to find
			\begin{equation*}
				\lim_{\b x\to\b y^*}\nabla \widetilde H_i(\b x)=-\frac{(\dt)^2}{2}\frac{\nabla w_i(\b y^*)}{y^*_i}.
			\end{equation*}
			Let us now suppose \eqref{neq:limNablaHi} is not satisfied and recall that $w_1=-\tfrac1c w_2$ holds true because of Lemma~\ref{Lem:w}. Hence, we would have
			\begin{equation}\label{eq:nablaw2/y2=pics.}
				\frac{\nabla w_2(\b y^*)}{y^*_2}=-\frac{\nabla w_2(\b y^*)}{cy^*_1}.
			\end{equation}
			From \eqref{eq:w_MatrixForm} it follows that $\nabla w_2$ is constant. We observe that $\nabla w_2(\b x)\neq \b 0$ for all $\b x$ as otherwise even $\b D \b w(\b x)=\b 0$ for all $\b x$, which contradicts Lemma~\ref{Lem:w} as \eqref{eq:w_MatrixForm} would imply that $\b w(\b y^n)=\b 0$ for all $\b y^n$. Hence, without loss of generality, assume $\partial_1w_2(\b x)\neq 0$. Then \eqref{eq:nablaw2/y2=pics.} implies
			\[ \partial_1w_2(\b x)(cy^*_1+y^*_2)=0,\]
			which is not true since $\partial_1w_2(\b x)\neq 0$ and $cy^*_1+y^*_2>0.$ Hence, \eqref{neq:limNablaHi} is true, so that Lemma~\ref{Lem:diff} implies $\b B\b G\notin \mathcal C^2$ in any neighborhood of $\b y^*$, and thus, the same holds for $\b g$.
		\end{enumerate}

	\end{proof}
	Note that part \ref{itThm:geco_c2} of Theorem~\ref{Thm:geco2} means that the assumptions of  \cite[Theorem 2.9]{IKM2122} are not fulfilled, while those of the generalization, Theorem~\ref{Thm_MPRK_stabil} are satisfied.
	\begin{rem}\label{rem:geco2}
		A numerical calculation shows that the stability function $R$ from \eqref{eq:stabfun_GeCo2} with $\dt\tr(\b S^-)=-z$ satisfies $\abs{R(z)}<1$ for $z\in (z^*,0]$ with $-3.9924\leq z^*\leq -3.9923$.  Hence, the stability region of GeCo2 when applied to \eqref{eq:IC}, \eqref{PDS_test} is almost twice as big as the one of the underlying Heun scheme which is $(-2,0]$. 
	\end{rem}
	To investigate $N\times N$ systems one needs to generalize Lemma~\ref{Lem:w} and Lemma~\ref{Lem:locallyLip} from the appendix, which is outside the scope of the present work.

\subsection{Generalized BBKS}\label{sec:stab_gBBKS}

When it comes to the analysis of gBBKS schemes, we face similar obstacles as for GeCo2. The aim of this work is to present results from \cite{gecostab} giving a first insight into the stability properties of these schemes. As done for GeCo2 we will discuss at the end of this section an ansatz to generalize the following analysis.
\subsubsection{Stability of first order gBBKS Schemes}\label{subsec:StabBBKS}
When applied to the system of differential equations \eqref{PDS_test}, \ie $\b y'=\bA\b y$ with $\bA=\vec{
	-a &\hphantom{-}bc\\ \hphantom{-}ac &-b}$, the first order gBBKS schemes \eqref{eq:gBBKS1Intro} are given by 
\begin{equation}\label{eq:gBBKS11}
	\b	y^{n+1}=\b y^n + \dt \bA\b y^n\Bigg(\prod_{m\in M^{n}} \frac{y^{n+1}_m}{\sigma^{n}_m}\Bigg)^{\ns r^{n}},\quad i=1,2,
\end{equation}
where
\begin{equation*}
	M^{n}=\{ m\in \{1,2\}\mid  (\bA\b y^n)_m<0\}.
\end{equation*}

In this section we investigate the stability properties of gBBKS schemes by first proving that the assumptions of Theorem \ref{Thm_MPRK_stabil} are met.
The existence and uniqueness of a function $\b g$ generating the iterates from \eqref{eq:gBBKS11}, i.\,e.\ $\b y^{n+1}=\b g(\b y^n)$, is already proven in \cite{gBBKS}. Thereby, $\b g$ is given by the unique solution to some equation 
\[\b F(\b x,\b g(\b x))=\b 0, \]
where $\b F\colon \R^2_{>0}\times  \R^2_{>0}\to \R^2$ with $(\b x,\b y)\mapsto\b F(\b x,\b y)$.  In the following we denote by 
\begin{equation*}
	\begin{aligned}
		\b D_\b x\b F(\b x,\b y)&=\frac{\partial \b F}{\partial \b x}(\b x,\b y),\\
		\b D_\b y\b F(\b x,\b y)&=\frac{\partial \b F}{\partial \b y}(\b x,\b y)
	\end{aligned}
\end{equation*}
the Jacobians of $\b F$ with respect to $\b x$ and $\b y$, respectively.

An intuitive way of proving $\b g\in \mathcal C^1(\mathcal D)$, where $\mathcal D$ is a neighborhood of a fixed point $\b y^*$ of $\b g$, is to use the implicit function theorem. Unfortunately, we will see in the following that in our case $\b F$ is not differentiable on $\mathcal D\times \mathcal D$. Since the existence and uniqueness of the map $\b g$ is already known here, the differentiability of $\b g$ can be obtained by weaker assumptions on $\b F$ as the next theorem states.
\begin{thm}[{\cite[Theorem 11.1]{LS14}}]\label{Thm:gDiff}
	Let $D\tm \R^2$ be open and $\b g\colon D\to D$ be continuous in $\b x_0$. Furthermore, let $\b F\colon D\times D\to \R^2$ with $(\b x,\b y)\mapsto\b F(\b x,\b y)$ be differentiable in $(\b x_0,\b g(\b x_0))^T$ and $\b D_\b y\b F(\b x_0,\b g(\b x_0))$ be invertible. Suppose that $\b F(\b x,\b g(\b x))=\b 0$ for all $\b x\in D$, then also $\b g$ is differentiable in $\b x_0$ and 
	\begin{equation*}
		\b D\b g(\b x_0)=-(\b D_\b y\b F(\b x_0,\b g(\b x_0)))^{-1}\b D_\b x\b F(\b x_0,\b g(\b x_0)).
	\end{equation*} 
\end{thm}

%
%

Before we formulate the stability theorem for gBBKS1, we introduce some assumptions on the exponent $r^n$ as well as $\sigma_m^n$ from \eqref{eq:gBBKS11}.
In particular, $r^n>0$ and $\sigma_m^n>0$ may depend on $\b y^n$ and hence will be interpreted as functions $r^n=r(\b y^n)$ and $\sigma_m^n=\sigma_m(\b y^n)$. For the analysis of the gBBKS1 schemes we do not further specify the expressions for the functions $r$ or $\sigma_m$. Instead, we assume some reasonable properties such as that $r$ and $\sigma_m$ are positive for all $\dt\geq0$. Furthermore, we require $\sigma_m(\b v)=v_m$ whenever $\b v\in\ker(\bA)\cap \R^2_{>0}$ which is in agreement with the literature \cite{gBBKS,BBKS2007,BRBM2008}. To guarantee the regularity of the map generating the iterates $\b y^n$, we also assume that $r,\sigma_1$ and $\sigma_2$ are in $\mathcal C^2$. In total, we prove the following theorem.

\begin{thm}\label{Thm:gBBKS1}
	Let  $\b y^*>\b 0$ be a steady state solution of \eqref{PDS_test}, and assume $\sigma_1,\sigma_2,r\in \mathcal{C}^2(\R^2_{>0},\R_{>0})$. Further, let $\mathcal D$ be a sufficiently small neighborhood of $\b y^*$ and suppose that $\bm \sigma(\b v)=\b v$ for all $\b v\in C=\ker(\bA)\cap \mathcal D$. Then the map $\b g$ generating the iterates of the gBBKS1 family, implicitly given by \eqref{eq:gBBKS11}, satisfies $\b g(\b v)=\b v$ for all steady states $\b v\in  C$ and the following statements hold.
	\begin{enumerate}
		\item\label{it:aThmgbbks1} The map $\b g$ satisfies  $\b g\in \mathcal{C}^1(\mathcal D)$ and $\b D\b g(\b y^*)=\b I+\dt\bA$.
		\item\label{it:cThmgbbks1}  The first derivatives of $\b g$ are bounded and Lipschitz continuous on $\mathcal D$.
		\item\label{it:bThmgbbks1} The map $\b g$ does not belong to $\mathcal C^2$ for any open neighborhood of $\b y^*$, if $\dt\neq(ac+b)^{-1}$.
	\end{enumerate}
\end{thm}
\begin{proof}
	Before we start the proof of \ref{it:aThmgbbks1}, we make some preparatory considerations.
	
	Since $(\bA\b y^n)_1=c(-ay_1^n+by_2^n)$ and $c(\bA\b y^n)_2=-(\bA\b y^n)_1$ we find
	\begin{equation*}
		M^n=\begin{cases}
			\{1\}, & y_1^n>\frac{b}{a}y^n_2,\\
			\emptyset, & y_1^n=\frac{b}{a}y^n_2,\\
			\{2\}, & y_1^n<\frac{b}{a}y^n_2.
		\end{cases}
	\end{equation*}
	Hence, when applied to \eqref{PDS_test}, \eqref{eq:IC} the scheme \eqref{eq:gBBKS11} turns into
	\begin{equation}\label{eq:gBBKS11_Test_Prob}
		\b y^{n+1}=\b y^n + \dt \bA \b y^n\begin{cases}
			\left(\frac{y^{n+1}_1}{\sigma^{n}_1}\right)^{r^n}, & y_1^n>\frac{b}{a}y^n_2,\\
			1, & y_1^n=\frac{b}{a}y^n_2,\\
			\left(\frac{y^{n+1}_2}{\sigma^{n}_2}\right)^{r^n}, & y_1^n<\frac{b}{a}y^n_2,
		\end{cases}
	\end{equation}
	where $(\frac{b}{a}y_2^n,y_2^n)^T\in \ker(\bA)$ is a steady state solution of \eqref{PDS_test}.
	
	Recall that the map $\b g$ generates the iterates $\b y^n$, that is $\b y^{n+1}=\b g(\b y^n)$. Hence, inserting $\b y^n=\b v\in C$ into equation \eqref{eq:gBBKS11_Test_Prob} yields $\b y^{n+1}=\b g(\b v)$ on the left and $\b v$ on the right, and thus $\b g(\b v)=\b v$. Furthermore, we introduce the function $\b F$ defined by
	\begin{equation}\label{eq:F(x).gBBKS1}
		\begin{aligned}
			\b F&\colon \R^2_{>0}\times \R^2_{>0}\to \R^2,\\ \b F(\b x,\b y)&=\b y-\b x-\dt\bA\b x H(\b x,\b y),\\
			H(\b x,\b y)&=\begin{cases}
				\widetilde H_1(\b x,\b y), & x_1>\frac{b}{a}x_2,\\
				1, & x_1=\frac{b}{a}x_2,\\
				\widetilde H_2(\b x,\b y), & x_1<\frac{b}{a}x_2,
			\end{cases} \\
			\widetilde H_i(\b x,\b y)&=\left(\frac{y_i}{\sigma_i(\b x)}\right)^{r(\b x)}, \quad i=1,2,
		\end{aligned}
	\end{equation}
	which satisfies $\b F(\b x,\b g(\b x))=\b 0$ for all $\b x>\b 0$.

	\begin{enumerate}
		\item We first show that $\b F$ is not differentiable on $\mathcal D\times \mathcal D$. For this, we choose $\b x_0\in C$ as well as $\b y_0>\b 0$ with $\frac{(\b y_0)_1}{(\b x_0)_1}\neq\frac{(\b y_0)_2}{(\b x_0)_2}$ and define $\Psi(\b x)=H(\b x,\b y_0)$. As a result of $\bm \sigma(\b x_0)=\b x_0$ and $r,\bm \sigma\in \mathcal C$ it follows that \[\lim_{h\searrow 0}\Psi(\b x_0+h\b e_1)=\lim_{h\searrow 0} \widetilde{H}_1(\b x_0+h\b e_1,\b y_0)=\left( \frac{(\b y_0)_1}{\sigma_1(\b x_0)}\right)^{r(\b x_0)}=\left( \frac{(\b y_0)_1}{(\b x_0)_1}\right)^{r(\b x_0)}.\]
		Analogously, we obtain
		\[\lim_{h\nearrow 0}\Psi(\b x_0+h\b e_1)=\lim_{h\nearrow 0} \widetilde{H}_2(\b x_0+h\b e_1,\b y_0)=\left( \frac{(\b y_0)_2}{\sigma_2(\b x_0)}\right)^{r(\b x_0)}=\left( \frac{(\b y_0)_2}{(\b x_0)_2}\right)^{r(\b x_0)},\]
		which  shows that $\Psi(\b x_0+h\b e_1)$ possesses several accumulation points as $h\to 0$, and hence, part \ref{it:blemdiff} of Lemma \ref{Lem:diff} from the appendix with $\mathbf \Phi(\b x)=\bA\b x$ implies that the $1$st partial derivative of $\b F$ does not exist.
		
		As mentioned above, this means that we can not apply the implicit function theorem to $\b F$ on $\mathcal D\times \mathcal D$ in order to prove that $\b g\in \mathcal C^1(\mathcal D)$. Nevertheless, $\b F$ is differentiable in $(\b x,\b y)\in E= \mathcal D\setminus\ker(\bA)\times \mathcal D$, since in this case we have
		\begin{equation}\label{eq:F(x,y):i=1,2}
			\b F(\b x,\b y)=\b y-\b x-\dt\bA\b x\left(\frac{y_i}{\sigma_i(\b x)}\right)^{r(\b x)}, \quad i=\begin{cases}
				1,& x_1>\frac{b}{a}x_2,\\
				2,& x_1<\frac{b}{a}x_2
			\end{cases} 
		\end{equation}
		with $\sigma_1,\sigma_2,r\in \mathcal{C}^2(\R^2_{>0},\R_{>0})$.
		In order to show that $\b g\in \mathcal C^1(\mathcal D\setminus \ker(\bA))$, we first show that the inverse of $\b D_\b y\b F(\b x,\b g(\b x))$ exists for all $\b x\in \mathcal D\setminus \ker(\bA)$. It is straightforward to verify that 
		\begin{equation*}
			\b D_\b y\b F(\b x,\b g(\b x))=\b I-\dt\bA \b x\nabla_\b y\widetilde H_i(\b x,\b g(\b x))=\b I-\dt\bA \b x\b e_i^T\frac{r(\b x)}{\sigma_i(\b x)}\left(\frac{g_i(\b x)}{\sigma_i(\b x)}\right)^{r(\b x)-1}
		\end{equation*}
		holds for $\b x\notin C$.
		Introducing the vectors
		\begin{equation}
			\b v^{(i)}(\b x)=\dt \bA\b x \frac{r(\b x)}{\sigma_i(\b x)}\left(\frac{g_i(\b x)}{\sigma_i(\b x)}\right)^{r(\b x)-1}=\dt \bA\b x \frac{r(\b x)}{g_i(\b x)}\left(\frac{g_i(\b x)}{\sigma_i(\b x)}\right)^{r(\b x)}\label{eq:v^(i)}
		\end{equation}
		for $i$ from \eqref{eq:F(x,y):i=1,2}, we can write the Jacobian in the compact form
		\begin{equation}\label{eq:DyFtilde=I-ve_i}
			\b D_\b y\b F(\b x,\b g(\b x))=\b I- \b v^{(i)}(\b x)\b e_i^T.
		\end{equation}
		Note that due to \eqref{eq:DyFtilde=I-ve_i},  the Jacobian of $\b F$ with respect to $\b y$ is a triangular matrix, depending on $i$ from \eqref{eq:F(x,y):i=1,2}. Nevertheless, in either case  we find
		\begin{equation}\label{eq:det()=1-v}
			\det(\b D_\b y\b F(\b x,\b g(\b x)))=1-v_i^{(i)}(\b x).
		\end{equation}
		Now, we know that $(\bA\b x)_i< 0$ for $i$ form \eqref{eq:F(x,y):i=1,2} by construction of the gBBKS schemes, which in particular means that
		\begin{equation}\label{eq:v_ineq1}
			v^{(i)}_i(\b x)\neq 1.
		\end{equation} As a result of \eqref{eq:det()=1-v}, \eqref{eq:v_ineq1} the inverse of $\b D_\b y\b F(\b x,\b g(\b x))$ exists.
		
		Considering a zero $(\b x_0,\b g(\b x_0))\in E$ of $\b F$, the implicit function theorem thus provides the existence of a unique $\mathcal C^1$-map $\widetilde{\b g}$ satisfying $\b F(\b x,\widetilde{\b g}(\b x))=\b 0$ in a sufficiently small neighborhood of $(\b x_0,\b g(\b x_0))$. Since $\b g$ and $\widetilde{\b g}$ are unique, we find $\b g=\widetilde{\b g}$, and since $\b x_0$ was arbitrary, we have shown that $\b g\in \mathcal C^1$ on $\mathcal D\setminus \ker(\bA)$, and in particular
		\begin{equation}\label{eq:Dg(x)}
			\b D\b g(\b x)=-(\b D_\b y\b F(\b x,\b g(\b x)))^{-1}\b D_\b x\b F(\b x,\b g(\b x))
		\end{equation}
		for $\b x\in\mathcal D\setminus \ker(\bA)$. It thus remains to show that $\b g\colon D\to D$ is also differentiable in any $\b x\in \ker(\bA)\cap \mathcal D=C$ and that the first derivatives are continuous in any $\b x\in C$. 
		
		To prove the differentiability of $\b g$ in any $\b x\in C$ we make use of Theorem~\ref{Thm:gDiff}, and hence we have to prove the following.
		\begin{enumerate}[label=\arabic*.]
			\item The map $\b g$ is continuous in any $\b x\in C$.
			\item The map $\b F$ is differentiable in $(\b x,\b g(\b x))$ for all $\b x\in C$.
			\item The Jacobian $\b D_\b y\b F(\b x,\b g(\b x))$ with respect to $\b y$ is invertible for all $\b x\in C$.
		\end{enumerate}
		If we have shown these properties, then Theorem \ref{Thm:gDiff} together with the considerations above implies that \eqref{eq:Dg(x)} even holds for all $\b x\in \mathcal D$. 		
		
		We first prove that $\b g$ is continuous on $C$. Since gBBKS schemes are positive and conserve all linear invariants, we find from \eqref{PDS_test} that
		\begin{equation}\label{eq:gnorm}
			\min\{1,c\}\Vert \b g(\b x)\Vert_1\leq  g_1(\b x)+cg_2(\b x)=x_1+cx_2\leq \max\{1,c\} \Vert \b x\Vert_1.
		\end{equation}
		Now, $\Vert\b x\Vert_1$ is bounded on a sufficiently small neighborhood $\mathcal D$ of $\b y^*$ as we can make sure that the closure of $\mathcal D$ is contained in the domain of $\b g$. And since norms on $\R^2$ are equivalent, we even find from \eqref{eq:gnorm} that $\Vert\b g\Vert$ is bounded on $C$. As a result, $H(\cdot,\b g(\cdot))$ is bounded on $\mathcal D$ since the reciprocal of $\bm \sigma\in \mathcal C^2$ as well as $r\in \mathcal C^2$ are bounded on a sufficiently small $\mathcal D$. It thus follows that $\bA\b x H(\b x,\b g(\b x))$ tends to $\b 0$ as $\b x\to \b y^*$. From \eqref{eq:F(x).gBBKS1} with $\b y=\b g(\b x)$ we therefore obtain \[\lim_{\b x\to\b y^*}\b g(\b x)=\b y^*=\b g(\b y^*), \]
		which means that $\b g\colon D\to D$ is continuous in all $\b x\in C$.
		
		Next, we show that $\b F$ is differentiable in $(\b x,\b g(\b x))$ for all $\b x\in C$.
		For this consider an $\b x_0\in C$ and set $\b y_0=\b g(\b x_0)$.
		Note that $\Psi=H(\cdot,\b y_0)$ is continuous in $\b x_0\in C$ with $\Psi(\b x_0)=1$ since $\b g(\b x_0)=\bm\sigma(\b x_0)=\b x_0$.
		
		In this case, part \ref{it:alemdiff} of Lemma \ref{Lem:diff} from the appendix with $\mathbf \Phi(\b x)=\bA\b x$ yields
		\begin{equation}\label{eq:DxF=-I-dtA}
			\b D_\b x\b F(\b x_0,\b g(\b x_0))=-\b I-\dt\bA.
		\end{equation}
		Furthermore, as $\bA\b x\widetilde H_i(\b x,\b y)=\b 0$ for all $\b x\in C$ and $\b y\in\R^2_{>0}$, it follows immediately that 
		\begin{equation}\label{eq:DyF=I}
			\b D_\b y\b F(\b x_0,\b g(\b x_0))=\b I,
		\end{equation}
		which shows that $\b F$ is partially differentiable in $(\b x_0,\b g(\b x_0))$. To prove that $\b F$ is differentiable in $(\b x_0,\b g(\b x_0))$, we show that the partial derivatives are continuous in $(\b x_0,\b g(\b x_0))$.
		Therefore, we consider the case $\b x\notin C$ and differentiate $\b F$ from \eqref{eq:F(x).gBBKS1} with respect to $\b x$ and $\b y$. We have
		
		\begin{equation}\label{eq:DxFtilde(x)}
			\b D_\b x\b F(\b x,\b g(\b x))=-\b I-\dt\left(\bA\widetilde H_i(\b x,\b g(\b x))+\bA\b x \nabla_\b x\widetilde H_i(\b x,\b g(\b x))\right),
		\end{equation}
		where the gradient denotes a row vector and
		\begin{equation*}
			i=\begin{cases}
				1, &x_1>\frac{b}{a}x_2, \\
				2, &x_1<\frac{b}{a}x_2.
			\end{cases}
		\end{equation*}
		Now, since $\b g,\bm \sigma>\b 0$ we can write $\widetilde H_i(\b x,\b g(\b x))=e^{r(\b x)\ln\left(\frac{g_i(\b x)}{\sigma_i(\b x)}\right)}$, from which it follows that 
		\begin{equation}\label{eq:nabla(g/sigma)}
			\nabla_\b x\widetilde H_i(\b x,\b g(\b x))=\widetilde H_i(\b x,\b g(\b x))\left(\nabla_\b x r(\b x)\ln\left(\frac{g_i(\b x)}{\sigma_i(\b x)}\right)-r(\b x)\frac{\nabla_\b x \sigma_i(\b x)}{\sigma_i(\b x)}\right)
		\end{equation}
		since
		\begin{equation*}
			\nabla_\b x \ln\left(\frac{y_i}{\sigma_i(\b x)}\right)=
			\nabla_\b x \ln(y_i)-\nabla_\b x\ln(\sigma_i(\b x))=-\frac{\nabla_\b x \sigma_i(\b x)}{´\sigma_i(\b x)}.
		\end{equation*}
		Plugging \eqref{eq:nabla(g/sigma)} into \eqref{eq:DxFtilde(x)}, we find
		\begin{equation}\label{eq:DxFtildegbbks1}
			\begin{aligned}
				\b D_\b x\b F(\b x,\b g(\b x))=-\b I-\dt\widetilde H_i(\b x,\b g(\b x))\Biggl(&\bA+\bA\b x\Biggl(\nabla_\b x r(\b x)\ln\left(\frac{g_i(\b x)}{\sigma_i(\b x)}\right)\\
				&-r(\b x)\frac{\nabla_\b x \sigma_i(\b x)}{\sigma_i(\b x)}\Biggr) \Biggr).
			\end{aligned}
		\end{equation}
		Furthermore, $\bA\b x_0=\b 0$ for $\b x_0\in C$ together with $\sigma_1,\sigma_2,r\in \mathcal C^2$ as well as equation \eqref{eq:DxFtildegbbks1} yield
		\begin{equation}\label{eq:limDxF=-I-dtA}
			\lim_{\b x\to\b x_0}\b D_\b x\b F(\b x,\b g(\b x))=-\b I-\dt\lim_{\b x\to\b x_0}\widetilde H_i(\b x,\b g(\b x))\bA=-\b I-\dt\bA.
		\end{equation}
		Moreover, due to \eqref{eq:DyFtilde=I-ve_i} and since $\b v^{(i)}$ is continuous with $\b v^{(i)}(\b x_0)=\b 0$ for $\b x_0\in C$, we find
		\begin{equation}\label{eq:limDyF=I}
			\lim_{\b x\to \b x_0}	\b D_\b y\b F(\b x,\b g(\b x))=\b I-\b v^{(i)}(\b x_0)\b e_i^T=\b I.
		\end{equation}
		As a result of \eqref{eq:DxF=-I-dtA}, \eqref{eq:limDxF=-I-dtA} and \eqref{eq:DyF=I}, \eqref{eq:limDyF=I}, we thus know that all partial first derivatives of $\b F$ are continuous in $(\b x_0,\b g(\b x_0))$ for all $\b x_0\in C$, which implies that $\b F$ is differentiable in $(\b x_0,\b g(\b x_0))$ for all $\b x_0\in C$.
		
		Finally, due to \eqref{eq:DyF=I} we know that $\b D_\b y\b F(\b x_0,\b g(\b x_0))$ is invertible for all $\b x_0\in C$.
		
		Altogether, all requirements of Theorem \ref{Thm:gDiff} are fulfilled, which implies that $\b g$ is differentiable on $C$ and that 
		\begin{equation}\label{eq:Dg(x)gbbks1}
			\b D\b g(\b x_0)=-(\b D_\b y\b F(\b x_0,\b g(\b x_0)))^{-1}\b D_\b x\b F(\b x_0,\b g(\b x_0))ö
		\end{equation} 
		holds for all $\b x_0\in C$. Moreover, all entries of the inverse of $\b D_\b y\b F(\b x,\b g(\b x))$ are continuous functions of $\b x$, which proves that $\b g\in \mathcal C^1(\mathcal D)$. Finally, \eqref{eq:DxF=-I-dtA} and \eqref{eq:DyF=I} yield \[\b D\b g(\b y^*)=\b I+\dt\bA.\]

		%
		\item
		In this part, we use the equations \eqref{eq:DxFtildegbbks1} and 	\begin{equation}\label{eq:DyFtildeInversegBBKS1}
			(\b D_\b y\b F(\b x,\b g(\b x)))^{-1}=\begin{cases}
				\frac{1}{1-v_1^{(1)}(\b x)}\begin{pmatrix*}[r]1 & 0\\ v_2^{(1)}(\b x) & 1-v_1^{(1)}(\b x)\end{pmatrix*}, & x_1>\tfrac{b}{a}x_2,\\
				\frac{1}{1-v_2^{(2)}(\b x)}\begin{pmatrix*}[r]1- v_2^{(2)}(\b x) &v_1^{(2)}(\b x)\\  0& 1\end{pmatrix*}, & x_1<\tfrac{b}{a}x_2\end{cases}
		\end{equation} to show that the first derivatives of $\b g$ are Lipschitz continuous on a sufficiently small neighborhood $\mathcal D$ of $\b y^*$. For this, we make use of the fact that the set of bounded Lipschitz continuous functions is closed under summation, multiplication and composition. Hence, all we need to prove is that each entry in the matrices \eqref{eq:DxFtildegbbks1} and \eqref{eq:DyFtildeInversegBBKS1} is bounded and Lipschitz continuous on $\mathcal D$, and to use the fact that the natural logarithm and each exponential function are locally Lipschitz continuous. 
		
		To bound the corresponding functions, we choose $\mathcal D$ in such a way that $g_i,\sigma_i$ and $1-v_i^{(i)}$ have an upper bound $C_i>0$ and lower bound $c_i>0$. This is possible by choosing $\overline{\mathcal D}\tm D$ since these functions are continuous at $\b y^*$ and satisfy $\b g(\b y^*)=\bm \sigma(\b y^*)=\b y^*>\b 0$ as well as $1-v_i^{(i)}(\b y^*)=1$. As a result, even the first two derivatives of $\bm \sigma$ and $r$ are bounded on $\mathcal{D}$. This way, we can compute the Lipschitz constants of $\bm \sigma$, its first derivatives and its reciprocal by using the mean value theorem, see \cite[Remark 8.12 (b)]{AE08} for the details.
		Analogously, $\b g$ as well as $\frac{1}{g_i}$ are bounded Lipschitz continuous functions for $i=1,2$ as their first derivatives are bounded on $\mathcal{D}$. By this reasoning, it is straightforward to verify that each matrix entry in \eqref{eq:DxFtildegbbks1} and \eqref{eq:DyFtildeInversegBBKS1} is a bounded Lipschitz continuous function.
		
		\item 
		Assume that $\b g\in\mathcal C^2$ for some appropriate neighborhood of $\b y^*$. Introducing \[
		d_{jk}(\b x)=\left(\b D_\b x \b F(\b x,\b g(\b x)\right)_{jk},\] equations \eqref{eq:DyFtildeInversegBBKS1} and \eqref{eq:Dg(x)gbbks1} 
		yield
		\begin{equation}\label{eq:partial_1g_2(x)}
			\partial_1g_2(\b x)=-\begin{cases}  \frac{v_2^{(1)}(\b x)}{1-v_1^{(1)}(\b x)}d_{11}(\b x)+d_{21}(\b x)&, x_1\geq \tfrac{b}{a}x_2 \\
				\frac{1}{1-v_2^{(2)}(\b x)}d_{21}(\b x)&,x_1\leq\tfrac{b}{a}x_2
			\end{cases}.
		\end{equation}
		Our strategy is to compute  $\partial_2\partial_1g_2(\b y^*)$ and derive $1=\dt(ca+b)$ from it. Using $\b v^{(i)}(\b y^*)=\b 0$ we get from \eqref{eq:partial_1g_2(x)}
		\begin{equation*}
			-\partial_2\partial_1g_2(\b y^*)=\partial_2v_2^{(1)}(\b y^*)d_{11}(\b y^*)+\partial_2d_{21}(\b y^*)
		\end{equation*}
		as well as
		\begin{equation*}
			-\partial_2\partial_1g_2(\b y^*)=\partial_2d_{21}(\b y^*)+\partial_2v_2^{(2)}(\b y^*)d_{21}(\b y^*).
		\end{equation*}
		As a result, we obtain
		\begin{equation}
			\partial_2v_2^{(1)}(\b y^*)d_{11}(\b y^*)=\partial_2v_2^{(2)}(\b y^*)d_{21}(\b y^*).\label{eq:proofb}
		\end{equation}
		Using \eqref{eq:v^(i)}, we find that
		\begin{equation*}
			\partial_jv_k^{(i)}(\b y^*)=\dt \lambda_{kj}\frac{r(\b y^*)}{y_i^*},
		\end{equation*}
		and from \eqref{eq:DxF=-I-dtA}, we know that $d_{jk}(\b y^*)=-(\b I+\dt\bA)_{jk}$, so that \eqref{eq:proofb} reads
		\begin{equation*}
			-\dt \lambda_{22}\frac{r(\b y^*)}{y_1^*}(\b I+\dt\bA)_{11}=-\dt \lambda_{22}\frac{r(\b y^*)}{y_2^*}(\b I+\dt\bA)_{21}.
		\end{equation*}
		Using the fact that $r>0$ and $\lambda_{22}\neq 0$, this equation reduces to
		\begin{equation*}
			(\b I+\dt\bA)_{11}= \frac{y_1^*}{y_2^*}(\b I+\dt\bA)_{21}=\frac{y_1^*}{y_2^*}(\dt\bA)_{21},
		\end{equation*}
		or equivalently,
		\begin{equation*}
			1=\dt\left(\tfrac{y_1^*}{y_2^*}\lambda_{21}-\lambda_{11}\right)\overset{\eqref{PDS_test}}{=}\dt(ca+b),
		\end{equation*}
		which finishes also this part of the proof. \qedhere
	\end{enumerate}
\end{proof}
It is worth mentioning that part \ref{it:bThmgbbks1} of the above theorem demonstrates, that in general $\b g\notin \mathcal C^2$. As a result we could not apply \cite[Theorem 2.9]{IKM2122}, however, the generalization Theorem \ref{Thm_MPRK_stabil} can be applied, which gives us the following statements due to $\b D\b g(\b y^*)=\b I+\dt\bA$.
\begin{cor}\label{cor:gBBKS1}
	Let $\b y^*>\b 0$ be an arbitrary steady state of \eqref{PDS_test}. Under the assumptions of Theorem \ref{Thm:gBBKS1}, the gBBKS1 schemes have the same stability function as the underlying Runge--Kutta method, i.\,e.\ $R(z)=1+z$ and the following holds.
	\begin{enumerate}
		\item If $\abs{R(-(ac+b)\dt)}<1$, then $\b y^*$ is a stable fixed point of each gBBKS1 scheme and there exists a $\delta>0$, such that $\vec{1\\ c}^T \b y^{0}=\vec{1\\ c}^T\b y^*$ and $\norm{\b y^0-\b y^*}<\delta$ imply $\b y^n\to \b y^*$ as $n\to \infty$.
		\item If $\abs{R(-(ac+b)\dt)}>1$, then $\b y^*$ is an unstable fixed point of each gBBKS1 scheme.
	\end{enumerate}
\end{cor}
\vfill
\subsubsection{Stability of second order gBBKS schemes}
In this subsection we investigate  the gBBKS2($\alpha$) schemes \eqref{eq:gBBKS2Intro} applied to \eqref{PDS_test}, \eqref{eq:IC}, which can be written in the form
\begin{subequations}\label{eq:gBBKS22b}
	\begin{align}\label{eq:stage1_gBBKS22b}
		&\begin{aligned}
			\b y^{(2)} &=\b y^n+\alpha \dt \bA\b y^n \Bigg(\prod_{j\in J^{n}} \frac{y^{(2)}_j}{\pi^{n}_j}\Bigg)^{\ns q^{n}},
		\end{aligned}\\ 
		&\begin{multlined}[b][.7\columnwidth]
			\b y^{n+1} = \b y^n+\dt \left(\Big(1-\frac{1}{2\alpha}\Big) \bA\b y^n+ \frac{1}{2\alpha}\bA\b y^{(2)}\right)\Bigg(\prod_{m\in M^{n}} \frac{y^{n+1}_m}{\sigma^{n}_m}\Bigg)^{\ns r^{n}},\label{eq:finalstage_gBBKS22c}
		\end{multlined}  
	\end{align}
\end{subequations}
for $i=1,2$, $\alpha\geq \tfrac12$ and
\begin{equation*}
	\begin{aligned}
		J^{n}&=\left\{ j\in \{1,2\}\mid (\bA\b y^n)_j<0\right\},\\
		M^{n}&=\left\{ m\in \{1,2\}\;\Big |\;\Big(1-\frac{1}{2\alpha}\Big) (\bA\b y^n)_m+ \frac{1}{2\alpha}(\bA\b y^{(2)})_m<0\right\}.
	\end{aligned}
\end{equation*}
Similarly to the gBBKS1 case, we introduce functions $r,q,\bm \pi$ and $\bm \sigma$ to describe the dependence of the parameters on $\b y^n$. Note that $\bm \sigma$ can depend on $\b y^n$ as well as $\b y^{(2)}$, see \cite{gBBKS,BBKS2007,BRBM2008}, and thus will be described by a map $\bm\sigma\colon\R^2_{>0}\times \R^2_{>0}\to \R^2_{>0}$. 
\begin{thm}\label{Thm:gBBKS2}
	Let $\pi_1,\pi_2,r,q\in \mathcal{C}^2(\R^2_{>0},\R_{>0})$, $\bm \sigma\in \mathcal{C}^2(\R^2_{>0}\times \R^2_{>0},\R^2_{>0})$ and $\b y^*>\b 0$ be a steady state solution of \eqref{PDS_test}. Also, let $\mathcal D$ be a sufficiently small neighborhood of $\b y^*$ and suppose that $\bm \sigma(\b v,\b v)=\bm \pi(\b v)=\b v$ is fulfilled for all $\b v\in C= \ker(\bA)\cap \mathcal D$. Then the map $\b g$ generating the iterates of the gBBKS2($\alpha$) family satisfies $\b g(\b v)=\b v$ for all steady states $\b v\in  C$ and the following statements are true.
	\begin{enumerate}
		\item\label{it:aThm} The map $\b g$ satisfies $\b g\in \mathcal{C}^1(\mathcal D)$ and $\b D\b g(\b y^*)=\b I+\dt\bA+\frac{(\dt)^2}{2}\bA^2$.
		\item\label{it:bThmgbbks2}  The first derivatives of $\b g$ are bounded and Lipschitz continuous on $\mathcal D$.
		\item\label{it:cThmgbbks2} The map $\b g$ does not belong to $\mathcal C^2$ for any open neighborhood of $\b y^*$, if $\dt\neq(ac+b)^{-1}$.
	\end{enumerate}
\end{thm}
\begin{proof}
	Our main strategy is to follow the ideas used in the proof of  Theorem~\ref{Thm:gBBKS1}. For this, we first compute the sets $J^n$ and $M^n$ in the case of the linear test problem \eqref{PDS_test}. Using \eqref{eq:stage1_gBBKS22b}, we obtain 
	\[ \Big(1-\frac{1}{2\alpha}\Big) (\bA\b y^n)_m+ \frac{1}{2\alpha}(\bA\b y^{(2)})_m=(\bA\b y^n)_m\left(1+\alpha \dt\Bigg(\prod_{j\in J^{n}} \frac{y^{(2)}_j}{\pi^{n}_j}\Bigg)^{\ns q^{n}}\right),\]
	so that\[M^n=J^n=\begin{cases}
		\{1\}, & y_1^n>\frac{b}{a}y^n_2,\\
		\emptyset, & y_1^n=\frac{b}{a}y^n_2,\\
		\{2\}, & y_1^n<\frac{b}{a}y^n_2
		
	\end{cases}\] follows as in the case of gBBKS1.  Next, we define 
	\begin{equation}
		\begin{aligned}
			\b y^{(2)}(\b x)&=\b x-\dt\alpha\bA\b x\begin{cases}
				\left(\frac{y_1^{(2)}(\b x)}{\pi_1(\b x)}\right)^{q(\b x)},& x_1>\frac{b}{a}x_2,\\
				1, & x_1=\frac{b}{a}x_2,\\
				\left(\frac{y_2^{(2)}(\b x)}{\pi_2(\b x)}\right)^{q(\b x)},& x_1<\frac{b}{a}x_2
			\end{cases} \label{eq:1.gBBKS2}
		\end{aligned}
	\end{equation}
	and
	\begin{equation}\label{eq:defFgbbks2}
		\b F(\b x,\b y)=\b y-\b x-\dt\left(\Big(1-\frac{1}{2\alpha}\Big)\bA\b x+\frac{1}{2\alpha}\bA\b y^{(2)}(\b x)\right)H(\b x,\b y),
	\end{equation}
	where
	\begin{equation*}
		H(\b x,\b y)=\begin{cases}
			\widetilde H_1(\b x,\b y),& x_1>\frac{b}{a}x_2,\\
			1, & x_1=\frac{b}{a}x_2,\\
			\widetilde H_2(\b x,\b y),& x_1<\frac{b}{a}x_2
		\end{cases}
	\end{equation*}
	as well as
	\begin{equation}\label{eq:HtildegBBKS2}
		\widetilde H_i(\b x,\b y)=\left(\frac{y_i}{\sigma_i(\b x,\b y^{(2)}(\b x))}\right)^{r(\b x)},\quad i=1,2,
	\end{equation}
	and point out that the function $\b g$ generating the gBBKS2($\alpha$) iterates is the unique solution to 
	\begin{equation}
		\begin{aligned}\label{eq:0=FtildegBBKS2}
			\b 0&=\b F(\b x,\b g(\b x)).
		\end{aligned}
	\end{equation}
	Note that equation \eqref{eq:1.gBBKS2} represents the gBBKS1 schemes applied to \eqref{PDS_test} with a time step size of $\dt\alpha$. Hence, Theorem \ref{Thm:gBBKS1} implies that the function $\b y^{(2)}$ is a $\mathcal{C}^1$-map on $\mathcal D$ with Lipschitz continuous first derivatives and \[\b D\b y^{(2)}(\b y^*)=\b I+\dt\alpha \bA.\]
	Furthermore, $\b v\in\ker(\bA)$ implies $\b y^{(2)}(\b v)=\b v$, and thus, inserting $\b x=\b v$ into \eqref{eq:defFgbbks2}, \eqref{eq:0=FtildegBBKS2} yields $\b g(\b v)=\b v$.
	\begin{enumerate}
		\item Along the same lines as in the proof of Theorem \ref{Thm:gBBKS1} we see that the map $\b F$ is not differentiable on $\mathcal D\times \mathcal D$ since $\bm \sigma(\b x_0,\b y^{(2)}(\b x_0))=\b x_0$ holds for all $\b x_0\in  C$. However, $\b F$ is differentiable in $(\b x,\b y)\in E=\mathcal D\setminus\ker(\bA)\times \mathcal D$ since 
		\begin{equation}\label{eq:FgBBKS2}
			\b F(\b x,\b y)=\b y-\b x-\dt\left(\Big(1-\frac{1}{2\alpha}\Big)\bA\b x+\frac{1}{2\alpha}\bA\b y^{(2)}(\b x)\right)\left(\frac{y_i}{\sigma_i(\b x,\b y^{(2)}(\b x))}\right)^{r(\b x)}
		\end{equation}
		for 
		\begin{equation}\label{eq:igbbks2}
			i=\begin{cases}
				1&, x_1>\frac{b}{a}x_2,\\
				2&, x_1<\frac{b}{a}x_2
			\end{cases}
		\end{equation}
		and $r\in \mathcal{C}^2(\R^2_{>0},\R_{>0})$, $\bm \sigma\in \mathcal{C}^2(\R^2_{>0}\times \R^2_{>0},\R^2_{>0})$ as well as $\b y^{(2)}\in\mathcal C^1(\mathcal D)$. Following the proof of Theorem \ref{Thm:gBBKS1}, we show that $\b D_\b y\b F(\b x,\b g(\b x))$ is nonsingular in order to show that $\b g\in \mathcal C^1$ on $\mathcal D\setminus\ker(\bA)$.
		First note that for $\b x\notin C$ we have
		\begin{equation}\label{eq:DyFgbbks2}
			\begin{aligned}
				\b D_\b y\b F(\b x,\b g(\b x))&=\b I-\dt\left(\Big(1-\frac{1}{2\alpha}\Big)\bA\b x+\frac{1}{2\alpha}\bA\b y^{(2)}(\b x)\right)\nabla_\b y \widetilde H_i(\b x,\b g(\b x))
			\end{aligned}
		\end{equation}
		for $i$ from \eqref{eq:igbbks2}. Now,  \eqref{eq:HtildegBBKS2} yields
		\begin{equation*}
			\nabla_\b y \widetilde H_i(\b x,\b g(\b x))=\frac{r(\b x)}{g_i(\b x)}\widetilde H_i(\b x,\b g(\b x))\b e_i^T
		\end{equation*}
		with the $i$th unit vector $\b e_i\in \R^2$  as in the proof of Theorem \ref{Thm:gBBKS1}.
		In order to see that $\b D_\b y\b F(\b x,\b g(\b x))$ is invertible, we introduce
		\begin{equation*}
			\begin{aligned}
				\b v^{(i)}(\b x)&=\dt\left(\left(1-\frac{1}{2\alpha}\right)\bA\b x+\frac{1}{2\alpha}\bA\b y^{(2)}(\b x)\right)\frac{r(\b x)}{g_i(\b x)}\widetilde H_i(\b x,\b g(\b x))\\
			\end{aligned}
		\end{equation*}
		and rewrite \eqref{eq:DyFgbbks2} as
		\begin{equation}\label{eq:DyFtildeInversegBBKS2}
			\b D_\b y\b F(\b x,\b g(\b x))=\b I-\b v^{(i)}(\b x)\b e_i^T.
		\end{equation}
		Hence, we obtain \[ \det(\b D_\b y\b F(\b x,\b g(\b x)))=1-v_i^{(i)}(\b x).\] 
		Using \eqref{eq:1.gBBKS2}, we see that 
		\[\b v^{(i)}(\b x)=\dt\bA\b x\left(1+\alpha \dt\Bigg( \frac{y_i^{(2)}(\b x)}{\pi_i(\b x)}\Bigg)^{\ns q(\b x)}\right)\frac{r(\b x)}{g_i(\b x)}\widetilde H_i(\b x,\b g(\b x)), \]
		where $(\bA\b x)_i<0$ by definition of the gBBKS2($\alpha$) schemes. As a result we know $v_i^{(i)}(\b x)<0$, and hence $\det(\b D_\b y\b F(\b x,\b g(\b x)))\neq 0$ proving that $\b D_\b y\b F(\b x,\b g(\b x))$ is invertible. This together with the corresponding arguments of Theorem~\ref{Thm:gBBKS1} implies that $\b g\in \mathcal C^1$ on $\mathcal D\setminus \ker(\bA)$ and 	
		\begin{equation}\label{eq:Dg(x)proofgbbks2}
			\b D\b g(\b x)=-(\b D_\b y\b F(\b x,\b g(\b x)))^{-1}\b D_\b x\b F(\b x,\b g(\b x))
		\end{equation}
		for $\b x\in\mathcal D\setminus \ker(\bA)$.
		To apply Theorem \ref{Thm:gDiff}, we proceed as in the proof of  Theorem \ref{Thm:gBBKS1}, i.\,e.\ we have to show that
		\begin{enumerate}[label=\arabic*.]
			\item\label{it:1} the map $\b g$ is continuous in any $\b x\in C$.
			\item\label{it:2} the map $\b F$ is differentiable in $(\b x,\b g(\b x))$ for all $\b x\in C$.
			\item\label{it:3} the Jacobian $\b D_\b y\b F(\b x,\b g(\b x))$ with respect to $\b y$ is invertible for all $\b x\in C$.
		\end{enumerate}
		
		The continuity of $\b g$ follows along the same lines as in the case of gBBKS1, where we additionally use $\b y^{(2)}\in \mathcal C^1(\mathcal D)$ for bounding $H(\cdot,\b g(\cdot))$. 
		
		For proving the differentiability of $\b F$ in $(\b x,\b g(\b x))$ for all $\b x\in C$ we consider an arbitrary element $\b x_0\in C$. Note that $\Psi(\b x)=H(\b x,\b g(\b x_0))$ is continuous in $\b x_0$ with $\Psi(\b x_0)=1$. Furthermore, \[\mathbf \Phi(\b x)=\Big(1-\frac{1}{2\alpha}\Big)\bA\b x+\frac{1}{2\alpha}\bA\b y^{(2)}(\b x)\] satisfies $\mathbf \Phi(\b x_0)=\b 0$, which means that part \ref{it:alemdiff} of Lemma \ref{Lem:diff} from the appendix together with $\b D\b y^{(2)}(\b x_0)=\b I+\dt\alpha\bA$ yields
		\begin{equation}\label{eq:DxFtildegBBKS2}
			\begin{aligned}
				\b D_\b x\b F(\b x_0,\b g(\b x_0))&=-\b I-\dt\left(\Big(1-\frac{1}{2\alpha}\Big)\bA+\frac{1}{2\alpha}\bA\b D\b y^{(2)}(\b x_0)\right)\\
				&=-\b I-\dt\left(\bA+\frac{\dt}{2}\bA^2\right).
			\end{aligned}
		\end{equation}
		Also, since $\mathbf\Phi(\b x_0)H(\b x_0,\b y)=\b 0$ for all $\b y\in \R^2_{>0}$, we find
		\begin{equation}\label{eq:DyF=IgBBKS2}
			\begin{aligned}
				\b D_\b y\b F(\b x_0,\b g(\b x_0))&=\b I,
			\end{aligned}
		\end{equation}
		which shows that $\b F$ is partially differentiable in $(\b x_0,\b g(\b x_0))$. We now prove that the partial derivatives of $\b F$ are also continuous in $(\b x_0,\b g(\b x_0))$, which shows the differentiability of $\b F$ in $(\b x,\b g(\b x))$ for all $\b x\in C$. To that end, we consider $\b x\notin C$ and differentiate $\b F$ from \eqref{eq:FgBBKS2} with respect to $\b x$ and $\b y$. First, due to \eqref{eq:DyFgbbks2} and since $\b v^{(i)}$ is continuous with $\b v^{(i)}(\b x_0)=\b 0$ for $\b x_0\in C$, we find
		\begin{equation*}
			\lim_{\b x\to \b x_0}	\b D_\b y\b F(\b x,\b g(\b x))=\b I-\b v^{(i)}(\b x_0)\b e_i^T=\b I
		\end{equation*}
		proving the continuity of the partial derivatives in $(\b x_0,\b g(\b x_0))$ with respect to $\b y$. Furthermore, we have 
		\begin{equation*}
			\begin{aligned}
				\b D_\b x\b F(\b x,\b g(\b x))=&-\b I-\dt\left(\Big(1-\frac{1}{2\alpha}\Big)\bA+\frac{1}{2\alpha}\bA\b D\b y^{(2)}(\b x)\right)\widetilde H_i(\b x,\b g(\b x))\\
				&-
				\dt\left(\Big(1-\frac{1}{2\alpha}\Big)\bA\b x+\frac{1}{2\alpha}\bA\b y^{(2)}(\b x)\right)\nabla_\b x \widetilde H_i(\b x,\b g(\b x)),
			\end{aligned}
		\end{equation*}
		whose entries converge to those of $\b D_\b x\b F(\b x_0,\b g(\b x_0))$ from \eqref{eq:DxFtildegBBKS2} because of the following. First, we have $\b y^{(2)}(\b x_0)=\b x_0$ and $\widetilde{H}_i\in \mathcal C^1(\mathcal D\times \mathcal D)$, which means that the last addend disappears as $\b x\to\b x_0\in C$. Additionally, inserting $\b D\b y^{(2)}(\b x_0)=\b I+\alpha\dt\bA$ and $\lim_{\b x\to \b x_0}\widetilde{H}_i(\b x,\b g(\b x))=1$ yield \eqref{eq:DxFtildegBBKS2}.

		Finally, it follows from \eqref{eq:DyF=IgBBKS2} that the Jacobian $\b D_\b y\b F(\b x,\b g(\b x))$ with respect to $\b y$ is invertible for all $\b x\in C$. Hence, Theorem \ref{Thm:gDiff} together with the considerations above proves that $\b g$ is differentiable in all $\b x\in \mathcal D$. 
		
		Moreover, since $\b F$ is continuously differentiable in $(\b x,\b g(\b x))$, we find due to \eqref{eq:Dg(x)proofgbbks2} that even $\b g\in \mathcal C^1(\mathcal D)$ holds true. Furthermore, inserting \eqref{eq:DxFtildegBBKS2} and \eqref{eq:DyF=IgBBKS2} into formula \eqref{eq:Dg(x)proofgbbks2} yields
		\[\b D\b g(\b x)=-(\b D_\b y\b F(\b x,\b g(\b x)))^{-1}\b D_\b x\b F(\b x,\b g(\b x))=\b I+\dt\bA+\frac{(\dt)^2}{2}\bA^2.\]
		
		\item We know that $\b y^{(2)}\in\mathcal C^1$ has Lipschitz continuous first derivatives on $\mathcal D$ and that $\bm \sigma\in \mathcal C^2$. Hence, with
		\begin{equation*}
			\begin{aligned}
				\nabla_\b x\widetilde H_i(\b x,\b g(\b x))=&\widetilde H_i(\b x,\b g(\b x))\Biggl(\nabla_\b x r(\b x)\ln\left(\frac{g_i(\b x)}{\sigma_i(\b x,\b y^{(2)}(\b x))}\right)\\
				&-r(\b x)\frac{\nabla_\b x \sigma_i(\b x,\b y^{(2)}(\b x))+\nabla_\b y \sigma_i(\b x,\b y^{(2)}(\b x))\b D\b y^{(2)}(\b x)}{\sigma_i(\b x)}\Biggr),
			\end{aligned}
		\end{equation*}
		which is analogous to \eqref{eq:nabla(g/sigma)}, this part can be proven along the same lines as in the proof of part \ref{it:cThmgbbks1} of Theorem \ref{Thm:gBBKS1}.
		\item Since \eqref{eq:DyFtildeInversegBBKS2} and \eqref{eq:Dg(x)proofgbbks2} are of the form \eqref{eq:DyFtildeInversegBBKS1}, \eqref{eq:Dg(x)gbbks1}, this part is proven along the same lines as part \ref{it:bThmgbbks1} of Theorem \ref{Thm:gBBKS1}.
	\end{enumerate}
\end{proof}
This theorem together with Theorem \ref{Thm_MPRK_stabil} and Theorem \ref{Thm:_Asym_und_Instabil}  allows us to conclude the following statements from $\b D\b g(\b y^*)=\b I+\dt\bA+\tfrac12(\dt\bA)^2$.
\begin{cor}\label{cor:gBBKS2}
	Let $\b y^*>\b 0$ be an arbitrary steady state of \eqref{PDS_test}. Under the assumptions of Theorem \ref{Thm:gBBKS2}, the gBBKS2($\alpha$) schemes have the same stability function as the underlying Runge--Kutta method, i.\,e.\ $R(z)=1+z+\frac{z^2}{2}$ and the following holds.
	\begin{enumerate}
		\item If $\abs{R(-(ac+b)\dt)}<1$, then $\b y^*$ is a stable fixed point of each gBBKS2($\alpha$) scheme and there exists a $\delta>0$, such that $\b y^n\to \b y^*$ as $n\to \infty$ for all $\b y^0$ satisfying $\vec{1\\ c}^T \b y^{0}=\vec{1\\ c}^T\b y^*$ and $\norm{\b y^0-\b y^*}<\delta$.
		\item If $\abs{R(-(ac+b)\dt)}>1$, then $\b y^*$ is an unstable fixed point of each gBBKS2($\alpha$) scheme.
	\end{enumerate}
\end{cor}
To summarize the presented analysis of gBBKS schemes, we conclude that the first and second order gBBKS schemes preserve the stability domain of the underlying Runge--Kutta method while preserving positivity. To generalize these results to $N\times N$ systems we need to exploit more properties of the particular choices of $r,q,\bm \pi$ and $\bm \sigma$ from the literature \cite{gBBKS}.

\newpage
\section{Summary of Stability Properties}
In the previous section we investigated several Patankar-type methods with respect to their stability. The purpose of this rather short section is to summarize our findings, see Table~\ref{tab:stabschemes}. We also recall that MPDeC(1) corresponds to MPE  and MPDeC(2) equals MPRK22($1$). 
Also, for more insights on the stability properties of MPRK43($\alpha,\beta$) we refer to Figure~\ref{fig:maxstab}, where a lower bound for the maximal opening angle of the stability domain is depicted for $(\alpha,\beta)$ pairs in the feasible domain in $[0,2]\times[0,\tfrac34]$ with a resolution of $102^2$ pairs per unit square. The opening angle estimate for MPDeC up to order $8$ can be found in Table~\ref{tab:mpdecstab}.
\begin{table}[!h]\centering
	\begin{tabular}{|c|c|c|}\hline
		Method & Parameter Specification & Unconditionally Stable? \\ \hline 
		MPE & -- &$\surd$\\
		MPRK22($\alpha$) & $\alpha\geq \frac12$ & $\surd$\\
		MPRK43($\alpha,\beta$)& $(\alpha,\beta)=(0.5,0.75)$ &$\surd$ \\
		MPRK43($\alpha,\beta$)& $(\alpha,\beta)=(1,0.5)$ &$\times$ \\
		MPRK43($\gamma$)& $\frac38\leq\gamma\leq \frac34 $ &$\surd$ \\
		SSPMPRK2($\alpha,\beta$) & $\alpha\leq \frac{1}{2\beta}$ & $\surd$\\
		SSPMPRK2($\alpha,\beta$) & $\alpha> \frac{1}{2\beta}$ & $\times$\\
		SSPMPRK3($\eta_2$) & $0\leq \eta_2\leq 0.37$ & $\surd$ \\ 
		MPDeC($p$)& $p\in\{1,2\}$& $\surd$\\
		MPDeC($p$)& $p=3,\dotsc,8$& $\times$\\
		MPDeCGL($p$)& $p=9,\dotsc,14$& ($\surd$)\\
		MPDeCEQ($p$)& $p=9,10,11,13$& ($\surd$)\\
		MPDeCEQ($p$)& $p\in\{12,14\}$& $\times$\\
		GeCo1 		& -- & $\surd$\\
		GeCo2 & -- & $\times$\\
		gBBKS1 & -- & $\times$\\
		gBBKS2$(\alpha)$ & $\alpha\geq \frac12$ & $\times$\\
		\hline
	\end{tabular}
	\caption{Overview on the stability properties of several Patankar-type methods from Chapter~\ref{chap:NumSchemes}. Here, \enquote{$(\surd)$} means that the unconditional stability is investigated only numerically and only for normal system matrices with real eigenvalues.}\label{tab:stabschemes} 
\end{table}

In the upcoming section we introduce several linear problems for testing the predictions contained in Table~\ref{tab:stabschemes} together with the corresponding numerical experiments. For a deeper insight into numerical experiments concerning oscillatory behavior we refer to \cite{ITOE22}.

\newpage
\section{Numerical Experiments}
As mentioned in the preceding section, this part of the thesis is dedicated to the numerical validation of the theoretical claims concerning stability, parts of which are summarized in Table~\ref{tab:stabschemes}. Thereby, we also incorporate the hypothesis mentioned and tested in \cite{izgin2023nonglobal} stating that the claimed properties of stability and convergence towards the steady state solution of \eqref{eq:PDS_Sys} are even of global nature for MPRK schemes that are based on a non-negative Butcher tableau. In fact, so far the only cases of MP methods where the stability properties were observed to be non-global are MPRK22($\alpha$) with $\alpha<\tfrac12$ and MPDeCEQ($p$) for some values $p\geq 8$. In all cases the schemes can be understood as MP methods based on RK schemes with a Butcher array containing also negative entries, see \cite{IOE22StabMP,izgin2023nonglobal}. In particular, we present the numerical experiments with MPDeCEQ(8) in this work to give an example of this phenomenon. 

For the numerical validation different test cases are of interest, which we will discuss in the following subsection.

\subsection{Test Problems}
In the following we only consider conservative problems, \ie the systems matrices we are going to introduce have an eigenvalue $\lambda=0$. Furthermore, most of the following test cases are chosen in such a way that all nonzero eigenvalues either lie in $\R^-$ or in $\C^-\setminus\R^-$. Moreover, as we are interested in testing part~\ref{it:Thmb} of Theorem~\ref{Thm_MPRK_stabil}, we also consider a test problem with two linear invariants. It is beneficial to consider these test cases rather than a single one with a spectrum in $\Cminus$ because this way we can test the stability domain of conditional stable methods at two distinct spots of the stability domain. Nevertheless, we will also include a test problem with real as well as complex eigenvalues for testing unconditionally stable schemes.

\subsubsection{Test problem with exclusively real eigenvalues}
The linear initial value problem 
\begin{equation}
	\b y'=100\Vec{-2& 1 &1\\1 &-4 &1\\1 &3 &-2}\b y,\quad  \b y(0)=\Vec{1\\ 9\\5}\label{eq:initProbReal}
\end{equation}
contains a system matrix, which has only positive off-diagonal elements and is therefore a Metzler matrix. Due to the positive initial values, this ensures that each component of the solution of the initial value problem is positive for all times. By a straightforward calculation of the eigenvalues $\lambda_1=0$, $\lambda_2=-300$ and $\lambda_3=-500$ of the system matrix as well as their associated eigenvectors, the solution reads \begin{equation}
	\b y(t)=c_1\begin{pmatrix}5\\ 3\\ 7\end{pmatrix}+c_2e^{-300t}\begin{pmatrix}-1\\ 0\\ 1\end{pmatrix}+c_3e^{-500t}\begin{pmatrix}0\\ -1\\ 1\end{pmatrix}\label{eq:exsolReal}
\end{equation}
with coefficients $c_1=1$, $c_2=4$ and $c_3=-6$ determined by the initial condition. Since only non-positive eigenvalues are present and the absolute values of the negative eigenvalues are large, there is a fast convergence to the equilibrium state 
\begin{equation*}
	\b y^*=\lim_{t\to\infty}\b y(t)=\Vec{5\\3\\7}
\end{equation*}
as depicted in Figure \ref{Fig:initProbReal}. Furthermore the zero eigenvalue is simple, and hence there exists exactly one linear invariant, which is given by $\bm 1^T\b y$ due to the fact that the sum of the elements in each column of the system matrix is always vanishing. This conservativity  can also be observed in Figure \ref{Fig:initProbReal}.

\begin{figure}[h!]
	\centering
	\includegraphics[width=0.5\textwidth]{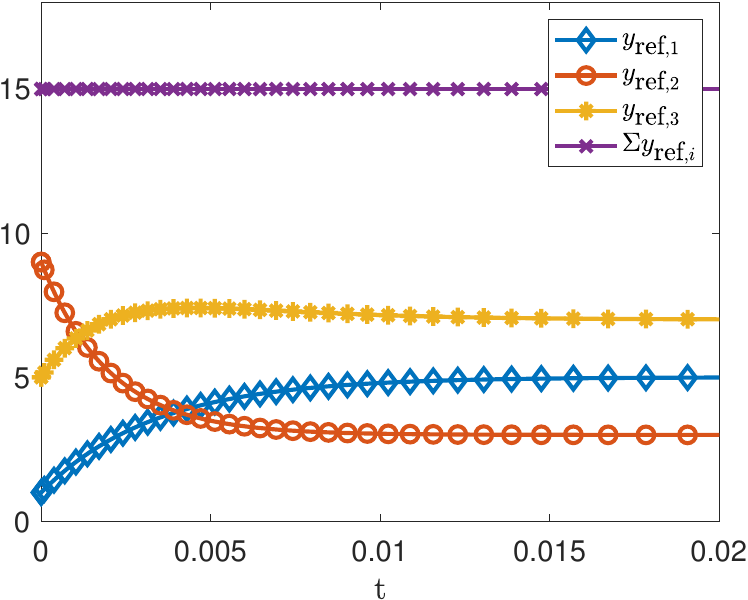}
	\caption{Exact solution \eqref{eq:exsolReal} of the initial value problem \eqref{eq:initProbReal} and the linear invariant $\bm 1^T\b y$.}\label{Fig:initProbReal}
\end{figure}

\subsubsection{Test problem with complex eigenvalues}

As a second test case, we consider the conservative system 
\begin{equation}
	\b y'=100\Vec{-4& 3 &1\\2 &-4 &3\\2 &1 &-4}\b y,\quad  \b y(0)=\Vec{9\\ 20\\8}\label{eq:initProbIm}.
\end{equation}
Again, the system matrix is a Metzler matrix, so that the solution of the initial value problem is always positive due to the positive initial conditions. Considering the eigenvalues $\lambda_1=0$ , $\lambda_2=100(-6+\ii)$ and $\lambda_3=\overline{\lambda_2}$ as well as the corresponding eigenvectors of the system matrix, the solution can be written in the form
\begin{equation}
	\begin{aligned}
		\b y(t)=&\begin{pmatrix}13\\ 14\\ 10\end{pmatrix}-2e^{-600t}\left(\cos \left(100t\right)\begin{pmatrix}-1\\ 0\\ 1\end{pmatrix}-\sin \left(100t\right)\begin{pmatrix}1\\ -1\\ 0\end{pmatrix}\right)\\&-6e^{-600t}\left(\cos \left(100t\right)\begin{pmatrix}1\\ -1\\ 0\end{pmatrix}+\sin \left(100t\right)\begin{pmatrix}-1\\ 0\\ 1\end{pmatrix}\right).\label{eq:exsolIm}
	\end{aligned}
\end{equation}
The nonzero complex eigenvalues have a negative real part with a large absolute value. Hence, one can expect a rapid convergence of the solution to the steady state given by
\begin{equation*}
	\b y^*=\lim_{t\to\infty}\b y(t)=\Vec{13\\14\\10}.
\end{equation*} 
Analogous to the first test case, the only linear invariant is $\bm 1^T\b y$, which is presented together with the exact solution in Figure \ref{Fig:initProbIm}.

\begin{figure}[h!]
	\centering
	\includegraphics[width=0.5\textwidth]{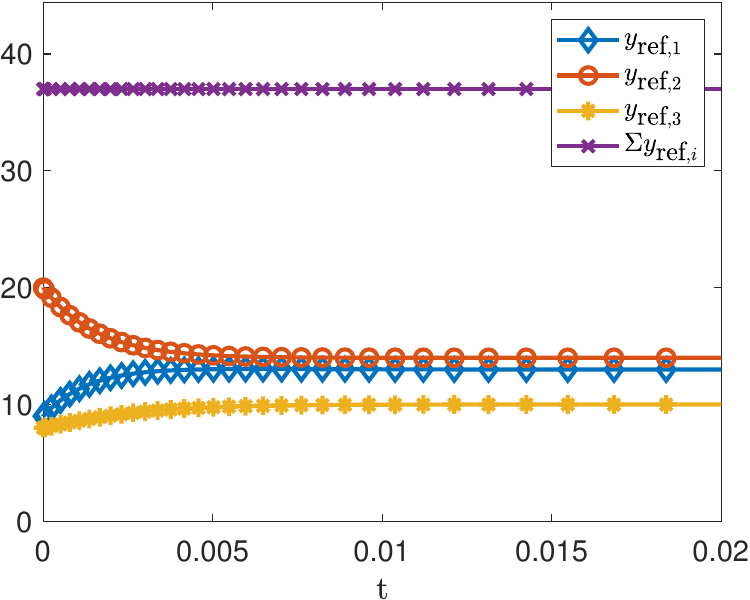}
	\caption{The exact solution \eqref{eq:exsolIm} of the initial value problem \eqref{eq:initProbIm} and the linear invariant $\bm 1^T\b y$.}\label{Fig:initProbIm}
\end{figure}

\subsubsection{Test problem with double zero eigenvalue}
Considering the linear initial value problem  
\begin{equation}
	\b y'=100\Vec{-2& 0 &0 &1\\0 &-4 &3& 0\\0 &4& -3 &0\\ 2 & 0&0&-1}\b y,\quad  \b y(0)=\Vec{4\\ 1\\9\\1},\label{eq:initProb4dim}
\end{equation}
we are faced with a Metzler matrix including a double zero eigenvalue $\lambda_1 = \lambda_2=0$. Therefore, besides $\bm 1^T\b y$, a second linear invariant $\b n^T\b y$ with $\b n=(1,2,2,1)^T$ is present. Due to the remaining eigenvalues $\lambda_3=-300$ and $\lambda_3=-700$ and the associated eigenvectors of all eigenvalues, the solution of the initial value problem writes
\begin{equation}
	\b y(t)=c_1\Vec{0\\1\\ \frac43\\0}+c_2\Vec{1\\0\\0\\2}+c_3e^{-700t}\Vec{0\\1\\-1\\0}+c_4e^{-300t}\Vec{1\\0\\ 0\\-1}\label{eq:exsol4dim}
\end{equation}
with coefficients
\begin{align*}
	c_1=\frac{30}{7},\quad c_2=\frac53, \quad c_3=-\frac{23}{7}\qta c_4=\frac73.
\end{align*}
Once again, a fast convergence to the equilibrium state
\begin{equation*}
	\b y^*=\lim_{t\to\infty}\b y(t)=c_1\Vec{0\\1\\ \frac43\\0}+c_2\Vec{1\\0\\0\\2}=\frac{1}{21}\Vec{7\\ 90\\ 120\\ 70}
\end{equation*}
takes place. The course of the solution together with the two linear invariants are shown in Figure \ref{Fig:initProb4dim}.
\begin{figure}[h!]
	\centering
	\includegraphics[width=0.5\textwidth]{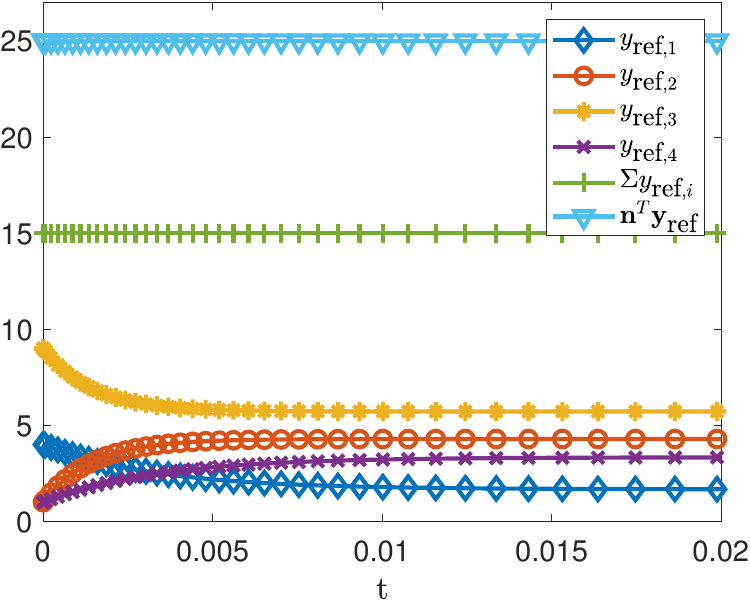}
	\caption{The exact solution \eqref{eq:exsol4dim} of the initial value problem \eqref{eq:initProb4dim} and the associated two linear invariants $\bm 1^T\b y$ and $\b n^T\b y$ with $\b n^T=(1,2,2,1)$.}\label{Fig:initProb4dim}
\end{figure}

At this point we want to note that the presented test cases represent stiff problems due to the occurrence of large absolute values of the corresponding eigenvalues. Hence, it is not surprising that the exact solution satisfies the inequality $\norm{\b y(t)-\b y^*}_2<2\cdot 10^{-2}$ at time $t=0.02$ for all of three problems. 

\subsubsection{Test Problem with mixed Eigenvalues}
Finally, we consider the initial value problem 
\begin{equation}\label{eq:testgeco1}
	\b y'=\bA\b y, \quad \b y(0)= \b y^0=(0,3,3,3,4)^T,
\end{equation}
where $\bA$ is the $5\times 5$ Metzler matrix
\begin{equation}\label{eq:Atest}
	\bA=\Vec{
		-4 & 2 & 1 & 2 & 2  \\
		1 & -4 & 1 & 0 & 2 \\
		0 & 0 & -4 & 2 & 0 \\
		2 & 2 & 2 & -4 & 0 \\
		1 & 0 & 0 & 0 & -4
	}.
\end{equation}
The spectrum of $\bA$ is given by $\sigma(\bA)=\{0,-5-\sqrt3,-5+\sqrt3,-5-\ii,-5+\ii\}\tm \overline{\C^-}$ including real as well as non-real eigenvalues. Furthermore, the kernel of $\bA^T$ is given by $\ker(\bA^T)=\Span(\b n)$ with
$\b n=(1,1,1,1,1)^T$. Hence, the total mass $\b n^T\b y(t)=\b n^T \b y^0=13$ is a linear invariant for the system, in correspondence of the initial value $\b y(0)=\b y^0$. The reference solution of the problem is depicted in Figure~\ref{Fig:lin5x5} and satisfies $\Vert \b y(t)-\b y^*\Vert_2<10^{-2}$ at time $t=1.61$.
\begin{figure}[h!]
	\centering
	\includegraphics[width=0.5\textwidth]{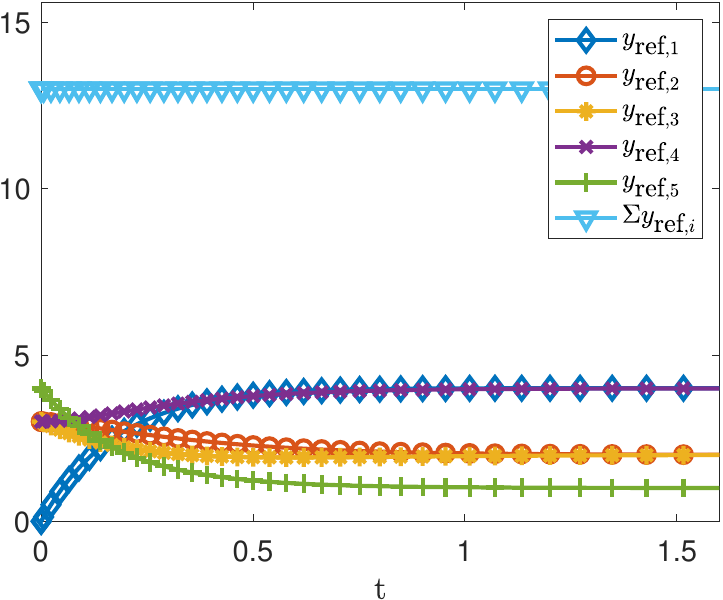}
	\caption{The reference solution of the initial value problem \eqref{eq:testgeco1}.}\label{Fig:lin5x5}
\end{figure}

We want to note that even though the stability functions of gBBKS and GeCo2 were obtained by analyzing a $2\times2$ system, we will see that the corresponding stability results are well reflected also for a larger system.

\subsection{Investigation of MPRK Schemes}
As in Chapter~\ref{chap:stab}, we consider here MPRK schemes up to order three. The stability analysis and numerical experiments for the fourth order MPRK method are left for future work.
In particular, the numerical experiments will be performed with MPE, MPRK22($\alpha$) for $\alpha\in\{0.5,1,5\}$, MPRK43($0.5,0.75$) and MPRK43($0.563$), all of which are proven to be unconditionally stable and locally converging towards the steady state solution. Hence, we consider the problem \eqref{eq:testgeco1} using a comparably large time step size of $\Delta t=5$. 

\subsubsection{MPE}
The results for MPE can be seen in Figure~\ref{fig:MPE}. As one can see, the method is stable and converging using the initial condition from \eqref{eq:testgeco1}. An error of around $10^{-14}$ is already obtained after $t=100$, that is after $20$ steps using $\dt=5$. Note again that this is a comparably large $\dt$ as the analytic solution satisfies $\norm{\b y(t)-\b y^*}_2<2\cdot 10^{-2}$ at time $t=1.61$. In Figure~\ref{fig:MPE4} on can see that the second linear invariant is also preserved.
\begin{figure}[!h]\centering
	\begin{subfigure}[t]{0.46\textwidth}
		\includegraphics[width=\textwidth]{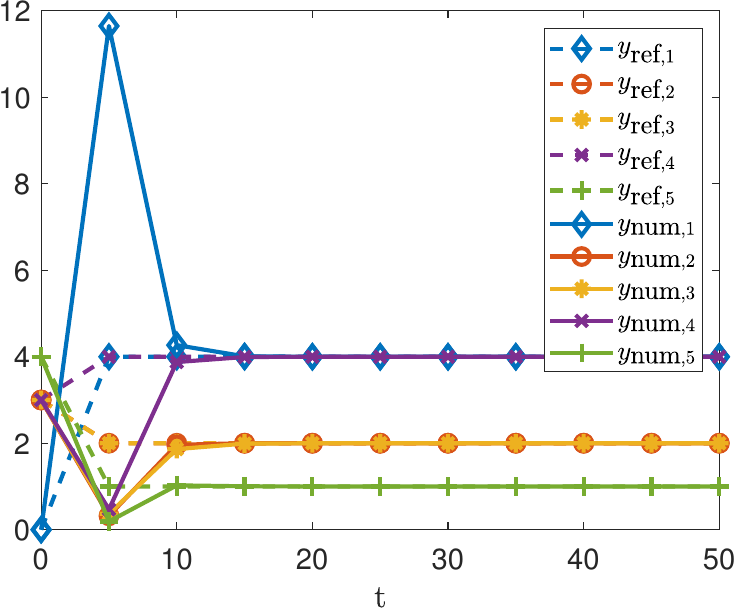}
	\end{subfigure}
	\begin{subfigure}[t]{0.49\textwidth}
		\includegraphics[width=\textwidth]{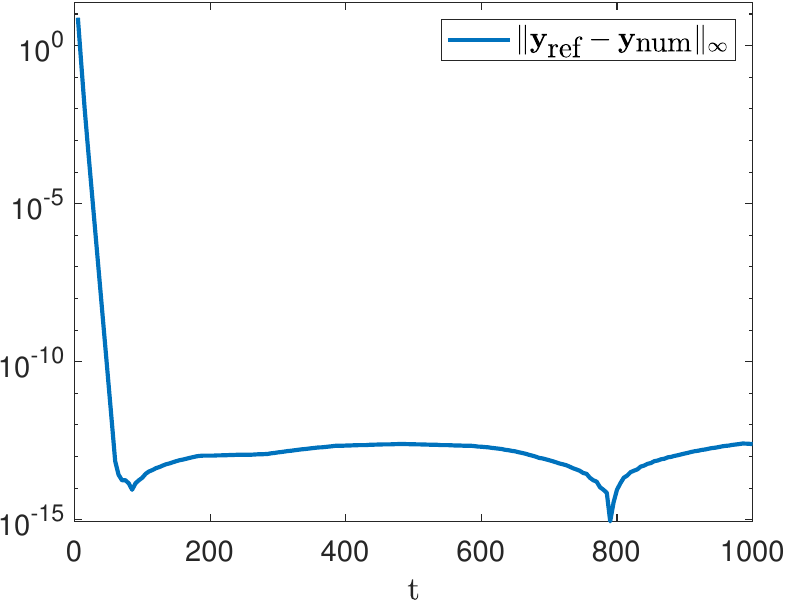}
	\end{subfigure}
	\caption{Numerical solution of \eqref{eq:testgeco1} and error plot using MPE. The dashed lines represent the reference solution.}\label{fig:MPE}
\end{figure}
\begin{figure}[!h]\centering
	\begin{subfigure}[t]{0.46\textwidth}
		\includegraphics[width=\textwidth]{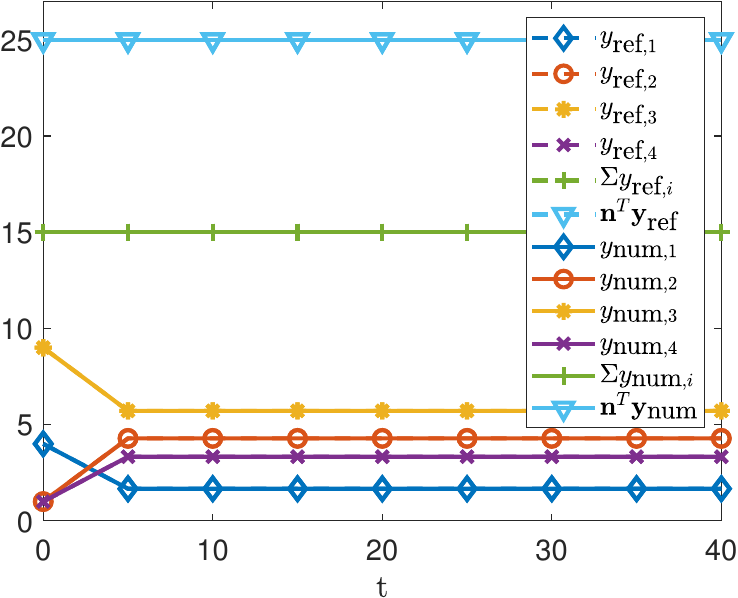}
	\end{subfigure}
	\caption{Numerical approximation of \eqref{eq:initProb4dim} using MPE. The dashed lines represent the exact solution \eqref{eq:exsol4dim} and coincide for this example with the numerical solution. The second linear invariant is determined by $\b n=(1,2,2,1)^T$.}\label{fig:MPE4}
\end{figure}

\subsubsection{MPRK22($\alpha$)}
In the Figures \ref{Fig:MPRK22} and \ref{Fig:MPRKinitProb4dim}, we compare the MPRK22($\alpha$) schemes for $\alpha\in\left\{\frac12,1,5\right\}$ and find that for $\alpha=1$ or $\alpha=5$ the methods produce errors near machine precision at $t=200$, \ie after around 40 steps, whereas for $\alpha=\frac12$ we cannot observe the convergence of the iterates towards $\b y^*$ within $t\in[0,50]$. Nevertheless, the results depicted on the top right show that even for the case $\alpha=\frac12$, the stability and convergence proved in Corollaries~\ref{Cor:MPRKstab} and \ref{Cor:MPRKstab1} can be confirmed numerically by extending the observation period. Moreover, the second linear invariant is also preserved, see Figure~\ref{Fig:MPRKinitProb4dim}.
\begin{figure}[!h]\centering
	\begin{subfigure}[t]{0.46\textwidth}
		\includegraphics[width=\textwidth]{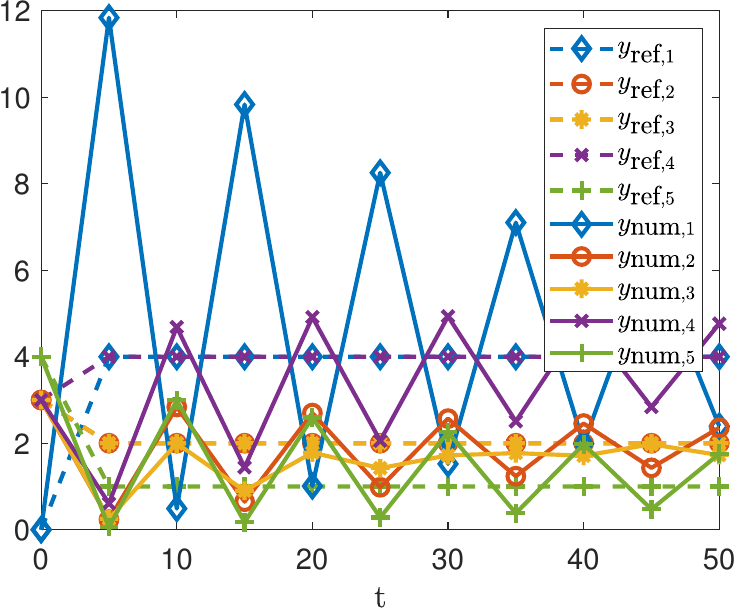}
		\subcaption{$\alpha=0.5$}
	\end{subfigure}
	\begin{subfigure}[t]{0.49\textwidth}
		\includegraphics[width=\textwidth]{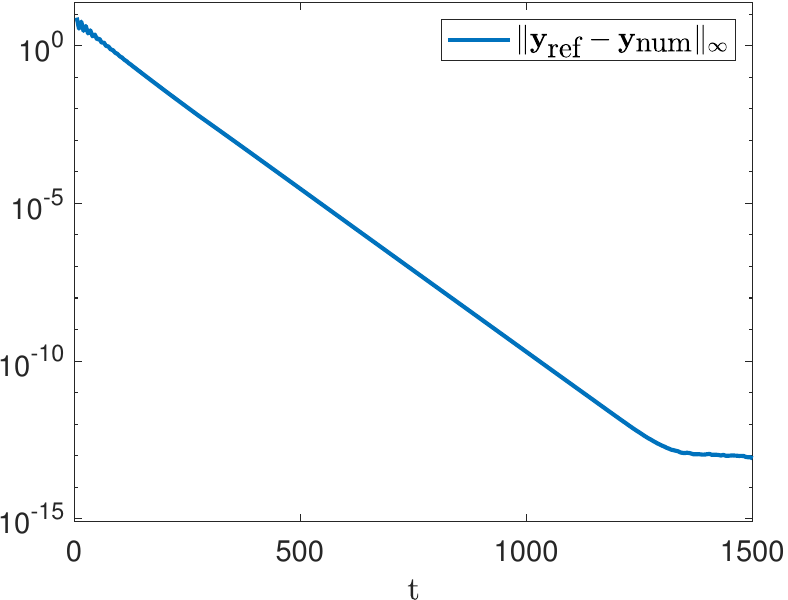}
		\subcaption{$\alpha=0.5$}
	\end{subfigure}\\
	\begin{subfigure}[t]{0.46\textwidth}
		\includegraphics[width=\textwidth]{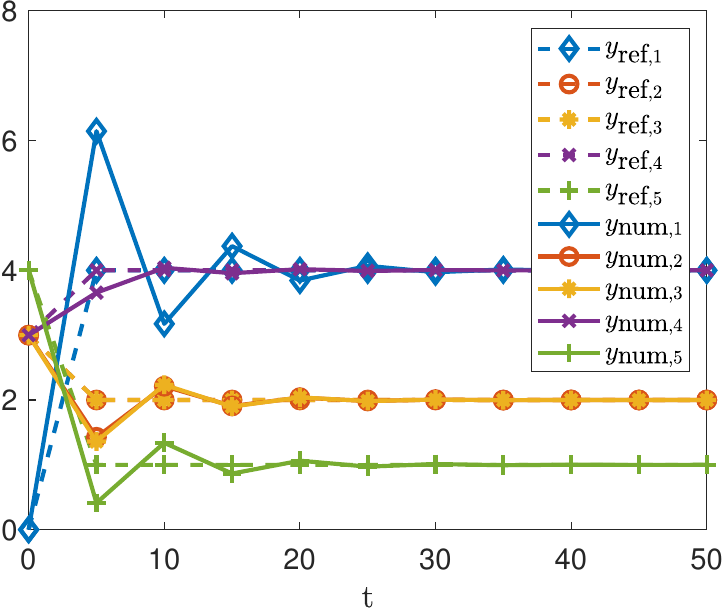}
		\subcaption{$\alpha=1$}
	\end{subfigure}
	\begin{subfigure}[t]{0.5\textwidth}
		\includegraphics[width=\textwidth]{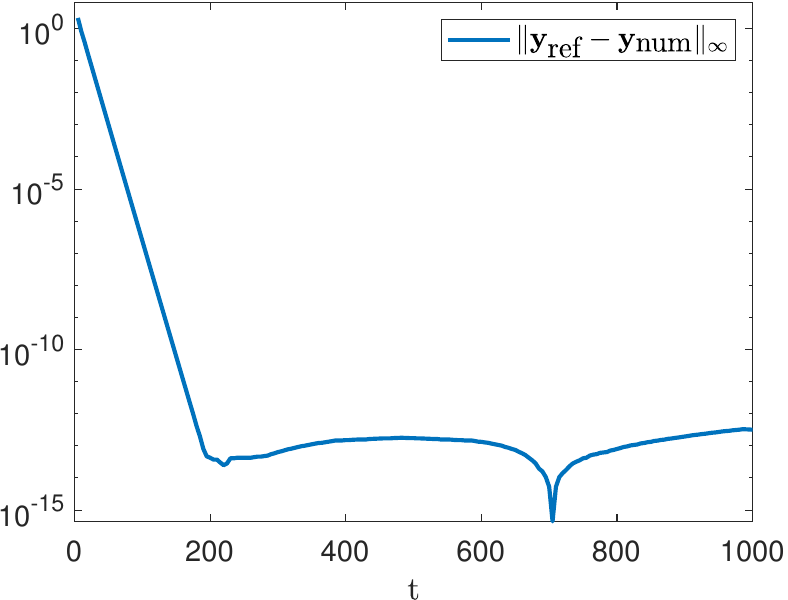}
		\subcaption{$\alpha=1$}
	\end{subfigure}\\
	\begin{subfigure}[t]{0.46\textwidth}
		\includegraphics[width=\textwidth]{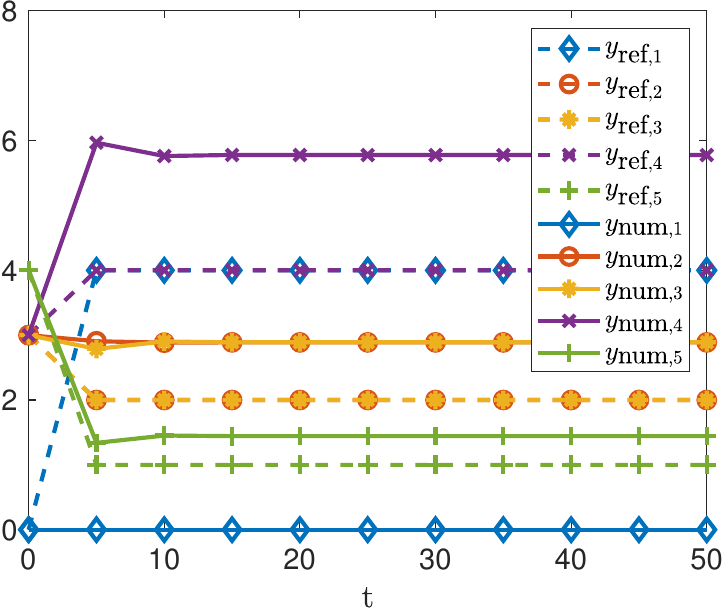}
		\subcaption{$\alpha=5$}
	\end{subfigure}
	\begin{subfigure}[t]{0.5\textwidth}
		\includegraphics[width=\textwidth]{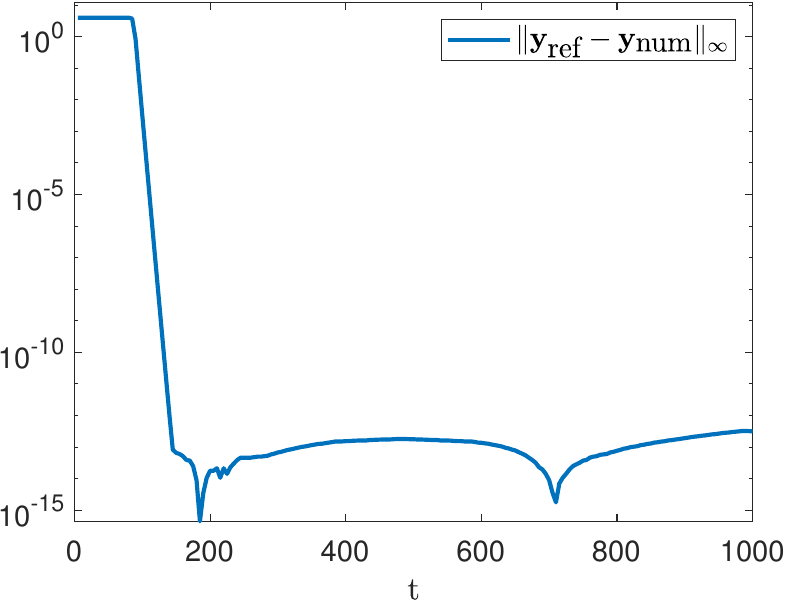}
		\subcaption{$\alpha=5$}
	\end{subfigure}
	\caption{Numerical solution of \eqref{eq:testgeco1} and error plots using MPRK22$(\alpha)$ schemes. The dashed lines represent the reference solution.}\label{Fig:MPRK22}
\end{figure}

\begin{figure}[!h]\centering
	\begin{subfigure}[t]{0.46\textwidth}
		\includegraphics[width=\textwidth]{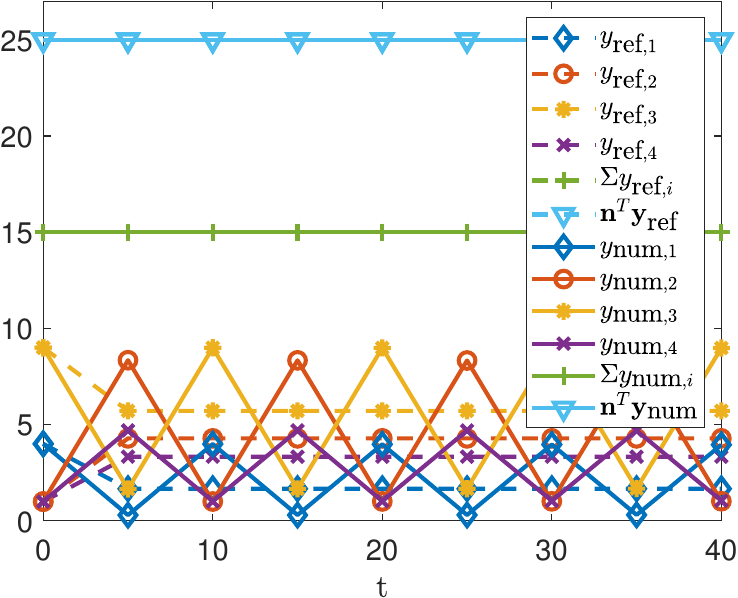}
		\subcaption{$\alpha=0.5$}
	\end{subfigure}
	\begin{subfigure}[t]{0.46\textwidth}
		\includegraphics[width=\textwidth]{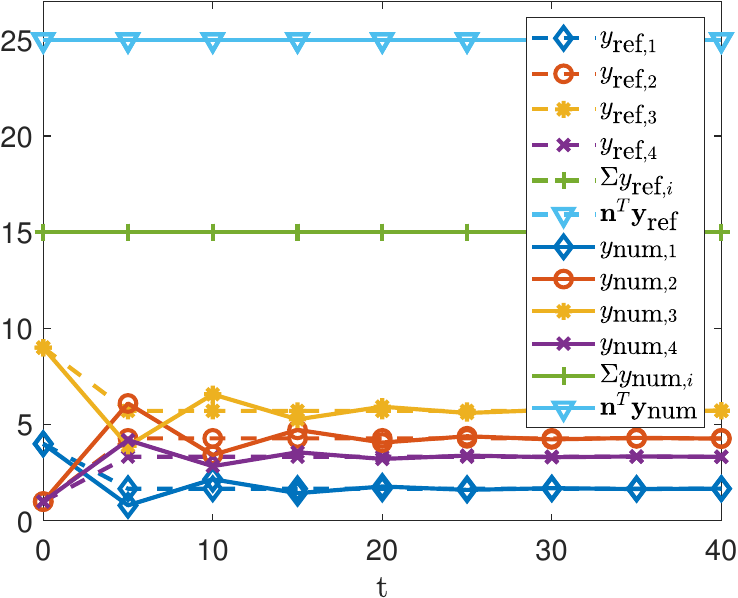}
		\subcaption{$\alpha=1$}
	\end{subfigure}
	\\
	\begin{subfigure}[t]{0.46\textwidth}
		\includegraphics[width=\textwidth]{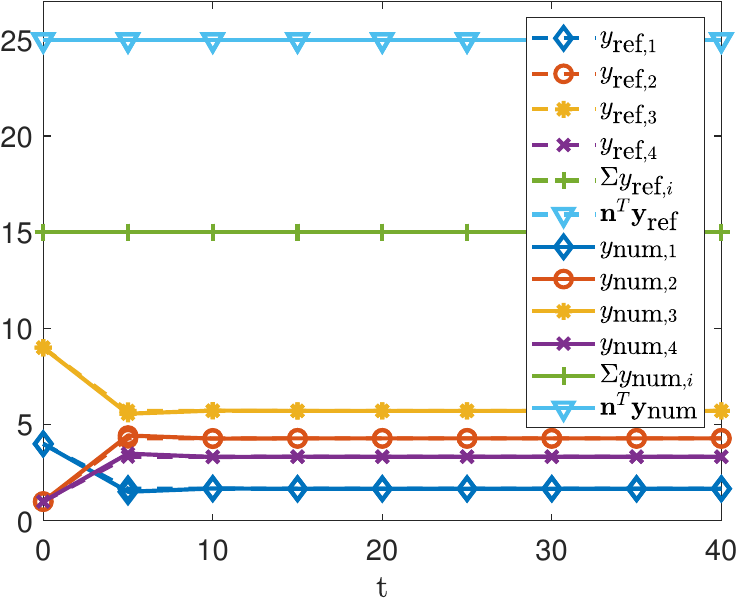}
		\subcaption{$\alpha=5$}
	\end{subfigure}
	\caption{Numerical approximations of \eqref{eq:initProb4dim} using MPRK22($\alpha$) schemes. The dashed lines indicate the exact solution \eqref{eq:exsol4dim} and $\b n=(1,2,2,1)^T$.}\label{Fig:MPRKinitProb4dim}
\end{figure}

\subsubsection{MPRK43($0.5,0.75$)} 
Similarly as before, all theoretical claims for MPRK43($0.5,0.75$) are well reflected in the numerical approximation of \eqref{eq:testgeco1}, see Figure~\ref{fig:MPRK43I} and Figure~\ref{fig:MPRK43I4}.
\begin{figure}[!h]\centering
	\begin{subfigure}[t]{0.46\textwidth}
		\includegraphics[width=\textwidth]{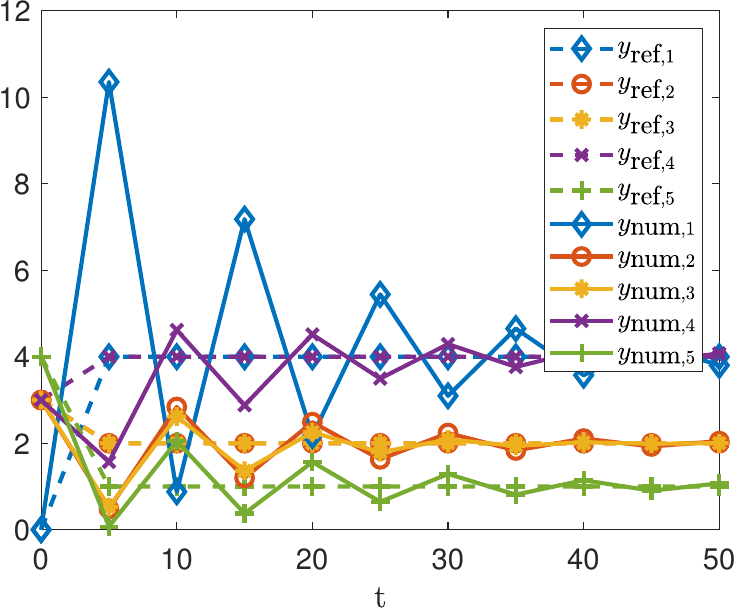}
	\end{subfigure}
	\begin{subfigure}[t]{0.49\textwidth}
		\includegraphics[width=\textwidth]{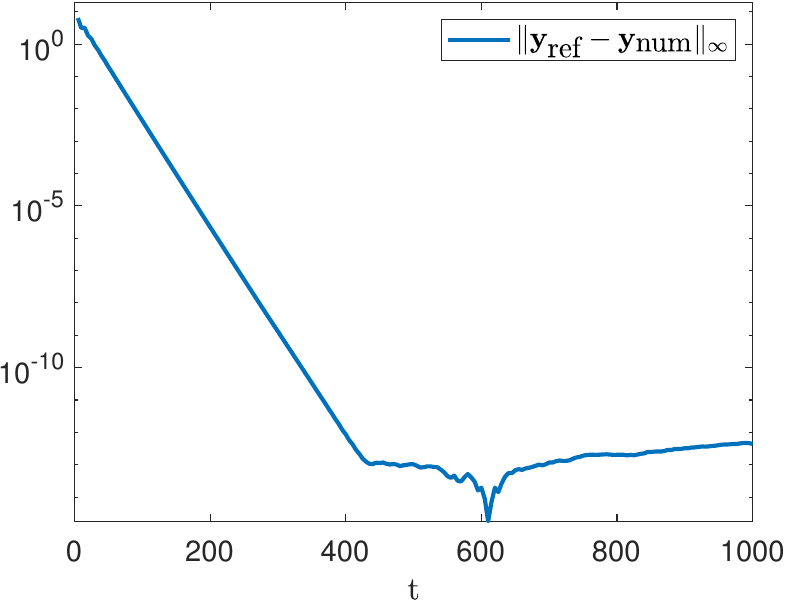}
	\end{subfigure}
	\caption{Numerical solution of \eqref{eq:testgeco1} and error plot using the third order MPRK43(0.5, 0.75) method. The dashed lines represent the reference solution.}\label{fig:MPRK43I}
\end{figure}
\begin{figure}[!h]\centering
	\begin{subfigure}[t]{0.46\textwidth}
		\includegraphics[width=\textwidth]{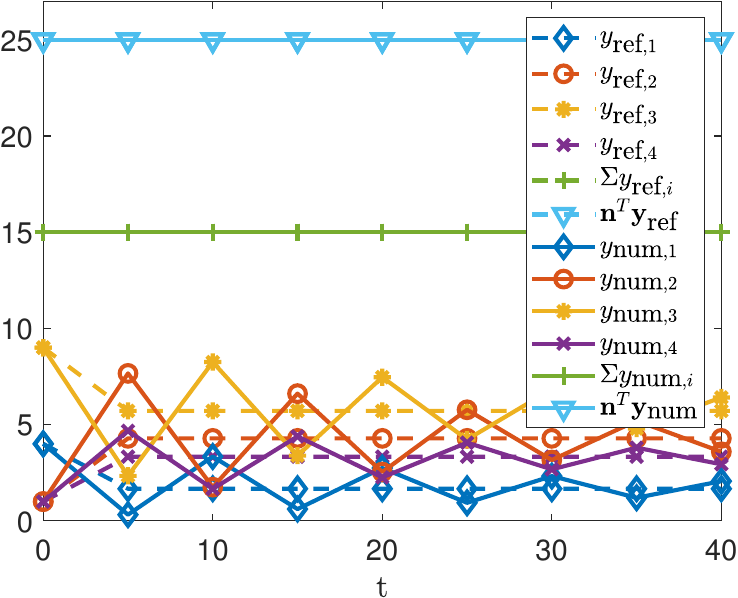}
	\end{subfigure}
	\caption{Numerical approximation of \eqref{eq:initProb4dim} using MPRK43(0.5, 0.75). The dashed lines represent the exact solution \eqref{eq:exsol4dim}  and $\b n=(1,2,2,1)^T$.}\label{fig:MPRK43I4}
\end{figure}
\subsubsection{MPRK43($0.563$)}
According to the investigation in \cite{ITOE22}, MPRK43$(\gamma$) has the largest $\dt$ bound for fulfilling the necessary condition for avoiding oscillations, if $\gamma\approx0.563$. This is why we restrict to this method hereafter. Since this method is also proven to be unconditionally stable, we proceed as for the previously discussed methods. The results can be found in Figure~\ref{fig:MPRK43II} and Figure~\ref{fig:MPRK43II4} and reflect well our theoretical claims from Corollary~\ref{Cor:MPRK43gamma}.
\begin{figure}[!h]\centering
	\begin{subfigure}[t]{0.46\textwidth}
		\includegraphics[width=\textwidth]{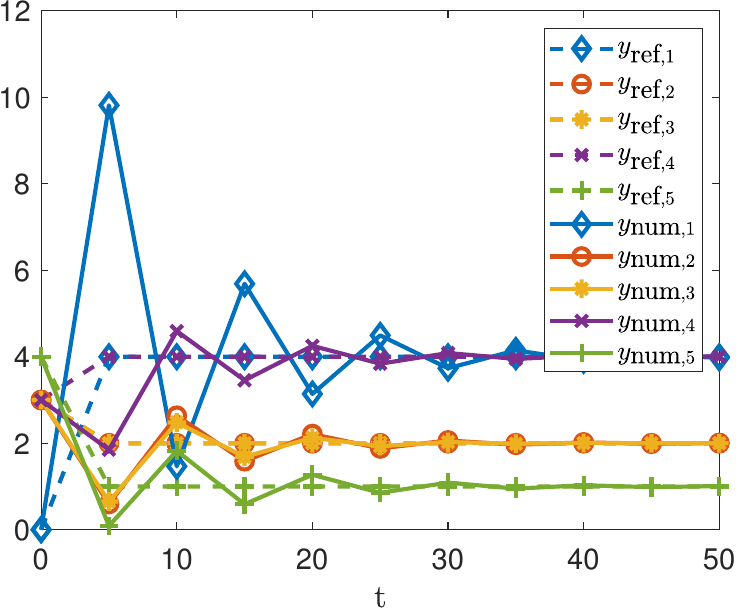}
	\end{subfigure}
	\begin{subfigure}[t]{0.49\textwidth}
		\includegraphics[width=\textwidth]{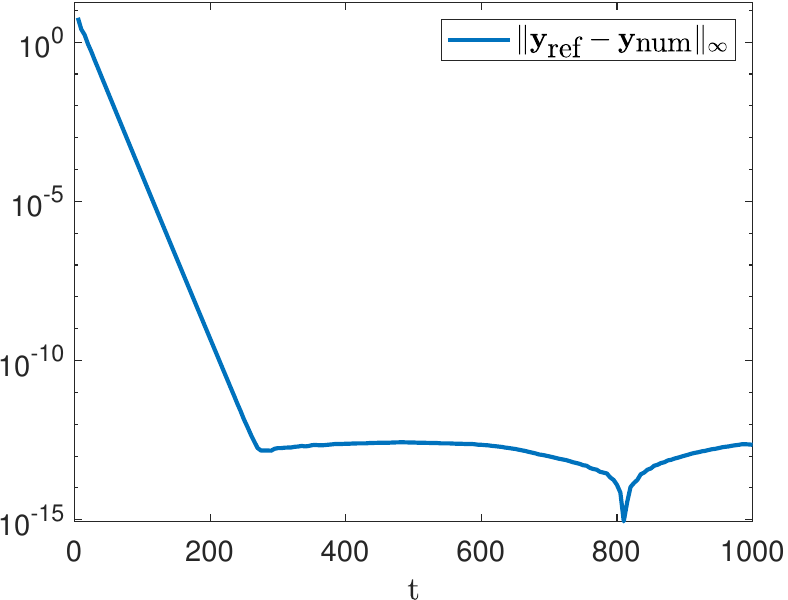}
	\end{subfigure}
	\caption{Numerical solution of \eqref{eq:testgeco1} and error plot using MPRK43(0.563). The dashed lines represent the reference solution.}\label{fig:MPRK43II}
\end{figure}
\begin{figure}[!h]\centering
	\begin{subfigure}[t]{0.46\textwidth}
		\includegraphics[width=\textwidth]{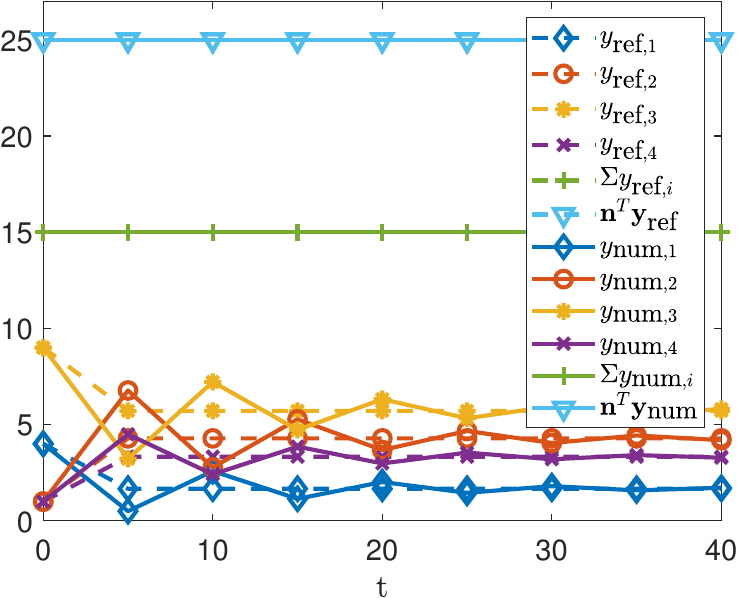}
	\end{subfigure}
	\caption{Numerical approximation of \eqref{eq:initProb4dim} using MPRK43(0.563). The dashed lines represent the exact solution \eqref{eq:exsol4dim}  and $\b n=(1,2,2,1)^T$.}\label{fig:MPRK43II4}
\end{figure}

\subsection{Investigation of SSPMPRK Schemes}
Hereafter, we confirm numerically that SSPMPRK schemes are stable as claimed in Corollary \ref{Cor:SSPMPRK2stab} and Corollary \ref{Cor:SSPMPRK3stab}. 
Furthermore, we investigate the local convergence to the steady state solution as stated in Corollary~\ref{Cor:SSPMPRK2stab1} and Corollary~\ref{Cor:SSPMPRK3stab1} by choosing $\b y^0=\b y(0)$ and $\Delta t=5$, if not stated otherwise. Indeed, in all experiments below the convergence in the stable case can be observed even for $\b y^0=\b y(0)$.

In particular, we are interested in the properties of SSPMPRK3($\frac{1}{3}$) which is the preferred scheme presented in \cite{SSPMPRK3}.
Moreover, we investigate SSPMPRK2($\alpha,\beta$) for three different pairs $(\alpha,\beta)$ covering all cases mentioned in Proposition~\ref{Prop:Stab_SSPMPRK2}.
For the case $\alpha>\frac{1}{2\beta}$ we choose the lower left vertex of the red rectangular from Figure \ref{Fig:alphabetagraph}, i.\,e.\ $(\alpha,\beta)=(0.2,3)$. In this case, we choose different time steps to demonstrate that the computed stability regions are correct. At this point we want to note that the eigenvalues of the system matrices from the test problems lie on the red or blue line depicted in Figure \ref{Fig:Stabmarker}. We scale the time step size $\Delta t$ in such a way that $\Delta t\rho(\b D\b g(\b y^*))=z_i$ for $i\in \{1,2,3,4\}$, respectively, so that for all test cases we consider the cases of stable as well as unstable fixed points.
\begin{figure}
	\centering
	\includegraphics[scale=0.7]{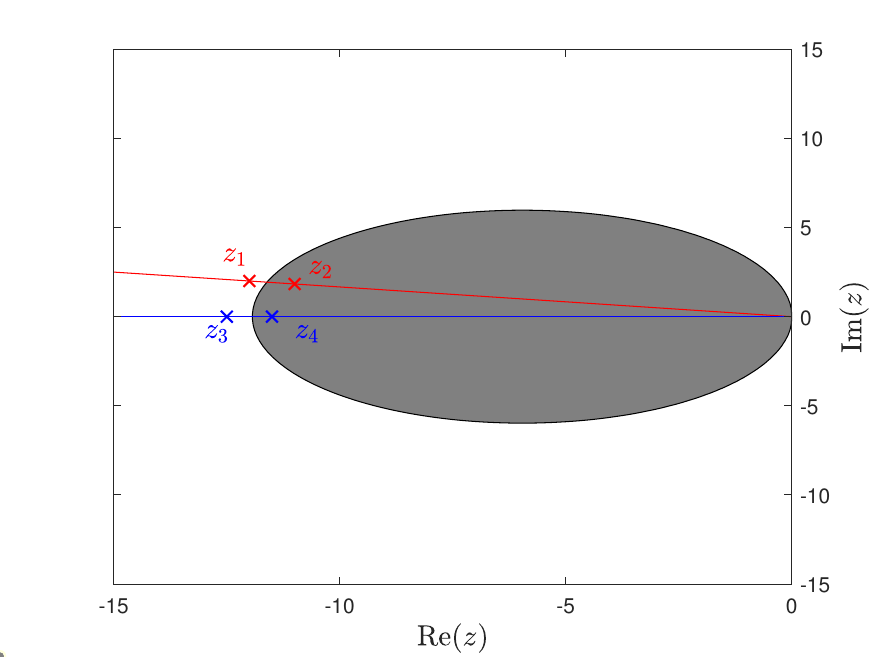}
	\caption{The stability region for SSPMPRK2$(0.2,3)$. The red line is the set $\{a(-6+\ii)\mid 2.5\leq a\leq 0\}$. In particular, the red marked complex numbers are $z_1=2(-6+\ii)$ and $z_2=\frac{11}{6}(-6+\ii)$.
		The blue line is the interval $[-15,0]$. In particular, the blue marked numbers are $z_3=-12.5$ and $z_4=-11.5$.}\label{Fig:Stabmarker}
\end{figure}

As a representative for the case $\alpha=\frac{1}{2\beta}$ we use $(\alpha,\beta)=(\frac12,1)$ which is the preferred choice presented in \cite{SSPMPRK2}. Finally, we choose $(\alpha,\beta)=(0.1,1)$ satisfying $\alpha<\frac{1}{2\beta}$.

\subsubsection{SSPMPRK2($\alpha, \beta$)}
In the subsequent figures, SSPMPRK2($\alpha, \beta$) schemes are used to solve the test problems. In all four figures \ref{Fig:SSPMPRKinitProbReal}, \ref{Fig:SSPMPRKinstabinitProbReal}, \ref{Fig:SSPMPRKinitProbIm}
and \ref{Fig:SSPMPRKinstabinitProbIm}, we can observe the same qualitative behavior. In Figure~\ref{Fig:SSPMPRKinitProbReal} and Figure~\ref{Fig:SSPMPRKinitProbIm}, the preferred choice of $(\alpha,\beta)=(\frac{1}{2},1)$ seems to be less damping than $(\alpha,\beta)=(0.1,1)$. However, in both cases a convergence towards the steady state solution can be observed. In Figure~\ref{Fig:SSPMPRKinstabinitProbReal} and Figure~\ref{Fig:SSPMPRKinstabinitProbIm}, the pair $(\alpha,\beta)$ lies in the critical region where the stability domain is bounded. If $\Delta t$ is chosen in such a way that $\Delta t\rho(\b D\b g(\b y^*))=z_i$ for $i=2$ or $i=4$, respectively, see Figure \ref{Fig:Stabmarker}, the numerical approximations behave as expected converging towards the corresponding steady state which is a stable fixed point of the method. However, increasing $\Delta t$ by approximately $2\cdot 10^{-3}$, we find that $\Delta t\rho(\b D\b g(\b y^*))=z_i$ for $i=1$ or $i=3$, respectively. As a result, even when we modify the starting vector to be $\b y^0=\b y^*+10^{-5}\b v$ with $\b v=(1,-2,1)^T$, the numerical approximation diverges from the steady state as predicted by the presented theory, can be observed. All parameters however lead to a scheme that also preserve the second linear invariant as Figure~\ref{Fig:SSPMPRKinitProb4Dim} suggests.
\begin{figure}[!h]\centering
	\begin{subfigure}[t]{0.46\textwidth}
		\includegraphics[width=\textwidth]{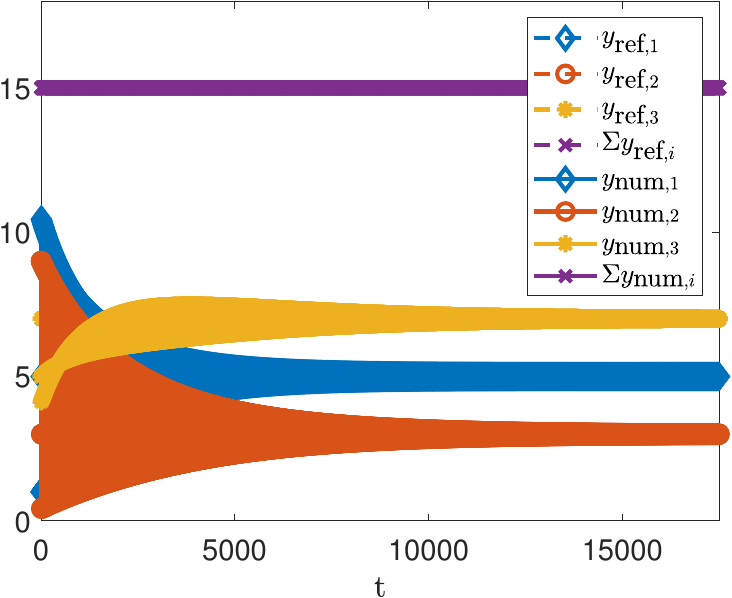}
		\subcaption{$(\alpha,\beta)=(\frac12,1)$, $\Delta t=5$}
	\end{subfigure}
	\begin{subfigure}[t]{0.475\textwidth}
		\includegraphics[width=\textwidth]{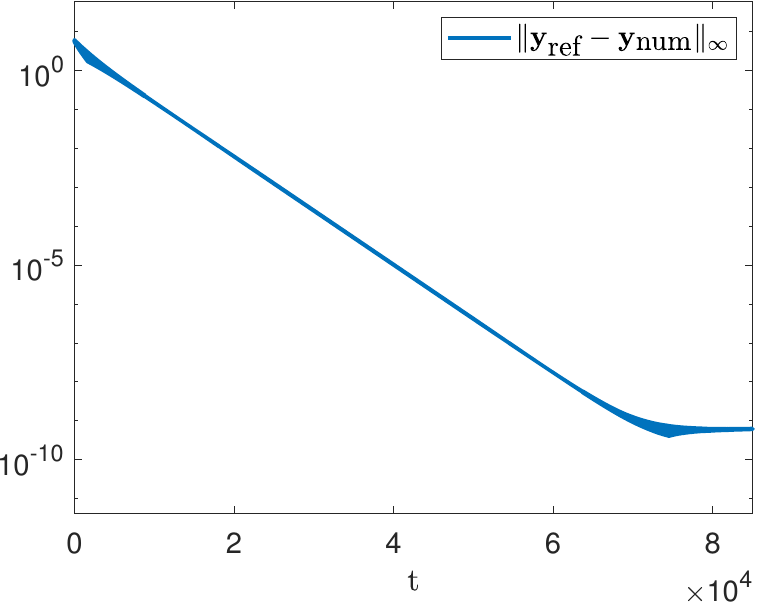}
		\subcaption{$(\alpha,\beta)=(\frac12,1)$, $\Delta t=5$}
	\end{subfigure}\\
	\begin{subfigure}[t]{0.46\textwidth}
		\includegraphics[width=\textwidth]{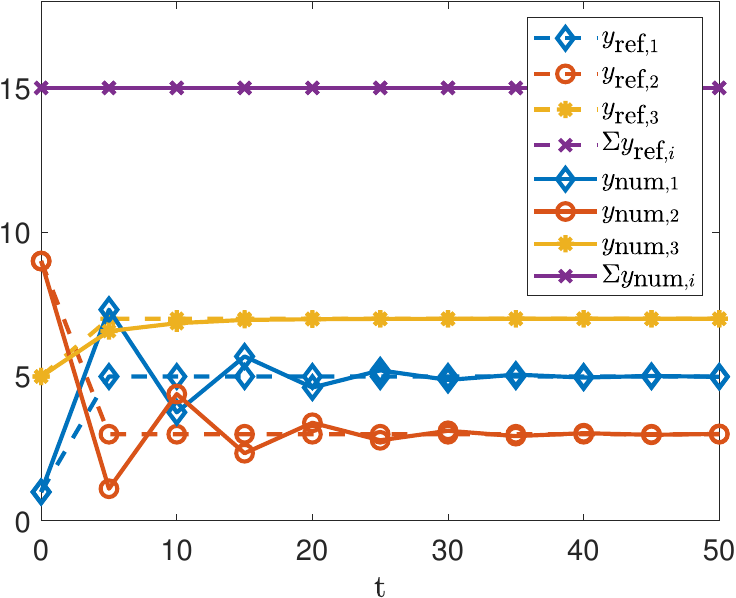}
		\subcaption{$(\alpha,\beta)=(0.1,1)$, $\Delta t=5$}
	\end{subfigure}
	\begin{subfigure}[t]{0.49\textwidth}
		\includegraphics[width=\textwidth]{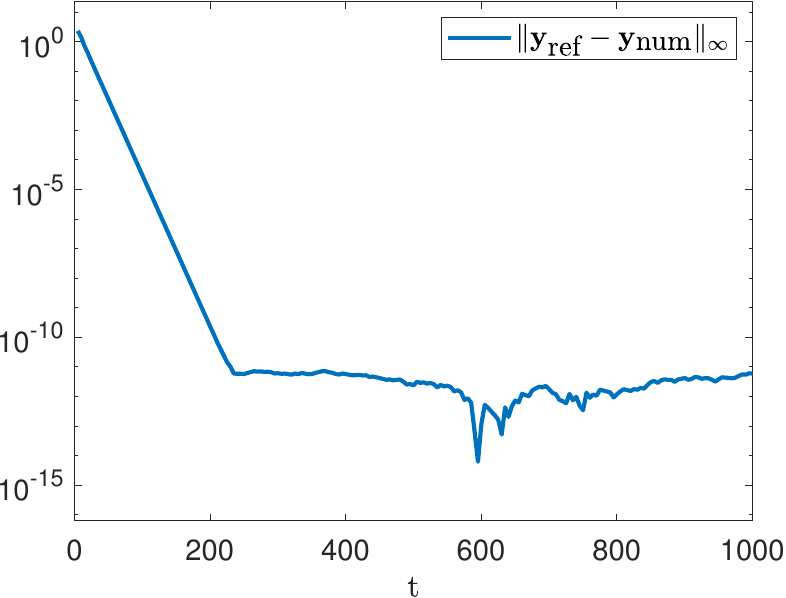}
		\subcaption{$(\alpha,\beta)=(0.1,1)$, $\Delta t=5$}
	\end{subfigure}
	\caption{Numerical approximations of \eqref{eq:initProbReal} using the SSPMPRK2 scheme. The dashed lines indicate the exact solution \eqref{eq:exsolReal}. }\label{Fig:SSPMPRKinitProbReal}
\end{figure}
\begin{figure}\centering
	\begin{subfigure}[t]{0.46\textwidth}
		\includegraphics[width=\textwidth]{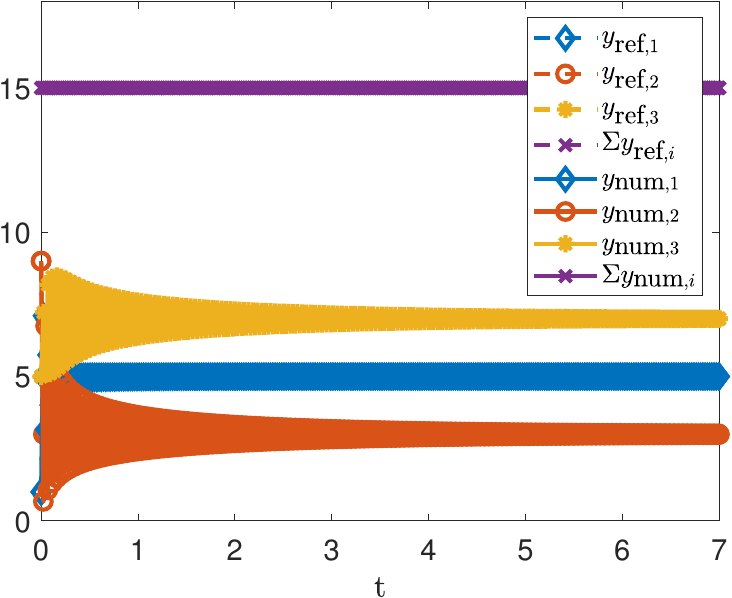}
		\subcaption{$\Delta t= 0.023$}
	\end{subfigure}
	\begin{subfigure}[t]{0.49\textwidth}
		\includegraphics[width=\textwidth]{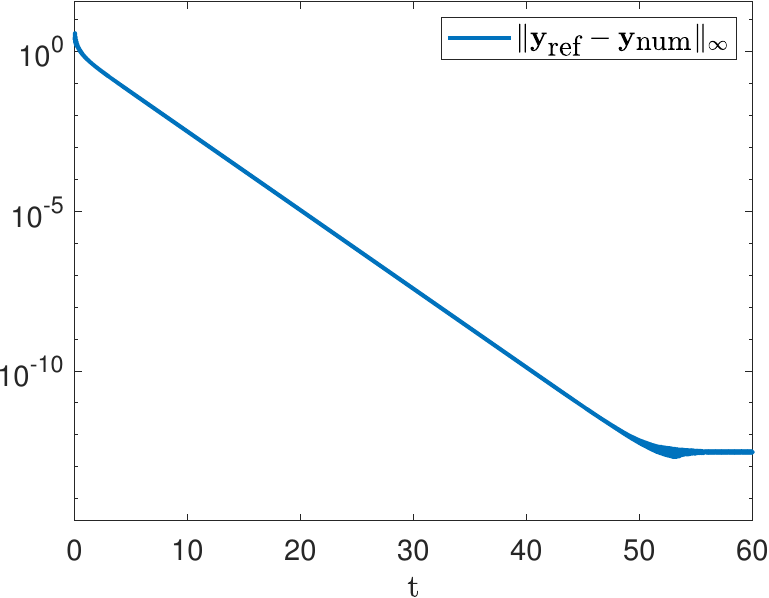}
		\subcaption{$\Delta t = 0.023$}
	\end{subfigure}\\
	\begin{subfigure}[t]{0.49\textwidth}
		\includegraphics[width=\textwidth]{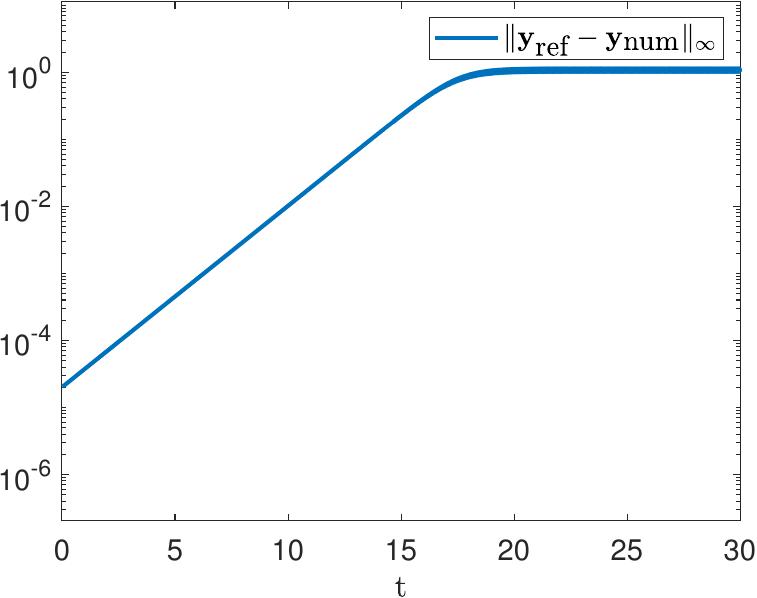}
		\subcaption{$\Delta t = 0.025$}\label{Subfig:DReal}
	\end{subfigure}
	\caption{Numerical approximations of \eqref{eq:initProbReal} using the SSPMPRK2($0.2,3$) scheme. The dashed lines indicate the exact solution \eqref{eq:exsolReal}. In \ref{Subfig:DReal}, we used $\b y^0=\b y^*+10^{-5}(1,-2,1)^T$. }\label{Fig:SSPMPRKinstabinitProbReal}
\end{figure}
\begin{figure}[!h]\centering
	\begin{subfigure}[t]{0.46\textwidth}
		\includegraphics[width=\textwidth]{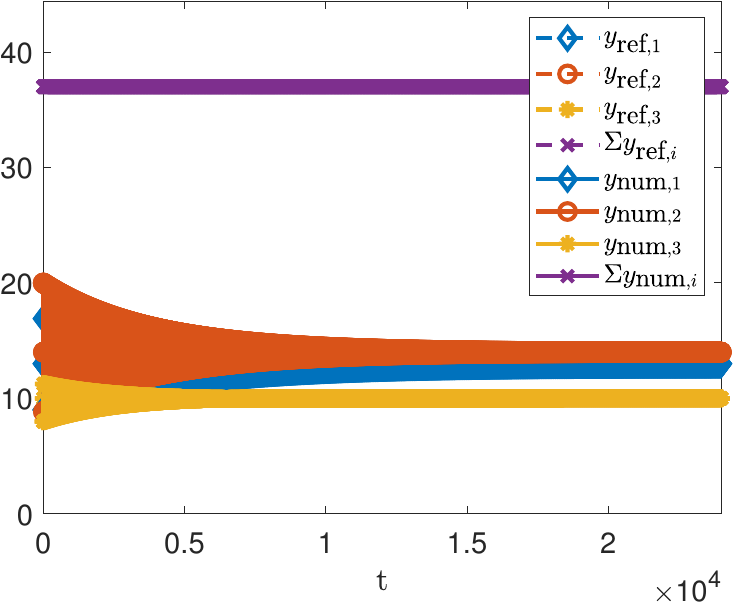}
		\subcaption{$(\alpha,\beta)=(\frac12,1)$, $\Delta t=5$}
	\end{subfigure}
	\begin{subfigure}[t]{0.48\textwidth}
		\includegraphics[width=\textwidth]{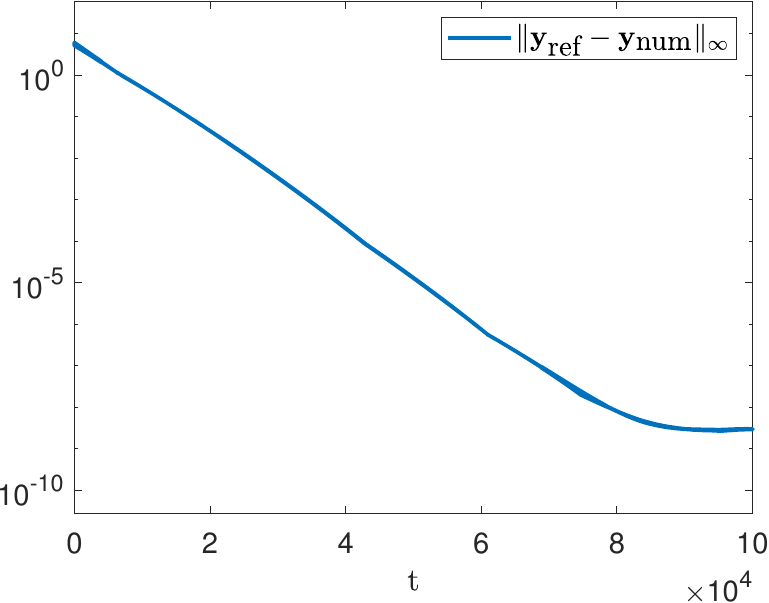}
		\subcaption{$(\alpha,\beta)=(\frac12,1)$, $\Delta t=5$}
	\end{subfigure}\\
	\begin{subfigure}[t]{0.46\textwidth}
		\includegraphics[width=\textwidth]{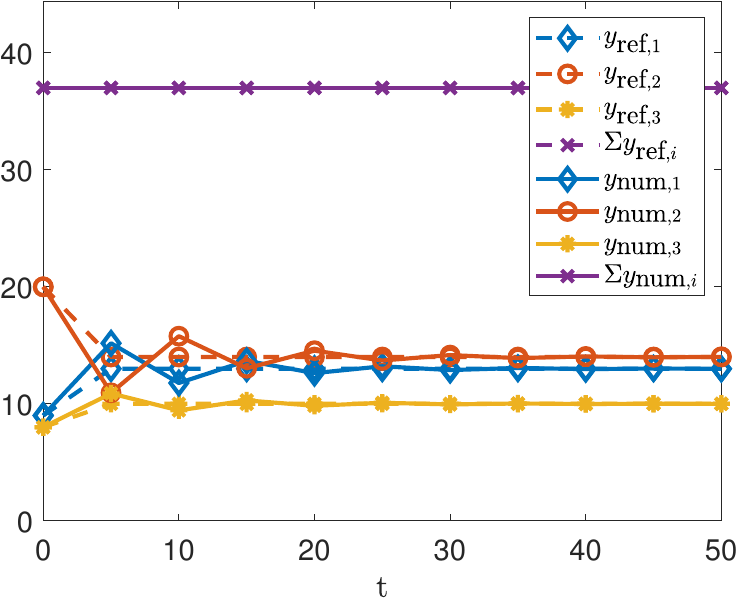}
		\subcaption{$(\alpha,\beta)=(0.1,1)$, $\Delta t=5$}
	\end{subfigure}
	\begin{subfigure}[t]{0.49\textwidth}
		\includegraphics[width=\textwidth]{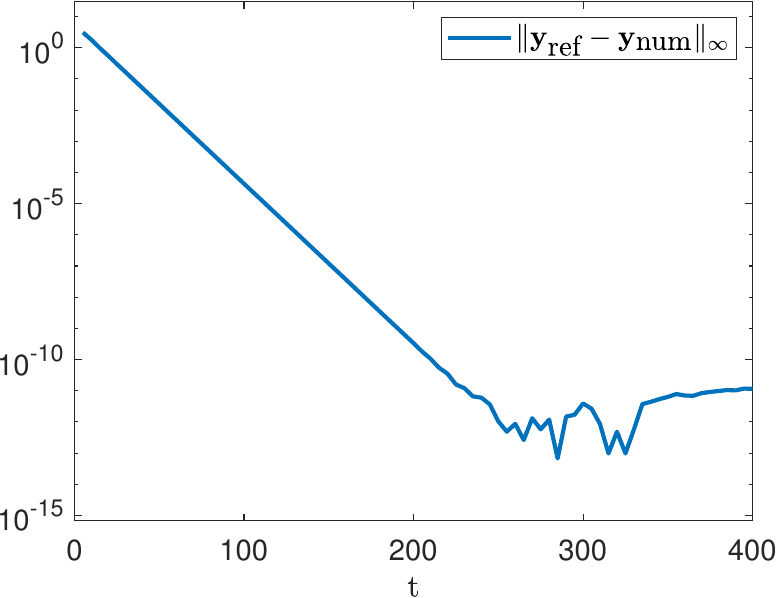}
		\subcaption{$(\alpha,\beta)=(0.1,1)$, $\Delta t=5$}
	\end{subfigure}
	\caption{Numerical approximations of \eqref{eq:initProbIm} using  the second order SSPMPRK scheme. The dashed lines indicate the exact solution \eqref{eq:exsolIm}.}\label{Fig:SSPMPRKinitProbIm}
\end{figure}
\begin{figure}[!h]\centering
	\begin{subfigure}[t]{0.46\textwidth}
		\includegraphics[width=\textwidth]{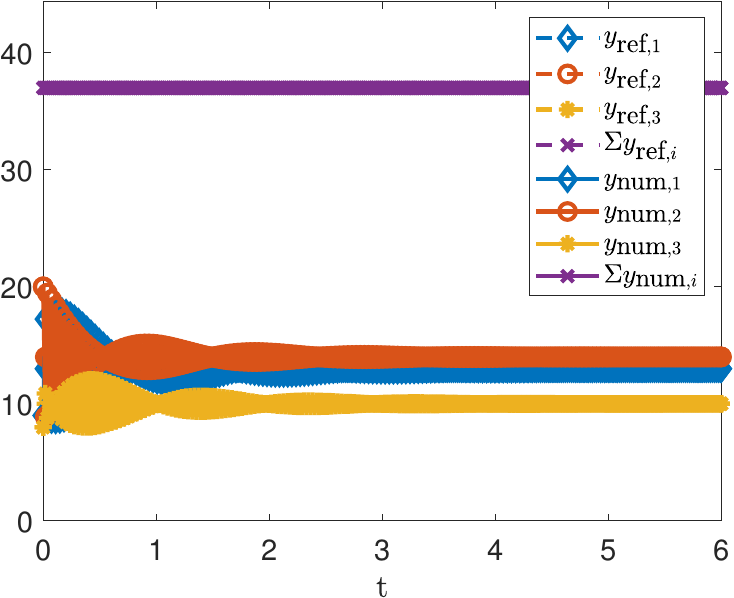}
		\subcaption{$\Delta t = 0.0183$}
	\end{subfigure}
	\begin{subfigure}[t]{0.49\textwidth}
		\includegraphics[width=\textwidth]{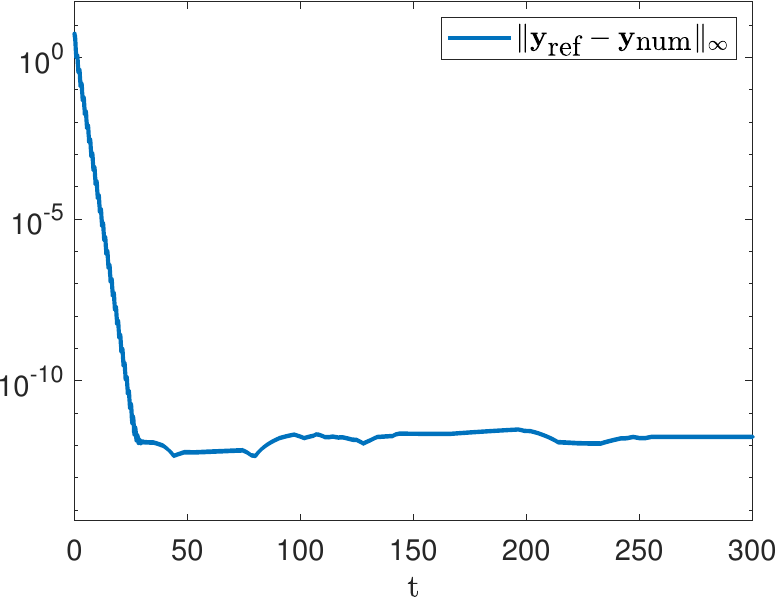}
		\subcaption{$\Delta t = 0.183$}
	\end{subfigure}\\
	\begin{subfigure}[t]{0.49\textwidth}
		\includegraphics[width=\textwidth]{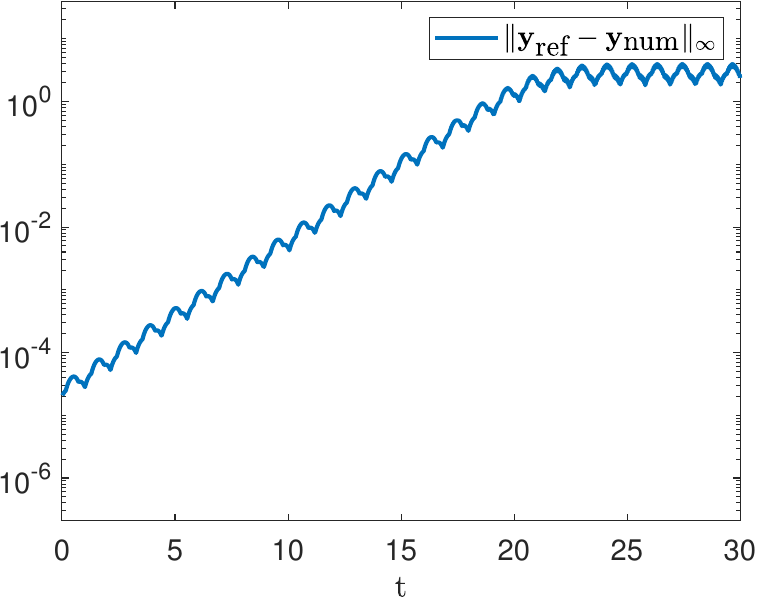}
		\subcaption{$\Delta t = 0.020$}\label{Subfig:DIm}
	\end{subfigure}
	\caption{Numerical approximations of \eqref{eq:initProbIm} using the SSPMPRK22($0.2,3$) scheme. The dashed lines indicate the exact solution \eqref{eq:exsolIm}. In \ref{Subfig:DIm}, the initial vector $\b y^0=\b y^*+10^{-5}(1,-2,1)^T$ is chosen. }\label{Fig:SSPMPRKinstabinitProbIm}
\end{figure}
\begin{figure}[!h]\centering
	\hspace{-0.8cm}
	\begin{subfigure}[t]{0.53\textwidth}
		\includegraphics[width=\textwidth]{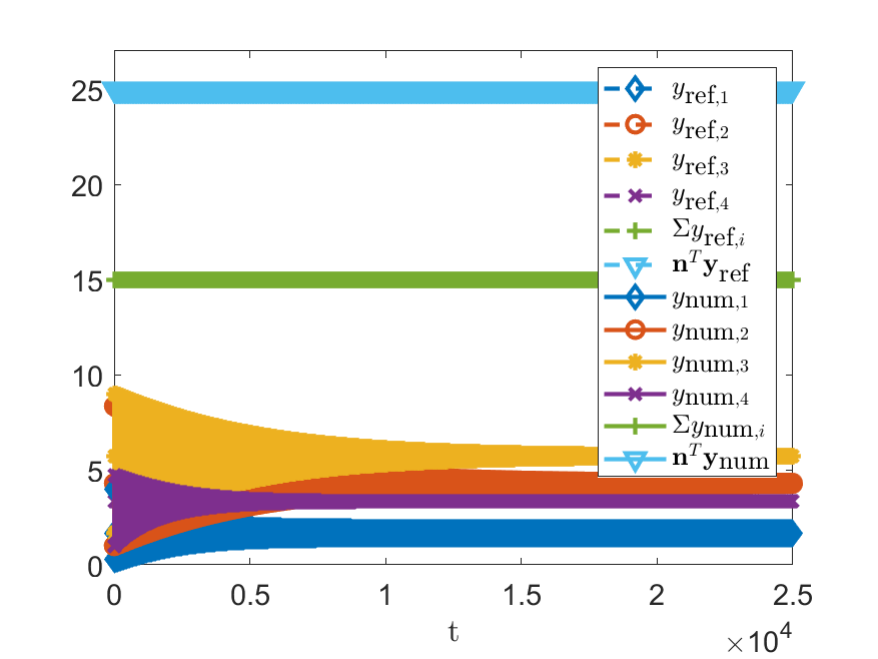}
		\subcaption{$(\alpha,\beta)=(\frac12,1)$, $\Delta t=5$}
	\end{subfigure}
	\begin{subfigure}[t]{0.46\textwidth}
		\includegraphics[width=\textwidth]{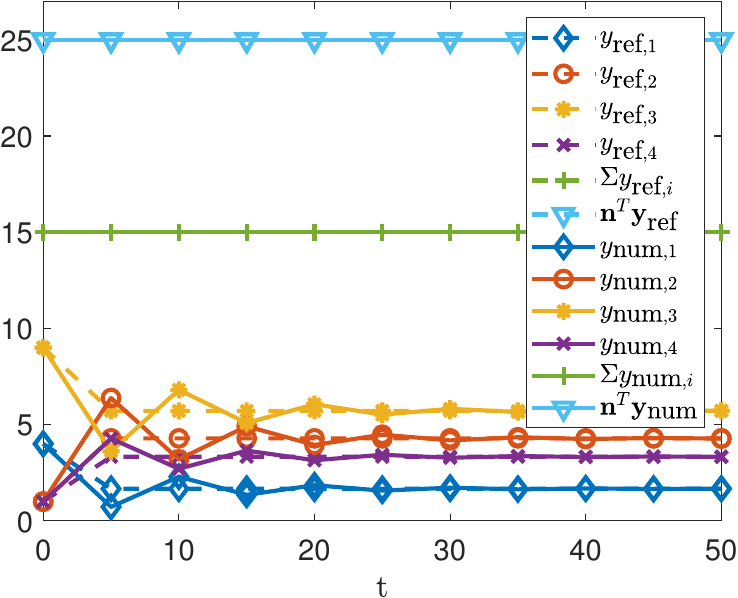}
		\subcaption{$(\alpha,\beta)=(0.1,1)$, $\Delta t=5$}	\end{subfigure}\\
	\begin{subfigure}[t]{0.46\textwidth}
		\includegraphics[width=\textwidth]{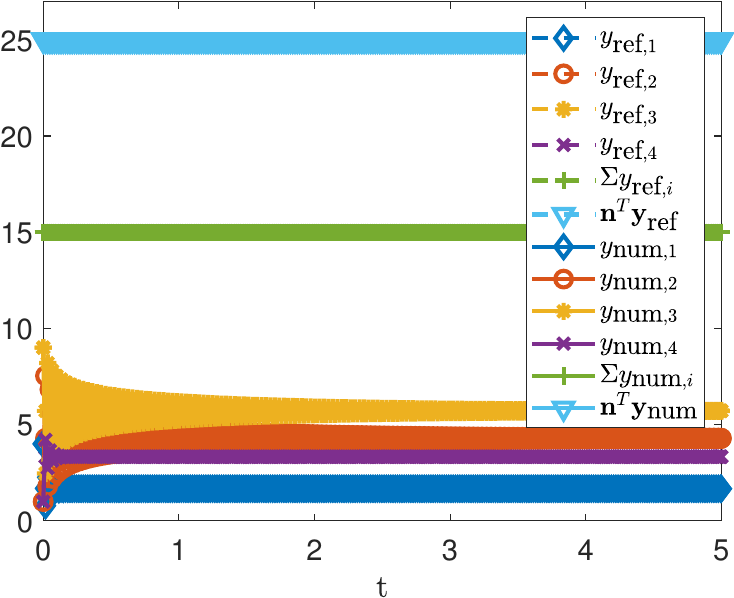}
		\subcaption{$(\alpha,\beta)=(0.2,3)$, $\Delta t = 0.0164$}
	\end{subfigure}
	\caption{Numerical approximations of \eqref{eq:initProb4dim} using the second order SSPMPRK scheme. The dashed lines indicate the exact solution \eqref{eq:exsol4dim}, where $\b n=(1,2,2,1)^T$.}\label{Fig:SSPMPRKinitProb4Dim}
\end{figure}
%

\subsubsection{SSPMPRK3($\frac13$)}
In Figure~\ref{Fig:SSPMPRK3initProblems}, the SSPMPRK3($\frac13$) scheme is used to integrate the test problems \eqref{eq:testgeco1} with mixed eigenvalues and \eqref{eq:initProb4dim} with a second linear invariant. The numerical experiments support the theoretical claims, \ie the fixed points seem to be stable and locally attracting. Moreover, all linear invariants are conserved by the method.

Altogether, the numerical experiments support very well the theoretical results from Section \ref{sec:stab_SSPMPRK} on SSPMPRK methods.
\begin{figure}[!h]\centering
	\begin{subfigure}[t]{0.46\textwidth}
		\includegraphics[width=\textwidth]{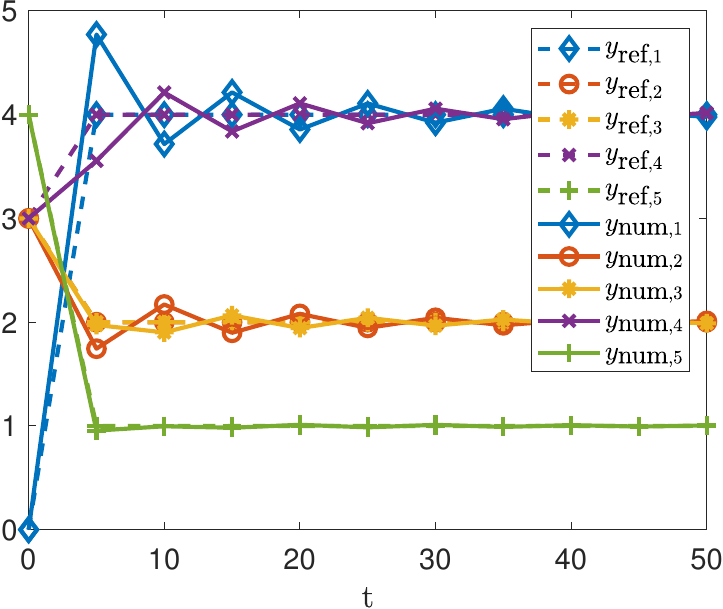}
	\end{subfigure}
	\begin{subfigure}[t]{0.49\textwidth}
		\includegraphics[width=\textwidth]{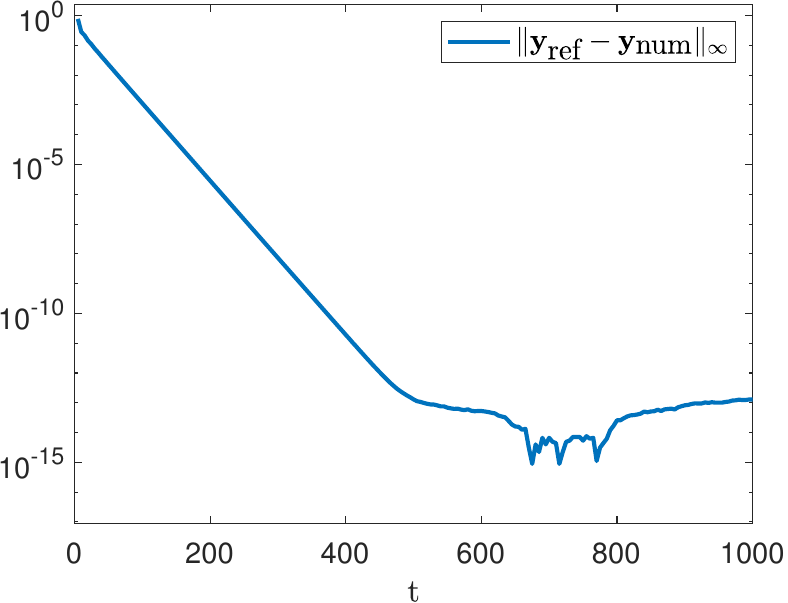}
	\end{subfigure}\\
	\begin{subfigure}[t]{0.46\textwidth}
		\includegraphics[width=\textwidth]{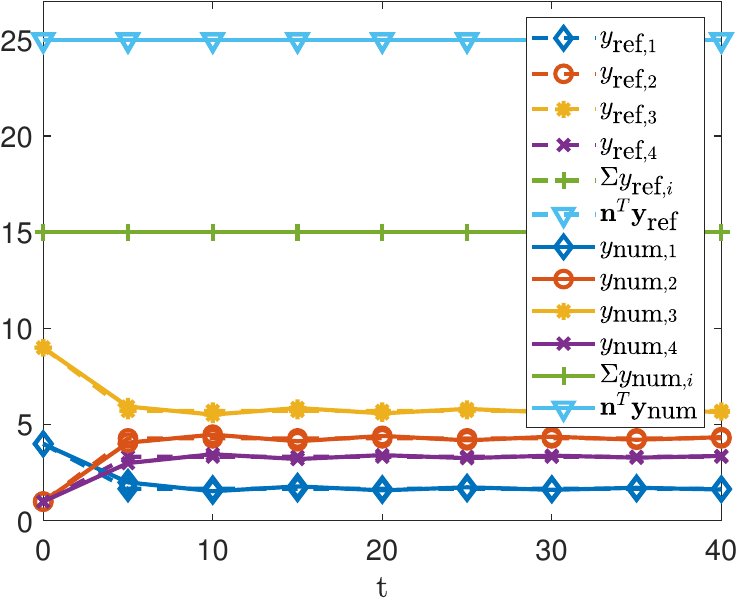}
	\end{subfigure}
	\caption{Numerical approximations of \eqref{eq:testgeco1} with error plot and \eqref{eq:initProb4dim} using SSPMPRK3($\frac13$) schemes and $\dt=5$. The dashed lines indicate the reference solutions and $\b n=(1,2,2,1)^T$.}\label{Fig:SSPMPRK3initProblems}
\end{figure}

\subsection{Investigation of MPDeC Schemes}
In this section we restrict to the investigation of MPDeC schemes with equidistant nodes and refer to \cite{IOE22StabMP} for the numerical experiments concerning MPDeCGL methods. As we have discovered in Figure~\ref{fig:stability}, MPDeCEQ($p$) for $p=12$ and $p=14$ have a bounded stability domain for problems with exclusively real eigenvalues. Moreover, it was observed in \cite[Figure B.9]{IssuesMPRK} that the iterates of MPDeCEQ$(8)$ only locally converge towards the steady state. This is in accordance with the presented theory, however, we did not observe this behavior within the numerical experiments of the previously discussed schemes. Nevertheless, MPDeCEQ$(8)$ is not the only scheme with that rather unpleasant property. Indeed, in \cite{izgin2023nonglobal}, which is based on the master thesis \cite{Schilling2023}, the authors demonstrate that this phenomenon also occurs with MPRK22$(\alpha)$ schemes for $\alpha<\frac12$. The common circumstance for both schemes is that both are based on RK methods with non-positive Butcher tableau. The resulting hypothesis was tested and supported with numerical experiments in \cite{izgin2023nonglobal}. 

Nevertheless, we want to mention that if we violate the stability condition, we can start arbitrary close to the steady state solution, and still, the iterates will not converge to $\b y^*$. 

Now, we reproduce the result from \cite{IssuesMPRK} investigating MPDeCEQ$(8)$, see Figure~\ref{fig:mpdeceq8}. Furthermore, we present experiments with the $12$th and $14$th order method when applied to \eqref{eq:initProbReal}, see Figure~\ref{fig:mpdec12} and Figure~\ref{fig:mpdec14}. In both cases the largest time step size is chosen such that $\lvert R(\dt^\pm \rho(\bA))\rvert =1\pm 0.1$ for the stable and unstable scenario, respectively, see Figure~\ref{fig:stability} for the graph of the stability functions. Since $\rho(\bA)=500$, the time step sizes for MPDeCEQ($12$) are \[\dt^+_{\eq(12)}\approx\frac{59}{500}=0.118\qta\dt^-_{\eq(12)}\approx\frac{20}{500}=0.04.\] In the case of MPDeCEQ($14$) they are \[\dt^+_{\eq(14)}=\approx\frac{12}{500}=0.024\qta\dt^-_{\eq(14)}\approx\frac{7.6}{500}=0.0152.\] Overall, the expected behavior can be observed. 
\begin{figure}[!h]\centering
	\begin{subfigure}[t]{0.49\textwidth}
		\includegraphics[width=\textwidth]{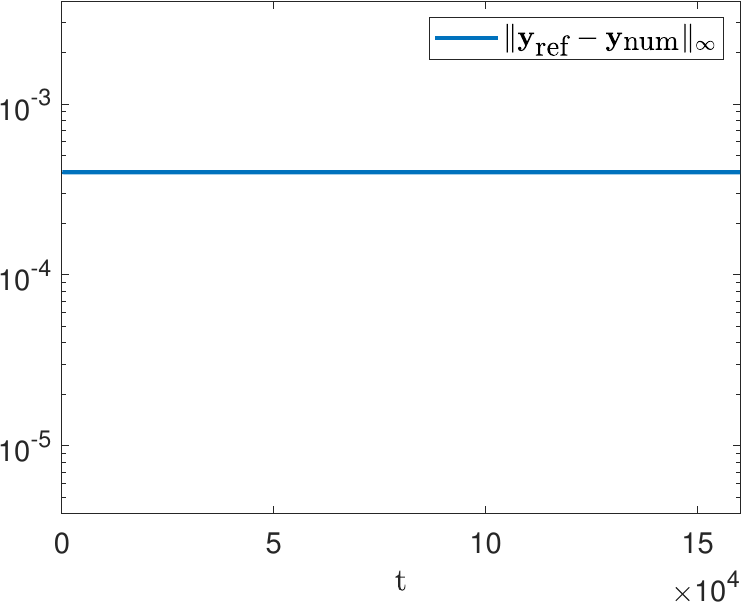}
		\subcaption{$\epsilon=0.1$}
	\end{subfigure}
	\begin{subfigure}[t]{0.498\textwidth}
		\includegraphics[width=\textwidth]{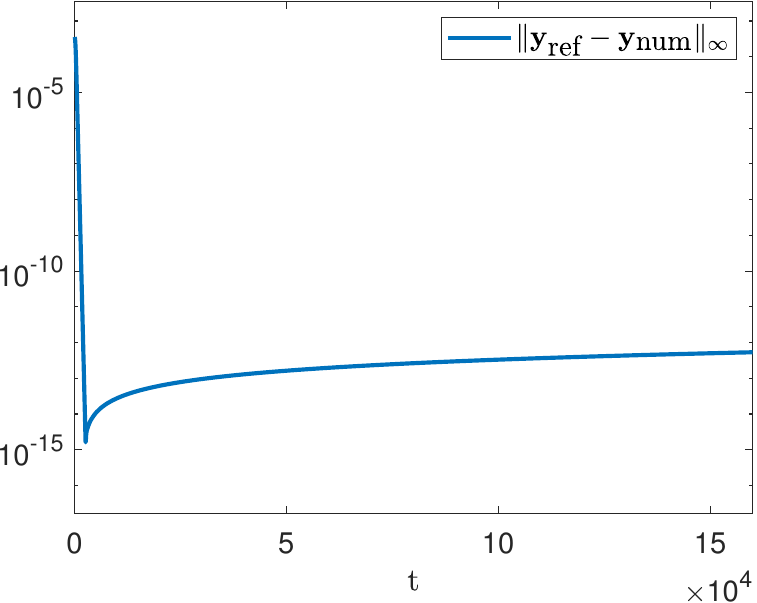}
		\subcaption{$\epsilon=0.01$}
	\end{subfigure}
	\caption{Error of the numerical solution of \eqref{eq:Testprob} using the MPDeCEQ$(8)$ scheme and $\b y^0=\b y^*+\epsilon(-1,1)^T$. }\label{fig:mpdeceq8}
\end{figure}
\begin{figure}[!h]\centering
	\begin{subfigure}[t]{0.47\textwidth}
		\includegraphics[width=\textwidth]{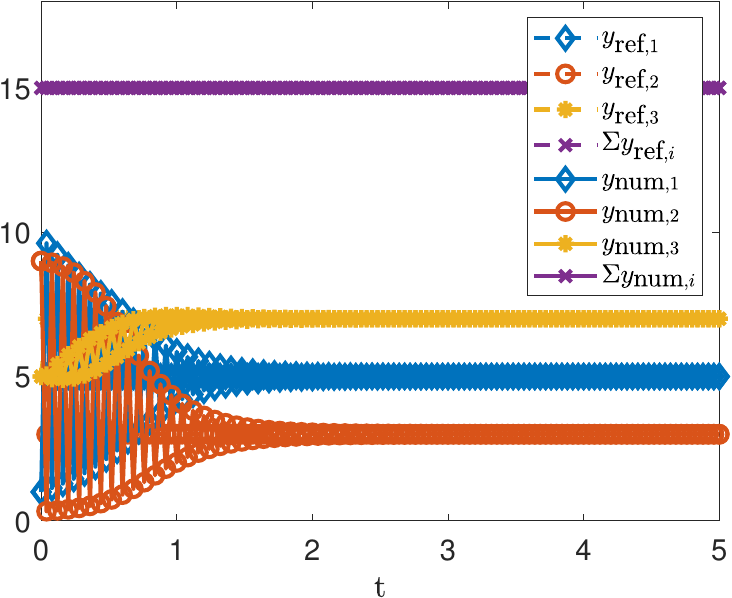}
		\subcaption{$\Delta t= 0.023$}
	\end{subfigure}
	\begin{subfigure}[t]{0.49\textwidth}
		\includegraphics[width=\textwidth]{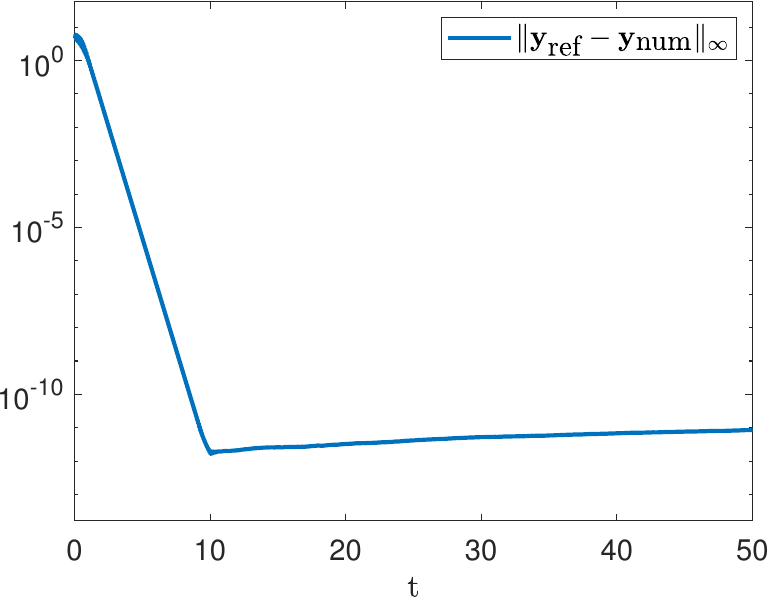}
		\subcaption{$\Delta t = 0.023$}
	\end{subfigure}\\
	\begin{subfigure}[t]{0.49\textwidth}
		\includegraphics[width=\textwidth]{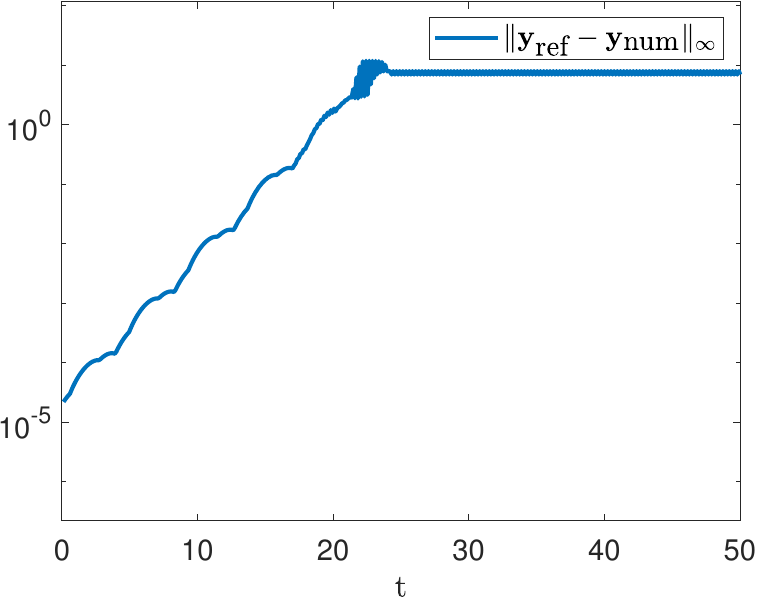}
		\subcaption{$\Delta t = 0.025$}\label{Subfig:mpdecDReal}
	\end{subfigure}
	\caption{Numerical approximations of \eqref{eq:initProbReal} using the MPDeCEQ($12$) scheme. The dashed lines indicate the exact solution \eqref{eq:exsolReal}. In Figure~\ref{Subfig:mpdecDReal}, we used $\b y^0=\b y^*+10^{-5}(1,-2,1)^T$. }\label{fig:mpdec12}
\end{figure}

\begin{figure}[!h]\centering
	\begin{subfigure}[t]{0.47\textwidth}
		\includegraphics[width=\textwidth]{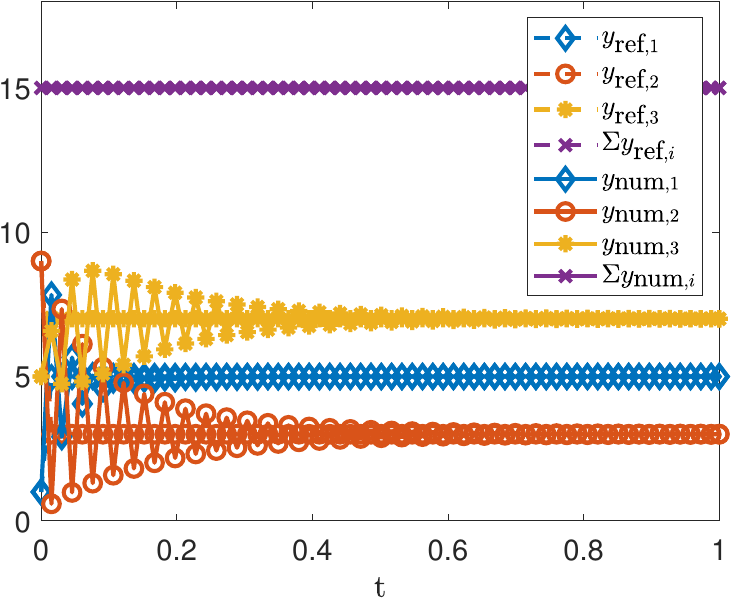}
		\subcaption{$\Delta t= 0.023$}
	\end{subfigure}
	\begin{subfigure}[t]{0.49\textwidth}
		\includegraphics[width=\textwidth]{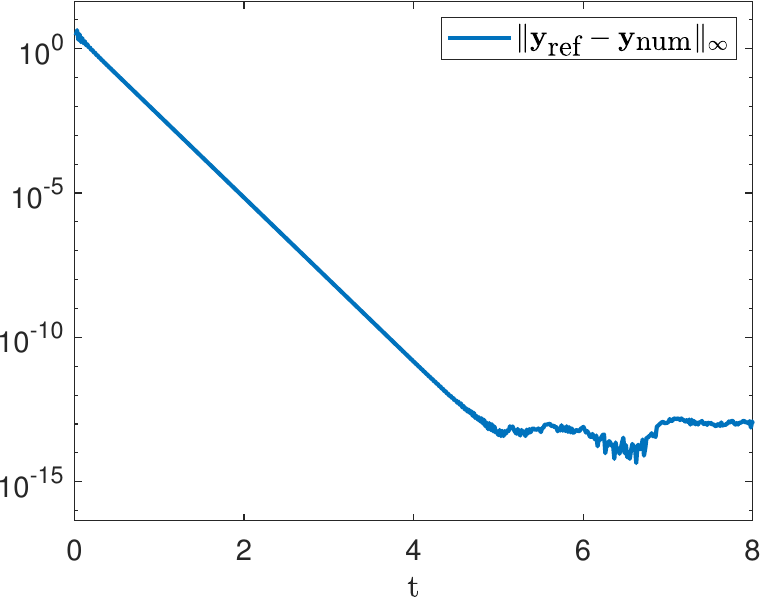}
		\subcaption{$\Delta t = 0.023$}
	\end{subfigure}\\
	\begin{subfigure}[t]{0.49\textwidth}
		\includegraphics[width=\textwidth]{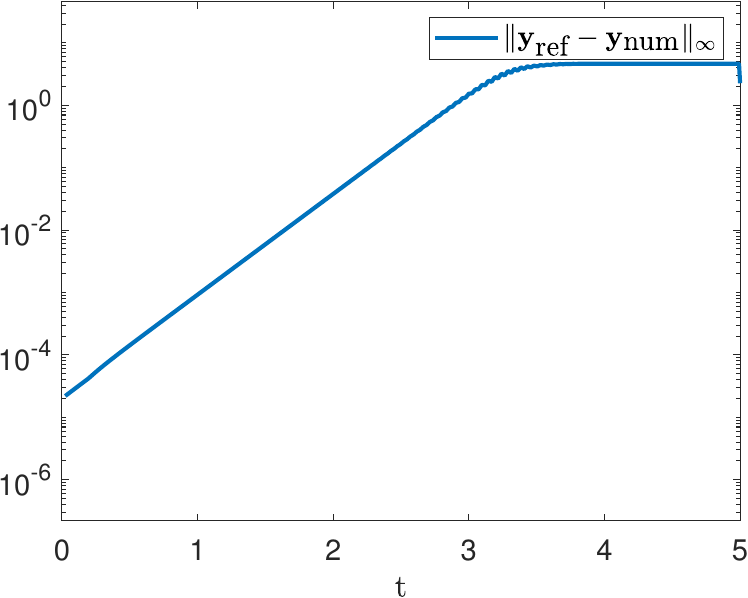}
		\subcaption{$\Delta t = 0.025$}\label{Subfig:mpdec14DReal}
	\end{subfigure}
	\caption{Numerical approximations of \eqref{eq:initProbReal} using the MPDeCEQ$14$) scheme. The dashed lines indicate the exact solution \eqref{eq:exsolReal}. In \ref{Subfig:mpdec14DReal}, the initial vector $\b y^0=\b y^*+10^{-5}(1,-2,1)^T$ is chosen. }\label{fig:mpdec14}
\end{figure}

\subsection{Investigation of GeCo Schemes}
\subsubsection{GeCo1}
Numerical solutions obtained by GeCo1 and the corresponding error plots are shown in Figure \ref{Fig:geco1}. In error plot \ref{errgeco11}, the convergence of the numerical solution to the steady state in the long run  can be seen, despite the low accuracy in the short run with the comparatively large time step of $\Delta t=5$. Hence, the result from Theorem~\ref{Thm:GeCo1} is well reflected here. Nevertheless, a shift of the numerical solution can be recognized for the chosen time step size. This can also be observed in Section~\ref{subsec:GeCostiff}, where we apply the method to increasingly stiff problems.



\begin{figure}
	\begin{subfigure}[t]{0.45\textwidth}
		\includegraphics[width=\textwidth]{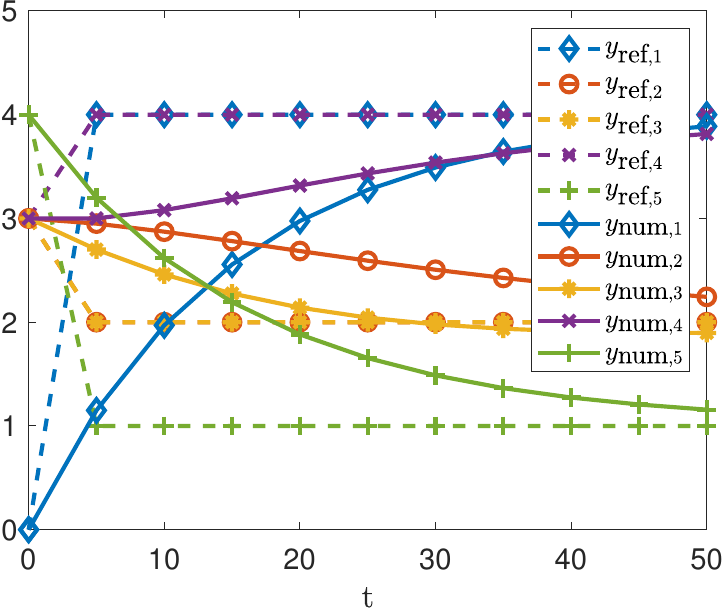}
		\subcaption{$\Delta t=5$ }\label{geco11}
	\end{subfigure}
	\begin{subfigure}[t]{0.49\textwidth}%
		\includegraphics[width=\textwidth]{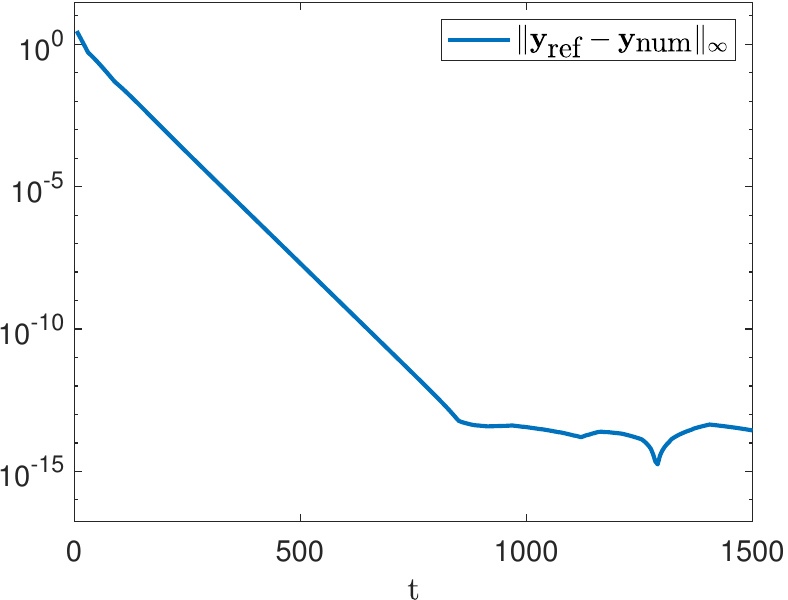}
		\subcaption{$\Delta t=5$}\label{errgeco11}
	\end{subfigure}
	\caption{Numerical approximations of \eqref{eq:testgeco1} and error plot using GeCo1.}
	\label{Fig:geco1}
\end{figure}

\subsubsection{GeCo2}
Based on the analysis for the system \eqref{PDS_test}, we use the function  \[R(z)=1+z+\frac12 z^2\varphi(\Delta t\tr(\b S^-))\]
even in the context of \eqref{eq:testgeco1} to determine the critical time step size $\Delta t_\text{GeCo2}$ of GeCo2.
For the system matrix \eqref{eq:Atest}, we find $\tr(\b S^-)=-\tr(\bA)=20$. A numerical calculation shows that $\abs{R(\Delta t\lambda)}<1$ for all $\lambda\in \sigma(\bA)\setminus\{0\}$ if $\Delta t<\Delta t_{\text{GeCo2}}\approx 0.3572$, where $\Delta t_{\text{GeCo2}}$ was rounded to five significant figures. Moreover, $\abs{R(\Delta t(-5-\sqrt{3}))}>1$ if $\Delta t>\Delta t_{\text{GeCo2}}$. 

In order to numerically confirm the stability results from Theorem~\ref{Thm:geco2} even in the context of the model problem \eqref{eq:testgeco1}, we solve the initial value problem \eqref{eq:testgeco1} using $\Delta t=\Delta t_{\text{GeCo2}}\cdot(1-10^{-3})\approx0.3569.$ The expected stable behavior of GeCo2 and the convergence of the iterates can be observed in Figures \ref{Fig:geco2a} and \ref{Fig:geco2b}. 
In order to demonstrate the expected divergence of the iterates when $\Delta t>\Delta t_\text{GeCo2}$ even for starting vectors that lie within a small neighborhood of the steady state solution, we choose $\Delta t=\Delta t_{\text{GeCo2}}\cdot(1+10^{-3})\approx0.3576$ and the initial value \[\widetilde{\b y}^0=\b y^*+10^{-5}\cdot(-2,1,1,-1,1)^T.\]  In Figure \ref{fig:geco2instab},  a small decrease of the error can observed before it increases to an error of approximately $10^{-3}$. Altogether, the numerical experiments reflect the expected behavior independent of $\b y^0$, at least for the selected model problem.
\begin{figure}
	\centering
	\begin{subfigure}[t]{0.46\textwidth}
		\includegraphics[width=\textwidth]{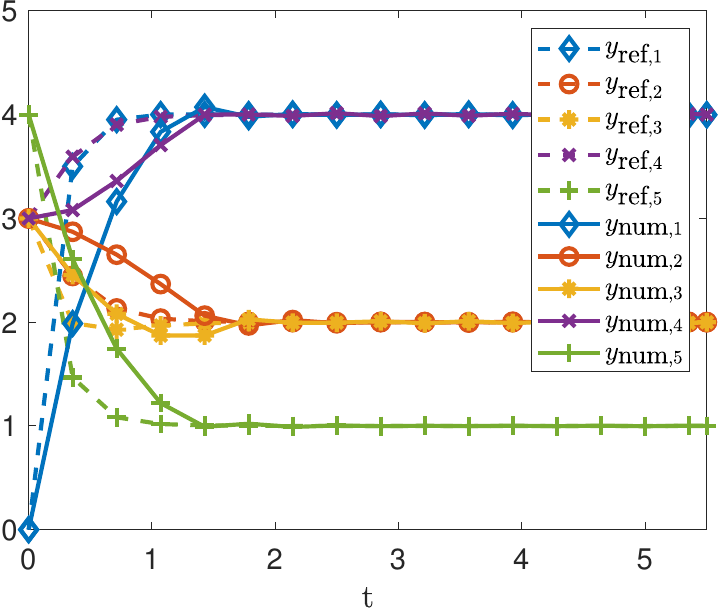}
		\subcaption{$\Delta t=\Delta t_{\text{GeCo2}}\cdot(1-10^{-3})$} \label{Fig:geco2a}
	\end{subfigure}
	\begin{subfigure}[t]{0.49\textwidth}%
		\includegraphics[width=\textwidth]{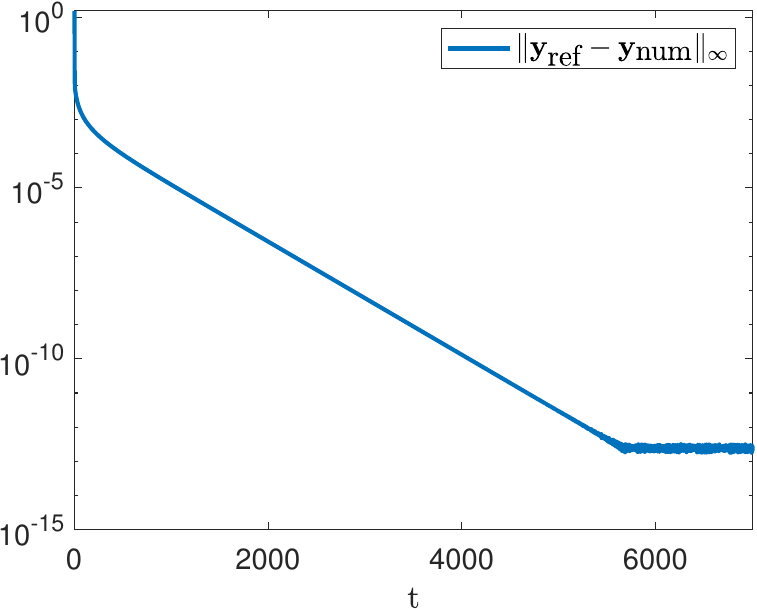}
		\subcaption{$\Delta t=\Delta t_{\text{GeCo2}}\cdot(1-10^{-3})$}\label{Fig:geco2b}
	\end{subfigure}\\
	\begin{subfigure}[t]{0.49\textwidth}
		\includegraphics[width=\textwidth]{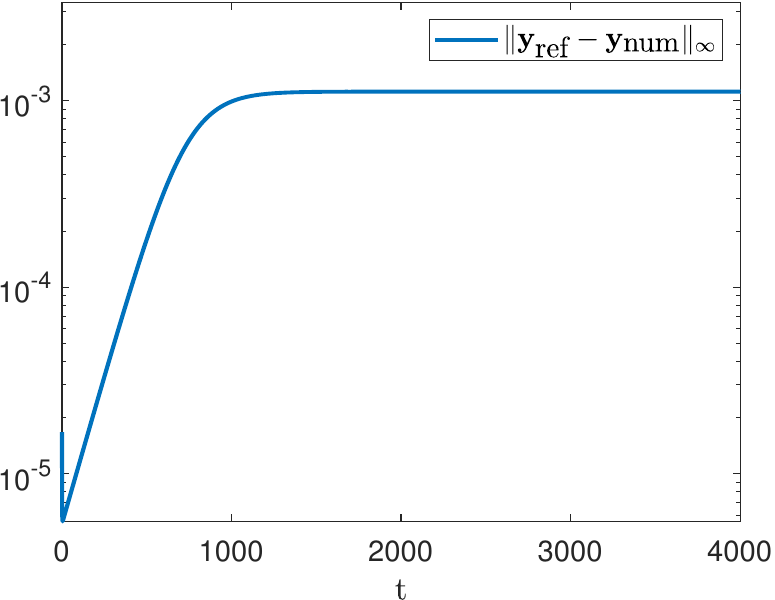}
		\subcaption{$\Delta t=\Delta t_{\text{GeCo2}}\cdot(1+10^{-3})$ }\label{fig:geco2instab}
	\end{subfigure}
	\caption{Numerical approximation of \eqref{eq:testgeco1} and error plots using GeCo2. In \ref{fig:geco2instab} the starting vector $\widetilde{\b y}^0=\b y^*+10^{-5}\cdot(-2,1,1,-1,1)^T$ was used.}
	\label{Fig:geco2}
\end{figure}

\subsection{Investigation of BBKS Schemes}
The stability functions of BBKS1 and BBKS2($1$) in the context of \eqref{PDS_test} are given by Theorem~\ref{Thm:gBBKS1} and Theorem~\ref{Thm:gBBKS2}, respectively. We apply the schemes to the initial value problem \eqref{eq:testgeco1} and test the stability for specific time step sizes. An elementary calculation reveals that the stability functions for both schemes satisfy $\abs{R(\Delta t\lambda)}<1$ for all $\lambda\in \sigma(\bA)\setminus\{0\}$ if $\Delta t < \Delta t_{\text{BBKS}}=\frac{5-\sqrt3}{11},$ and $\abs{R(\Delta t(-5-\sqrt3))}>1$  if $\Delta t>\Delta t_{\text{BBKS}}$. As we did for GeCo2, we investigate the BBKS schemes by varying the time step size around $\Delta t_{\text{BBKS}}$ by multiplying with $1\pm 10^{-3}$, respectively. Furthermore, we also choose $\widetilde{\b y}^0=\b y^*+10^{-5}\cdot(-2,1,1,-1,1)^T$ in the case $\Delta t>\Delta t_{\text{BBKS}}$ in order to highlight the expected divergence of the iterates.

In Figure \ref{Fig:mBBKS1} the numerical solutions of \eqref{eq:testgeco1}  and the error plots using BBKS1 are shown. In \ref{fig:mBBKS1a}, corresponding to the step size \[\Delta t=\Delta t_{\text{BBKS}}\cdot(1-10^{-3})\approx 0.2968,\] all components of the numerical solution tend to the reference solution in the long run, with an error between $10^{-13}$ and $10^{-12}$.  In the unstable case, see Figure \ref{fig:mBBKS1c}, when \[\Delta t =\Delta t_{\text{BBKS}}\cdot(1+10^{-3})\approx 0.2974,\] the error increases almost to $10^{-3}$. 
Similar conclusions can be deduced by looking at Figure~\ref{Fig:mBBKS2}, where the numerical solutions and the error plots of BBKS2($1$) are shown, in correspondence of the same step sizes used for BBKS1.

Altogether, the stability properties shown in Figures~\ref{Fig:mBBKS1} and \ref{Fig:mBBKS2} are in accordance with the stability results expected from the theory presented in Section~\ref{sec:stab_gBBKS}.
\begin{figure}[!h]
	\centering
	\begin{subfigure}[t]{0.46\textwidth}
		\includegraphics[width=\textwidth]{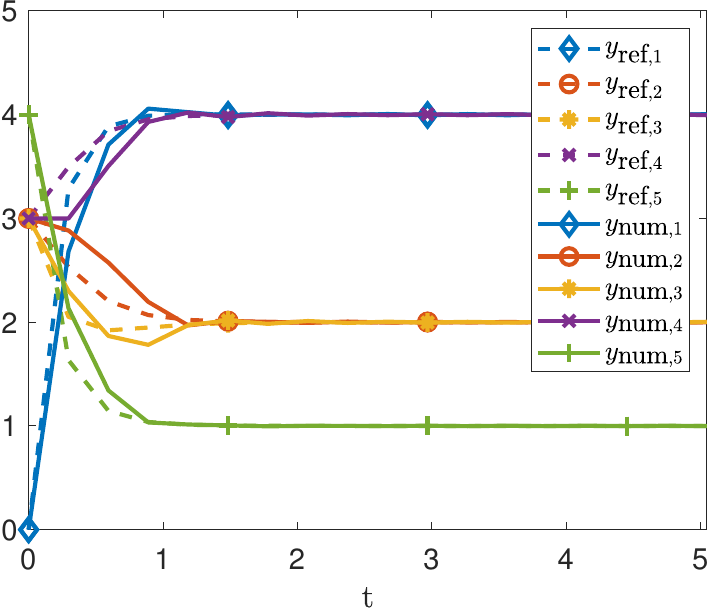}
		\subcaption{$\Delta t=\Delta t_{\text{BBKS}}\cdot(1-10^{-3})$}\label{fig:mBBKS1a}
	\end{subfigure}
	\begin{subfigure}[t]{0.49\textwidth}%
		\includegraphics[width=\textwidth]{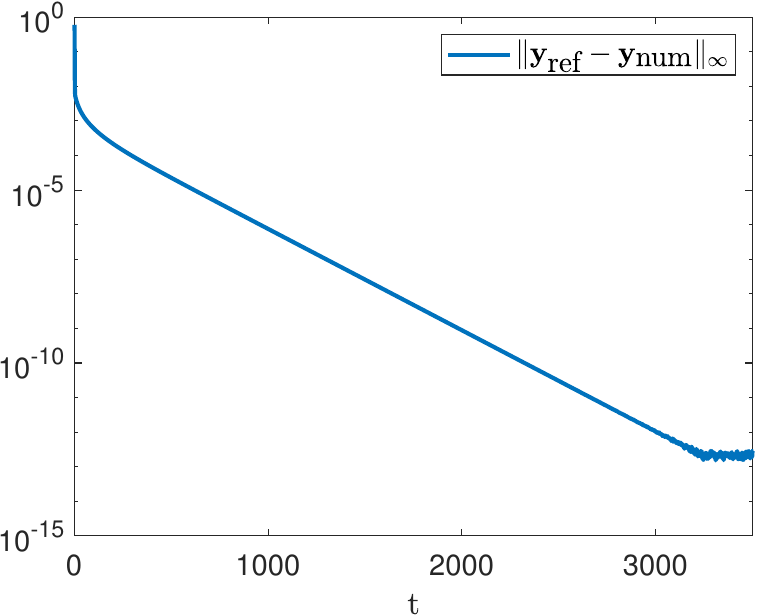}
		\subcaption{$\Delta t=\Delta t_{\text{BBKS}}\cdot(1-10^{-3})$}\label{fig:mBBKS1b}
	\end{subfigure}\\
	\begin{subfigure}[t]{0.49\textwidth}%
		\includegraphics[width=\textwidth]{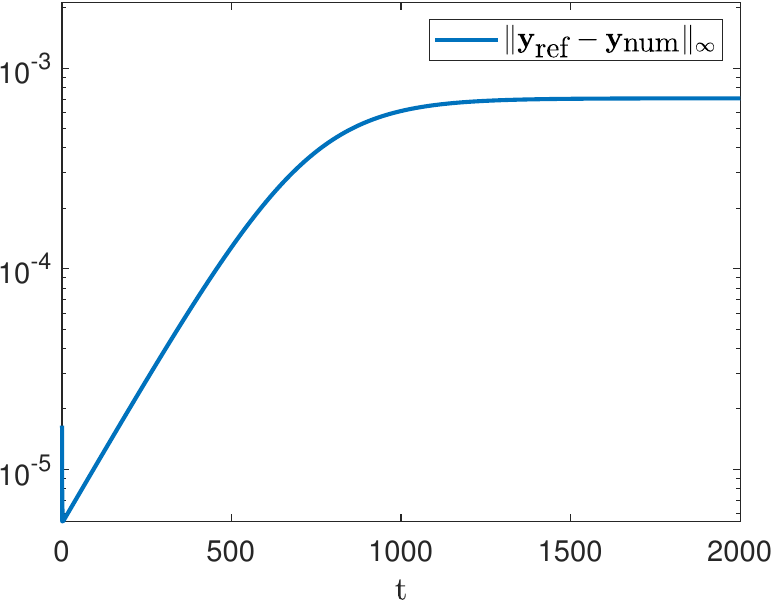}
		\subcaption{$\Delta t=\Delta t_{\text{BBKS}}\cdot(1+10^{-3})$}\label{fig:mBBKS1c}
	\end{subfigure}
	\caption{Numerical approximations of \eqref{eq:testgeco1} and error plots using BBKS1. The starting vector $\widetilde{\b y}^0=\b y^*+10^{-5}\cdot(-2,1,1,-1,1)^T$ was chosen in \ref{fig:mBBKS1c}.}
	\label{Fig:mBBKS1}
\end{figure}
\begin{figure}[!h]
	\centering
	\begin{subfigure}[t]{0.46\textwidth}
		\includegraphics[width=\textwidth]{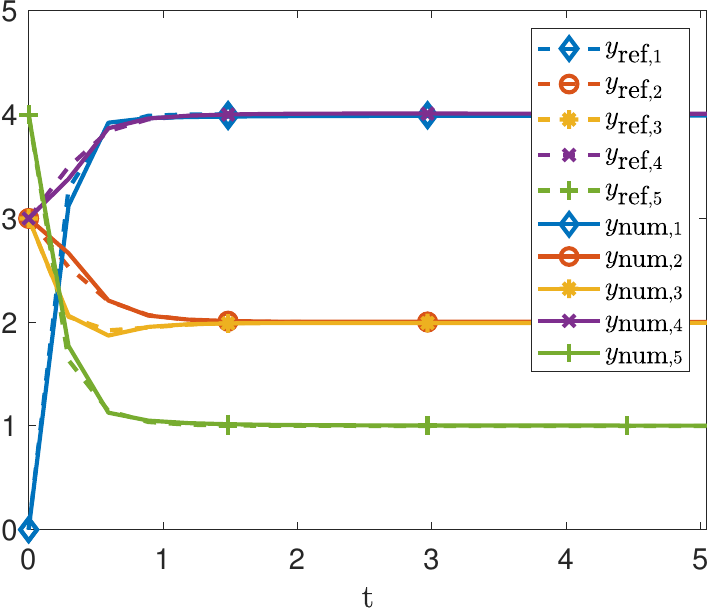}
		\subcaption{$\Delta t=\Delta t_{\text{BBKS}}\cdot(1-10^{-3})$}\label{fig:mBBKS2a}
	\end{subfigure}
	\begin{subfigure}[t]{0.49\textwidth}%
		\includegraphics[width=\textwidth]{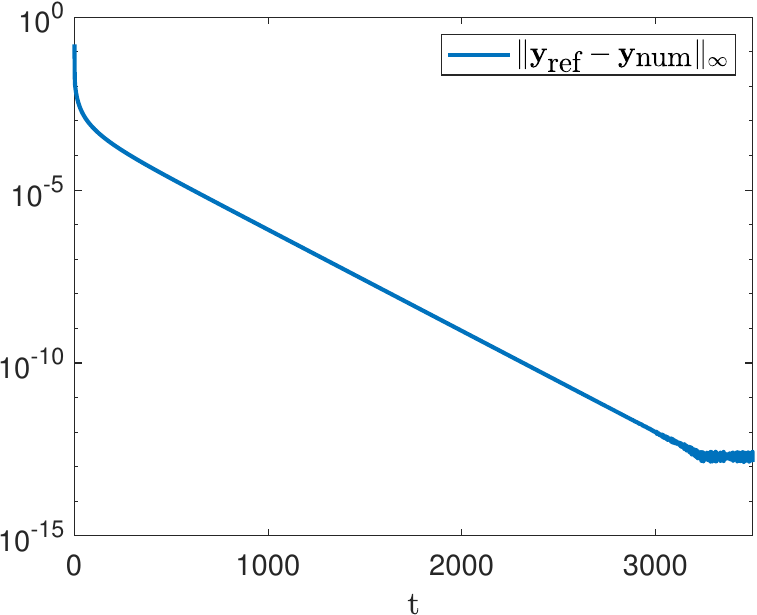}
		\subcaption{$\Delta t=\Delta t_{\text{BBKS}}\cdot(1-10^{-3})$}\label{fig:mBBKS2b}
	\end{subfigure}
	\begin{subfigure}[t]{0.49\textwidth}%
		\includegraphics[width=\textwidth]{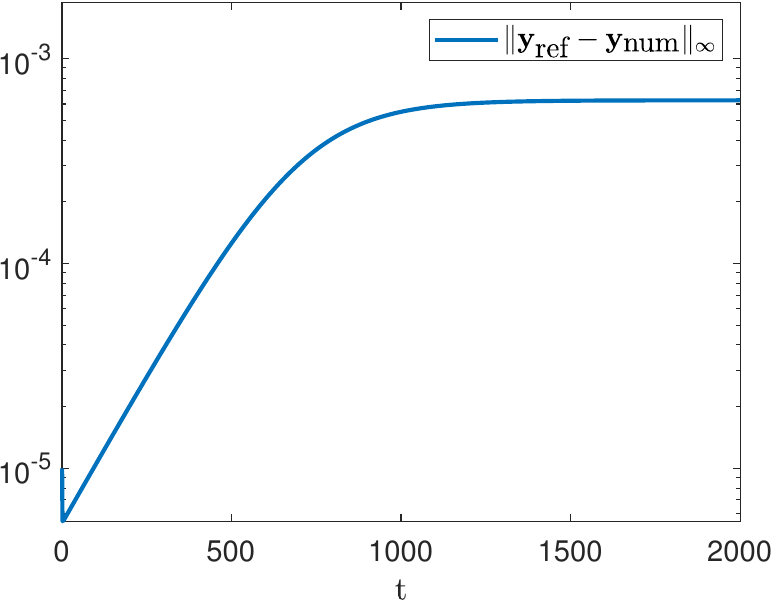}
		\subcaption{$\Delta t=\Delta t_{\text{BBKS}}\cdot(1+10^{-3})$}\label{fig:mBBKS2c}
	\end{subfigure}
	\caption{Numerical approximations of \eqref{eq:testgeco1} and error plots using BBKS2(1). In \ref{fig:mBBKS2c} the starting vector $\widetilde{\b y}^0=\b y^*+10^{-5}\cdot(-2,1,1,-1,1)^T$ was chosen.}
	\label{Fig:mBBKS2}
\end{figure}

\subsection{Applicability of GeCo1 to Stiff Problems}\label{subsec:GeCostiff}
Since the GeCo1 scheme is stable for arbitrary time step sizes, at least locally, this scheme might be able to solve stiff problems.
Unfortunately, this is not true as demonstrated in \cite{gecostab}. In the following we present the investigation from Kopecz performed therein.

To assess the usability for stiff problems, Kopecz \cite{gecostab} proposed to consider the linear initial value problem $\b y'=\bA\b y$, $\b y(0)=\b y^0$ with
\begin{equation}\label{eq:linsys_3x3}
	\bA=\Vec{-K & 0 & 0\\\hphantom{-}K & -1 & 0\\0 & \hphantom{-}1& 0},\quad \b y^0=\Vec{0.98\\ 0.01\\ 0.01}.
\end{equation}
This system becomes increasingly stiff as the value of $K>0$ is increased. For $K\ne 1$ the solution reads
\begin{align*}
	y_1(t)&=\frac{49e^{-K t}}{50},\quad y_2(t)=\frac{(99K-1)e^{-t}}{100(K-1)}-\frac{49Ke^{-K t}}{50(K-1)},\\
	y_3(t)&=1 -\frac{(99K-1)e^{-t}}{100(K-1)} +\frac{49e^{-K t}}{50(K-1)}.
\end{align*}
Defining $\hat{\b y}(t)=\lim_{K\to\infty}\b y(t)$ we find
\[
\hat y_1(t)=0,\quad \hat y_2(t)=\frac{99}{100}e^{-t},\hat  y_3(t)=1-\frac{99}{100}e^{-t}\]
for $t>0$.
In the limit $K\to\infty$, $y_2$ and $y_3$ should therefore be equal at approximately $t=0.7$. In Figure~\ref{fig:linsys_3x3}, we present the plots from \cite{gecostab} of GeCo1 solving \eqref{eq:linsys_3x3} for different values of $K$.
\begin{figure}
	\begin{subfigure}[t]{0.5\textwidth}
		\includegraphics[width=\textwidth]{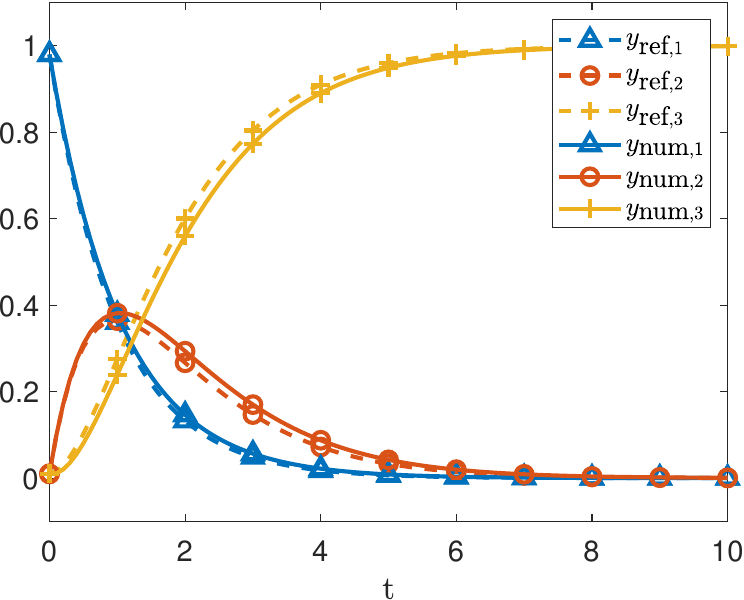}
		\subcaption{$K=1$}\label{subfig:geco1_1_1}
	\end{subfigure}%
	\begin{subfigure}[t]{0.5\textwidth}%
		\includegraphics[width=\textwidth]{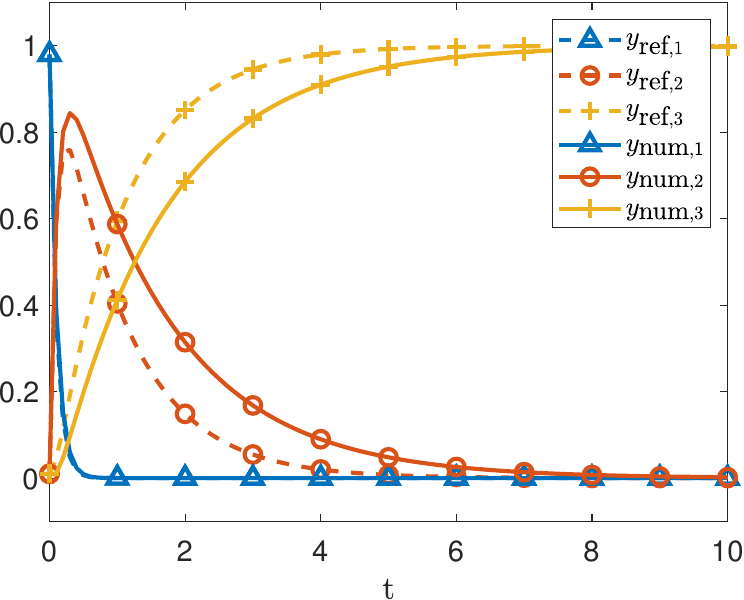}
		\subcaption{$K=10$}\label{subfig:geco1_10_1}
	\end{subfigure}
	
	\begin{subfigure}[t]{0.5\textwidth}
		\includegraphics[width=\textwidth]{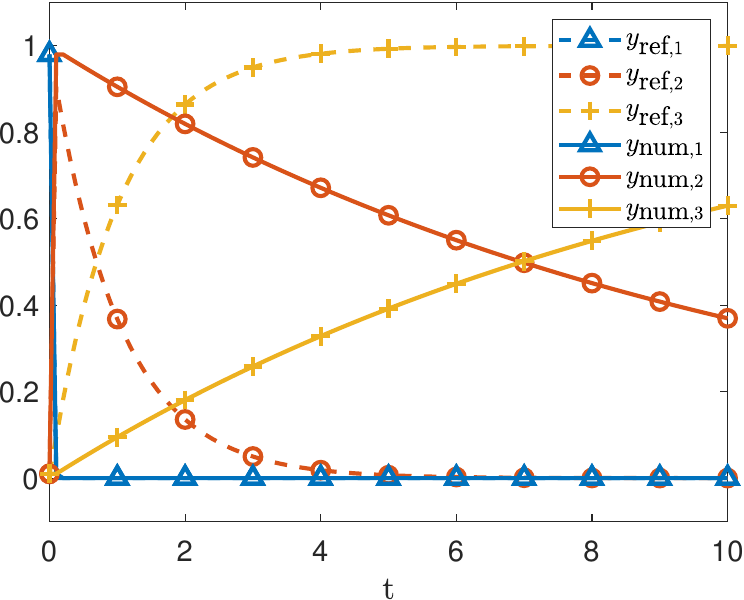}
		\subcaption{$K=100$}\label{subfig:geco1_100_1}
	\end{subfigure}%
	\begin{subfigure}[t]{0.5\textwidth}%
		\includegraphics[width=\textwidth]{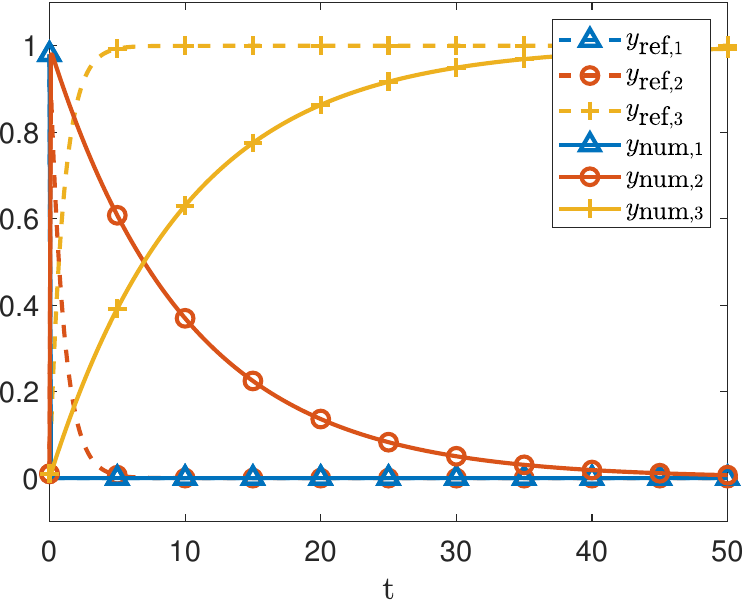}
		\subcaption{$K=100$}\label{subfig:geco1_100_1b}
	\end{subfigure}   
	
	\begin{subfigure}[t]{0.5\textwidth}
		\includegraphics[width=\textwidth]{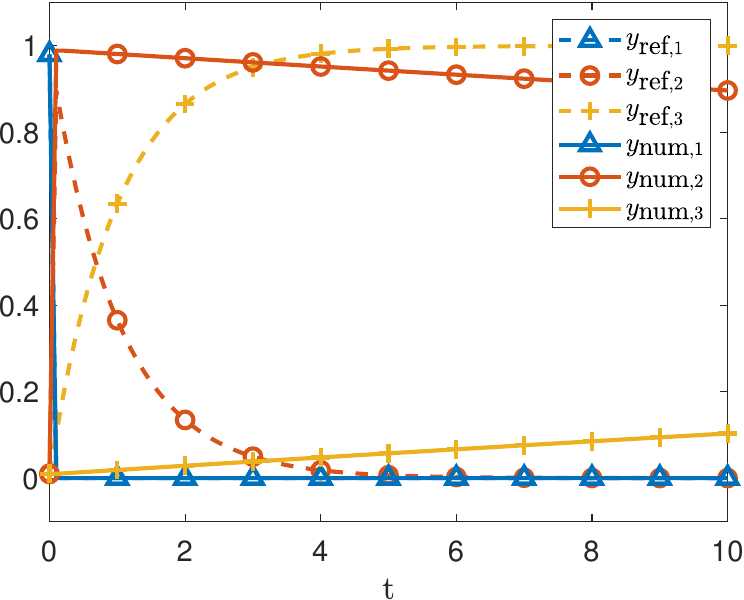}
		\subcaption{$K=1000$}\label{subfig:geco1_1000_1}
	\end{subfigure}%
	\begin{subfigure}[t]{0.5\textwidth}%
		\includegraphics[width=\textwidth]{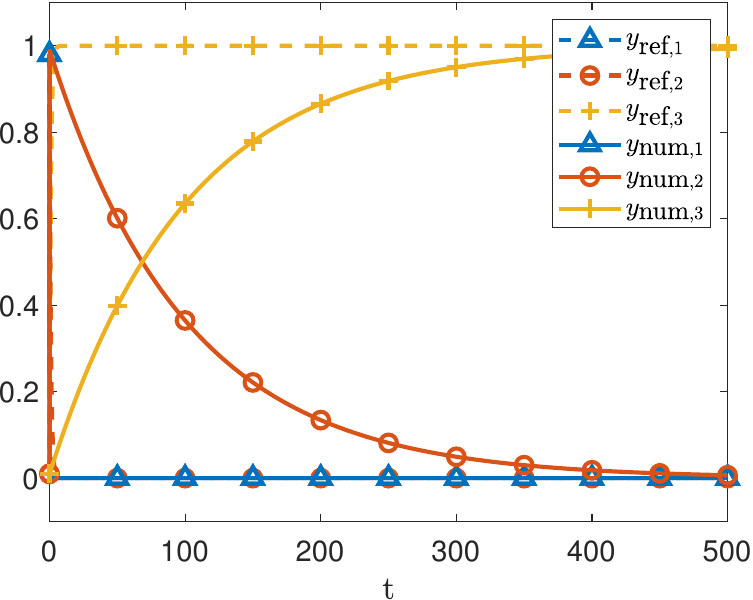}
		\subcaption{$K=1000$}\label{subfig:geco1_1000_1b}
	\end{subfigure}   
	\caption{Numerical solutions of \eqref{eq:linsys_3x3} computed with GeCo1 for different values of $K$ \cite{gecostab}. The step size used is $\Delta t=0.1$. The dashed lines indicate the reference solution.}
	\label{fig:linsys_3x3}
\end{figure}
As already observed in our previous numerical experiments, there is a significant phase error so that $y_2$ and $y_3$ are equal at about $t=7$ for $K=10$ and about $t=70$ for $K=100$, which is far from $t=0.7$. Hence, for increasingly stiff problems, GeCo1 gets less accurate if the time step size is not adapted correspondingly. Altogether, this means that GeCo1 can hardly be regarded as a stiff solver.

\chapter{Conclusion and Outlook}\label{chap:conclusion}
The present work dealt with two major topics concerning the numerical analysis of Runge--Kutta-like methods, namely their stability and order of convergence. 

We motivated and introduced modified Patankar schemes as a subclass of Runge--Kutta-like methods and emphasized their importance. The first major part of this thesis was then dedicated to providing a tool for deriving order conditions for MP methods. The proposed approach may yields implicit order conditions, which can be rewritten in explicit form using the NB-series of the stages \cite{NSARK}. The obtained explicit order conditions can be further reduced using Gröbner bases computations. With the presented approach, it was possible for the first time to obtain conditions for the construction of 3rd and 4th order GeCo as well as 4th order MPRK schemes. Moreover, we constructed a new 4th order MPRK method using our theory and validated the order of convergence numerically. Future work within this topic include the adaptation of this approach for further nonlinear methods such as SSPMPRK schemes and the construction of higher order schemes. In particular, constructing 4th order MPRK methods with a minimal number of stages is of interest. Furthermore, to investigate the order of GeCo and gBBKS methods in the context of non-autonomous problems is to the authors best knowledge still an open task.

The second major part was concerned with the stability of nonlinear time integrators preserving at least one linear invariant. We discussed how the given approach generalizes the notion of $A$-stability. The main difficulty in the analysis comes from the presence of linear invariants, so that any steady state of the corresponding linear system of ODEs resulted in a non-hyperbolic fixed point of the steady state preserving nonlinear method. Even though the investigation of non-hyperbolic fixed points in general is a case by case study, we were able to find an exception for steady states forming a subspace, as is the case for the linear test problem we considered. As a result, we were able to prove that investigating the Jacobian of the generating map is sufficient to understand the stability of the nonlinear method in a neighborhood of the steady state. This approach allowed for the first time the investigation of several modified Patankar schemes such as MPRK, SSPMPRK, MPDeC, GeCo and gBBKS methods which was performed in \cite{izgin2022stability,HIKMS22,IOE22StabMP,gecostab}, also presented and extended within this work. In particular, we tackled the question of unconditional stability for all of the above mentioned methods and summarized our findings in Table~\ref{tab:stabschemes}. In addition to that, we demonstrated that GeCo2 and gBBKS methods are not in $\mathcal C^2$ and proved asymptotic stability for GeCo1 schemes and, for some PDRS, also for MPRK methods. The investigation of MPRK schemes together with the analysis for gBBKS and GeCo2 methods applied to general linear systems represents a future research topic. In the particular case of MPRK schemes, we computed the stability function for arbitrary MPRK schemes in a way that can be easily adapted to the case of PDRS while we pointed out ideas how to generalize our findings for GeCo2 and gBBKS. Finally, our findings support the numerically observed robustness of MP methods while we were able to provide sharp bounds on the time step in the case of conditional stability. Moreover, it might be interesting to apply the presented stability theory in the context of linear multistep methods.

We also connected the approach coming from dynamical systems with that of \cite{IssuesMPRK} concerning oscillatory behavior of nonlinear methods. Here, the zeros of the respective stability function are interlinked with a necessary condition for avoiding oscillatory behavior, which was numerically validated in \cite{ITOE22}. 

Although the proven stability properties are initially local in nature, the work \cite{izgin2023nonglobal} suggests that they can be provably global if the underlying Butcher tableau contains only non-negative entries, while there are schemes with negative Butcher entries for which the stability properties are only local. To further investigate or even prove this claim is of high importance and will be part of my future research.

Also, the implications of this approach for the analysis of numerical methods in the context of partial differential equations (PDEs) is of interest. In particular, generalizing the main stability result, Theorem~\ref{Thm_MPRK_stabil}, to the infinite dimensional case promises interesting applications in the field of numerical analysis of PDEs, as refining the grid in space of the semi-discrete system corresponds to increasingly larger systems of ODEs.

It is also worth mentioning that there are two further tasks arsing naturally as future research topics. 

First, there is not much work available concerning the efficiency of modified Patankar schemes using a time step controller. To the authors knowledge, there is only \cite{KMP21adap}, where standard step size controller were applied to Patankar--Runge--Kutta methods. However, also considering more general controllers from digital signal processing	\cite{soderlind2006time,soderlind2006adaptive,S2002,soderlind2003digital,gustafsson1988pi, gustafsson1991control,G1994,Z1964} might result in even better performances. The exploration of such controllers is one of my future research topics. 

A second aspect related to efficiency is the construction of dense output formulae for MP schemes. The major task here is to provide not only an approximation for any point in time within a given order of accuracy but to force the approximation to be also positive and conservative, or linear invariant preserving in general. Following the idea from \cite{KLJK17}, it seems to be possible to construct second order dense output formulae, \ie for third order MP methods, using our approach of NB-series. However, in the same work the authors find a negative result, \ie using their approach there is no third order dense output formula for MP scheme based on a Butcher tableau with non-negative entries \cite[Theorem~1]{KLJK17}. However, as we wish to use only non-negative Butcher arrays for reasons of stability, we are forced to take a different approach for constructing even third order dense output formulae for MP methods. To construct such a formula together with the above mentioned properties is to my best knowledge still an open problem, yet of high importance. If such formulae are available they might be also useful to construct higher order MPRK methods since the PWDs need to be positive approximations to classical Runge--Kutta stages, that is to the exact solution at intermediate times.

\appendix
\chapter[Intermediate Results on Stability]{Intermediate Results for the Stability Analysis}
In this appendix, we present results with rather technical proofs.

\begin{lem}\label{lem:stabconditionR(rexp(phi))}
	Let $R(z)=\frac{\sum_{j=0}^4n_jz^j}{\sum_{j=0}^4d_jz^j}$ with $d_0=1$. Then $\abs{R(re^{\ii\varphi})}<1$ with $r>0$ and $\varphi\in [0,2\pi)$ is equivalent to 
	\begin{equation*}
		\begin{aligned}
			p_\varphi(r)=&(-d_4^{2}+n_4^{2}) r^{8}+2 (-d_3 d_4+n_3 n_4) \cos(\varphi) r^{7}\\
			&+(4 (-d_2 d_4+n_2 n_4) \cos(\varphi)^{2}+2 d_2 d_4-d_3^{2}-2 n_2 n_4+n_3^{2}) r^{6}
			\\&+(8 (-d_1 d_4+n_1 n_4) \cos(\varphi)^{3}+2 (3 d_1 d_4-d_2 d_3-3 n_1 n_4+n_2 n_3) \cos(\varphi)) r^{5}\\
			&+(16 (n_0 n_4-d_4) \cos(\varphi)^{4}+4 (-d_1 d_3-4 n_0 n_4+n_1 n_3+4 d_4) \cos(\varphi)^{2}\\
			&+2 d_1 d_3-d_2^{2}+2 n_0 n_4-2 n_1 n_3+n_2^{2}-2 d_4) r^{4}\\ 
			&+(8 (n_0 n_3-d_3) \cos(\varphi)^{3}+2 (-d_1 d_2-3 n_0 n_3+n_1 n_2+3 d_3) \cos(\varphi)) r^{3}\\
			&+(4 (n_0 n_2-d_2) \cos(\varphi)^{2}-d_1^{2}-2 n_0 n_2+n_1^{2}+2 d_2) r^{2}\\
			&+2 (n_0 n_1-d_1) \cos(\varphi) r+n_0^{2}-1<0
		\end{aligned}
	\end{equation*}
	Furthermore,  $\abs{R(re^{\ii\varphi})}>1$ is equivalent to $p_\varphi(r)>0$.
\end{lem}
\begin{proof}
	A straightforward calculation rewriting \begin{equation*}
		1>\abs{R(re^{\ii\varphi})}^2=\frac{\left(\sum_{j=0}^4r^jn_j\cos(j\varphi)\right)^2+\left(\sum_{j=0}^4r^jn_j\sin(j\varphi)\right)^2}{\left(\sum_{j=0}^4r^jd_j\cos(j\varphi)\right)^2+\left(\sum_{j=0}^4r^jd_j\sin(j\varphi)\right)^2}
	\end{equation*}
	yields the result.
\end{proof}

The next statement provides us conditions under which the product of a scalar continuous function and a partially differentiable vector field is partially differentiable again, and conditions under which a partial derivative of the product does not exist. \newpage
\begin{lem}\label{Lem:diff}
	Let $D\tm \R^N$ be open and $\b e_i$ denote the $i$th unit vector in $\R^N$.  Furthermore, let  
	$\mathbf \Phi\from D\to\R^N$ be partially differentiable in $\mathbf x_0\in D$ with $\mathbf \Phi(\mathbf x_0)=\mathbf 0$ and let $\Psi\from D\to\R$.
	\begin{enumerate}
		\item\label{it:alemdiff}  If $\Psi$ is continuous in $\mathbf x_0$, then the product $\Psi \cdot\mathbf \Phi\from D\to\R^N$ is partially differentiable in $\mathbf x_0$ with
		\[\mathbf D(\Psi\cdot\mathbf \Phi)(\mathbf x_0)=\Psi(\mathbf x_0)\mathbf D\mathbf \Phi(\mathbf x_0).\]
		\item\label{it:blemdiff} If $\Psi(\b x_0+\b e_ih)$ has several accumulation points as $h\to0$ and $\partial_i\mathbf \Phi(\b x_0)\neq\b 0$, then the $i$th partial derivative of $\Psi\cdot \mathbf \Phi$ does not exists.
	\end{enumerate}
	
\end{lem}
\begin{proof}
	\begin{enumerate}
		\item Since $\mathbf \Phi(\mathbf x_0)=\mathbf 0$ we have
		\begin{equation}\label{eq:proofLemdiff}
			\begin{aligned}&\frac{\Psi(\mathbf x_0 + h\mathbf e_i)\mathbf \Phi(\mathbf x_0 + h\mathbf e_i)-\Psi(\mathbf x_0)\mathbf \Phi(\mathbf x_0)}{h}\\&\hphantom{lllllll}=\frac{\Psi(\mathbf x_0 + h\mathbf e_i)\mathbf \Phi(\mathbf x_0 + h\mathbf e_i)-\Psi(\mathbf x_0+h\mathbf e_i)\mathbf \Phi(\mathbf x_0)}{h}\\
				&\hphantom{lllllll}=\Psi(\mathbf x_0 + h\mathbf e_i)\cdot \frac{\mathbf \Phi(\mathbf x_0 + h\mathbf e_i)-\mathbf \Phi(\mathbf x_0)}{h}.
			\end{aligned}
		\end{equation}
		Passing to the limit $h\to 0$ on both sides shows
		\[\frac{\partial(\Psi\mathbf \Phi)}{\partial x_i}(\mathbf x_0)=\Psi(\mathbf x_0)\frac{\partial \mathbf\Phi}{\partial x_i}(\mathbf x_0),\quad i=1,\dots,N,\]
		and hence
		\[\mathbf D(\Psi\mathbf \Phi)(\mathbf x_0)=\Psi(\mathbf x_0)\mathbf D\mathbf \Phi(\mathbf x_0).\]
		\item If $\Psi(\b x_0+\b e_ih)$ possesses several accumulation points as $h\to0$, then this is also true for 
		\[\Psi(\mathbf x_0 + h\mathbf e_i)\cdot \frac{\mathbf \Phi(\mathbf x_0 + h\mathbf e_i)-\mathbf \Phi(\mathbf x_0)}{h}\]
		as $\frac{\partial \mathbf \Phi}{\partial x_i}(\b x_0)\neq\b 0$. As a result of \eqref{eq:proofLemdiff} we thus obtain that $\frac{\partial(\Psi\cdot \mathbf \Phi)}{\partial x_i}(\b x_0)$ does not exist.
	\end{enumerate}
\end{proof}
The last result of this section is concerned with sufficient conditions for a map $T\colon \R^2_{>0}\to \R$ to be locally Lipschitz continuous even though it is not in $\mathcal C^1$ on its entire domain.
\begin{lem}\label{Lem:locallyLip}
	Let $\bA\in \R^{2\times 2}$ be given by  \eqref{PDS_test} and set $D_1=\{\b x\in \R^2_{>0}\mid x_1>\frac{b}{a}x_2\}$, $D_2=\{\b x\in \R^2_{>0}\mid x_1<\frac{b}{a}x_2\}$ and $C=\ker(\bA)\cap \R^2_{>0}$. Let $T:\R^2_{>0}\to \R$ be continuous with $T\vert_C=\operatorname{const}$ and $T\vert_{D_i}\in \mathcal C^1$ for $i=1,2$. If $\lim_{\b x\to\b c}\nabla T(\b x)$ exists for any $\b c\in C$, then $T$ is locally Lipschitz continuous.
\end{lem}
\begin{proof}
	
	Note that $C=\partial{D_1}= \partial{D_2}$ and that $T$ is locally Lipschitz on $D_1$ and $D_2$ because $T\vert_{D_i}\in \mathcal C^1$ for $i=1,2$. As a first step, we prove that $T$ is also locally Lipschitz on $\overline{D_i}=D_i\cup C$. For this, we consider closed half balls \[H_{\epsilon,i}(\b v)=\overline{B_\epsilon(\b v)}\cap\overline{D_i},\] where $\b v\in C$ and $B_\epsilon(\b v)$ denotes the open ball with center $\b v$ and radius $\epsilon>0$.
	
	As the limit $\lim_{\b x\to\b c}\nabla T(\b x)$ exists for any $\b c\in C$, we can consider the continuous extension of $\nabla T$ to the set $H_{\epsilon,i}(\b v)$, denoted by $\widetilde{\b T}$. Thus, the mean value theorem and the Cauchy--Schwarz inequality yield
	\begin{equation}\label{eq:HLipschitz}
		\lvert T(\b x_1)-T(\b x_2)\rvert\leq \sup_{\b x\in\overline{B_\epsilon(\b v)}\cap D_i}\Vert \nabla T(\b x)\Vert_2 \Vert \b x_1-\b x_2\Vert_2=\max_{\b x\in H_{\epsilon,i}(\b v)}\Vert \widetilde{\b T}(\b x)\Vert_2 \Vert \b x_1-\b x_2\Vert_2
	\end{equation}
	for $\b x_1,\b x_2\in \overline{B_\epsilon(\b v)}\cap D_i$, which means that $T$ is Lipschitz continuous on $\overline{B_\epsilon(\b v)}\cap D_i$ for $i\in\{1,2\}$.
	
	Note that $T\vert_C=\operatorname{const}$ implies that $T$ is Lipschitz continuous on $C$. Hence, to prove the Lipschitz continuity on the closed half ball $H_{\epsilon,i}(\b v)$ it remains to consider the case $\b x_1\in C$ and $\b x_2\in \overline{B_\epsilon(\b v)}\cap D_i$ with $i\in \{1,2\}$. For this, we introduce a sequence $(\b x^n)_{n\in \N}\tm \overline{B_\epsilon(\b v)}\cap D_i$ with $\lim_{n\to\infty}\b x^n=\b x_1$. As $T$ is continuous we therefore find $N_0\in \N$ such that for all $n\geq N_0$ we have 
	\begin{equation}\label{eq:H(x^n)-H(x1)<1/n}
		\lvert T(\b x_1)-T(\b x^n)\rvert<\frac{1}{n}.
	\end{equation}
	Altogether, using $L_i=\max_{\b x\in H_{\epsilon,i}(\b v)}\Vert \widetilde{\b T}(\b x)\Vert_2$ we obtain from \eqref{eq:HLipschitz} and \eqref{eq:H(x^n)-H(x1)<1/n}
	\begin{equation*}
		\lvert T(\b x_1)-T(\b x_2)\rvert\leq \lvert T(\b x_1)-T(\b x^n)\rvert+\lvert T(\b x^n)-T(\b x_2)\rvert< \frac{1}{n}+L_i\Vert \b x^n-\b x_2\Vert_2,
	\end{equation*}
	and passing to the limit, we see that $T$ is even Lipschitz continuous on the closed half ball with a Lipschitz constant $L_i$.
	
	Next, we prove that for any $\b x\in D_1$ and $\b y\in D_2$ there exists a $\b z\in C$ such that 
	\begin{equation}
		\Vert \b x-\b y\Vert_2=\Vert \b x-\b z\Vert_2+\Vert \b z-\b y\Vert_2.\label{eq:triang.equal}
	\end{equation}
	That is to say that  $\b z$ lies on the straight line between $\b x$ and $\b y$. Indeed, setting 
	\[ \b z=\b x+c(\b y-\b x),\quad c=\frac{x_1-\frac{b}{a}x_2}{x_1-\frac{b}{a}x_2+\frac{b}{a}y_2-y_1},\]
	we find $c\in(0,1)$ as $\b x\in D_1$ and $\b y\in D_2$. Additionally, $\b z\in \ker(\bA)$ since
	\begin{align*}
		z_1-\frac{b}{a}z_2&=x_1+c(y_1-x_1)-\frac{b}{a}(x_2+c(y_2-x_2))\\
		&=x_1-\frac{b}{a}x_2-c\left(x_1-y_1+\frac{b}{a}y_2-\frac{b}{a}x_2\right)=0,
	\end{align*}
	and $\b z>\b 0$ since it is on the line between $\b x>\b 0$ and $\b y>\b 0$.
	
	Let us now prove that $T$ is Lipschitz continuous on $\overline{B_\epsilon(\b v)}$. For this, let $\b x\in D_1$ and $\b y\in D_2$, then choose $\b z\in C$ such that \eqref{eq:triang.equal} is satisfied. As a result we obtain
	\begin{equation*}
		\begin{aligned}
			\lvert T(\b x)-T(\b y)\rvert&\leq \lvert T(\b x)-T(\b z)\rvert+\lvert T(\b z)-T(\b y)\rvert\\&\leq\max\{L_1,L_2\} (\Vert \b x-\b z\Vert_2+\Vert \b z-\b y\Vert_2)=\max\{L_1,L_2\}\Vert \b x-\b y\Vert_2,
		\end{aligned}
	\end{equation*}
	and since $\b v\in C$ and $\epsilon>0$ are arbitrary, we have proven hat $T$ is locally Lipschitz continuous.
\end{proof}

\chapter[Intermediate Results on NB-Series]{Intermediate Results for Nonstandard NB-Series}\label{appendix:NB}
In this appendix we present and prove intermediate results that are analogous to statements in \cite{B16}.
We start by recalling Theorem 308A from \cite{B16}, for which we briefly introduce the notation.

Let $m\in \N$ and $I$ be a non-decreasing and finite sequence of integers from the set $\{1,2,\dotsc,m\}$ and $J_m$ the set of all such $I$, where we also include the empty sequence $\varnothing\in J_m$. If $I$ contains $k_j$ occurrences of $j$ for each $j=1,\dotsc, m$ then we define 
\[\hat\sigma(I)=\prod_{j=1}^mk_j!\]
and set $\hat\sigma(\varnothing)=1$.
Now let $\bm{\delta}^{(1)},\dotsc,\bm{\delta}^{(m)}\in \R^d$ and define for $I=(i_1,\dotsc,i_l)\in J_m $ the quantity $\lvert I\rvert=l$ as well as
\[\bm{\delta}^{I}=(\bm{\delta}^{(i_1)},\dotsc,\bm{\delta}^{(i_l)})\in (\R^d)^l,\] and we set $\bm{\delta}^{\varnothing}=\varnothing$ as well as $\lvert \varnothing\rvert=0.$ Next, for a map $\b f\in \mathcal C^{p+1}( \R^d, \R^d)$ we define $\b f^{(0)}(\b y)\varnothing=\b f(\b y)$ and \[\b f^{(l)}(\b y)\bm{\delta}^I=\sum_{j_1,\dotsc,j_l=1}^d\partial_{j_1\dotsc j_l}\b f(\b y)\delta^{(i_1)}_{j_1}\cdots\delta^{(i_l)}_{j_l},\quad 1\leq l\leq p+1,\]
which allows us to formulate \cite[Theorem 308A]{B16}, where we truncate the series using the Lagrangian remainder.
\begin{thm}\label{thm:308a} 
	Let $p\in \N$ and $f\in \mathcal C^{p+1}(\R^d,\R)$ as well as $\b y,\bm{\delta}^{(1)},\dotsc,\bm{\delta}^{(m)}\in \R^d$. Then
	\[ f\left(\b y +\sum_{i=1}^m\bm{\delta}^{(i)}\right)=\sum_{\substack{I\in J_m\\ \lvert I\rvert \leq p}}\frac{1}{\hat\sigma(I)} f^{(\lvert I\rvert)}(\b y)\bm{\delta}^I+  R_p\left(\b y+\sum_{i=1}^m\bm{\delta}^{(i)},\sum_{i=1}^m\bm{\delta}^{(i)}\right),\]
	where, using the multi index notation, we have
	\[R_p\left(\mathbf x,\mathbf a\right)=\sum_{\lvert \bm \alpha \rvert =p+1}\frac{\partial^{\bm \alpha} f(\bm \xi)}{\bm \alpha!}(\mathbf x-\mathbf a)^{\bm\alpha}\]
	with $\bm{\alpha}\in \N_0^d$ and $\xi_j$ between $x_j$ and $a_j$. 
\end{thm}
The key observation is that this equality holds true for any values of $\bm{\delta}^{(i)}$, that is also for solution-dependent vectors $\bm{\delta}^{(i)}=\bm{\delta}^{(i)}(\b y^n,\dt)$.

Our aim is to apply Theorem \ref{thm:308a} to each addend of the right-hand side of the differential equations \eqref{ivp}. Following the idea from \cite[Lemma 310B]{B16}, we prove the following result.
\begin{lem}\label{lem:hFl}
	Let $p\in \N$ and $\Fnu\in \mathcal C^{p+1}$ for $\nu=1\dotsc,N$. Then 
	\begin{align*}
		\dt\Fnu&\left(\b y^n+\sum_{\tau\in NT_{p-1}}\theta(\tau,\b y^n,\dt)\frac{\dt^{\lvert\tau\rvert}}{\sigma(\tau)}\dF(\tau)(\b y^n) +\O(\dt^p)\right)\\
		&=\sum_{\tau\in NT_p}\widetilde\theta_\nu(\tau,\b y^n,\dt)\frac{\dt^{\lvert\tau\rvert}}{\sigma(\tau)}\dF(\tau)(\b y^n)+\O(\dt^{p+1}),
	\end{align*}
	where
	\begin{equation}\label{eq:widetildetheta}
		\widetilde\theta_\nu(\tau,\b y^n,\dt)=\begin{cases}
			\delta_{\nu\mu},& \tau=\rt[]^{[\mu]},\\
			\delta_{\nu\mu}\prod_{i=1}^l\theta(\tau_i,\b y^n,\dt),& \tau=[\tau_1,\dotsc,\tau_l]^{[\mu]}
		\end{cases}
	\end{equation}
	and $\delta_{\nu\mu}$ denotes the Kronecker delta.
\end{lem}
\begin{proof} Let $NT_{p-1}=\{\tau^{(i)}\mid i=1,\dotsc,\lvert NT_{p-1}\rvert\}$.
	We want to apply Theorem \ref{thm:308a} to each component of $\Fnu$ by first writing 
	\begin{equation*}
		\begin{aligned}
			\dt\Fnu&\left(\b y^n+\sum_{\tau\in NT_{p-1}}\theta(\tau,\b y^n,\dt)\frac{\dt^{\lvert\tau\rvert}}{\sigma(\tau)}\dF(\tau)(\b y^n) +\O(\dt^p)\right)\\&=\dt\Fnu\left(\b y^n+\sum_{j=1}^{\lvert NT_{p-1}\rvert}\bm{\delta}^{(j)}+\O(\dt^p)\right)\\
			&=\dt\Fnu\left(\b y^n+\sum_{j=1}^{\lvert NT_{p-1}\rvert}\bm{\delta}^{(j)}\right)+\O(\dt^{p+1})
		\end{aligned}
	\end{equation*}
	with
	\begin{equation}
		\bm{\delta}^{(j)}=\theta(\tau^{(j)},\b y^n,\dt)\frac{\dt^{\lvert \tau^{(j)}\rvert}}{\sigma(\tau^{(j)})}\dF(\tau^{(j)})(\b y^n). \label{eq:delta(i)}
	\end{equation}
	To that end, we first introduce for $I=(i_1,\dotsc,i_r)\in J_m$ with $m=\lvert NT_{p-1}\rvert$ the quantity $\Sigma_I=\sum_{j=1}^r\lvert\tau^{(i_j)}\rvert$ and set $\Sigma_\varnothing=0$. With that, Theorem \ref{thm:308a} and \eqref{eq:delta(i)} yield
	\begin{align*}
		\dt\Fnu&\left(\b y^n+\sum_{\tau\in NT_{p-1}}\theta(\tau,\b y^n,\dt)\frac{\dt^{\lvert\tau\rvert}}{\sigma(\tau)}\dF(\tau)(\b y^n) +\O(\dt^p)\right)\\&=\sum_{\substack{I\in J_m\\ \Sigma_I\leq p-1}}\frac{\dt}{\hat\sigma(I)}(\Fnu)^{(\lvert I\rvert)}(\b y^n)\bm{\delta}^I +\O(\dt^{p+1}).
	\end{align*}
	To prove the claim, we show that 
	\[\sum_{\substack{I\in J_m\\ \Sigma_I\leq p-1}}\frac{\dt}{\hat\sigma(I)}(\Fnu)^{(\lvert I\rvert)}(\b y^n)\bm{\delta}^I =\sum_{\tau\in NT_p}\widetilde\theta_\nu(\tau,\b y^n,\dt)\frac{\dt^{\lvert\tau\rvert}}{\sigma(\tau)}\dF(\tau)(\b y^n)\]
	by means of an induction.
	
	If $p=1$, we find $\frac{\dt}{\hat\sigma(\varnothing)}(\Fnu)^{(0)}(\b y^n)\varnothing=\dt\Fnu(\b y^n)$ and 
	\[\sum_{\mu=1}^N\widetilde\theta_\nu(\rt[]^{[\mu]},\b y^n,\dt)\frac{\dt^{\lvert\tau\rvert}}{\sigma(\rt[]^{[\mu]})}\dF(\rt[]^{[\mu]})(\b y^n)=\sum_{\mu=1}^N\delta_{\nu\mu}\dt\Fmu(\b y^n)=\dt\Fnu(\b y^n),\]
	so that
	\begin{equation*}
		\begin{aligned}
			\sum_{\substack{I\in J_m\\ \Sigma_I\leq 0}}\frac{\dt}{\hat\sigma(I)}(\Fnu)^{(\lvert I\rvert)}(\b y^n)\bm{\delta}^I&=\frac{\dt}{\hat\sigma(\varnothing)}(\Fnu)^{(0)}(\b y^n)\varnothing =\dt\Fnu(\b y^n)\\&=\sum_{\mu=1}^N\widetilde\theta_\nu(\rt[]^{[\mu]},\b y^n,\dt)\frac{\dt^{\lvert\tau\rvert}}{\sigma(\rt[]^{[\mu]})}\dF(\rt[]^{[\mu]})(\b y^n)\\& = \sum_{\tau\in NT_1}\widetilde\theta(\tau,\b y^n,\dt)\frac{\dt^{\lvert\tau\rvert}}{\sigma(\tau)}\dF(\tau)(\b y^n) 
		\end{aligned}
	\end{equation*} is true.
	By induction we can now assume that 
	\[\sum_{\substack{I\in J_m\\ \Sigma_I\leq p-2}}\frac{\dt}{\hat\sigma(I)}(\Fnu)^{(\lvert I\rvert)}(\b y^n)\bm{\delta}^I =\sum_{\tau\in NT_{p-1}}\widetilde\theta_\nu(\tau,\b y^n,\dt)\frac{\dt^{\lvert\tau\rvert}}{\sigma(\tau)}\dF(\tau)(\b y^n)\]
	holds true for some $p\geq 2$, so that it remains to show
	\begin{equation}\label{eq:induction}
		\sum_{\substack{I\in J_m\\ \Sigma_I=p-1}}\frac{\dt}{\hat\sigma(I)}(\Fnu)^{(\lvert I\rvert)}(\b y^n)\bm{\delta}^I =\sum_{\tau\in NT_p\setminus NT_{p-1}}\widetilde\theta_\nu(\tau,\b y^n,\dt)\frac{\dt^{\lvert\tau\rvert}}{\sigma(\tau)}\dF(\tau)(\b y^n)
	\end{equation}
	to finish the proof by induction. For this, let us consider an arbitrary element $\tau\in NT_p\setminus NT_{p-1}$, which can be written as
	\begin{equation*}
		\tau=[\tau_1,\dotsc,\tau_l]^{[\mu]}=[(\tau^{(i_1)})^{m_1},\dotsc,(\tau^{(i_j)})^{m_j}]^{[\mu]}, \quad \sum_{r=1}^j\lvert\tau^{(i_r)}\rvert m_r=p-1\footnote{The root node is not counted here.},
	\end{equation*}
	where we point out that $\tau^{(i_r)}\in NT_{p-1}$ for $r=1,\dotsc,j$. Without loss of generality, we can assume that $i_1<i_2<\dotsc <i_j$. 
	For each such $\tau$ we can define the uniquely determined and non-decreasing sequence \[ \hat I=(\underbrace{i_1,\dotsc,i_1}_{m_1 \text{times}},\underbrace{i_2,\dotsc,i_2}_{m_2 \text{times}},\dotsc,\underbrace{i_j,\dotsc,i_j}_{m_j \text{times}})\]
	satisfying $\hat I\in J_m$ and $\Sigma_{\hat I}=\sum_{r=1}^l\lvert \tau_r \rvert=\sum_{r=1}^j\lvert\tau^{(i_r)}\rvert m_r=p-1$, so that equation \eqref{eq:induction} follows by proving
	\begin{align*}
		\frac{\dt}{\hat\sigma(\hat I)}(\Fnu)^{(\lvert \hat I\rvert)}(\b y^n)\bm{\delta}^{\hat I} =\sum_{\mu=1}^N\frac{\dt^p\widetilde\theta_\nu([\tau_1,\dotsc,\tau_l]^{[\mu]},\b y^n,\dt)}{\sigma([(\tau^{(i_1)})^{m_1},\dotsc,(\tau^{(i_j)})^{m_j}]^{[\mu]})}\dF([\tau_1,\dotsc,\tau_l]^{[\mu]})(\b y^n),
	\end{align*}
	since then any addend on the left-hand side of \eqref{eq:induction} is uniquely associated with the sum over the different root colors of a tree $\tau\in NT_p\setminus NT_{p-1}$. 
	Using \eqref{eq:delta(i)} and the definitions of $\sigma$, $\dF$ and $\widetilde\theta_\nu$ from \eqref{eq:sigmagamma}, \eqref{eq:elemdiff} and \eqref{eq:widetildetheta}, we indeed find
	\begin{equation*}
		\begin{aligned}
			&\frac{\dt}{\hat\sigma(\hat I)}(\Fnu)^{(\lvert \hat I\rvert)}(\b y^n)\bm{\delta}^{\hat I}\\&=\frac{\dt\prod_{r=1}^j\left(\frac{(\theta(\tau^{(i_r)},\b y^n,\dt))^{m_r} \dt^{m_r\lvert \tau^{(i_r)}\rvert}}{(\sigma(\tau^{(i_r)}))^{m_r}}\right)}{\prod_{r=1}^jm_r!}  (\Fnu)^{(l)}(\b y^n)(\dF(\tau_1)(\b y^n),\dotsc,\dF(\tau_l)(\b y^n))  \\
			&= \frac{\dt^p\widetilde \theta_\nu([(\tau^{(i_1)})^{m_1},\dotsc,(\tau^{(i_j)})^{m_j}]^{[\nu]},\b y^n,\dt)}{\sigma([(\tau^{(i_1)})^{m_1},\dotsc,(\tau^{(i_j)})^{m_j}]^{[\nu]})}\dF([\tau_1,\dotsc,\tau_l]^{[\nu]})(\b y^n)\\
			&=\sum_{\mu=1}^N\frac{\dt^p\widetilde\theta_\nu([\tau_1,\dotsc,\tau_l]^{[\mu]},\b y^n,\dt)}{\sigma([(\tau^{(i_1)})^{m_1},\dotsc,(\tau^{(i_j)})^{m_j}]^{[\mu]})}\dF([\tau_1,\dotsc,\tau_l]^{[\mu]})(\b y^n)
		\end{aligned}
	\end{equation*}
	finishing the proof.
\end{proof}
With Lemma \ref{lem:hFl} we can prove the following result, which is the analogue to Lemma 313A in \cite{B16}. 
\begin{lem}\label{lem:hfnu}
	Define $d_i$ and $g_i^{[\nu]}$ for $i=1,\dotsc,s$ and $\nu=1,\dotsc, N$ as in \eqref{eq:pertcond}. Furthermore, let $p\in \N$ and $ \Fnu\in \mathcal C^{p+1}$ for $\nu=1,\dotsc, N$.  If
	\[\byi=\b y^n+\sum_{\tau\in NT_{p-1}}\frac{\dt^{\lvert \tau\rvert}}{\sigma(\tau)}d_i(\tau,\b y^n,\dt) \dF(\tau)(\b y^n)+\O(\dt^p)\]
	then
	\[\dt\Fnu(\byi)=\sum_{\tau\in NT_{p}}\frac{\dt^{\lvert \tau\rvert}}{\sigma(\tau)}g^{[\nu]}_i(\tau,\b y^n,\dt) \dF(\tau)(\b y^n)+\O(\dt^{p+1}).\]
\end{lem}
\begin{proof}
	The claim follows using Lemma \ref{lem:hFl} with $\theta(\tau,\b y^n,\dt)=d_i(\tau,\b y^n,\dt)$, which gives us
	\begin{equation*}
		\begin{aligned}
			\widetilde\theta_\nu (\tau,\b y^n,\dt)&=\begin{cases}
				\delta_{\nu\mu},& \tau=\rt[]^{[\mu]},\\
				\delta_{\nu\mu}\prod_{i=1}^ld_i(\tau_i,\b y^n,\dt),& \tau=[\tau_1,\dotsc,\tau_l]^{[\mu]}
			\end{cases} \\
			&=g_i^{[\nu]}(\tau,\b y^n,\dt).
		\end{aligned}
	\end{equation*}
	\vspace{-\baselineskip}
\end{proof}
\section{Results for Reducing Order Conditions of NSARK methods}
As final intermediate results, we prove the following lemmas which are helpful to reduce the conditions for 3rd and 4th order MPRK and GeCo methods. Both families of schemes can be written in the form of an NSARK method with 
\begin{equation*}
	a_{ij}^{[\nu]}=a_{ij}\gamma_\nu^{(i)},\quad b_i^{[\mu]}=b_i\delta_\mu
\end{equation*}
for suitable solution-dependent functions $\delta_\mu$ and $\gamma_\nu^{(i)}$, which we previously referred to as NSWs. In the following we use these general functions to reduce the order conditions \eqref{eq:condp=3} and \eqref{eq:condp=4} for 3rd and 4th order, respectively.  As we assume for Theorem~\ref{thm:main} that $a_{ij}^{[\nu]}=\O(1)$ as $\dt\to 0$, it suffices to prove the following results.
\begin{lem}\label{lem:equivalent}
	Let $\b A, \b b, \b c$ be the coefficients of an explicit 3-stage RK scheme of order 3, and let $\gamma_\nu^{(i)}=\O(1)$ as $\dt\to 0$. Then the conditions
	\begin{subequations}\label{eq:MPRKcondp=3}
		\begin{align}
			\delta_\mu&=1 +\O(\dt^3), &\mu&=1,\dotsc,N,\\
			\sum_{i=2}^3 b_ic_i\gamma^{(i)}_\nu&=\frac12 +\O(\dt^2), &\nu&=1,\dotsc,N,\label{eq:cond3b}\\
			\sum_{i=2}^3 b_ic_i^2\gamma^{(i)}_\nu\gamma^{(i)}_\xi&=\frac13 +\O(\dt), &\nu,\xi&=1,\dotsc,N,\label{eq:cond3c}\\
			\sum_{i,j=2}^3 b_i a_{ij}c_j\gamma^{(i)}_\nu\gamma^{(j)}_\xi&=\frac16 +\O(\dt), &\nu,\xi&=1,\dotsc,N\label{eq:cond3d}
		\end{align}
	\end{subequations}
	and
	\begin{equation}\label{eq:MPRKcondp'=3}
		\begin{aligned}
			\delta_\mu&=1 +\O(\dt^3), &\mu&=1,\dotsc,N,\\
			\sum_{i=2}^3 b_ic_i\gamma^{(i)}_\nu&=\frac12 +\O(\dt^2), &\nu&=1,\dotsc,N,\\
			\gamma^{(i)}_\nu&=1+\O(\dt),&\nu&=1,\dotsc,N,\quad i=2,3
		\end{aligned}
	\end{equation} 
	are equivalent for any solution and step-size dependent values of $\delta_\mu$ and $\gamma^{(i)}_\nu$ for $i=2,3$ and $\mu,\nu=1,\dotsc,N$.
\end{lem}
\begin{proof}
	It is easy to see that the conditions \eqref{eq:MPRKcondp=3} are fulfilled by any solution of \eqref{eq:MPRKcondp'=3}.
	To see that any solution of \eqref{eq:MPRKcondp=3} must satisfy \eqref{eq:MPRKcondp'=3}, consider the conditions from \eqref{eq:MPRKcondp=3} as $\dt\to 0$. From $a_{ij}^{[\nu]}=\O(1)$, any accumulation point of $\gamma^{(i)}_\nu$ is neither $\infty$ nor $-\infty$. In the following, we denote by $\Gamma^{(i)}_\nu$ an arbitrary accumulation point of  $\gamma^{(i)}_\nu$ as $\dt\to 0$. Moreover, since the underlying RK scheme is explicit with three stages, the only addend remaining on the left-hand side of \eqref{eq:cond3d} is 
	$b_3a_{32}c_2\gamma^{(3)}_\nu\gamma^{(2)}_\xi=\frac16\gamma^{(3)}_\nu\gamma^{(2)}_\xi.$
	Hence, for any accumulation point $\Gamma^{(i)}_\nu$, the conditions \eqref{eq:cond3b}, \eqref{eq:cond3c} with $\nu=\xi$, and \eqref{eq:cond3d} together with $c_1=0$ imply
	\begin{equation*}
		\begin{aligned}
			\sum_{i=2}^3 b_ic_i\Gamma^{(i)}_\nu&=\frac12, &\nu&=1,\dotsc,N,\\
			\sum_{i=2}^3 b_ic_i^2(\Gamma^{(i)}_\nu)^2&=\frac13, &\nu&=1,\dotsc,N,\\
			\Gamma^{(3)}_\nu\Gamma^{(2)}_\xi&=1, &\nu,\xi&=1,\dotsc,N.
		\end{aligned}
	\end{equation*}
	This system of equations possesses for any pair $(\nu,\xi)$ the unique solution $\Gamma^{(2)}_\xi=1$ and $\Gamma^{(3)}_\nu=1$ for all $\xi,\nu=1,\dotsc,N$, see \cite[Lemma 7]{KM18Order3}. Finally, \eqref{eq:cond3c} with $\nu=\xi$  thus implies that $\gamma^{(i)}_\nu=1+\O(\dt)$ proving that \eqref{eq:MPRKcondp=3} and \eqref{eq:MPRKcondp'=3} are equivalent.
\end{proof}
To come up with an analogue for $4$-stage RK methods of $4$th order, we can follow the same ideas as in the last proof, however, this time we need to come up with a substitute for \cite[Lemma 7]{KM18Order3}. The precise procedure is based on Gröbner bases computations as we will see in the proof of the following lemma.

\begin{lem}\label{lem:equivalent4}
	Let $\b A, \b b, \b c$ be the coefficients of an explicit 4-stage RK scheme of order 4, and let $\gamma_\nu^{(i)}=\O(1)$ as $\dt\to 0$. Then the conditions
	\begin{subequations}\label{eq:MPRKcondp=4}
		\begin{align}
			\delta_\mu&=1 +\O(\dt^4), &\mu&=1,\dotsc,N,\\
			\sum_{i=2}^4 b_ic_i\gamma^{(i)}_\nu&=\frac12 +\O(\dt^3), &\nu&=1,\dotsc,N,\label{eq:cond4b}\\
			\sum_{i=2}^4 b_ic_i^2\gamma^{(i)}_\nu\gamma^{(i)}_\xi&=\frac13 +\O(\dt^2), &\nu,\xi&=1,\dotsc,N,\label{eq:cond4c}\\
			\sum_{i,j=2}^4 b_i a_{ij}c_j\gamma^{(i)}_\nu\gamma^{(j)}_\xi&=\frac16 +\O(\dt^2), &\nu,\xi&=1,\dotsc,N,\label{eq:cond4d}\\
			\sum_{i,j=2}^4 b_ic_i a_{ij}c_j\gamma^{(i)}_\nu\gamma^{(i)}_\xi\gamma^{(j)}_\eta&=\frac18 +\O(\dt), &\nu,\xi,\eta&=1,\dotsc,N,\label{eq:cond4e}\\
			\sum_{i=2}^4 b_ic_i^3 \gamma^{(i)}_\nu\gamma^{(i)}_\xi\gamma^{(i)}_\eta&=\frac14 +\O(\dt), &\nu,\xi,\eta&=1,\dotsc,N,\label{eq:cond4f}\\
			\sum_{i,j,k=2}^4 b_ia_{ij}a_{jk}c_k\gamma^{(i)}_\nu\gamma^{(j)}_\xi\gamma^{(k)}_\eta&=\frac{1}{4!}+\O(\dt), &\nu,\xi,\eta&=1,\dotsc,N,\label{eq:cond4g}\\
			\sum_{i,j=2}^4 b_i a_{ij}c_j^2\gamma^{(i)}_\nu\gamma^{(j)}_\xi\gamma^{(j)}_\eta&=\frac{1}{12} +\O(\dt), &\nu,\xi,\eta&=1,\dotsc,N\label{eq:cond4h}
		\end{align}
	\end{subequations}
	and
	\begin{equation}\label{eq:MPRKcondp'=4}
		\begin{aligned}
			\delta_\mu&=1 +\O(\dt^4), &\mu&=1,\dotsc,N,\\
			\sum_{i=2}^4 b_ic_i\gamma^{(i)}_\nu&=\frac12 +\O(\dt^3), &\nu&=1,\dotsc,N,\\
			\gamma^{(i)}_\nu&=1+\O(\dt^2),&\nu&=1,\dotsc,N,\quad i=2,3,4
		\end{aligned}
	\end{equation} 
	are equivalent for any solution and step-size dependent values of $\delta_\mu$ and $\gamma^{(i)}_\nu$ for $i=2,3,4$ and $\mu,\nu=1,\dotsc,N$.
\end{lem}
\begin{proof}
	We first note that the conditions \eqref{eq:MPRKcondp=4} are fulfilled by any solution of \eqref{eq:MPRKcondp'=4}. 
	To see that any solution of \eqref{eq:MPRKcondp=4} must satisfy \eqref{eq:MPRKcondp'=4}, consider the conditions from \eqref{eq:MPRKcondp=4} as $\dt\to 0$. From $a_{ij}^{[\nu]}=\O(1)$, any accumulation point of $\gamma^{(i)}_\nu$ is neither $\infty$ nor $-\infty$. In the following, we denote by $\Gamma^{(i)}_\nu$ an arbitrary accumulation point of  $\gamma^{(i)}_\nu$ as $\dt\to 0$. Moreover, since the underlying RK scheme is explicit with four stages, the only addend remaining on the left-hand side of \eqref{eq:cond4g} is 
	\[b_4a_{43}a_{32}c_2\gamma^{(4)}_\nu\gamma^{(3)}_\xi\gamma^{(2)}_\eta=\frac{1}{4!}\gamma^{(4)}_\nu\gamma^{(3)}_\xi\gamma^{(2)}_\eta.\]
	Hence, for any accumulation point $\Gamma^{(i)}_\nu$, the conditions \eqref{eq:MPRKcondp=4} together with the order conditions for the underlying RK method and $\nu=\xi=\eta$ imply 
	{\allowdisplaybreaks
		\begin{align}
			\sum_{i=2}^4 b_ic_i\left(\Gamma^{(i)}_\nu-1\right)&=0, &\nu&=1,\dotsc,N,\nonumber\\
			\sum_{i=2}^4 b_ic_i^2\left((\Gamma^{(i)}_\nu)^2-1\right)&=0, &\nu&=1,\dotsc,N,\nonumber\\
			\sum_{i,j=2}^4 b_i a_{ij}c_j\left(\Gamma^{(i)}_\nu\Gamma^{(j)}_\nu-1\right)&=0, &\nu&=1,\dotsc,N,\nonumber\\
			\sum_{i,j=2}^4 b_ic_i a_{ij}c_j\left((\Gamma^{(i)}_\nu)^2\Gamma^{(j)}_\nu-1\right)&=0, &\nu&=1,\dotsc,N,\label{eq:redPGS_cond4}\\
			\sum_{i=2}^4 b_ic_i^3 \left((\Gamma^{(i)}_\nu)^3-1\right)&=0, &\nu&=1,\dotsc,N,\nonumber\\
			\Gamma^{(4)}_\nu\Gamma^{(3)}_\nu\Gamma^{(2)}_\nu-1&=0, &\nu&=1,\dotsc,N,\nonumber\\
			\sum_{i,j=2}^4 b_i a_{ij}c_j^2\left(\Gamma^{(i)}_\nu(\Gamma^{(j)}_\nu)^2-1\right)&=0, &\nu&=1,\dotsc,N.\nonumber
	\end{align}}
	In what follows we fix $\nu\in \{1,\dotsc,N\}$. Then, we compute a reduced Gröbner basis\footnote{We refer to our Maple repository \cite{GBrepository} for the computation of the Gröbner bases for this work.} of the corresponding polynomial ideal generated by the polynomials on the left-hand sides of \eqref{eq:redPGS_cond4} in the ring $\R[\Gamma^{(2)}_\nu,\Gamma^{(3)}_\nu,\Gamma^{(4)}_\nu]$, yielding $\{\Gamma^{(2)}_\nu-1,\Gamma^{(3)}_\nu-1,\Gamma^{(4)}_\nu-1\}$. Hence, $\Gamma^{(i)}_\nu=1$ for $\nu=1,\dotsc,N$ and $i=2,3,4$ is the unique solution to \eqref{eq:redPGS_cond4}. As a result, \eqref{eq:cond4f} with $\nu=\xi=\eta$ implies that $\gamma^{(i)}_\nu=1+\O(\dt)$. This already allows us to neglect the conditions \eqref{eq:cond4e} to \eqref{eq:cond4h} in the following as they are now fulfilled by  $\gamma^{(i)}_\nu=1+\O(\dt)$. Substituting the ansatz\footnote{Formally, $x^{(i)}_\nu$ is an arbitrary accumulation point of $\tfrac{\gamma^{(i)}_\nu-1}{\dt}$ as $\dt\to 0$, which due to  $\gamma^{(i)}_\nu=1+\O(\dt)$ cannot be $\pm \infty$. However, for the sake of simplicity, we refrain to introduce several $\gamma^{(i)}_\nu$ for every occurring accumulation point.} $\gamma^{(i)}_\nu=1+x^{(i)}_\nu \dt +\O(\dt^2)$ into the remaining conditions \eqref{eq:cond4b} to \eqref{eq:cond4d}, the resulting coefficients of $\dt$ must vanish, that is
	\begin{equation}\label{eq:redPGS_cond4b}
		\begin{aligned}
			\sum_{i=2}^4 b_ic_ix^{(i)}_\nu&=0,\\
			\sum_{i=2}^4 b_ic_i^22x^{(i)}_\nu&=0, \\
			\sum_{i,j=2}^4 b_i a_{ij}c_j(x^{(i)}_\nu+x^{(j)}_\nu)&=0.
		\end{aligned}
	\end{equation}
	We again compute a reduced Gröbner basis of the ideal generated by the left-hand side polynomials from \eqref{eq:redPGS_cond4b} in the polynomial ring $\R[x^{(2)}_\nu,x^{(3)}_\nu,x^{(4)}_\nu]$. The resulting Gröbner basis reads $\{x^{(2)}_\nu,x^{(3)}_\nu,x^{(4)}_\nu\}$ proving that the unique solution to the above polynomial system is given by $x^{(i)}_\nu=0$ for $\nu=1,\dotsc,N$ and $i=2,3,4$. With that we have demonstrated that $\gamma^{(i)}_\nu=1 +\O(\dt^2)$ which finishes the proof.
\end{proof}

\bibliographystyle{alpha}
\addcontentsline{toc}{chapter}{Bibliography}
	\bibliography{cas-refs}

\newcommand{\etalchar}[1]{$^{#1}$}
\begin{thebibliography}{ALM{\"O}T22}

\bibitem[AE08]{AE08}
H.~Amann and J.~Escher.
\newblock {\em Analysis. {II}}.
\newblock Birkh\"{a}user Verlag, Basel, 2008.
\newblock Translated from the 1999 German original by Silvio Levy and Matthew
  Cargo.

\bibitem[AGKM21]{AGKM_Oliver}
A.~I. \'{A}vila, G.~J. Gonz\'{a}lez, S.~Kopecz, and A.~Meister.
\newblock Extension of modified {P}atankar-{R}unge-{K}utta schemes to
  nonautonomous production-destruction systems based on {O}liver's approach.
\newblock {\em J. Comput. Appl. Math.}, 389:Paper No. 113350, 13, 2021.

\bibitem[AKM20]{gBBKS}
A.~I. \'{A}vila, S.~Kopecz, and A.~Meister.
\newblock A comprehensive theory on generalized {BBKS} schemes.
\newblock {\em Appl. Numer. Math.}, 157:19--37, 2020.

\bibitem[ALM{\"O}T22]{abgrall2022relaxation}
R.~Abgrall, {\'E}.~Le~M{\'e}l{\'e}do, P.~{\"O}ffner, and D.~Torlo.
\newblock Relaxation deferred correction methods and their applications to
  residual distribution schemes.
\newblock {\em SMAI J. Comput. Math.}, 8:125--160, 2022.

\bibitem[AMSS97]{Ntrees}
A.~L. Ara\'{u}jo, A.~Murua, and J.~M. Sanz-Serna.
\newblock Symplectic methods based on decompositions.
\newblock {\em SIAM J. Numer. Anal.}, 34(5):1926--1947, 1997.

\bibitem[ARS97]{ARS1997}
U.~M. Ascher, S.~J. Ruuth, and R.~J. Spiteri.
\newblock Implicit-explicit runge-kutta methods for time-dependent partial
  differential equations.
\newblock {\em Applied Numerical Mathematics}, 25(2):151--167, 1997.
\newblock Special Issue on Time Integration.

\bibitem[BBK{\etalchar{+}}06]{BBKMNU2006}
H.~Burchard, K.~Bolding, W.~K\"{u}hn, A.~Meister, T.~Neumann, and L.~Umlauf.
\newblock {Description of a flexible and extendable physical--biogeochemical
  model system for the water column}.
\newblock {\em Journal of Marine Systems}, 61(3--4):180--211, 2006.
\newblock Workshop on Future Directions in Modelling Physical-Biological
  Interactions (WKFDPBI)Workshop on Future Directions in Modelling
  Physical-Biological Interactions (WKFDPBI).

\bibitem[BBKS07]{BBKS2007}
J.~Bruggeman, H.~Burchard, B.~W. Kooi, and B.~Sommeijer.
\newblock A second-order, unconditionally positive, mass-conserving integration
  scheme for biochemical systems.
\newblock {\em Appl. Numer. Math.}, 57(1):36--58, 2007.

\bibitem[BC78]{BC78}
C.~Bolley and M.~Crouzeix.
\newblock Conservation de la positivit\'{e} lors de la discr\'{e}tisation des
  probl\`emes d'\'{e}volution paraboliques.
\newblock {\em RAIRO Anal. Num\'{e}r.}, 12(3):237--245, iv, 1978.

\bibitem[BDM03]{BDM03}
H.~Burchard, E.~Deleersnijder, and A.~Meister.
\newblock A high-order conservative {P}atankar-type discretisation for stiff
  systems of production-destruction equations.
\newblock {\em Appl. Numer. Math.}, 47(1):1--30, 2003.

\bibitem[BDM05]{BDM2005}
H.~Burchard, E.~Deleersnijder, and A.~Meister.
\newblock {Application of modified Patankar schemes to stiff biogeochemical
  models for the water column}.
\newblock {\em Ocean Dynamics}, 55(3):326--337, 2005.

\bibitem[Ber96]{bertolazzi1996positive}
E.~Bertolazzi.
\newblock Positive and conservative schemes for mass action kinetics.
\newblock {\em Comput. Math. Appl.}, 32(6):29--43, 1996.

\bibitem[BF04]{StabMetz}
L.~Benvenuti and L.~Farina.
\newblock Eigenvalue regions for positive systems.
\newblock {\em Systems \& Control Letters}, 51(3-4):325--330, 2004.

\bibitem[BIM21]{blanes2021positivitypreserving}
S.~Blanes, A.~Iserles, and S.~Macnamara.
\newblock Positivity--preserving methods for population models, 2021.

\bibitem[BIM22]{blanes2021positivity}
S.~Blanes, A.~Iserles, and S.~Macnamara.
\newblock Positivity-preserving methods for ordinary differential equations.
\newblock {\em ESAIM Math. Model. Numer. Anal.}, 56(6):1843--1870, 2022.

\bibitem[BMZ07]{BMZ2007}
J.~Benz, A.~Meister, and P.~Andrea Zardo.
\newblock A positive and conservative second order finite volume scheme applied
  to a phosphor cycle in canals with sediment.
\newblock In {\em PAMM: Proceedings in Applied Mathematics and Mechanics},
  volume~7, pages 2040045--2040046. Wiley Online Library, 2007.

\bibitem[BMZ09]{BMZ2009}
J.~Benz, A.~Meister, and P.~A. Zardo.
\newblock {A conservative, positivity preserving scheme for
  advection-diffusion-reaction equations in biochemical applications}.
\newblock In E.~Tadmor, J.-G. Liu, and A.~Tzavaras, editors, {\em {Hyperbolic
  Problems: Theory, Numerics and Applications}}, volume 67.2 of {\em
  {Proceedings of Symposia in Applied Mathematics}}, pages 399--408. {American
  Mathematical Society}, Providence, Rhode Island, 2009.

\bibitem[BRBM08]{BRBM2008}
N.~Broekhuizen, G.~J. Rickard, J.~Bruggeman, and A.~Meister.
\newblock An improved and generalized second order, unconditionally positive,
  mass conserving integration scheme for biochemical systems.
\newblock {\em Appl. Numer. Math.}, 58(3):319--340, 2008.

\bibitem[But16]{B16}
J.~C. Butcher.
\newblock {\em Numerical methods for ordinary differential equations}.
\newblock John Wiley \& Sons, Ltd., Chichester, third edition, 2016.
\newblock With a foreword by J. M. Sanz-Serna.

\bibitem[Car81]{carr1982}
J.~Carr.
\newblock {\em Applications of centre manifold theory}, volume~35 of {\em
  Applied Mathematical Sciences}.
\newblock Springer-Verlag, New York, 1981.

\bibitem[CD16]{colonna2016plasma}
G.~Colonna and A.~D’Angola, editors.
\newblock {\em Plasma Modeling}.
\newblock 2053-2563. IOP Publishing, 2016.

\bibitem[CM{\"O}T22]{CMOT21}
M.~Ciallella, L.~Micalizzi, P.~{\"O}ffner, and D.~Torlo.
\newblock An arbitrary high order and positivity preserving method for the
  shallow water equations.
\newblock {\em Comput. \& Fluids}, 247:Paper No. 105630, 21, 2022.

\bibitem[Coh03]{C03}
P.~M. Cohn.
\newblock {\em Basic algebra}.
\newblock Springer-Verlag London, Ltd., London, 2003.
\newblock Groups, rings and fields.

\bibitem[Cro80]{Crouzeix1980}
M.~Crouzeix.
\newblock Une m\'{e}thode multipas implicite-explicite pour l'approximation des
  \'{e}quations d'\'{e}volution paraboliques.
\newblock {\em Numer. Math.}, 35(3):257--276, 1980.

\bibitem[Cry73]{C73}
C.~W. Cryer.
\newblock A new class of highly-stable methods: $a_0$-stable methods.
\newblock {\em BIT Numerical Mathematics}, 13(2):153--159, 1973.

\bibitem[Dah63]{D63}
G.~G. Dahlquist.
\newblock A special stability problem for linear multistep methods.
\newblock {\em Nordisk Tidskr. Informationsbehandling (BIT)}, 3:27--43, 1963.

\bibitem[DB02]{DB2002}
P.~Deuflhard and F.~Bornemann.
\newblock {\em Scientific computing with ordinary differential equations},
  volume~42 of {\em Texts in Applied Mathematics}.
\newblock Springer-Verlag, New York, 2002.
\newblock Translated from the 1994 German original by Werner C. Rheinboldt.

\bibitem[DGR00]{dutt2000spectral}
A.~Dutt, L.~Greengard, and V.~Rokhlin.
\newblock Spectral deferred correction methods for ordinary differential
  equations.
\newblock {\em BIT}, 40(2):241--266, 2000.

\bibitem[DK06]{DK06}
D.~T. Dimitrov and H.~V. Kojouharov.
\newblock {Positive and elementary stable nonstandard numerical methods with
  applications to predator--prey models}.
\newblock {\em Journal of Computational and Applied Mathematics},
  189(1--2):98--108, 2006.
\newblock Proceedings of The 11th International Congress on Computational and
  Applied MathematicsThe 11th International Congress on Computational and
  Applied Mathematics.

\bibitem[FS11a]{FSpos}
L.~Formaggia and A.~Scotti.
\newblock Positivity and conservation properties of some integration schemes
  for mass action kinetics.
\newblock {\em SIAM J. Numer. Anal.}, 49(3):1267--1288, 2011.

\bibitem[FS11b]{formaggia2011positivity}
L.~Formaggia and A.~Scotti.
\newblock Positivity and conservation properties of some integration schemes
  for mass action kinetics.
\newblock {\em SIAM Journal on Numerical Analysis}, 49(3/4):1267--1288, 2011.

\bibitem[GLS88]{gustafsson1988pi}
K.~Gustafsson, M.~Lundh, and G.~S{\"o}derlind.
\newblock A {PI} stepsize control for the numerical solution of ordinary
  differential equations.
\newblock {\em BIT Numerical Mathematics}, 28(2):270--287, 1988.

\bibitem[Gre17]{Gressel2017}
O.~Gressel.
\newblock Toward realistic simulations of magneto-thermal winds from
  weakly-ionized protoplanetary disks.
\newblock In {\em Journal of Physics: Conference Series}, volume 837, page
  012008. IOP Publishing, 2017.

\bibitem[Gus91]{gustafsson1991control}
K.~Gustafsson.
\newblock Control theoretic techniques for stepsize selection in explicit
  {R}unge-{K}utta methods.
\newblock {\em {ACM Trans. Math. Software}}, 17(4):533--554, 1991.

\bibitem[Gus94]{G1994}
K.~Gustafsson.
\newblock Control-theoretic techniques for stepsize selection in implicit
  {R}unge-{K}utta methods.
\newblock {\em ACM Trans. Math. Software}, 20(4):496--517, 1994.

\bibitem[HB10a]{HenseBeckmann2010}
I.~Hense and A.~Beckmann.
\newblock {The representation of cyanobacteria life cycle processes in aquatic
  ecosystem models}.
\newblock {\em Ecological Modelling}, 221(19):2330--2338, 2010.

\bibitem[HB10b]{HenseBurchard2010}
I.~Hense and H.~Burchard.
\newblock {Modelling cyanobacteria in shallow coastal seas}.
\newblock {\em Ecological Modelling}, 221(2):238--244, 2010.

\bibitem[HIK{\etalchar{+}}22]{repoSSPMPRK}
J.~Huang, T.~Izgin, S.~Kopecz, A.~Meister, and C.-W. Shu.
\newblock {Lyapunov Stability of third order SSPMPRK schemes (code)}.
\newblock \url{https://github.com/IzginThomas/LyapunovSSPMPRK.git}, December
  2022.

\bibitem[HIK{\etalchar{+}}23]{HIKMS22}
J.~Huang, T.~Izgin, S.~Kopecz, A.~Meister, and C.-W. Shu.
\newblock On the stability of strong-stability-preserving modified
  {P}atankar--{R}unge--{K}utta schemes.
\newblock {\em ESAIM Math. Model. Numer. Anal.}, 57(2):1063--1086, 2023.

\bibitem[HNW93]{HNW1993}
E.~Hairer, S.~P. N{\o}rsett, and G.~Wanner.
\newblock {\em Solving ordinary differential equations. {I}}, volume~8 of {\em
  Springer Series in Computational Mathematics}.
\newblock Springer-Verlag, Berlin, second edition, 1993.
\newblock Nonstiff problems.

\bibitem[H{\"O}T21]{han2021dec}
M.~{Han Veiga}, P.~{\"O}ffner, and D.~{Torlo}.
\newblock De{C} and {ADER}: similarities, differences and a unified framework.
\newblock {\em {J. Sci. Comput.}}, 87(1):35, 2021.
\newblock Id/No 2.

\bibitem[HS19]{SSPMPRK2}
J.~Huang and C.-W. Shu.
\newblock Positivity-preserving time discretizations for production-destruction
  equations with applications to non-equilibrium flows.
\newblock {\em J. Sci. Comput.}, 78(3):1811--1839, 2019.

\bibitem[HW74]{HWButcherTrees}
E.~Hairer and G.~Wanner.
\newblock On the {B}utcher group and general multi-value methods.
\newblock {\em Computing (Arch. Elektron. Rechnen)}, 13(1):1--15, 1974.

\bibitem[HW10]{HNWII}
E.~Hairer and G.~Wanner.
\newblock {\em Solving ordinary differential equations. {II}}, volume~14 of
  {\em Springer Series in Computational Mathematics}.
\newblock Springer-Verlag, Berlin, Berlin, 2010.
\newblock Stiff and differential-algebraic problems, Second revised edition,
  paperback.

\bibitem[HZS19]{SSPMPRK3}
J.~Huang, W.~Zhao, and C.-W. Shu.
\newblock A third-order unconditionally positivity-preserving scheme for
  production-destruction equations with applications to non-equilibrium flows.
\newblock {\em J. Sci. Comput.}, 79(2):1015--1056, 2019.

\bibitem[IKM21]{IKM21}
T.~Izgin, S.~Kopecz, and A.~Meister.
\newblock Recent developments in the field of modified
  patankar-runge-kutta-methods.
\newblock {\em PAMM}, 21(1):e202100027, 2021.

\bibitem[IKM22a]{IKM2122}
T.~Izgin, S.~Kopecz, and A.~Meister.
\newblock On {L}yapunov stability of positive and conservative time integrators
  and application to second order modified {P}atankar--{R}unge--{K}utta
  schemes.
\newblock {\em ESAIM Math. Model. Numer. Anal.}, 56(3):1053--1080, 2022.

\bibitem[IKM22b]{izgin2022stability}
T.~Izgin, S.~Kopecz, and A.~Meister.
\newblock On the stability of unconditionally positive and linear invariants
  preserving time integration schemes.
\newblock {\em SIAM J. Numer. Anal.}, 60(6):3029--3051, 2022.

\bibitem[IKM23a]{GBrepository}
T.~Izgin, D.~I. Ketcheson, and A.~Meister.
\newblock {Order conditions for {NSARK} methods (code)}.
\newblock \url{https://github.com/IzginThomas/NSARK}, May 2023.

\bibitem[IKM23b]{NSARK}
T.~Izgin, D.~I. Ketcheson, and A.~Meister.
\newblock Order conditions for {R}unge--{K}utta-like methods with
  solution-dependent coefficients.
\newblock {\em https://arxiv.org/abs/2305.14297}, 2023.

\bibitem[IKM23c]{IKMnonlin22}
T.~Izgin, S.~Kopecz, and A.~Meister.
\newblock A stability analysis of modified {P}atankar–{R}unge–{K}utta
  methods for a nonlinear production–destruction system.
\newblock {\em PAMM}, 22(1):e202200083, 2023.

\bibitem[IKMM23]{gecostab}
T.~Izgin, S.~Kopecz, A.~Martiradonna, and A.~Meister.
\newblock On the dynamics of first and second order geco and gbbks schemes.
\newblock {\em Applied Numerical Mathematics}, 193:43--66, 2023.

\bibitem[IKMS23]{izgin2023nonglobal}
T.~Izgin, S.~Kopecz, A.~Meister, and Amandine Schilling.
\newblock On the non-global linear stability and spurious fixed points of
  {MPRK} schemes with negative {RK} parameters.
\newblock {\em https://arxiv.org/abs/2305.14297}, 2023.

\bibitem[I{\"O}23]{IOE22StabMP}
T.~Izgin and P.~{\"O}ffner.
\newblock A study of the local dynamics of modified {P}atankar {D}e{C} and
  higher order modified {P}atankar--{RK} methods.
\newblock {\em ESAIM Math. Model. Numer. Anal.}, 57(4):2319--2348, 2023.

\bibitem[Ioo79]{iooss1979}
G.~Iooss.
\newblock {\em Bifurcation of maps and applications}, volume~36 of {\em
  North-Holland Mathematics Studies}.
\newblock North-Holland Publishing Co., Amsterdam-New York, 1979.

\bibitem[I{\"O}T22]{ITOE22}
T.~Izgin, P.~{\"O}ffner, and D.~Torlo.
\newblock A necessary condition for non oscillatory and positivity preserving
  time-integration schemes.
\newblock {\em https://arxiv.org/abs/2211.08905}, 2022.

\bibitem[Jac09]{J09}
Z.~Jackiewicz.
\newblock {\em General linear methods for ordinary differential equations}.
\newblock John Wiley \& Sons, Inc., Hoboken, New Jersey, 2009.

\bibitem[KLJK17]{KLJK17}
D.~I. Ketcheson, L.~L\'{o}czi, A.~Jangabylova, and Adil Kusmanov.
\newblock Dense output for strong stability preserving {R}unge-{K}utta methods.
\newblock {\em J. Sci. Comput.}, 71(3):944--958, 2017.

\bibitem[KM10]{KlarMuecket2010}
J.~S. Klar and J.~P. M\"{u}cket.
\newblock {A detailed view of filaments and sheets in the warm-hot
  intergalactic medium}.
\newblock {\em Astronomy {\&} Astrophysics}, 522:A114, 2010.

\bibitem[KM18a]{KM18}
S.~Kopecz and A.~Meister.
\newblock On order conditions for modified {P}atankar-{R}unge-{K}utta schemes.
\newblock {\em Appl. Numer. Math.}, 123:159--179, 2018.

\bibitem[KM18b]{KM18Order3}
S.~Kopecz and A.~Meister.
\newblock Unconditionally positive and conservative third order modified
  {P}atankar-{R}unge-{K}utta discretizations of production-destruction systems.
\newblock {\em BIT}, 58(3):691--728, 2018.

\bibitem[KM19a]{KMpos}
S.~Kopecz and A.~Meister.
\newblock A comparison of numerical methods for conservative and positive
  advection-diffusion-production-destruction systems.
\newblock {\em PAMM}, 19(1):e201900209, 2019.

\bibitem[KM19b]{MPRK3ex}
S.~Kopecz and A.~Meister.
\newblock On the existence of three-stage third-order modified
  {P}atankar-{R}unge-{K}utta schemes.
\newblock {\em Numer. Algorithms}, 81(4):1473--1484, 2019.

\bibitem[KMP21]{KMP21adap}
S.~Kopecz, A.~Meister, and Helmut Podhaisky.
\newblock On adaptive patankar runge–kutta methods.
\newblock {\em PAMM}, 21(1):e202100235, 2021.

\bibitem[Koo00]{kooijman2000dynamic}
S.~A. L.~M. Kooijman.
\newblock {\em Dynamic Energy and Mass Budgets in Biological Systems}.
\newblock Cambridge University Press, 2 edition, 2000.

\bibitem[KV12]{KV2012}
B.~Korte and J.~Vygen.
\newblock {\em Combinatorial optimization}, volume~21 of {\em Algorithms and
  Combinatorics}.
\newblock Springer, Heidelberg, fifth edition, 2012.
\newblock Theory and algorithms.

\bibitem[LD21]{lacitignola2021using}
D.~Lacitignola and F.~Diele.
\newblock {Using awareness to Z-control a SEIR model with overexposure:
  Insights on Covid-19 pandemic}.
\newblock {\em Chaos, Solitons \& Fractals}, 150:111063, 2021.

\bibitem[LS14]{LS14}
L.~H. Loomis and S.~Sternberg.
\newblock {\em Advanced calculus}.
\newblock World Scientific Publishing Co. Pte. Ltd., Hackensack, NJ, 2014.

\bibitem[Lue79]{Luen79}
D.~G. Luenberger.
\newblock {\em {Introduction to Dynamic Systems: Theory, Models, and
  Applications}}.
\newblock Wiley, 1979.

\bibitem[MB10]{MeisterBenz2010}
A.~Meister and J.~Benz.
\newblock {\em {Phosphorus Cycles in Lakes and Rivers: Modeling, Analysis, and
  Simulation}}.
\newblock Springer Berlin Heidelberg, Berlin, Heidelberg, 2010.

\bibitem[MCD20]{martiradonna2020geco}
A.~Martiradonna, G.~Colonna, and F.~Diele.
\newblock {\it {G}e{C}o}: {G}eometric {C}onservative nonstandard schemes for
  biochemical systems.
\newblock {\em Appl. Numer. Math.}, 155:38--57, 2020.

\bibitem[Mic21]{mickens1994nonstandard}
R.~E. Mickens.
\newblock {\em Nonstandard finite difference schemes---methodology and
  applications}.
\newblock World Scientific Publishing Co. Pte. Ltd., Hackensack, NJ, [2021]
  \copyright 2021.
\newblock Expanded second edition of [ 1275372].

\bibitem[MM76]{mccracken1976hopf}
J.~E. Marsden and M.~McCracken.
\newblock {\em The {H}opf bifurcation and its applications}, volume~19 of {\em
  Applied Mathematical Sciences, Vol. 19}.
\newblock Springer-Verlag, New York, 1976.
\newblock With contributions by P. Chernoff, G. Childs, S. Chow, J. R. Dorroh,
  J. Guckenheimer, L. Howard, N. Kopell, O. Lanford, J. Mallet-Paret, G. Oster,
  O. Ruiz, S. Schecter, D. Schmidt and S. Smale.

\bibitem[MO14]{MeisterOrtleb2014}
A.~Meister and S.~Ortleb.
\newblock {On unconditionally positive implicit time integration for the DG
  scheme applied to shallow water flows}.
\newblock {\em International Journal for Numerical Methods in Fluids},
  76(2):69--94, 2014.

\bibitem[NRK21a]{nusslein2021positivity}
S.~N\"{u}sslein, H.~Ranocha, and D.~I. Ketcheson.
\newblock Positivity-preserving adaptive {R}unge-{K}utta methods.
\newblock {\em Commun. Appl. Math. Comput. Sci.}, 16(2):155--179, 2021.

\bibitem[NRK21b]{nuesslein2021positivitypreserving}
S.~N\"{u}sslein, H.~Ranocha, and D.~I. Ketcheson.
\newblock Positivity-preserving adaptive {R}unge-{K}utta methods.
\newblock {\em Commun. Appl. Math. Comput. Sci.}, 16(2):155--179, 2021.

\bibitem[OH17]{OH17}
S.~Ortleb and W.~Hundsdorfer.
\newblock Patankar-type {Runge-Kut}ta schemes for linear {PDE}s.
\newblock In {\em AIP Conference Proceedings}, volume 1863, page 320008. AIP
  Publishing LLC, 2017.

\bibitem[Osi12]{Osipenko2009}
G.~Osipenko.
\newblock Center manifolds.
\newblock In {\em Mathematics of complexity and dynamical systems. {V}ols.
  1--3}, pages 48--62. Springer, New York, 2012.

\bibitem[{\"O}T20]{MPDeC}
P.~{\"O}ffner and D.~Torlo.
\newblock Arbitrary high-order, conservative and positivity preserving
  {P}atankar-type deferred correction schemes.
\newblock {\em Appl. Numer. Math.}, 153:15--34, 2020.

\bibitem[Pat80]{Patankar1980}
S.~V. Patankar.
\newblock {\em Numerical heat transfer and fluid flow}.
\newblock {Series in computational methods in mechanics and thermal sciences}.
  Hemisphere Pub. Corp. New York, Washington, 1980.

\bibitem[San01]{sandu2001positive}
A.~Sandu.
\newblock Positive numerical integration methods for chemical kinetic systems.
\newblock {\em J. Comput. Phys.}, 170(2):589--602, 2001.

\bibitem[San02]{Sandu02}
A.~Sandu.
\newblock Time-stepping methods that favor positivity for atmospheric chemistry
  modeling.
\newblock In {\em Atmospheric modeling ({M}inneapolis, {MN}, 2000)}, volume 130
  of {\em IMA Vol. Math. Appl.}, pages 21--37. Springer, New York, 2002.

\bibitem[Sch23]{Schilling2023}
Amandine Schilling.
\newblock {E}igenschaften {m}odifizierter
  {P}atankar--{R}unge--{K}utta-{V}erfahren {m}it {n}egativen {RK}-{P}arametern,
  2023.
\newblock {U}niversit{\"a}t {K}assel, 2023, master thesis (written in German).

\bibitem[SD17]{SemeniukDastoor2017}
K.~Semeniuk and A.~Dastoor.
\newblock Development of a global ocean mercury model with a methylation cycle:
  outstanding issues.
\newblock {\em Global Biogeochemical Cycles}, pages n/a--n/a, 2017.
\newblock 2016GB005452.

\bibitem[SG15]{SG2015}
A.~Sandu and M.~G\"{u}nther.
\newblock A generalized-structure approach to additive {R}unge-{K}utta methods.
\newblock {\em SIAM J. Numer. Anal.}, 53(1):17--42, 2015.

\bibitem[SH98]{SH98}
A.~Stuart and A.~R. Humphries.
\newblock {\em Dynamical systems and numerical analysis}, volume~2.
\newblock Cambridge University Press, Cambridge, 1998.

\bibitem[Sha86]{shampine1986conservation}
L.~F. Shampine.
\newblock Conservation laws and the numerical solution of {ODE}s.
\newblock {\em Comput. Math. Appl. Part B}, 12(5-6):1287--1296, 1986.

\bibitem[SM03]{SM2003}
E.~Süli and D.~F. Mayers.
\newblock {\em An Introduction to Numerical Analysis}.
\newblock Cambridge University Press, 2003.

\bibitem[SO88]{SO88}
C.-W. Shu and S.~Osher.
\newblock Efficient implementation of essentially non-oscillatory
  shock-capturing schemes.
\newblock {\em Journal of Computational Physics}, 77(2):439--471, 1988.

\bibitem[S{\"o}d02]{S2002}
G.~S{\"o}derlind.
\newblock Automatic control and adaptive time-stepping.
\newblock {\em Numer. Algorithms}, 31(1-4):281--310, 2002.
\newblock Numerical methods for ordinary differential equations (Auckland,
  2001).

\bibitem[S{\"o}d03]{soderlind2003digital}
G.~S{\"o}derlind.
\newblock Digital filters in adaptive time-stepping.
\newblock {\em ACM Transactions on Mathematical Software (TOMS)}, 29(1):1--26,
  2003.

\bibitem[S{\"o}d06]{soderlind2006time}
G.~S{\"o}derlind.
\newblock Time-step selection algorithms: {A}daptivity, control, and signal
  processing.
\newblock {\em Applied Numerical Mathematics}, 56(3-4):488--502, 2006.

\bibitem[SS03]{SS03}
E.~M. Stein and R.~Shakarchi.
\newblock {\em Complex analysis}, volume~2 of {\em Princeton Lectures in
  Analysis}.
\newblock Princeton University Press, Princeton, NJ, 2003.

\bibitem[STKB05]{STKB}
L.~F. Shampine, S.~Thompson, J.~A. Kierzenka, and G.~D. Byrne.
\newblock Non-negative solutions of {ODE}s.
\newblock {\em Appl. Math. Comput.}, 170(1):556--569, 2005.

\bibitem[SVV18]{SVV18}
A.~J. Steyer and E.~S. Van~Vleck.
\newblock {A Lyapunov and Sacker--Sell spectral stability theory for one-step
  methods}.
\newblock {\em BIT Numerical Mathematics}, 58(3):749--781, 2018.

\bibitem[SW06]{soderlind2006adaptive}
G.~S{\"o}derlind and L.~Wang.
\newblock Adaptive time-stepping and computational stability.
\newblock {\em Journal of Computational and Applied Mathematics},
  185(2):225--243, 2006.

\bibitem[TGA96]{TGA96}
E.~H. Twizell, A.~B. Gumel, and M.~A. Arigu.
\newblock Second-order, {$L_0$}-stable methods for the heat equation with
  time-dependent boundary conditions.
\newblock {\em Adv. Comput. Math.}, 6(3-4):333--352 (1997), 1996.
\newblock John Crank 80th birthday special issue.

\bibitem[Tit39]{T39}
E.~C. Titchmarsh.
\newblock {\em The theory of functions}.
\newblock Oxford University Press, Oxford, second edition, 1939.

\bibitem[T{\"O}R22]{IssuesMPRK}
D.~Torlo, P.~{\"O}ffner, and H.~Ranocha.
\newblock Issues with positivity-preserving {P}atankar-type schemes.
\newblock {\em Appl. Numer. Math.}, 182:117--147, 2022.

\bibitem[Var00]{V00}
R.~S. Varga.
\newblock {\em Matrix iterative analysis}, volume~27 of {\em Springer Series in
  Computational Mathematics}.
\newblock Springer-Verlag, Berlin, expanded edition, 2000.

\bibitem[WHK13]{WHK2013}
A.~Warns, I.~Hense, and A.~Kremp.
\newblock Modelling the life cycle of dinoflagellates: a case study with
  {B}iecheleria baltica.
\newblock {\em J. Plankton. Res}, 35(2):379--392, 2013.

\bibitem[WS22]{WS21}
S.~Wei and R.~J. Spiteri.
\newblock Qualitative property preservation of high-order operator splitting
  for the sir model.
\newblock {\em Appl. Numer. Math.}, 172:332--350, 2022.

\bibitem[Zon64]{Z1964}
J.~A. Zonneveld.
\newblock {\em Automatic numerical integration}, volume~8 of {\em Mathematical
  Centre Tracts}.
\newblock Mathematisch Centrum, Amsterdam, 1964.

\end{thebibliography}
\end{document}